\documentclass[11pt]{amsart}
\usepackage{amsmath}
\usepackage{hyperref} 

\usepackage{amsfonts}
\usepackage{amssymb}
\usepackage{amsthm}
\usepackage[arrow, matrix, curve]{xy} 
\usepackage{enumerate}
\usepackage{todonotes}
\usepackage{amssymb,amscd,url,tikz,stmaryrd,mathtools,fixltx2e}
\usepackage{amssymb,amscd,url}

\usetikzlibrary{intersections}
\usetikzlibrary{calc}
\usetikzlibrary{decorations.markings}

\tikzset{%
	from end of path/.style={
		insert path={
			\pgfextra{%
				\expandafter\pgfprocesspathextractpoints%
				\csname tikz@intersect@path@name@#1\endcsname%
				\pgfpointlastonpath%
				\pgfgetlastxy\lastx\lasty
			}
			(\lastx,\lasty)
}}}

\tikzset{
	anticlockwise arc centered at/.style={
		to path={
			let \p1=(\tikztostart), \p2=(\tikztotarget), \p3=(#1),
			\n{anglestart}={atan2(\y1-\y3,\x1-\x3)},
			\n{angletarget}={atan2(\y2-\y3,\x2-\x3)},
			\n{angletarget}={\n{angletarget} < \n{anglestart} ? \n{angletarget}+360 : \n{angletarget}},
			\n{radius}={veclen(\x1-\x3,\y1-\y3)}
			in arc(\n{anglestart}:\n{angletarget}:\n{radius})  -- (\tikztotarget)
		},
	},
	clockwise arc centered at/.style={
		to path={
			let \p1=(\tikztostart), \p2=(\tikztotarget), \p3=(#1),
			\n{anglestart}={atan2(\y1-\y3,\x1-\x3)},
			\n{angletarget}={atan2(\y2-\y3,\x2-\x3)},
			\n{angletarget}={\n{angletarget} > \n{anglestart} ? \n{angletarget} - 360 : \n{angletarget}},
			\n{radius}={veclen(\x1-\x3,\y1-\y3)}
			in arc(\n{anglestart}:\n{angletarget}:\n{radius})  -- (\tikztotarget)
		},
	},
}

\usepackage{color}

\newcommand{\A}{{\mathbb A}}
\newcommand{\Z}{{\mathbb Z}}
\let\oldS\S
\renewcommand{\S}{{\mathbb S}}
\newcommand{\R}{{\mathbb R}}
\newcommand{\C}{{\mathbb C}}
\newcommand{\D}{{\mathbb D}}
\newcommand{\N}{{\mathbb N}}

\def\Res{{\,\rm Res}}
\def\Re{{\rm Re}}
\def\Im{{\rm Im}}
\def\Id{{\rm Id}}
\def\ii{{\rm i}}

\def\sl{\mathfrak{sl}}

\def\tr{{\rm trace}}

\def\SL{{\rm SL}}
\def\SU{{\rm SU}}

\def\traceL{\tau}
\renewcommand{\matrix}[1]{\left(\begin{array}{cc} #1\end{array}\right)}

\newcommand{\wt}[1]{\widetilde{#1}}
\newcommand{\wh}[1]{\widehat{#1}}
\newcommand{\cal}[1]{{\mathcal #1}}
\theoremstyle{plain}
\newtheorem{theorem}{Theorem}
\newtheorem{lemma}{Lemma}
\newtheorem{proposition}[lemma]{Proposition}
\newtheorem{remark}[lemma]{Remark}
\newtheorem{corollary}[lemma]{Corollary}

\newtheorem{definition}[lemma]{Definition}

\newtheorem{example}[lemma]{Example}

\DeclareFontFamily{U}{mathx}{\hyphenchar\font45}
\DeclareFontShape{U}{mathx}{m}{n}{
      <5> <6> <7> <8> <9> <10>
      <10.95> <12> <14.4> <17.28> <20.74> <24.88>
      mathx10
      }{}
\DeclareSymbolFont{mathx}{U}{mathx}{m}{n}
\DeclareFontSubstitution{U}{mathx}{m}{n}
\DeclareMathAccent{\widecheck}{0}{mathx}{"71}
\DeclareMathAccent{\widetilde}{0}{mathx}{"72}
\DeclareMathAccent{\widebar}{0}{mathx}{"73}
\DeclareMathAccent{\widevec}{0}{mathx}{"74}
\DeclareMathAccent{\widehat}{0}{mathx}{"70}
\DeclareMathAccent{\widefrown}{0}{mathx}{"75}
\DeclareMathAccent{\chinesehat}{0}{mathx}{"69}

\let\eps\varepsilon
\usepackage{xfrac}
\newcommand{\half}{{\tfrac{1\!}{2}}}
\def\Li{{\rm Li}}

\def\Res{{\,\rm Res}}
\def\Re{{\rm Re}}
\def\Im{{\rm Im}}
\def\Id{{\rm Id}}
\def\ii{{\rm i}}

\def\sl{\mathfrak{sl}}

\def\tr{{\rm trace}}

\def\SL{{\rm SL}}
\def\SU{{\rm SU}}

\def\Res{{\,\rm Res}}
\def\Re{{\rm Re}}
\def\Im{{\rm Im}}
\def\Id{{\rm Id}}
\def\ii{{\rm i}}

\def\sl{\mathfrak{sl}}

\def\tr{{\rm trace}}

\def\SL{{\rm SL}}
\def\SU{{\rm SU}}

\renewcommand{\matrix}[1]{\left(\begin{array}{cc} #1\end{array}\right)}

\theoremstyle{plain}

\DeclareFontFamily{U}{mathx}{\hyphenchar\font45}
\DeclareFontShape{U}{mathx}{m}{n}{
      <5> <6> <7> <8> <9> <10>
      <10.95> <12> <14.4> <17.28> <20.74> <24.88>
      mathx10
      }{}
\DeclareSymbolFont{mathx}{U}{mathx}{m}{n}
\DeclareFontSubstitution{U}{mathx}{m}{n}
\DeclareMathAccent{\widecheck}{0}{mathx}{"71}
\DeclareMathAccent{\widetilde}{0}{mathx}{"72}
\DeclareMathAccent{\widebar}{0}{mathx}{"73}
\DeclareMathAccent{\widevec}{0}{mathx}{"74}
\DeclareMathAccent{\widehat}{0}{mathx}{"70}
\DeclareMathAccent{\widefrown}{0}{mathx}{"75}
\DeclareMathAccent{\chinesehat}{0}{mathx}{"69}

\def\Res{{\,\rm Res}}
\def\Re{{\rm Re}}
\def\Im{{\rm Im}}
\def\Id{{\rm Id}}
\def\ii{{\rm i}}

\def\sl{\mathfrak{sl}}

\def\tr{{\rm trace}}

\def\SL{{\rm SL}}
\def\SU{{\rm SU}}

\renewcommand{\matrix}[1]{\left(\begin{array}{cc} #1\end{array}\right)}

\usepackage{shuffle}
\usepackage{bbm}

\newcommand{\wc}[1]{\widecheck{#1}}

\newcommand{\cv}[1]{\underline{#1}}

\makeatletter\let\@wraptoccontribs\wraptoccontribs\makeatother

\def\low{\mbox{\scriptsize lower}}
\renewcommand{\matrix}[1]{\left(\begin{array}{cc} #1\end{array}\right)}
\title[Minimal surfaces and alternating MZVs]{Minimal surfaces and alternating multiple zetas}
\author[S. Charlton]{Steven Charlton}
\address{Max Planck Institute for Mathematics, Vivatsgasse 7, Bonn 53111, Germany}
\email{charlton@mpim-bonn.mpg.de 
}
\author[L. Heller]{Lynn Heller}
\address{Beijing Institute of Mathematical Sciences and Applications, Beijing, China}
\email{lynn@bimsa.cn}
\author[S. Heller]{Sebastian Heller}
\address{Beijing Institute of Mathematical Sciences and Applications, Beijing, China}
\email{sheller@bimsa.cn}
\author[M. Traizet]{Martin Traizet}
\address{Institut Denis Poisson, CNRS UMR 7350 
Universit\'e de Tours, France }
\email{martin.traizet@univ-tours.fr }

\begin{document}

\begin{abstract}
In this paper we show for every sufficiently large integer $g$ the existence of a complete family of
closed and embedded constant mean curvature (CMC) surfaces deforming the Lawson surfaces $\xi_{1,g}$ parametrized by their conformal type. When specializing to the minimal case, we discover a pattern resulting in the coefficients of the involved expansions being alternating multiple zeta values (MZVs),  which generalizes the notion of Riemann's zeta values to multiple integer variables. This allows us to extend a new existence proof of the Lawson surfaces $\xi_{1,g}$ to all $g\geq 3$ using complex analytic methods  and to give closed form expressions of their area expansion up to order $7$. For example, the third order coefficient is $\tfrac{9}{4}\zeta(3)$ (the first and second order term were shown to be $\log(2)$ and $0$ respectively in \cite{HHT}). As a corollary, we obtain that the area of  $\xi_{1,g}$ is monotonically increasing in their genus $g$ for all $g\geq 0.$
\end{abstract}
\thanks{\\
LH and SH are supported by the Beijing Natural Science Foundation IS23002 (LH) and IS23003 (SH).\\MT is supported by the French ANR project Min-Max (ANR-19-CE40-0014).\\
SC is grateful to the Max Planck Institute for Mathematics, for support, hospitality and excellent working conditions during the preparation of this work.  During the initial developments on this work at Universit\"at Hamburg, they were supported by DFG Eigene Stelle grant CH 2561/1-1, for Projektnummer 442093436.
}
\maketitle
\begin{figure}[h]
\vspace{-0.8cm}
\centering
\includegraphics[width=0.18\textwidth]{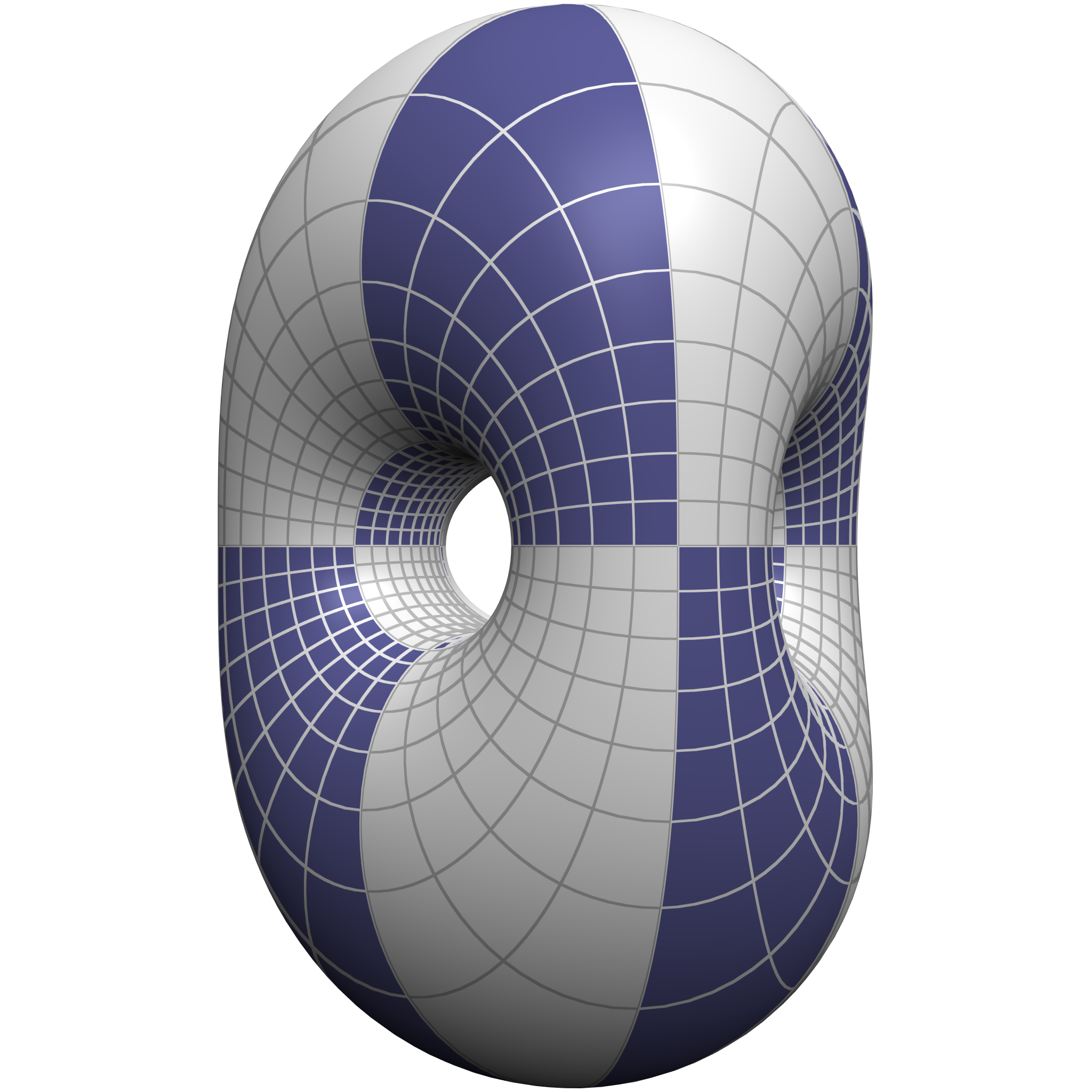}\includegraphics[width=0.18\textwidth]{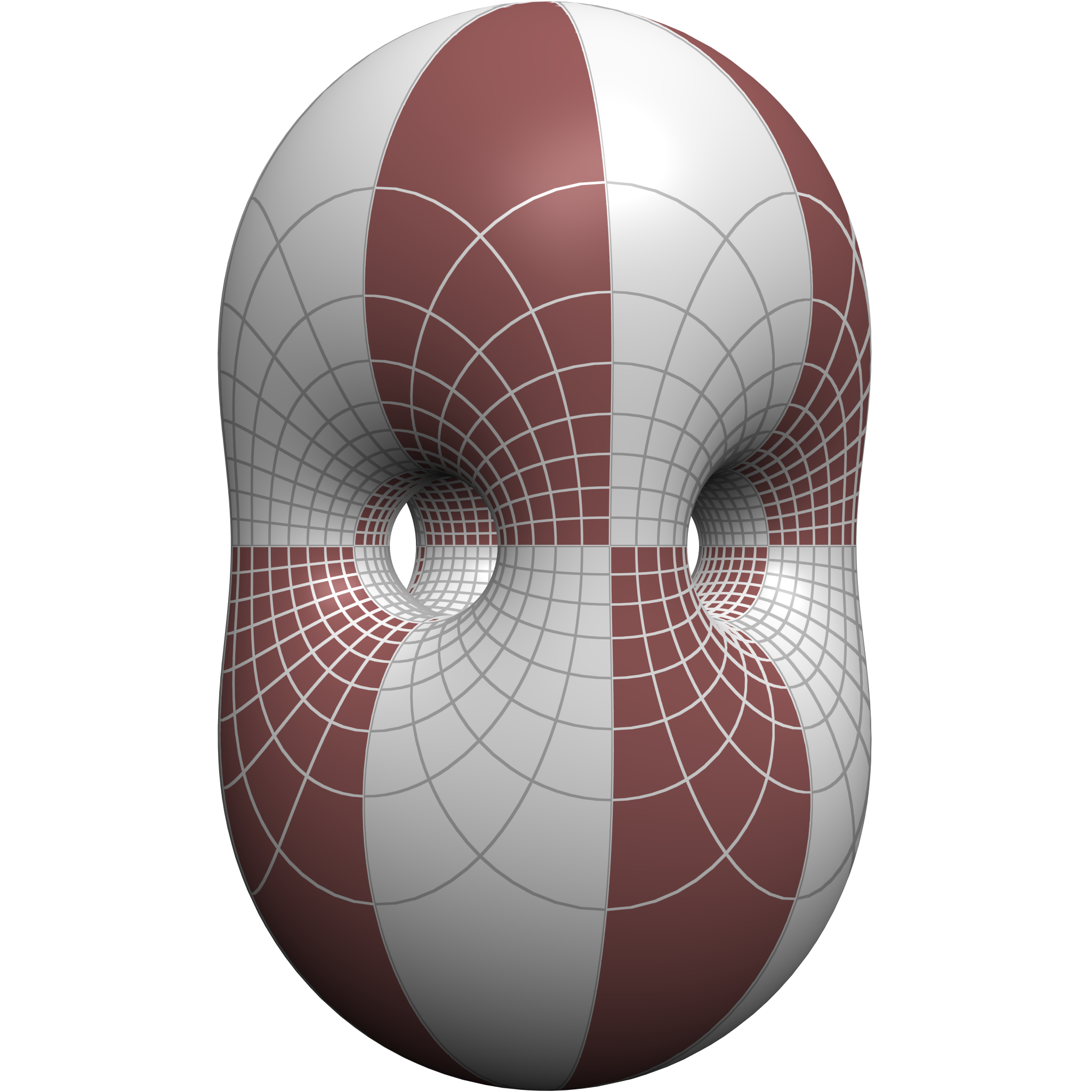}\includegraphics[width=0.18\textwidth]{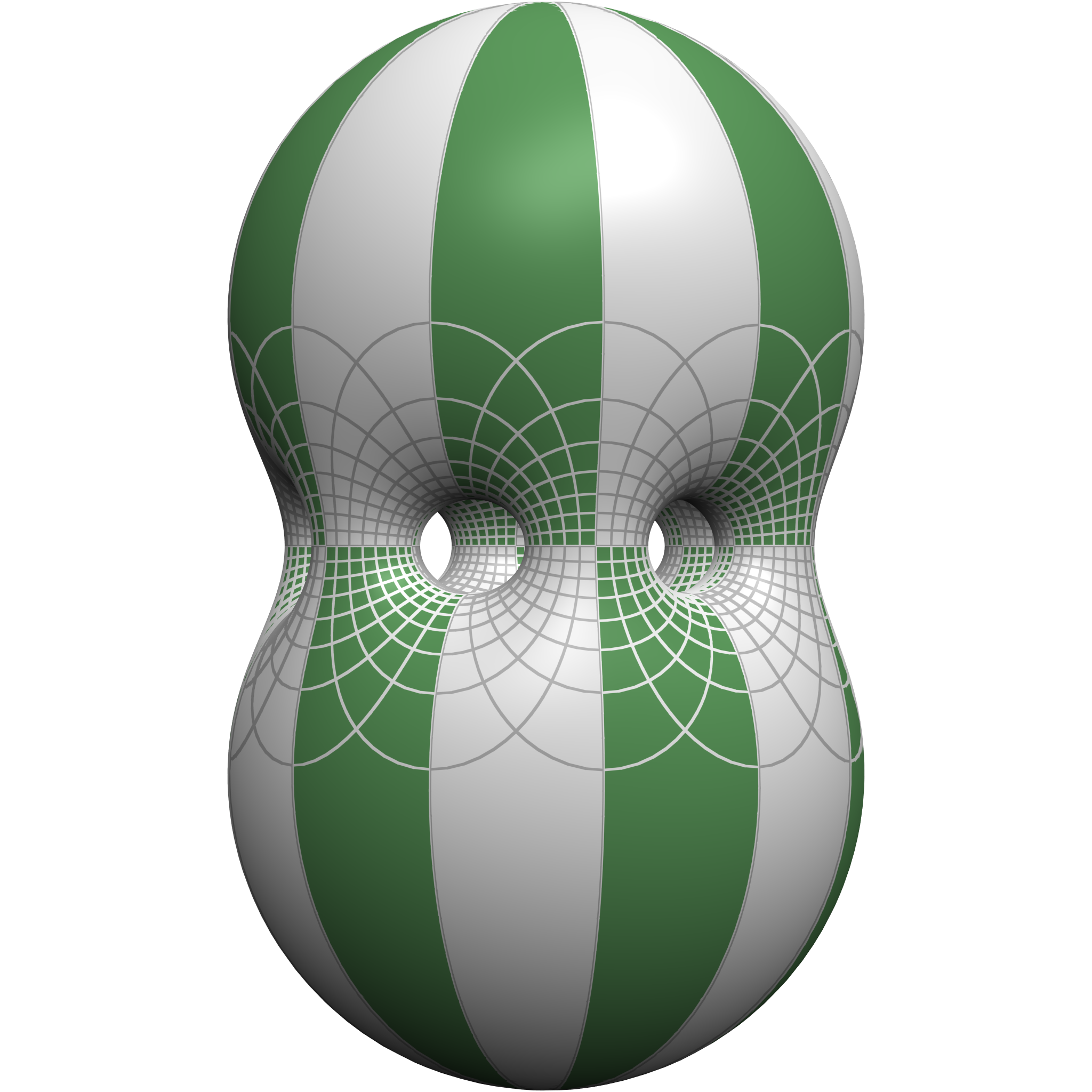}\includegraphics[width=0.18\textwidth]{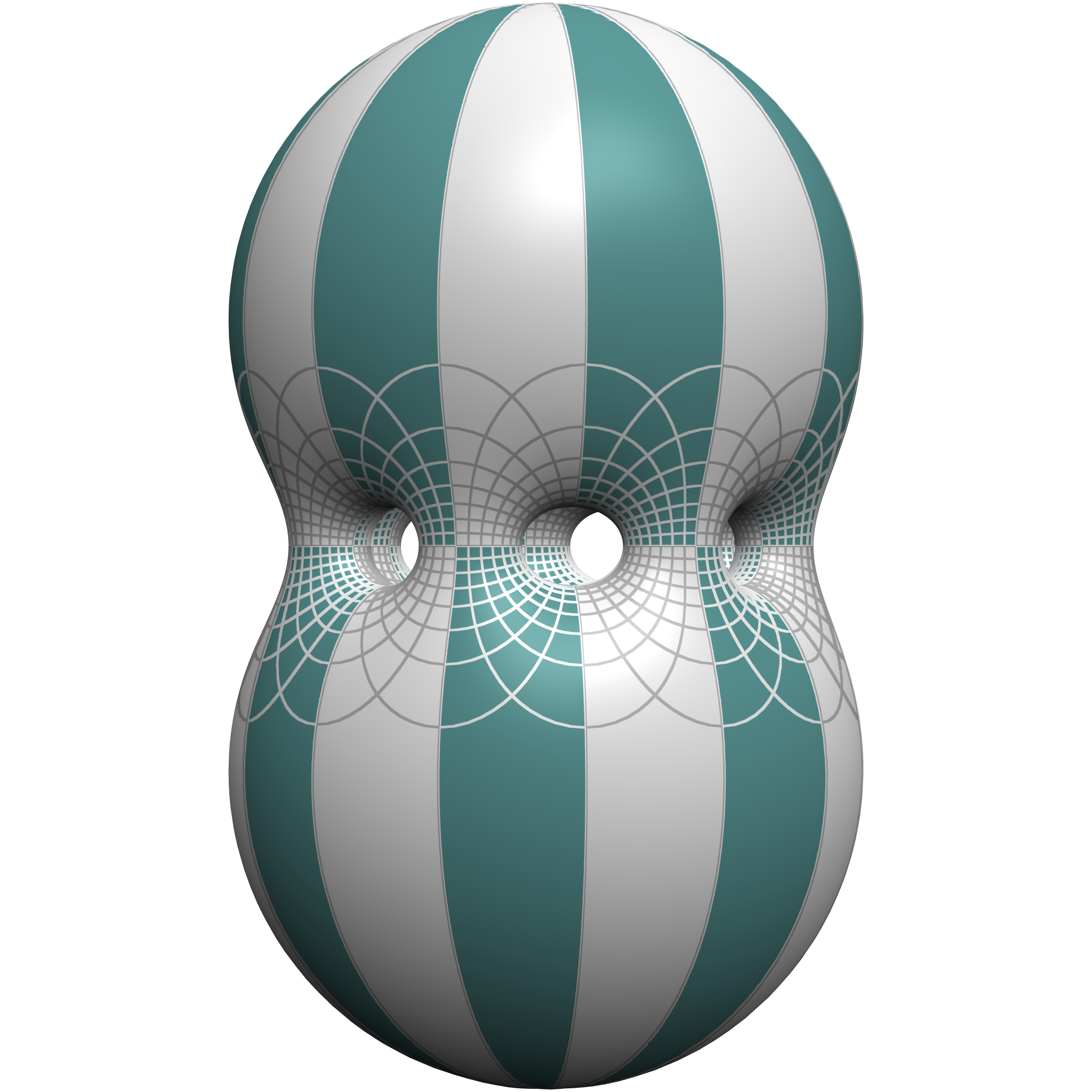}  \includegraphics[width=0.16\textwidth]{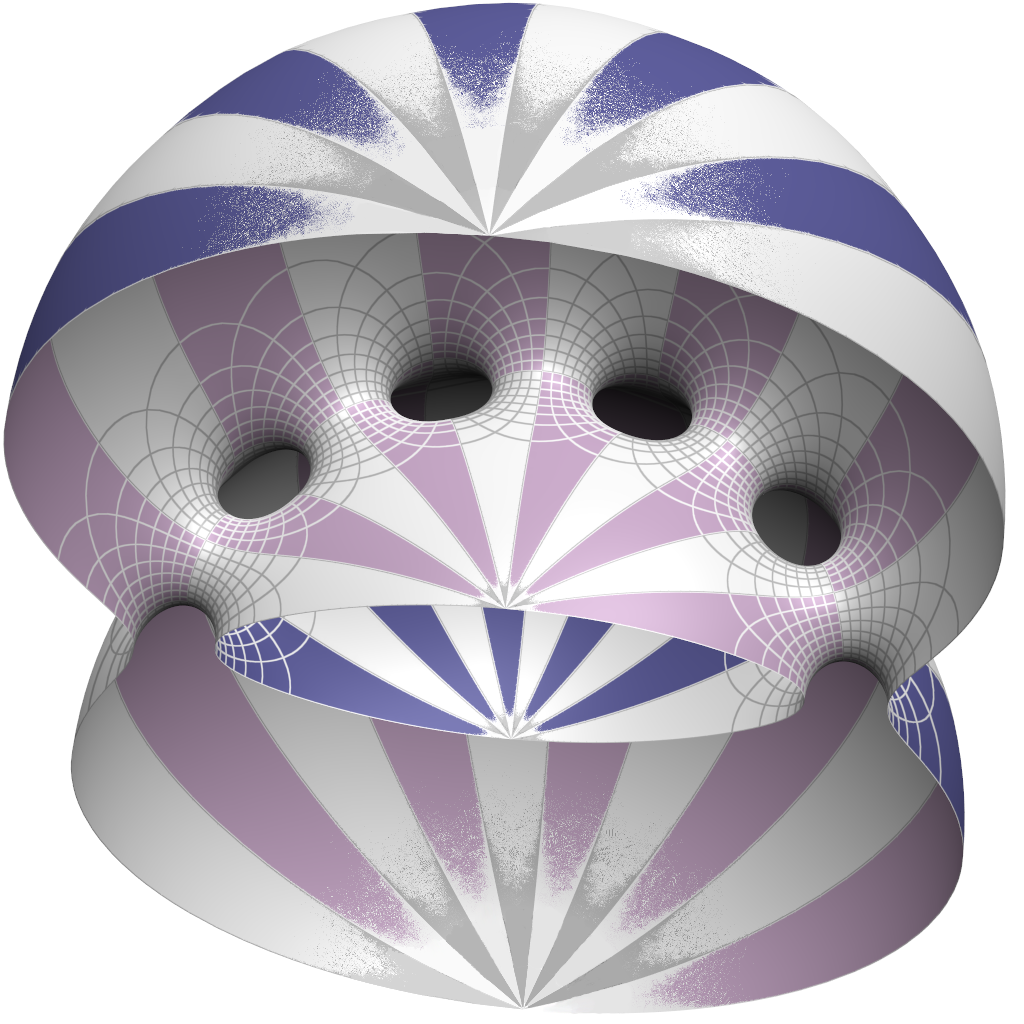}
\renewcommand{\figurename}{\footnotesize Figure}
\caption{
\footnotesize{Lawson's minimal surfaces $\xi_{1,g}$ for $g= 2, 3, 4, 5$ and $g=9$. The last picture in a cutaway view alluding to the convergence to two perpendicularly intersecting 2-spheres for $g \rightarrow \infty.$ Images by Nick Schmitt.
}}
\label{l9}
\end{figure}

\setcounter{tocdepth}{1}
\tableofcontents
\section{Introduction}
Minimal surfaces, and more generally constant mean curvature (CMC) surfaces, in three-dimensional space forms have been the object of intensive study for centuries. Local properties of these surfaces are completely understood via (generalizations of) the Weierstra\ss $\;$ representation, which gives explicit parametrizations of the surface in terms of holomorphic data. Determining global properties, such as the topology, area or embeddedness of minimal or CMC surfaces remains very challenging. \\

Due to the maximum principle, the behaviour of solutions and their moduli space depends  crucially on the (constant) curvature of the ambient space. The most difficult case is the round $3$-sphere, where very few examples have been constructed. The first compact embedded minimal surfaces in the $3$-sphere of all genera were discovered by Lawson \cite{Lawson} using the solution of the Plateau problem with respect to a polygonal boundary curve which are then reflected and rotated along geodesics. Similarly, Karcher-Pinkall-Sterling \cite{KPS} have constructed compact minimal surfaces with platonic symmetries. The other class of closed minimal surfaces in $\S^3$ were constructed by Kapouleas \cite{Kap1} via glueing equatorial 2-spheres using catenoidal handles. Though topology and embeddedness of these examples are known, other geometric properties, for example stability or area are very difficult to determine. While the stability for the Lawson surfaces $\xi_{1,g}$ were proven by Kapouleas and Wiygul \cite{KapWiy}, no area of any minimal surface of genus $g\geq2$ has been explicitly computed. In a very recent paper \cite{KKMS}, new minimal surfaces in the round sphere are constructed by equivariant optimization of the Laplace first eigenvalue. They show that the number of such minimal surfaces grows at least linearly with the genus. Thus showing that the Lawson surfaces are indeed energy minimizing seems to be (even) more subtle than in the genus one case, and new techniques are needed.\\

An alternate  approach to constructing minimal and CMC surfaces in space forms uses the integrable systems structure of harmonic maps. This can be interpreted as a global version of the Weierstra\ss $\;$ representation, which is often referred to as the DPW method \cite{DPW} in this context. For tori, the integrable systems approach was pioneered by Pinkall-Sterling \cite{PS} and Hitchin \cite{Hitchin}  around 1990, and Bobenko \cite{Bobenko} gave explicit parametrizations of all CMC tori in all 3-dimensional space forms.
 \\

Consider a conformally parametrized minimal immersion $f \colon M_g \rightarrow \mathbb S^3$ from a compact genus $g$ Riemann surface into the round $3$-sphere. Then the property of $f$ being harmonic gives rise to a symmetry of the Gauss-Codazzi equations inducing an associated $\S^1$-family of (isometric) minimal surfaces on the universal covering of $M_g$ by rotating the Hopf differential. The gauge theoretic counterpart of this symmetry is manifested in an associated $\C^*$-family 
of flat $\text{SL}(2,\mathbb C)$-connections $\nabla^\lambda$  \cite{Hitchin} on the trivial $\C^2$-bundle over $M_g$ of the form
$$\nabla^\lambda=\nabla + \lambda^{-1}\Psi - \lambda \Psi^*,$$
where $\Psi$ is referred to as Higgs field.  This family of flat connections satisfies 
\begin{enumerate}\label{closingconditions}
\item[(i)]  intrinsic closing: $\nabla^\lambda$ is unitary for all $\lambda\in\S^1$;
\item[(ii)] extrinsic closing: $\nabla^\lambda$ has trivial  monodromy for $\lambda_1=-1$ and $\lambda_2 = 1.$
\item[(iii)] conformality: $\Psi\in \Omega^{(1,0)}(M_g,\mathfrak{sl}(2,\mathbb C)) $ is nilpotent.
\end{enumerate}
The minimal surface can be reconstructed from the associated family of flat connections $\nabla^\lambda$ as the gauge between
$\nabla^{-1}$ and $\nabla^1, $ i.e., $\Psi = \tfrac{1}{2}(f^{-1}df)^{1,0}.$ Constructing minimal surfaces is thus equivalent to writing down appropriate families of such flat connections. Surfaces with constant mean curvature $H$ into space forms can be obtained using an associated family of the form 
$$\nabla^\lambda=\nabla + \lambda^{-1}\wt \Psi - \lambda \wt\Psi^*$$
with the difference that $\nabla$ is no longer the Levi-Civita connection of the CMC surface $f$ and  $2\wt \Psi $ is no longer the $(1,0)$-part of its Maurer-Cartan form, but rather correspond to  the data of an associated minimal surface obtained via Lawson correspondence. More explicitly,
let $\wt \Psi = \tfrac{\lambda_2}{\lambda_2-\lambda_1}(f^{-1}df)^{1,0}$, with $\lambda_1 = -1$ and $\lambda_2 = \tfrac{-iH + 1}{i H +1}$ and $\nabla = d+ \wt \Psi-\wt\Psi^*.$ Then $\nabla^{\lambda_1}=d$ and $\nabla^{\lambda_2} = d+ f^{-1}df$ has trivial  monodromy  and the gauge between the connections $\nabla^{-1}$ and $\nabla^{\lambda_2}$ is the immersion $f$.  \\

The DPW method \cite{DPW} is a way to generate such families of flat connections on $M_g$ from so-called {\em DPW potentials}, denoted by  $\eta = \eta^\lambda,$ 
using loop group factorisation. In fact, $\eta^\lambda$ determines the gauge class of the connections $\nabla^\lambda$ as
$$d+ \eta^\lambda \in [\nabla^\lambda].$$ On simply connected domains $\D$, all DPW potentials give rise to minimal surfaces from $\D$. Whenever the domain has non-trivial topology, finding DPW potentials satisfying  
conditions equivalent to (i)-(iii) for general Sym points $\lambda_1, \lambda_1 \in \mathbb S^1$ becomes difficult.  The problem of finding DPW potentials that fulfill these types of conditions is referred to as {\em Monodromy Problem}.\\

Though successful in the case of tori, the first embedded and closed minimal surfaces of genus $g>1$ using DPW were only recently constructed in \cite{HHT}. This is due to the fact that in contrast to tori the fundamental group of a higher genus surface is non-abelian. A global version of DPW has been developed in \cite{He1, He2} under certain symmetry assumptions. The main challenge to actually construct higher genus minimal and CMC surfaces is to determine infinitely many parameters in the holomorphic ``Weierstra\ss-data'' -- referred to as spectral data. The key idea in our approach is to determine these missing parameters by starting at a well-understood surface, relax some closing conditions, and deform the known spectral data in a direction that changes the genus of the surface such that the Monodromy Problem is solved at rational times.
These ideas were first implemented in \cite{HHS} to deform homogenous and 2-lobed Delaunay tori in direction of higher genus CMC surfaces giving rise to families of closed but branched CMC surfaces in the 3-sphere. A similar philosophy was independently pursued in \cite{nnoids} to find CMC spheres with Delaunay ends  in euclidean space, and more generally, CMC surfaces close to a chain of spheres have been constructed in \cite{nodes}. Combining both approaches embedded minimal surfaces using the DPW approach were constructed in \cite{HHT}. In particular, we gave an alternate existence proof of the Lawson surfaces $\xi_{1,g,}$ by constructing a family $f^t$ of minimal surfaces for $t\sim0$ starting at two orthogonally intersecting geodesic spheres and deform its DPW potential into the direction of a Scherk surface such that $f^t= \xi_{1,g}$ at $t= \tfrac{1}{2(g+1)}$. \\

In this paper we start with constructing {\em Fuchsian} potentials to extend this idea of obtaining closed CMC surfaces by desingularizing two intersecting 2-spheres to Scherk surfaces that intersect at angle $2 \varphi,$ for $\varphi \in (0,\tfrac{\pi}{2})$. For better exposition the $t$-parameter in this paper is no longer $\tfrac{1}{2g+2}=s$ but agrees with $s$ up to order $\geq2$ at $t=s=0$. Our first main theorem is

\begin{theorem}(Existence)\label{MT}
For every $g \in \N$ sufficiently large,  there exists  a smooth family of conformal CMC embeddings $f_{g,\varphi} \colon M_{g, \varphi} \longrightarrow \S^3$ from a compact surface of genus $g$ with 
parameter $\varphi \in (0, \tfrac{\pi}{2})$ 
\begin{equation*}
M_{g,\varphi} \colon y^{g+1} = \frac{(z-p_1)(z-p_4)}{(z - p_2)(z-p_3)},
\end{equation*}
with $p_1 = e^{\ii \varphi}, p_2= -e^{-\ii\varphi}, p_3 = -e^{\ii \varphi}$ and $p_4 = e^{-\ii \varphi}$ satisfying
\begin{itemize}
\item for $\varphi \rightarrow 0, \tfrac{\pi}{2}$ the immersion $f_{g, \varphi}$ smoothly converges to a doubly covered geodesic 2-sphere with $2g+2$ branch points, i.e., the family $f_{g,\varphi}$ cannot be extended in the parameter $\varphi$ in the space of immersions;
\item $f_{g, \varphi} = f_{g, \tfrac{\pi}{2} - \varphi}$ up to  reparametrization and (orientation reversing) isometries of $\S^3$;
\item $f_{g, \tfrac{\pi}{4}}$ is the Lawson surface $\xi_{1,g}$ of genus $g$;
\item the (constant) mean curvature $H_{g, \varphi}$ of $f_{g, \varphi}$ is zero if and only if $\varphi =\tfrac{\pi}{4}$.
\end{itemize}
\end{theorem}

Combining the results by Kusner-Mazzeo-Pollack \cite{KMP} and Kapouleas-Wiygul \cite{KapWiy} the moduli space of genus $g$ CMC surfaces is $1$-dimensional at the Lawson surface $\xi_{1,g}$. Therefore, the families we construct in Theorem \ref{MT} give the first global result on the structure of the moduli space of CMC surfaces of genus $g>1.$
By an estimate of Li-Yau \cite{LiYau} surfaces $f:M_g \longrightarrow \mathbb S^3$ with Willmore energy 
\begin{equation*}
\mathcal W (f) = \int_{M_g} (H^2+1) dA,
\end{equation*}
below $8 \pi$ are automatically embedded. To apply this theorem to the families of CMC surfaces constructed in Theorem \ref{MT}, we estimate their Willmore energy.

\begin{theorem}(Energy expansion)\label{MT2}
Let  $f_{g,\varphi} \colon M_{g, \varphi} \longrightarrow \S^3$ be the  smooth family of conformal embeddings constructed in Theorem \ref{MT}.
Then there exists an iterative algorithm to compute the DPW potential as well as the area and Willmore energy of $f_{g,\varphi}$ in terms multiple polylogarithms (MPLs). In particular, we have
\begin{itemize}
\item the Willmore energy of $f_{g,\varphi}$ is strictly monotonically decreasing in $\varphi$ for $\varphi \in (0, \tfrac{\pi}{4})$ from $8 \pi$ to Area$(\xi_{1,g})= \mathcal W(\xi_{1,g})$  with Taylor expansion  at $g=\infty$ given by  
\begin{equation}\label{eq:Willmore-series}
\mathcal W(f_{g,\varphi})= 8 \pi \left(1-  \sum_{\substack{k=1 \\ k\text{ odd}}}^\infty \mathcal W_k  \tfrac{1}{(2g + 2)^k}\right)
\end{equation}
with $\mathcal W_1=  -2\left[ \cos(\varphi)^2 \log(\cos(\varphi)) + \sin(\varphi)^2\log(\sin(\varphi)) \right].$ 

\item furthermore, for $\varphi = \tfrac{\pi}{4}$ we have 
\begin{equation}
\label{eq:area-series}
\text{Area}(\xi_{1,g})=8 \pi\left(1-\sum_{\substack{k=1 \\ k\text{ odd}}}^{\infty}\frac{\alpha_k}{(2g+2)^k}\right)
\end{equation}

with $\alpha_1 = \log (2)$ and $\alpha_3= \tfrac{9}{4}\zeta (3),$
where $\zeta$ is the Riemann $\zeta$-function.
\end{itemize}
\end{theorem}

\begin{remark}
The $\alpha_1= \log(2)$ was computed in \cite{HHT}.  
Identifying $\alpha_3$ to be $\tfrac{9}{4}\zeta(3)$ was first suggested by Wolframalpha. The verification of this identity required considerable additional effort compared to $\alpha_1$.
\end{remark}

An immediate corollary of the energy estimates is that the infimum Willmore energy in the conformal class of $f_{g, \varphi}$ is below $8 \pi$ for all $\varphi \in (0, \tfrac{\pi}{2})$ and $g\gg1$. Therefore, by \cite{KuwertSchatzle} the infimum is attained at an embedding.
\begin{corollary}
For $g\gg1$  the constrained Willmore infimum
\begin{equation*}
\textrm{Inf}\;\{\mathcal W(f) \; | \; f \colon M_{g,\varphi} \longrightarrow \S^3, \text{conformal immersion}\}.
\end{equation*}
is attained at a smooth and conformal embedding. 
\end{corollary}

The main idea of proving the two  Theorems \ref{MT} and \ref{MT2} is as follows. Let $\pi\colon M_{g,\varphi} \longrightarrow \C P^1$ be the projection from $M_{g, \varphi},$ see \eqref{riemannsurface}, to $\C P^1$ totally branched over the four points $p_1, ..., p_4$. Showing the existence of a (Fuchsian) DPW potential solving a Monodromy Problem on the 4-punctured sphere $\C P^1\setminus\{p_1, ..., p_4\}$ then gives rise to minimal or CMC  surface patches
$\tilde f_{g, \varphi}$ in the round 3-sphere with boundary. The closed surfaces $f_{g, \varphi} \colon M_{g, \varphi} \longrightarrow \S^3$ are then obtained through analytic continuation of $\tilde f_{g, \varphi}$.  Since the potential used in  \cite{HHT}  withstood the generalization to general wing angles (in a way that an implicit function theorem argument can be applied to obtain compact surfaces), we give in Section \ref{ssec:thepotential} an alternate ansatz using the same philosophy which allows the intersection angle and the wing angle to be $2\varphi$ with $\varphi \in [0, \tfrac{\pi}{2}]$. For $\varphi \in (0, \frac{\pi}{2})$ fixed, we then show the existence of a unique family $\eta_{t, x(t), \varphi}$ of Fuchsian DPW potentials given by the parameter vector $x(t)$ for $t \sim 0$ solving the Monodromy Problem. An advantage of these Fuchsian potentials $\eta_{t, x(t), \varphi}$ is that the limiting behaviour for $\varphi\rightarrow 0$ (and $\varphi \rightarrow \tfrac{\pi}{2}$) can be understood (within the same setup) to obtain a uniform existence interval in $t$ for all $\varphi$ in Section \ref{limitvarphi}. This gives rise to complete families of CMC surfaces for genus $g \gg1$. \\

Through the DPW approach it is possible to compute geometric properties of the surfaces explicitly from the potential $\eta_{t, x(t), \varphi}$. The major improvement of the paper at hand is that we construct {\em Fuchsian} DPW potentials that incorporate all symmetries of the immersion.  In stark contrast to \cite{HHT} this new approach not only allows 
to obtain a complete family of CMC surfaces $f_{g, \varphi}$ deforming the Lawson surface $\xi_{1,g}$ but also allows for an iterative algorithm to compute the Taylor expansions of the DPW potential as well as the area of $f_{g, \varphi}$ explicitly. For a deeper investigation of the Taylor series, we specialize to the case of Lawson's minimal surfaces starting from Section \ref{mincasereduce}. When writing the area expansion for $\varphi = \tfrac{\pi}{4}$ as in \eqref{eq:area-series} we have an iterative algorithm to compute the coefficients $\alpha_k$ in terms of certain iterated ($\Omega$-) integrals. Evaluating these integrals gave
\begin{align*}
&\alpha_1=\log(2)\\
&\alpha_3=\frac{9}{4}\zeta(3)\\
&\alpha_5\simeq 3.6996269944\ldots\\
&\alpha_7\simeq -53.1688000602\ldots\\
&\alpha_9\simeq -459.5656763714\ldots
\end{align*}
while the even order coefficients vanish. Note that the coefficients have changing signs, so it is not obvious whether the area is monotonic in $g.$ Then in Section \ref{section:multizetas} we first prove a structure theorem for the coefficients $\alpha_k$. With the notations of Section \ref{section:multizetas} we have
\begin{theorem}
\label{thm1}
Every $\alpha_k$ can be expressed as a weight $k$ linear combination ( with coefficients in $\mathbb Q$)  of products of alternating multiple zeta values at arguments in $\{\overline{1},2,\overline{2}\}$.
\end{theorem}

For example, for the first coefficient \( \alpha_1 \), it is elementary that
\[
	\zeta(\overline{1}) = \sum_{k=1}^\infty \frac{(-1)^k}{k} = -\log(2) \,,
\]	
and only the multiple remains to be fixed.  The second coefficient \( \alpha_3 \) must be a linear combination of weight 3 alternating multi-zetas, for which \( \{ \zeta(3), \zeta(2)\log(2), \log(2)^3 \} \) is known to be a spanning set.  The vanishing of the \( \zeta(2)\log(2) \) and \( \log(2)^3 \) coefficient is established by direct calculation in Section \ref{section:multizetas}, but not yet understood and hints perhaps at some deeper structures.
The pattern underlying the proof of Theorem \ref{thm1} results in a simplification of the algorithm, which allows us to compute the coefficients $\alpha_k$ up to $k=21$ numerically, and prove close form expressions for those in terms of multiple zeta values up to $k=11$ (up to $k=7$ is given in Appendix \ref{appendix:numalpha}.) Furthermore, this also allows us to give a proof identifying $\alpha_3 = \tfrac{9}{4} \zeta(3)$ without the need of any computer algebra system.  \\

A natural question that arises in Theorem \ref{MT}  is about the interval $t \in [0, T)$ for which the implicit function theorem arguments holds. By complexifying the underlying equations and going into details of the proof of the implicit function theorem using contraction mapping principle we estimate
\begin{theorem}
\label{thm2}
The Taylor series for the Fuchsian DPW potential, and in particular the series $\eqref{eq:area-series}$,  converges for $g\geq 2.65404$. 
\end{theorem}

\begin{remark}
This is by far not optimal, as we conjecture convergence of our potential for all $g \geq1.$ In fact, when plugging in $g=1$ and $\varphi = \tfrac{\pi}{4}$ in \eqref{eq:area-series}  truncated at $k = 21,$ we obtain the area of the Clifford torus, which is $2 \pi^2$, up to an error of only $10^{-3}$.  For $g=0$ on the other hand we get a large number, suggesting that the convergence radius should lie between $g=0$ and $g=1.$
\end{remark}

The proof of Theorem \ref{thm2} relies on the evaluation of a large number of iterated integrals and has a number of interesting consequences: it can be used to numerically compute the area of Lawson surfaces of genus $g\geq 3$ with an explicit bound on the error,
see Section \ref{section:area}.
Though it strikes weird at first that the bound on the genus we give here is not an integer, the explicit value, or that the value is significantly below $3$ is essential to prove that the area is an strictly increasing function of the genus $g$ for all $g\geq 3$
(Proposition \ref{prop:monotonicity}). Together with the resolution of the Willmore conjecture \cite{MN} and an upper bound on the area  of Lawson's genus $2$ surface using a coarse triangularization of the fundamental piece in Proposition \ref{prop:genus2} we obtain:

\begin{theorem}
\label{thm3}
The area of the Lawson surfaces  $\xi_{1,g}$  of genus $g$ is strictly monotonically increasing in $g$ for all $g\geq 0$.
\end{theorem}
More generally, it is  natural to conjecture
\vspace{-0.3cm}
\subsection*{Conjecture}
Consider for $(k,l) \in \N^2$ the area $\mathcal A_{k,l}$ of the Lawson surfaces $\xi_{k,l}.$ Then, $\mathcal A_{k,l}$ is monotonically increasing with respect to the partial ordering of $\N^2$ given by $(k,l) \preceq (k',l')$ if and only if $k\leq k'$ and $l \leq l'.$  

\begin{remark}
For low genus, the experimental values of  $\mathcal A_{k,l}$ for $k,l=1, ..., 12$ given in \cite{BoHeSch} suggest that the conjecture holds in these cases. When $l$ is fixed, the area estimates in \cite{HHT} at $k=\infty$ confirms the conjecture for $k$ sufficiently large. 
\end{remark}

The paper is organized as follows. We first give in Section \ref{sec:prelimi} the necessary preliminaries to loop groups, DPW approach and character varieties. The ansatz for the Fuchsian DPW potential is explained in Section \ref{sec:potential} which is then shown to solve a reformulated monodromy problem. When building the surface from the DPW potential, we show that the symmetries of the potential actually carry over to symmetries of the surface. We give a detailed study of these symmetries and of the limiting behaviour of the compact surfaces when $g \rightarrow \infty.$
The $\varphi \rightarrow 0, \tfrac{\pi}{2}$ limit is considered in Section \ref{limitvarphi} giving us an uniform existence interval for all $\varphi.$
\medskip

When solving the monodromy problem, we use so-called half-trace coordinates denoted by $\mathfrak p$ and $\mathfrak q.$ 
The DPW potential and the area (respectively the Willmore energy) of the corresponding surfaces can be determined from the expansions in $t$ of these coordinates.
In Section \ref{iterated} we discover a pattern in the Taylor expansion of $\mathfrak p$ and $\mathfrak q$, and use these in 
 Section \ref{sec:1storder} to compute the first derivatives of the parameters when they solve the monodromy problem. An algorithm to compute higher order derivatives is presented in Section \ref{sec:higheroder}, and derivatives up to order three are computed in the minimal case. The coefficients are given by iterated integrals, which we call $\Omega$-values.  When specializing to the Lawson case of $\varphi = \tfrac{\pi}{4}$ we discover in Section \ref{section:multizetas} that 
every $\Omega$-value can be expressed using one single multiple zeta value using so called iterated $\beta$ integrals introduced by Hirose-Sato \cite{hsBetaIntegral}. This drastically reduces the complexity of computing the $\Omega$-values, which was a crucial step to compute $\alpha_k$ in \eqref{eq:area-series} up to $k=21.$
\medskip

Quantitative estimates on the existence interval for the implicit function theorem argument are given in Section \ref{sec:quantativeimplicitfunction} which yields that the Taylor series of our DPW potential converges for $g \geq 2.65404.$  Using this convergence radius and the expansion of the area to order $k= 21$, we prove monotonicity for the area of the Lawson surfaces for all $g \geq 0$ in Section \ref{sec:monotonicity}.\\

The proofs for some technical or folklore Lemmas  are included in the appendix. Moreover, we also attached numerical outputs, such as the computation of the third order derivatives, and the values of the $\alpha_k$. Supplementary Mathematica notebooks with the implementation of the algorithm to compute higher order derivatives in the minimal and CMC surface case are provided. Moreover, we also provide Mathematica notebooks to estimate the $\Omega$-values and to compute the convergence radius for the quantitive implicit function theorem of Section \ref{sec:quantativeimplicitfunction}.

\section{Preliminaries}\label{sec:prelimi}
The DPW method \cite{DPW} is a technique to parametrize minimal and CMC surfaces in space forms using holomorphic data and loop groups. We first set the basic definitions and the necessary notations here, for more details adapted to our approach see \cite{HHT}. \\
\subsection{Loop groups}
\label{loop-groups}
Let $G$ be a Lie group and $\mathfrak g$ its Lie algebra. Then the associated 
loop group 
 $$\Lambda G:= \{\text{ real analytic maps (loops) }\Phi\colon \S^1\longrightarrow G,\quad \lambda \longmapsto \Phi^\lambda\},$$
 is an infinite dimensional Frechet Lie group via pointwise multiplication. Its Lie algebra is given by
$$\Lambda \mathfrak{g}:=\{\text{ real analytic maps (loops) }\eta\colon \S^1\longrightarrow\mathfrak{g},\quad  \lambda \longmapsto \eta^\lambda\}.$$ 
Let $\overline\D:=\{\lambda\in\C\mid |\lambda|\leq1\}$ and $G$ a complex Lie group. Then 
\[\Lambda^+G:=\{\Phi\in\Lambda G\mid \Phi \text{ extends holomorphically to } \overline\D\}\]
denotes the positive part of the loop group $\Lambda G$. Similarly, 
 \[\Lambda^+\mathfrak g:=\{\eta\in\Lambda \mathfrak g\mid \eta \text{ extends holomorphically to } \overline\D\}\]
 denotes the non-negative part of its Lie algebra.\\

Fix $\rho>1$ and define for $u\in L^2(\S^1,\mathbb C)$, given by its Fourier series
\[u=\sum_{k\in\Z} u_k \lambda^k,\] the norm
\[\parallel u\parallel_{\rho}=\sum_{k\in\Z} |u_k| \rho^{|k|}\leq\infty.\]
As in  \cite{nnoids, HHT} the following functional spaces are considered. Let
 \[\mathcal W_{\rho}:=\{u\in  L^2(\S^1,\mathbb C)\mid \;\parallel u\parallel_{\rho}<\infty\}\] 
be the space of absolutely convergent Fourier series on the annulus 
$$\A_{\rho}=\{\lambda\in\C\mid \tfrac{1}{\rho} < |\lambda| <{\rho}\}.$$
This generalizes the classical Wiener algebra (which has $\rho=1$).
The most important property of $\mathcal W_{\rho}$ is that it is a Banach algebra.
Let
\[\cal{W}^{\geq 0}_{\rho}:=\{u=\sum_k u_k \lambda^k\in \mathcal W_{\rho}\mid
u_k=0\; \;\;\forall  \;k<0\}\]
denote the space of functions $u\in L^2(\S^1,\mathbb C)$ that can be extended holomorphically to the disk $\D_{\rho}=D(0,\rho)$.  Similarly, let
\[\cal{W}^{>0}_{\rho}:=\{u=\sum_k u_k \lambda^k\in \mathcal W_{\rho}\mid u_k=0\; \;\;\forall\; k\leq0\,\}\]
\[\cal{W}^{<0}_{\rho}:=\{u=\sum_k u_k \lambda^k\in \mathcal W_{\rho}\mid u_k=0\;\;\;\forall \;k\geq0\,\}\]
denote the positive and negative space, respectively. Therefore we can decompose every $u\in \mathcal W_{\rho}$
$$u=u^+ + u^0+ u^-$$
into its positive and negative component $u^\pm \in W^{\gtrless 0}_{\rho}$, and a constant component $u^0 = u_0$.  \\

We define the star and conjugation involutions on $\mathcal W_{\rho}$ by
$$u^{*}(\lambda)=\overline{u(1/\overline{\lambda})}
\quad\mbox{ and }\quad \overline{u}(\lambda)=\overline{u(\overline{\lambda})}$$
and we denote by $\mathcal W_{\R,\rho}$ the space of functions $u\in\mathcal W_{\rho}$
with $u=\overline{u}$.
Note that $u=u^*$ means that $u$ is real on the unit circle, while $u=\overline{u}$ means that $u$ is real on the real line.

 \begin{remark}
For an arbitrary matrix group $G$,
$\Lambda G_{\rho}$ denotes the subspace of $\Lambda G$ consisting of loops whose entries
are in $\mathcal W_{\rho}$. Then $\Lambda G_{\rho}$ is a Banach Lie group.
To enhance exposition we will omit the subscript $\rho$ most of the time.
 \end{remark}

\subsection{DPW approach}
A DPW potential $\eta$ on a Riemann surface $M$ is 
a holomorphic 1-form
\[\eta\in\Omega^{1,0}(M,\Lambda \mathfrak{sl}(2,\C))\]
with  
\[\lambda \eta\in \Omega^{1,0}(M,\Lambda^+\mathfrak{sl}(2,\C))\]
such that  its residue at $\lambda=0$ 
\[\eta_{-1}:=\text{Res}_{\lambda=0} (\eta)\]
is a nowhere vanishing and nilpotent 1-form.  

For a given DPW potential $\eta$ the extended frame $\Phi$ is a solution on the universal cover $\widetilde M$ of $M$ of 
$$d_{M} \Phi =  \Phi \eta$$
with some initial value $\Phi (z_0)= \Phi_0\in\Lambda\SL(2,\C)$ at a fixed $z_0 \in M$. In other words, $\Phi$ is a parallel section of $\wt \nabla^\lambda.$ The Iwasawa decomposition is the unique splitting of $\Phi$ into a unitary and a positive factor:
$$\Phi = F B$$ 
with $F \in \Lambda \SU (2)$ and $B \in \Lambda^+_{\R}\SL(2,\C)$,
where $\Lambda^+_{\R}SL(2,\C)\subset\Lambda^+ SL(2,\C)$ denotes the space of positive loops $B$ such that $B(0)$ is upper triangular with positive real numbers on the diagonal.
The splitting depends smoothly on $z \in M$, thus the unitary factor $F$ is also smooth in $z$.
\begin{remark} 
The Iwasawa decomposition is a smooth diffeomorphism between the Banach Lie groups
$\Lambda SL(2,\C)_{\rho}$ and $\Lambda SU(2)_{\rho}\times\Lambda^+_{\R} SL(2,\C)_{\rho}$
(see  Theorem 5 in \cite{minoids}).
\end{remark}

Consider two unitary complex numbers $\lambda_1\neq\lambda_2\in\mathbb S^1$, called the Sym-points.
A conformal immersion of constant mean curvature $H = \ii  \tfrac{\lambda_1 + \lambda_2}{\lambda_1- \lambda_2}$ can be reconstructed from the unitary factor $F$ by the Sym-Bobenko formula
$$f = F(\lambda_1) F(\lambda_2)^{-1}\colon\widetilde M\longrightarrow\mathbb S^3.$$ 
For any element $\gamma\in\pi_1(M,z_0)$ of the fundamental group, let $\cal{M}(\Phi,\gamma)$ denote the monodromy of $\Phi$ with respect to $\gamma.$ The conditions for the DPW potential to give a well-defined immersion $f$ on $M$ are

\begin{equation} \label{monodromy-problem}
\forall \gamma\in\pi_1(M,z_0)\qquad
\left\{\begin{array}{l}
\cal{M}(\Phi,\gamma)\in \Lambda SU(2)\\
\cal{M}(\Phi,\gamma)|_{\lambda=\lambda_1}=\cal{M}(\Phi,\gamma)_{\lambda=\lambda_2}=\pm\Id_2\end{array}\right.\end{equation}
We refer to these conditions in \eqref{monodromy-problem} as the {\em Monodromy Problem}.

\subsection{The Riemann surface}
In this paper we restrict ourselves to genus $g$ Riemann surfaces $M_{g,\varphi}$ with a $\Z_{g+1}$-symmetry. More explicitly,  the Riemann surface $M_{g,\varphi}$, where $\varphi\in(0,\frac{\pi}{2})$ is a parameter, is determined by the algebraic equation
\begin{equation}
\label{riemannsurface}
M_{g, \varphi} \colon y^{g+1} = \frac{(z-p_1)(z-p_3)}{(z-p_2)(z-p_4)},
\end{equation}
so that it admits a $(g+1)$-fold covering
$$\pi \colon M_{g,\varphi}\longrightarrow \C P^1$$
totally branched over the four points
\begin{equation}\label{realvarphipj}p_1=e^{\ii\varphi},\quad p_2=-e^{-\ii\varphi},\quad p_3=-e^{\ii\varphi},\quad p_4=e^{-\ii\varphi}.\end{equation}

\begin{remark}
On a compact Riemann surface
there exists no  DPW potential solving the Monodromy Problem without singularities (e.g. poles).
To obtain compact CMC surfaces we require the necessary singularities to be apparent, i.e., removable by suitable local gauge transformation.
\end{remark}
Rather than constructing the DPW potential directly on $M_{g,\varphi}$, we consider a DPW potential $\eta$ on
$$\Sigma = \Sigma_\varphi:= \C P^1  \setminus \{p_1, ..., p_4\}$$
with simple poles at $p_1,\cdots,p_4$.

\subsection{Fuchsian systems}
A \SL$(2, \C)$ Fuchsian system on the 4-punctured sphere $\Sigma$ is a holomorphic connection on the trivial $\C^2$-bundle over $\Sigma$ of the form $\nabla= d+ \xi$ with
$$\xi = \sum_{j=1}^4 A_j \tfrac{dz}{z-p_j},$$
where $A_j \in \mathfrak {sl}(2, \C)$ and ${\displaystyle \sum_{j=1}^4 A_j= 0}$ to avoid a further singularity at $z= \infty.$  Two Fuchsian systems are equivalent (when fixing the punctures $p_j$), if there exist an invertible matrix $G$ such that $\tilde A_j = G^{-1}A_j G.$
Due to its form, a Fuchsian system is automatically flat, and we can consider the associated monodromy 
representation of the first fundamental group $\pi_1(\Sigma)$. 
Via the monodromy representation the space of these (irreducible) Fuchsian systems (modulo equivalence) 
is biholomorphic to an open dense subset of the space of (irreducible) representations of the first fundamental group of $\Sigma$  to $\mathrm{SL}(2,\C)$ (modulo overall conjugation). Of particular interest for the construction of CMC surfaces
are  Fuchsian systems admitting a unitary monodromy representation. 
 
 \begin{definition}
 A \SL$(2, \C)$ Fuchsian system is called unitarizable if there exist a hermitian metric $h$  on $\underline{\C}^2\to\Sigma$ such that the connection $d+\xi$ is unitary with respect to $h.$
 \end{definition}
  
The analogous definition on the representation side is
\begin{definition}
A monodromy representation is unitarizable if it lies in the conjugacy class of a unitary representation.
\end{definition}
We will make use of the following classical theorem, which dates back to Vogt, and Fricke and Klein, see for example \cite{BeGo}. We give a short and self-contained proof in Appendix \ref{AppCV}.
\begin{theorem}\label{characterLLLL}
Let $L_1,\dots,L_4\in\mathrm{SL}(2,\C)$ satisfying $L_1L_2L_3L_4=\Id$ with
\[\tr(L_1)=\tr(L_2)=\tr(L_3)=0\quad\text{and}\quad \tr(L_4)=\traceL\in(0,2).\]
Then, $L_1,\dots,L_4$ is uniquely determined up to conjugation by its
trace coordinates $x=\tr(L_1L_2),$ $y=\tr(L_2L_3)$ and $z=\tr(L_2L_4).$
These satisfy \[x^2+y^2+z^2+xyz-4+\traceL^2=0.\]
Conversely, every solution $(x,y,z)$ to this equation gives rise to  $L_1,\dots,L_4$ as above with $(x,y,z)$ as trace coordinates.
Moreover,  there exists a $U\in\mathrm{SL}(2,\C)$ such that  $$U L_j U^{-1} \in \mathrm{SU}(2) \quad \text{ for all } \quad j= 1, ..., 4 $$ if and only if the trace coordinates satisfy $(x,y,z)\in(-2,2)\times (-2,2)\times \mathbb R.$ The unitarizer $U$ is uniquely determined up to multiplication by $\mathrm{SU}(2).$ 
\end{theorem}

\section{A Fuchsian DPW potential for Lawson-type CMC surfaces of high genus}
\label{sec:potential}
\subsection{The potential}
\label{ssec:thepotential}
Fix $\varphi \in (0, \tfrac{\pi}{2})$.
For small $t>0$, we consider on the 4-punctured sphere $\Sigma_{\varphi}$ a so-called {\em Fuchsian DPW potential} of the form
\begin{equation}\label{thegeneraleta}\eta=t\;{\sum_{j=1}^4 }A_j(\lambda)\frac{dz}{z-p_j}\end{equation}
where the residues $A_j \in \Lambda\mathfrak{sl}(2, \C)_{\rho}$ are given as follows.
Due to the choice of $p_j$ in \eqref{realvarphipj}, the Riemann surface $\Sigma_{\varphi}$ has three symmetries given by
$$\delta(z)=-z,  \quad \tau(z)= \tfrac{1}{z}, \quad  \text{ and }\quad  \sigma(z) = \bar z.$$
We require that the potential $\eta$ is equivariant with respect to these symmetries, i.e., it satisfies 
\begin{align*}
\delta^*\eta&=D^{-1}\eta D \quad \text{ with }\quad D=\matrix{\ii&0\\0&{-\ii}}\\
\tau^*\eta &= C^{-1}\eta C\quad \text{ with }\quad C=\matrix{0&\ii\\\ii& 0 }\\
\sigma^*\overline{\eta}&=\eta\quad \text{ where }\quad \overline{\eta}(z,\lambda)=\overline{\eta(z,\overline{\lambda})}.
\end{align*}
These are meant to encode the symmetries of Lawson surfaces. We will see in Section \ref{section:building} that provided the Monodromy Problem is solved, the resulting immersion indeed has the desired symmetries.
The symmetries of $\eta$ are respectively equivalent to
\begin{align*}
&A_3=D^{-1}A_1 D\quad\text{and}\quad  A_4=D^{-1}A_2 D\\ 
&A_4=C^{-1}A_1 C\quad \text{and}\quad A_3=C^{-1}A_2C\\
&A_4=\overline{A_1} \quad\text{and}\quad
A_3=\overline{A_2}.
\end{align*}
In particular, $A_1=C^{-1} \overline{A_1} C$ so $A_1$ is of the form
$$A_1=\begin{pmatrix}x_1 & x_2+\ii x_3\\ x_2-\ii x_3& -x_1\end{pmatrix}$$
with $\overline{x_1}=-x_1$, $\overline{x_2}=x_2$ and $\overline{x_3}=x_3$.
In other words, $x_1\in \ii\mathcal W_{\R,\rho}$ and $x_2,x_3\in\mathcal W_{\R,\rho}$.
The other residues are then given by
\begin{equation*}
\begin{split}
A_2&=\begin{pmatrix}-x_1 & -x_2+\ii x_3\\ -x_2-\ii x_3& x_1\end{pmatrix}\\
A_3&=\begin{pmatrix}x_1 & -x_2-\ii x_3\\ -x_2+\ii x_3& -x_1\end{pmatrix}\\
A_4&=\begin{pmatrix}-x_1 & x_2-\ii x_3\\ x_2+\ii x_3& x_1\end{pmatrix}.
\end{split}
\end{equation*}
In particular,
\[\sum_{j=0}^4A_j=0,\]
so the potential is regular at $z= \infty$. 
The functions $x_1,x_2,x_3$ are the parameters of our construction. To emphasize the dependence of the potential on $t$ and $x=(x_1,x_2,x_3)$ we write $\eta_{t,x}$ and even $\eta_{t,x,\varphi}$ (in this order) if we need to emphasise the dependence on the angle $\varphi$.
We can rewrite the potential in the form $$\eta_{t,x}=t\begin{pmatrix}
x_1\omega_1&x_2\omega_2+\ii x_3\omega_3\\
x_2\omega_2-\ii x_3\omega_3&-x_1\omega_1\end{pmatrix}
=t\sum_{j=1}^3 x_j\mathfrak m_j\omega_j$$
where the meromorphic 1-forms $\omega_j$ are given by
\[
\begin{split}
\omega_1 = \left (\frac{1}{z-p_1} - \frac{1}{z-p_2}+\frac{1}{z-p_3}- \frac{1}{z-p_4}\right)dz
=\frac{4 \ii\sin(2\varphi)z\,dz}{z^4-2\cos(2\varphi)z^2+1},\\
\omega_2 = \left (\frac{1}{z-p_1} - \frac{1}{z-p_2}-\frac{1}{z-p_3} + \frac{1}{z-p_4}\right)dz
=\frac{4\cos(\varphi)(z^2-1)\,dz}{z^4-2\cos(2\varphi)z^2+1},\\
\omega_3 = \left (\frac{1}{z-p_1} +\frac{1}{z-p_2}-\frac{1}{z-p_3}- \frac{1}{z-p_4}\right)dz
=\frac{4\ii\sin(\varphi)(z^2+1)\,dz}{z^4-2\cos(2\varphi)z^2+1)},
 \end{split}\]
and
\[
 \mathfrak m_1=\begin{pmatrix} 1&0\\0&-1\end{pmatrix},\qquad
\mathfrak m_2=\begin{pmatrix} 0&1\\1&0\end{pmatrix},\qquad
\mathfrak m_3=\begin{pmatrix} 0&i\\-i&0\end{pmatrix}.
\]
Recall that in order for $\eta$ to be DPW potential, we need $\Res_{\lambda=0}\eta$ to be nilpotent. Using an index $-1$ to denote the coefficient of $\lambda^{-1}$ in the expansion, we have
\begin{eqnarray*}
\det(\Res_{\lambda=0}\eta)&=&-t^2\sum_{j=1}^3 x_{j,-1}^2\omega_j^2\\
&=&\frac{16 t^2\left(
x_{1,-1}^2\sin^2(2\varphi)z^2
-x_{2,-1}^2\cos^2(\varphi)(z^2-1)^2
+x_{3,-1}^2\sin^2(\varphi)(z^2+1)^2\right)}{(z^4-2\cos(2\varphi)z^2+1)^2}.\end{eqnarray*}
This gives us the equations
$$\begin{cases}
\sin^2(2\varphi)x_{1,-1}^2+2\cos^2(\varphi)x_{2,-1}^2+2\sin^2(\varphi)x_{3,-1}^2=0\\
x_{2,-1}^2\cos^2(\varphi)=x_{3,-1}^2\sin^2(\varphi)\end{cases}$$
which determines
$$(x_{1,-1},\;x_{2,-1},\;x_{3,-1})=r(\ii,\pm \sin(\varphi),\pm \cos(\varphi)), \quad r\in\R.$$
We choose $r=\frac{1}{2}$ and the following signs:
\begin{equation}
\label{eq:negative-part}
(x_{1,-1},\;x_{2,-1},\;x_{3,-1})=\tfrac{1}{2}(\ii,-\sin(\varphi),-\cos(\varphi)).
\end{equation}
There is no loss in generality in fixing $r$ because of the time parameter $t,$ and choosing other signs corresponds to conjugating the potential by $C$ or $D$.
\subsection{The Monodromy Problem}$\;$\\
The symmetry $\sigma$ indicates that the Sym-points, denoted by $\lambda_1,\lambda_2 \in \S^1$ in the following, should be complex conjugate to each other. Hence, our ansatz for the Sym-points is
\[\lambda_1(\theta) :=e^{ \ii\theta}\quad \text{ and }\quad  \lambda_2(\theta):=e^{-\ii\theta}\]
for some parameter $\theta\in\R.$
Let $\wt\Sigma$ be the universal cover of $\Sigma$ and $\Phi_{t,x}:\wt\Sigma\to\Lambda SL(2,\C)$ be the solution of the Cauchy Problem
\begin{equation}\label{eqn-sol}
d_\Sigma \Phi_{t,x} = \Phi_{t,x} \eta_{t,x} \quad \text{ with initial condition } \quad \Phi_{t,x}(z=0)=\Id.
\end{equation}
Let $\gamma_1,\cdots,\gamma_4$ be generators of the fundamental group
$\pi_1(\Sigma,0)$, with
$\gamma_j$ enclosing only the singularity $p_k$ and
$\gamma_1\gamma_2\gamma_3\gamma_4=1$. Let
$M_k(t,x)=\cal{M}(\Phi_{t,x},\gamma_j)$ be the monodromy of $\Phi_{t,x}$ along $\gamma_j$.
Following \cite{HHT}, the goal is to solve the following Monodromy Problem:

\begin{equation}
\label{monodromy-problem2}
\left\{\begin{array}{l}
(i) \quad\exists U\in\Lambda SL(2,\C),\;\forall j,\;U^{-1}M_j U\in\Lambda SU(2)\;;\\ 
(ii) \quad\exists s>0,\;\forall j,\;M_j\mbox{ has constant eigenvalues $e^{\pm 2\pi\ii s}$}\;;\\
(iii) \quad\exists \theta\in\R, \;\forall j,\;M_j(\lambda=e^{\pm \ii\theta})\mbox{ is diagonal}\;.
\end{array}\right.
\end{equation}
The first point means that the monodromy representation is unitarizable.
Point (iii) ensures that the monodromies commute at the Sym-points.
If Point (ii) is satisfied and $s=\frac{1}{2g+2}$, $\eta_{t,x}$ lifts to $M_{g,\varphi}$ to
a potential $\wt{\eta}_{t,x}$ with monodromies $-\Id$ around the points $\wt{p}_j=\pi^{-1}(p_j)$.
We will see that the points $\wt{p}_j$ are in fact apparent singularities, as desired.
This therefore yields a closed CMC surface $f \colon M_{g,\varphi} \rightarrow \S^3$.

Because of the symmetries imposed on the potential, it suffices to solve Problem \ref{monodromy-problem2} for $j=1$.
We solve this problem in Section \ref{section:IFT} using the implicit function theorem at $t=0$. This will determine
the parameter $x=(x_1,x_2,x_3)$, the unitarizer $U$, $s$ and the Sym-point angle $\theta$ as functions of $t,$ for $t \sim 0$ small enough.

\subsection{Half-trace coordinates}
\label{section:half-trace-coord}
Our goal is to reformulate the Monodromy Problem \eqref{monodromy-problem2}
in terms of traces using Theorem \ref{characterLLLL}.
In the following we will denote by $\mathcal U\subset\C P^1$ the simply connected domain obtained by removing the radial rays from $p_j$ to $\infty$, for $j = 1, ...4$, see Figure \ref{domain}.

 \begin{figure}[h]
\centering  \vspace{-0.8cm}
 \includegraphics[height=0.25\textwidth]{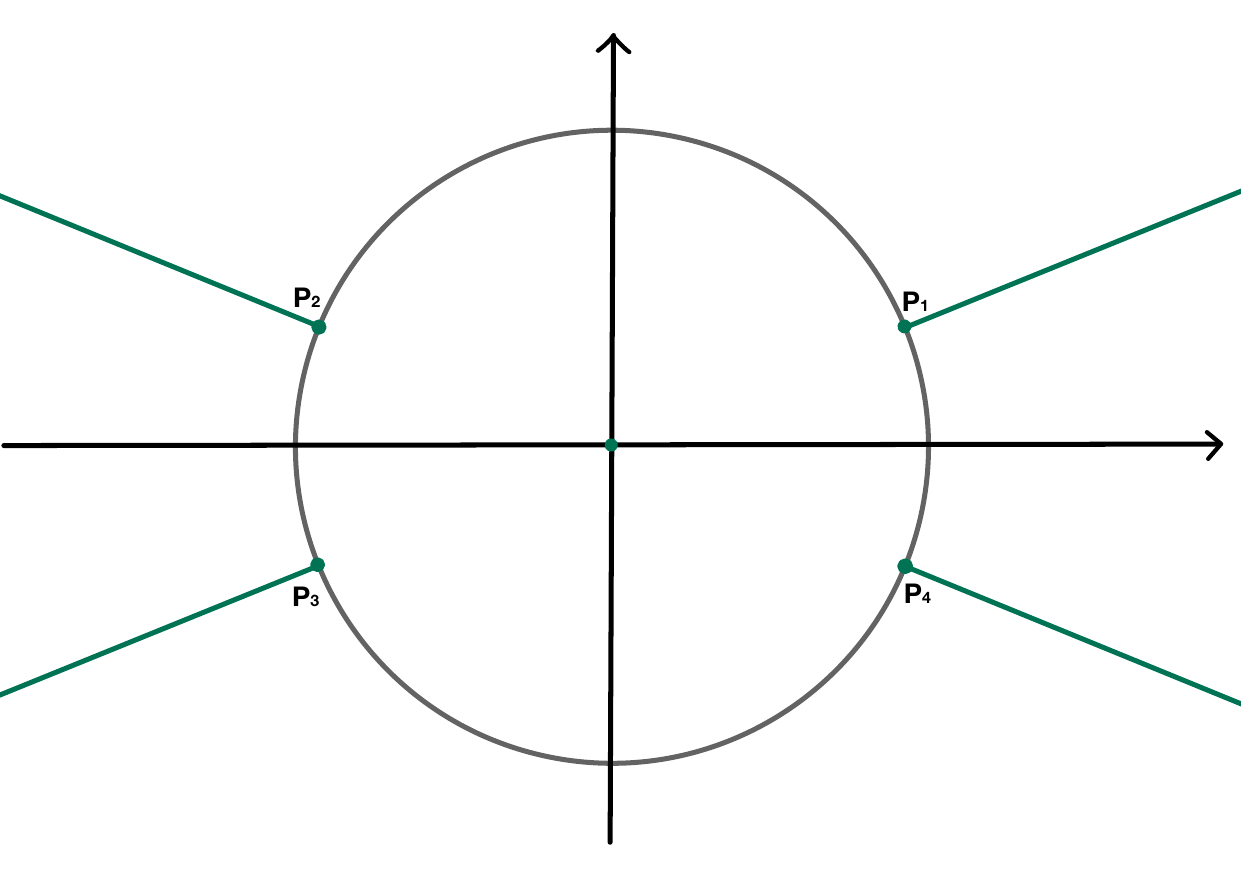} 
  \vspace{-0.5cm}
 \caption{The simply connected domain $\mathcal U$ is $\C$ with the green lines removed.}\label{domain}
\end{figure}

In this section $\Phi=\Phi_{t,x}$ denotes the solution of $d\Phi=\Phi\eta$ in $\mathcal U$ with initial condition $\Phi(0)=\Id$.
Define
$$\mathcal P=\mathcal P(t,x)=\Phi(z=1)\quad\text{ and }\quad
\mathcal Q=\mathcal Q(t,x)=\Phi(z=\ii).$$

 \begin{figure}[h]
\centering
\includegraphics[width=0.75\textwidth]{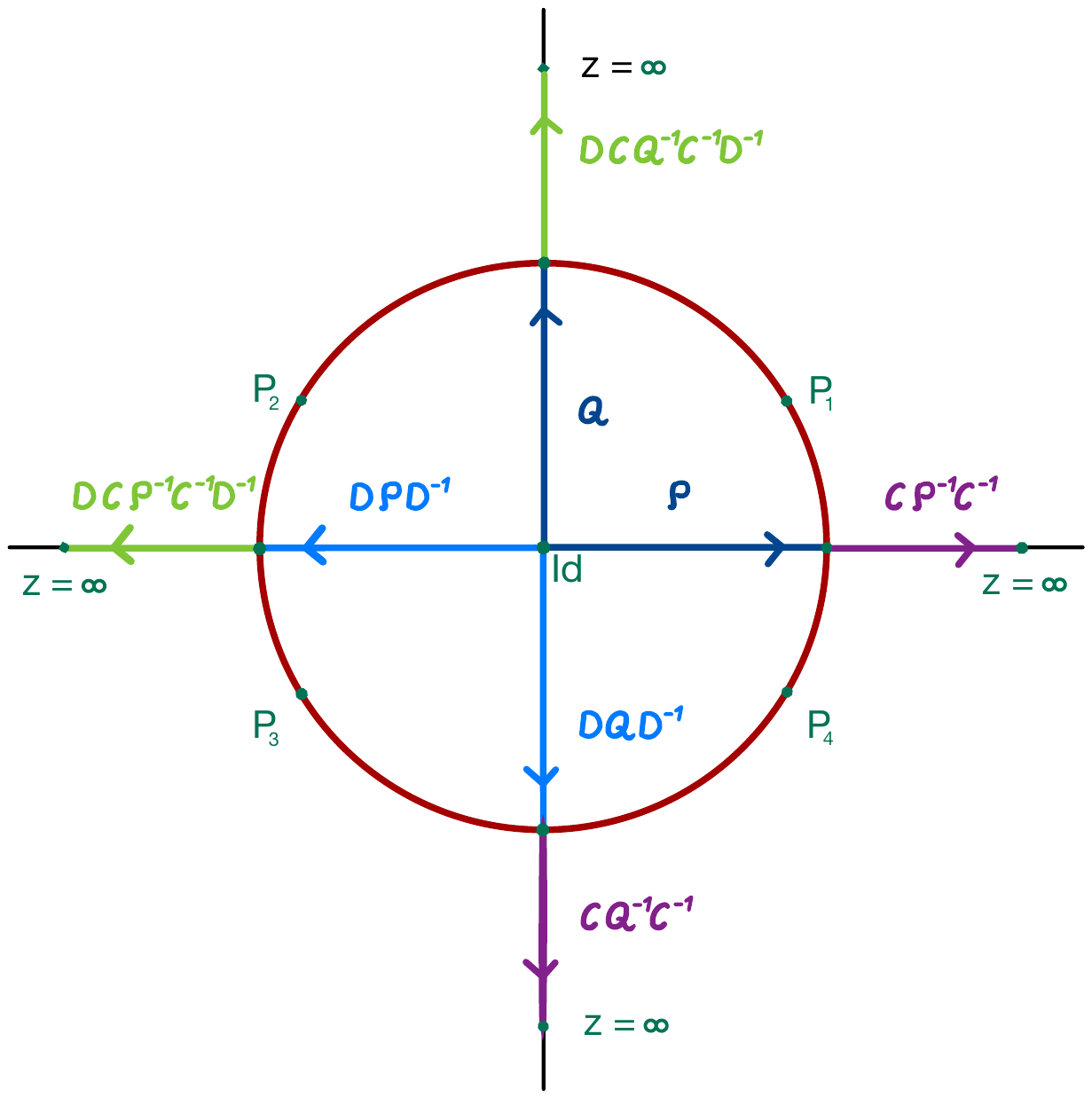}
\hspace{0.25cm}
\caption{
\footnotesize The figure shows a cartoon of the $4$-punctured sphere $\Sigma$ with punctures $p_1, ..., p_4,$ where the points $z= \infty$ need to be identified.    The initial value of $\Phi$ at $z=0$ is $\Id$. The extended frame $\Phi$ along the dark blue curves are given by $\mathcal P$ and $\mathcal Q$ respectively. The principal solution along the light blue curves are given by the principal solution along the dark blue curves together with the symmetry $\delta.$ The $x$-axis and $y$-axis each consists of a dark blue, light blue, green and purple line segment, and on $\Sigma$ each axis is homotopic to a curve around two punctures $p_1, p_2$ and $p_2, p_3,$ respectively.
}
\label{fig:mon}
\end{figure}
\begin{proposition}\label{Pro:monoPQ}
The monodromies $M_1$, $M_2$, $M_3$ and $M_4$ are given by
\begin{align*}
&M_1=-\mathcal P C\mathcal P^{-1}D\mathcal Q CD\mathcal Q^{-1}\\
&M_2=\mathcal Q D C \mathcal Q^{-1} \mathcal P C\mathcal P^{-1}D\\
&M_3=D^{-1} M_1 D\\
&M_4=D^{-1} M_2 D
\end{align*}
where the matrices $C$ and $D$ are defined in Section \ref{ssec:thepotential}.
\end{proposition}
\begin{proof}
Recall that $\eta$ is not singular at $z=\infty$.
Denote by $\Phi(+\infty)$, $\Phi(+\ii\infty)$, $\Phi(-\infty)$ and $\Phi(-\ii\infty)$
the value of $\Phi$ at $\infty$ obtained by analytic continuation along
the positive real axis, the positive imaginary axis, the negative real axis and the negative imaginary axis, respectively, see Figure \ref{fig:mon}.
Using the $\delta$ and $\tau$-symmetries, we obtain
\[\begin{split}
\Phi(+\infty)&=\mathcal PC\mathcal P^{-1}C^{-1},\\
\Phi(+\ii\infty)&=\mathcal QDC\mathcal Q^{-1}C^{-1}D^{-1},\\
\Phi(-\infty)&=D\mathcal PC\mathcal P^{-1}C^{-1}D^{-1}.\\
\end{split}\]
Then
$$M_1=\Phi(+\infty)\Phi(+\ii\infty)^{-1}=\mathcal PC\mathcal P^{-1}C^{-1}DC\mathcal Q C^{-1}D^{-1}\mathcal Q^{-1}$$
$$M_2=\Phi(+\ii\infty)\Phi(-\infty)^{-1}=\mathcal QDC\mathcal Q^{-1}C^{-1}D^{-1}C\mathcal PC^{-1}\mathcal P^{-1}$$
which simplify to the formulas of Proposition \ref{Pro:monoPQ} using
$DC=-CD$ and $C^2=D^2=-\Id$.
The formulas for $M_3$ and $M_4$ follow from the $\delta$-symmetry.
\end{proof}
We will apply Theorem \ref{characterLLLL} to the representation given by
\begin{equation}\label{defLs}
L_1=\mathcal QCD\mathcal Q^{-1},\quad
 L_2=D,\quad 
 L_3=\mathcal PC\mathcal P^{-1}
\quad\text{and}\quad L_4=M_1=-\mathcal P C\mathcal P^{-1}D\mathcal Q CD\mathcal Q^{-1}.\end{equation}
Then $L_1 L_2 L_3 L_4=\Id$ and $\tr(L_1)=\tr(L_2)=\tr(L_3)=0$.
Using Proposition \ref{Pro:monoPQ}, we thus obtain
\begin{equation}
\label{eq:MiLi}
M_1=L_4,\quad M_2=L_1^{-1}L_4L_1,\quad M_3=L_2^{-1}L_4L_2,\quad M_4=L_2^{-1}L_1^{-1}L_4L_1L_2.
\end{equation}
Observe that if $L_1,\dots,L_4$ are simultaneously unitarizable, i.e, there exist a $U\in \text{SL}(2, \C)$ such that $U L_k U^{-1} \in \text{SU}(2)$, then $M_1,\dots,M_4$ are also simultaneously unitarizable.
In view of Theorem \ref{characterLLLL}, we define the half-trace coordinates $\mathfrak p,\mathfrak q,\mathfrak r$ by
$$\mathfrak p(t,x)=\tfrac{1}{2}\tr(L_2 L_3),\quad
\mathfrak q(t,x)=-\tfrac{1}{2}\tr(L_1 L_2)\quad\text{ and }\quad
\mathfrak r(t,x)=-\tfrac{1}{2}\tr(L_2 L_4).$$
Moreover, let
$$\mathcal K(x)=-\det(A_1)=x_1^2+x_2^2+x_3^2.$$
\begin{proposition}
\label{prop:monodromy-reduction}
Let $t>0$ and $x$ be the parameter vector of the DPW potential $\eta$ such that:
\begin{enumerate}[(i)]
\item $\mathfrak p(t,x)$ and $\mathfrak q(t,x)$ are real along the unit circle $\lambda\in\S^1$.
\item $\mathcal K(x)$ is a positive constant (with respect to $\lambda$) which satisfies
\begin{equation}
\label{eq:hypothesisK}
0<t \sqrt{\mathcal K(x)}<\frac{1}{4}.\end{equation}
\item There exists $\lambda_1=e^{\ii\theta}\in\S^1$ such that $\mathfrak p(t,x)(\lambda_1)=\mathfrak q(t,x)(\lambda_1)=0$.
\end{enumerate}
Then there exists $U\in\Lambda SL(2,\C)$ such that $U L_j U^{-1}\in\Lambda SU(2)$ for all $j = 1, ... 4$.
The unitarizer $U$ is unique up to left multiplication by $\Lambda SU(2)$.
Moreover, the Monodromy Problem \eqref{monodromy-problem2} is solved.
\end{proposition}
\begin{proof}
Let
$$s=t\sqrt{\mathcal K}.$$
Then the residue $tA_1$ of the potential at $p_1$ has eigenvalues $\pm s$.
By the hypothesis \eqref{eq:hypothesisK}, the Fuchsian system given by $d\Phi=\Phi\eta$ is non-resonant at $p_1$,
so the monodromy $M_1$ has eigenvalues $e^{\pm 2\pi \ii s}$, proving Point (ii) of Problem \eqref{monodromy-problem2}. Moreover,
$$\traceL=\tr(L_4)=\tr(M_1)=2\cos(2\pi s)\in (-2,2).$$
By Theorem \ref{characterLLLL}, the half-traces coordinates $\mathfrak p,\mathfrak q,\mathfrak r$ satisfy the equation (after dividing by 4)
\begin{equation}
\label{eq:quadratic}
\mathfrak p^2 +\mathfrak q^2 +\mathfrak r^2 +2 \mathfrak p\mathfrak q\mathfrak r -1+\cos^2(2\pi s) =0.
\end{equation}
The function $\mathfrak r$ is a well-defined holomorphic function of $\lambda$ in a neighborhood of the unit circle. On the other hand,
$\mathfrak r$ is given as the solution of
the quadratic polynomial \eqref{eq:quadratic}.
Since $\mathfrak p$ and $\mathfrak q$ are real along the unit circle, this polynomial \eqref{eq:quadratic} has real coefficients,
and since $\mathfrak r$ is well-defined and holomorphic in $\lambda$, its discriminant
  \[\Delta=4(1-\mathfrak p^2)(1-\mathfrak q^2)-4\cos^2(2\pi s)\]
  has constant sign on the unit-circle (its zeros, if any, must have even order).
  At $\lambda_1$, we have $\Delta=4\sin^2(2\pi s)>0$, so $\Delta\geq 0$ on the unit circle.
  Hence $\mathfrak r$ must be real on the unit circle.
If $p^2=1$ or $q^2=1$, then 
$\Delta<0$. Hence, $|\mathfrak p|$ and $|\mathfrak q|$ are bounded by 1 along $\S^1$.
So $\tr(L_1 L_2)\in (-2,2)$, $\tr(L_2 L_3)\in (-2,2)$ and $\tr(L_2 L_4)\in\R$.
By Theorem \ref{characterLLLL}, there exists $U\in \Lambda SL(2,\C)$ such that $U L_j U^{-1}\in \Lambda SU(2)$
for $j= 1, ..., 4$. Moreover, the unitarizer is diagonal by Proposition \ref{prop:symmetriesU}, depends smoothly on $\lambda$ along the unit circle by Theorem \ref{unitarizer}, and unique up to left multiplication by $\Lambda SU(2)$.

Thus point (i) of Problem \eqref{monodromy-problem2} follows from Equation \eqref{eq:MiLi}.
Finally, to prove Point (iii) of Problem \eqref{monodromy-problem2}, we compute
\[\Phi(+\infty)=\mathcal PC\mathcal P^{-1}C^{-1}=\begin{pmatrix}\mathcal P_{11}^2-\mathcal P_{12}^2&-\mathfrak p\\\mathfrak p&\mathcal P_{22}^2-\mathcal P_{21}^2\end{pmatrix}\]
\[\Phi(\ii\infty)=\mathcal Q DC\mathcal Q^{-1}C^{-1}D^{-1}=\begin{pmatrix}\mathcal Q_{11}^2+\mathcal Q_{12}^2&-\ii\mathfrak q\\ -\ii\mathfrak q& \mathcal Q_{22}^2+\mathcal Q_{21}^2\end{pmatrix}.\]
Hence both are diagonal at $\lambda_1$ and so is $M_1$.
By Proposition \ref{prop:pq} below, we have
$\mathfrak p(\overline{\lambda_1})=\overline{\mathfrak p(\lambda_1)}=0$ and
similarly
$\mathfrak q(\overline{\lambda_1})=0$, so $M_1$ is diagonal at $\lambda_2=e^{-\ii\theta}$, proving Point (iii)
of Problem \eqref{monodromy-problem2}.
\end{proof}

\begin{proposition}
\label{prop:pq}\mbox{}
\begin{enumerate}
\item The half-trace coordinate $\mathfrak p$ and $\mathfrak q$ can be expressed in terms of the entries of $\mathcal P$ and $\mathcal Q$ by
$$\mathfrak p=\mathcal P_{11}\mathcal P_{21}-\mathcal P_{12}\mathcal P_{22}\quad\text{and}\quad
\mathfrak q =\ii (\mathcal Q_{11}\mathcal Q_{21}+\mathcal Q_{12}\mathcal Q_{22}).$$
\item They satisfy $\bar{\mathfrak p}=\mathfrak p$ 
(which means $\overline{\mathfrak p(\overline{\lambda})}=\mathfrak p$) and $\bar{\mathfrak q}=\mathfrak q$.
\item At $t=0$, we have $\mathfrak p(0,x)=\mathfrak q(0,x)=0$ and the derivatives of $\mathfrak p$ and $\mathfrak q$ with respect to $t$ at $t=0$ are given by
$$\mathfrak p'(0,x)=2\pi x_3\quad\text{ and }\quad
\mathfrak q'(0,x)=2\pi x_2.$$
\end{enumerate}
\end{proposition}
\begin{proof} The first point is a direct computation using the definitions of $\mathfrak p$ and $\mathfrak q$.
 Recalling the symmetries of the potential, we have $\sigma^*\overline{\Phi}=\Phi$ so
 $\overline{\mathcal P}=\mathcal P$, which gives
$\overline{\mathfrak p}=\mathfrak p$. In the same way,
$(\delta\circ\sigma)^*\overline{\Phi}=D^{-1}\Phi D$ so  $\overline{\mathcal Q}=D^{-1}\mathcal Q D$.
This gives $\overline{\mathfrak q}=\mathfrak q$.

At $t=0$, we have $\mathcal P(0,x)=\mathcal Q(0,x)=\Id$ so $\mathfrak p(0,x)=\mathfrak q(0,x)=0$.
Differentiating the equation $d_{\Sigma}\Phi_{t,x}=\Phi_{t,x}\eta_{t,x}$ with respect to $t$ we obtain
$$d_{\Sigma}\Phi'_{0,x}=\eta'_{0,x}
=\sum_{j=1}^3 x_j\mathfrak m_j\omega_j.$$
Integrating, we obtain for $z\in\mathcal U$ --- the simply connected domain in $\C P^1$ obtained by removing the radial rays from $p_k$ to $z=\infty$ ---
$$\Phi'_{0,x}(z)=\int_0^z\eta'_{0,x}=\matrix{x_1\Omega_1(z)&x_2\Omega_2(z)+\ii x_3\Omega_3(z)\\
x_2\Omega_2(z)-\ii x_3\Omega_3(z)&-x_1\Omega_1(z)}$$
where
$$\Omega_j(z)= \int_0^z \omega_j.$$
For further reference, we list the values of all $\Omega_j$ at $z=1$ and $z=\ii$
(which are easily computed using complex logarithms)
\begin{equation}\label{eq:Omega1}
\begin{split}
\Omega_1(1)=\ii(\pi-2\varphi), \quad 
\Omega_2(1)=\log\left(\frac{1-\cos(\varphi)}{1+\cos(\varphi)}\right),\quad
\Omega_3(1)=\ii \pi
\end{split}
\end{equation}
\begin{equation}\label{eq:OmegaI}
\begin{split}
\Omega_1(\ii)=-2\ii\varphi,\quad
\Omega_2(\ii)=-\ii\pi,\quad
\Omega_3(\ii)=\log\left(\frac{1-\sin(\varphi)}{1+\sin(\varphi)}\right).
\end{split}
\end{equation}
This gives
\begin{equation*}
\begin{split}
\mathfrak p'(0,x)&=\mathcal P_{21}'(0,x)-\mathcal P_{12}'(0,x)=-2\ii x_3\,\Omega_3(1)=2\pi x_3\\
\mathfrak q'(0,x)&=\ii(\mathcal Q_{21}'(0,x)+\mathcal Q_{12}'(0,x))=2\ii x_2\,\Omega_2(\ii)=
2\pi x_2.
\end{split}
\end{equation*}
\end{proof}
Using Iwasawa decomposition, we may choose a positive unitarizer $U\in\Lambda^+_{\R} SL(2,\C)$.
It then has additional properties following from the symmetries:
\begin{proposition}
\label{prop:symmetriesU}
The unitarizer $U$ is diagonal and satisfies $U=\overline{U}$, i.e. $\overline{U(\overline{\lambda})}=U(\lambda)$.
\end{proposition}
\begin{proof}
Since $U$ unitarizes $L_2$, we have
$$U D U^{-1} D^{-1}=U L_2 U^{-1} D^{-1}\in \Lambda SU(2)$$
and since $D$ is constant diagonal, also
$$U D U^{-1} D^{-1}\in \Lambda^+_{\R} SL(2,\C).$$
Hence $UDU^{-1}D^{-1}=\Id$, so the unitarizer $U$ is diagonal.

For the second point, we prove that $\overline{U}$ unitarizes all $L_j$.
From the $\sigma$ symmetry, we have $M_1=\overline{M_4}^{-1}$, so
$$\overline{U}L_4\overline{U}^{-1}=\overline{U} M_1\overline{U}^{-1}=\overline{ U M_4^{-1} U^{-1}}\in\Lambda SU(2).$$
Since $\mathcal P=\overline{\mathcal P}$,
$$\overline{U} L_3\overline{U}^{-1}=\overline{U} \mathcal P C \mathcal P^{-1}\overline{U}^{-1}
=-\overline{ U \mathcal P C \mathcal P^{-1} U^{-1}}=-\overline{ U L_3 U^{-1}}\in\Lambda SU(2).$$
In the same way, using $\mathcal Q= D \overline{\mathcal Q}D^{-1}$ and that $U$ is diagonal, we have
$$\overline{U} L_1\overline{U}^{-1}=\overline{U}(D\overline{Q}D^{-1})CD(D\overline{Q}^{-1}D^{-1})\overline{U}^{-1}
=-D\overline{U L_1 U^{-1}}D^{-1}\in\Lambda SU(2).$$
Finally, since $D$ is diagonal, $\overline{U}L_2 \overline{U}^{-1}=D\in\Lambda SU(2)$.
By uniqueness of the unitarizer of $L_1,\cdots, L_4$ up to left multiplication by $\Lambda SU(2)$, we have $\overline{U}U^{-1}\in\Lambda SU(2)$. On the other hand, $\overline{U}U^{-1}\in\Lambda^+_{\R} SL(2,\C)$, so $\overline{U}U^{-1}=\Id$.
\end{proof}

\subsection{Solving the Monodromy Problem for small $t$ using implicit functions}
\label{section:IFT}
If $f$ is an analytic function of $t$ such that $f(0)=0$, we denote by $\wh{f}$ the analytic function
\[\wh{f}(t)=\left\{\begin{array}{ll}
t^{-1}f(t)\quad&\text{ if } t\neq 0\\
f'(0)&\text{ if } t=0.\end{array}\right.\]
Consider the analytic functions $\wh{\mathfrak p}$ and $\wh{\mathfrak q}$ and  define
\[
\begin{split}
&\mathcal F_1(t,x)=\wh{\mathfrak p}(t,x)-\wh{\mathfrak p}(t,x)^*\\
&\mathcal F_2(t,x)=\wh{\mathfrak q}(t,x)-\wh{\mathfrak q}(t,x)^*\\
&\mathcal H_1(t,x,\theta)=\Re\left(\wh{\mathfrak p}(t,x)\mid_{\lambda=e^{\ii\theta}}\right)\\
&\mathcal H_2(t,x,\theta)=\Re\left(\wh{\mathfrak q}(t,x)\mid_{\lambda=e^{\ii\theta}}\right),
\end{split}
\]
where $\theta$ is a real parameter. We shall solve the following problem:
\begin{equation}
\label{monodromy-problem3}
\left\{\begin{array}{l}
(i)\quad\;\;\mathcal F_1(t,x)=\mathcal F_2(t,x)= 0\\ 
(ii)\quad\;\mathcal K(x)\text{ is constant with respect to $\lambda$}.\\
(iii)\quad\;\mathcal H_1(t,x,\theta)=\mathcal H_2(t,x,\theta)=0.
\end{array}\right.
\end{equation}
Note that $\mathcal F_1=0$ is equivalent to  $\mathfrak p$ being real on the unit circle.
In this case, $\mathfrak p(e^{\ii\theta})\in\R$ and thus $\mathcal H_1=0$ is equivalent to  $4\mathfrak p(e^{\ii\theta})=0$.
Hence, by Proposition \ref{prop:monodromy-reduction}, a solution $(t,x,\theta)$ of Problem \eqref{monodromy-problem3} gives a solution of the Monodromy Problem \eqref{monodromy-problem2}, provided $\mathcal K$ satisfies the bound \eqref{eq:hypothesisK}.
\begin{proposition}
\label{prop:symmetryFGK}
The maps $\mathcal F_i$ and $\mathcal K$ have the following symmetries:
\[\mathcal F_i=-\mathcal F_i^* \quad\text{ and }\quad\overline{\mathcal F_i}=\mathcal F_i \]
\[\overline{\mathcal K}=\mathcal K.\]
 In particular, the equation $\mathcal F_i= 0$ is equivalent to $\mathcal F_i^+=0$. 
\end{proposition}
\begin{proof}
The proposition follows directly from the fact that $*$ is an involution, the symmetries of $\mathfrak p$ and  $\mathfrak q$ in Proposition \ref{prop:pq} and $\overline{x_i}=\pm x_i$.
\end{proof}
\begin{proposition}
\label{prop:central}
At $t=0$, the Monodromy Problem \eqref{monodromy-problem3} has a unique solution
such that $\mathcal K>0$. It is given by
\begin{align*}
x_1&=\tfrac{i}{2}(\lambda^{-1}-\lambda):=\cv{x}_1\\
x_2&=-\tfrac{1}{2}\sin(\varphi)(\lambda^{-1}+\lambda):=\cv{x}_2\\
x_3&=-\tfrac{1}{2}\cos(\varphi)(\lambda^{-1}+\lambda):=\cv{x}_3\\
\theta&=\tfrac{\pi}{2}:=\cv{\theta}
\end{align*}
and satisfies $\mathcal K=1$.
We use an underscore to denote the value of the parameters at $t=0$ and call it the central value.
\end{proposition}
\begin{proof}
Remember that we have fixed the negative part of each parameter $x_i$ in \eqref{eq:negative-part}.
At $t=0$, we have by Proposition \ref{prop:pq}
\begin{equation}
\label{eq:pq0}
\wh{\mathfrak p}(0,x)=2\pi x_3\quad\text{and}\quad
\wh{\mathfrak q}(0,x)=2\pi x_2.\end{equation}
The equation $\mathcal F_1=0$ gives $x_3=x_3^*$, hence $x_3$ is a degree-1 Laurent polynomial:
$$x_3=-\tfrac{1}{2}\cos(\varphi)(\lambda^{-1}+\lambda)+x_{3,0}.$$
Then $\mathcal H_1=0$ gives
$$x_{3,0}=\pi\cos(\varphi)\cos(\theta).$$
In the same way, the equations $\mathcal F_2=0$ and $\mathcal H_2=0$ give
$$x_2=-\tfrac{1}{2}\sin(\varphi)(\lambda^{-1}+\lambda)+\pi\sin(\varphi)\cos(\theta).$$
From $\mathcal K$ being constant, we see that $x_1$ must also be a Laurent polynomial of degree 1 which we write as $x_1=\tfrac{\ii}{2}\lambda^{-1}+x_{1,0}+x_{1,1}\lambda$. Then $\mathcal K$ expands as
$$\mathcal K=\lambda^{-1}(\ii x_{1,0}-\pi\cos(\theta))
+\lambda^0(x_{1,0}^2+\ii x_{1,1}+\pi^2\cos^2(\theta)+\tfrac{1}{2})
+\lambda(2 x_{1,0}x_{1,1}-\pi\cos(\theta))
+\lambda^2(x_{1,1}^2+\tfrac{1}{4}).$$
Since $\mathcal K$ is constant in $\lambda$, we obtain
$$\begin{cases}
x_{1,0}=-\ii\pi\cos(\theta)\\
x_{1,1}=\frac{1}{2}\varepsilon\ii\quad\text{ with }\quad \varepsilon=\pm 1\\
\pi\cos(\theta)(\varepsilon-1)=0\\
\mathcal K=\frac{1}{2}(1-\varepsilon).
\end{cases}$$
The choice $\varepsilon=1$ leads to $\mathcal K=0$ which is excluded. Hence
$\varepsilon=-1$ and $\cos(\theta)=0$.
\end{proof}

\begin{proposition}\label{prop:IFT} For $\varphi\in(0,\tfrac{\pi}{2})$ fixed,
there exists $T=T(\varphi)>0$ such that for $|t|<T$, the Monodromy Problem \eqref{monodromy-problem3} has a unique solution $(x(t),\theta(t))$, analytic in $t$, such that
$(x(0),\theta(0))=(\cv{x},\cv{\theta}).$
\end{proposition}

\begin{proof}
We write $x_i=\cv{x}_i+y_i$ where $\cv{x}_i$ is given by Proposition \ref{prop:central}
and
$$y=(y_1,y_2,y_3)\in \ii\mathcal W^{\geq 0}_{\R}\times \mathcal W^{\geq 0}_{\R}\times \mathcal W^{\geq 0}_{\R}.$$
By Equation \eqref{eq:pq0}, we have at $t=0$
\[\begin{split}
\mathcal F_1(0,x)&=2\pi (y_3-y_3^*)\\
\mathcal F_2(0,x)&=2\pi (y_2-y_2^*)\\
\mathcal H_1(0,x,\theta)&=2\pi \left(-\cos(\varphi)\cos(\theta)+y_3(e^{\ii\theta})\right)\\
\mathcal H_2(0,x,\theta)&=2\pi \left(-\sin(\varphi)\cos(\theta)+y_2(e^{\ii\theta})\right).
\end{split}\]
So the differential of $(\mathcal F^+,\mathcal G^+,\mathcal H_1,\mathcal H_2)$ with respect to $(y,\theta)$ at $(t,y,\theta)=(0,0,\cv{\theta})$ is given by
\[\begin{split}
d\mathcal F_1^+&=2\pi d y_3^+\\
d\mathcal F_2^+&=2\pi d y_2^+\\
d\mathcal H_1&=2\pi\cos(\varphi)d\theta+2\pi d y_3(\ii)\\
d\mathcal H_2&=2\pi\sin(\varphi)d\theta+2\pi d y_2(\ii).
\end{split}\]
The partial differential with respect to $(y_2,y_3)$ is clearly an
isomorphism from $(\mathcal W^{\geq 0}_{\R})^2$ to $(\mathcal W^{>0}_{\R})^2\times\R^2$.
By the implicit function theorem, for $(t,y_1,\theta)$ in a neighborhood of $(0,0,\cv{\theta})$, there exists unique values of $y_2$ and $y_3$
in $\mathcal W^{\geq 0}_{\R}$ solving $\mathcal F_1^+=\mathcal F_2^+=0$ and
$\mathcal H_1=\mathcal H_2=0$.
Moreover, the differential of $y_2(t,y_1,\theta)$ and $y_3(t,y_1,\theta)$ with respect to the remaining parameters $(y_1,\theta)$ at $(t,y_1,\theta)=(0,0,\cv{\theta})$ is given by
$$d y_2=d y_2^0=-\sin(\varphi)d\theta\quad\text{ and }\quad
d y_3=d y_3^0 = -\cos(\varphi) d\theta.$$
Then the differential of $\mathcal K$ with respect to $(y_1,\theta)$ at
$(t,y_1,\theta)=(0,0,\cv{\theta})$ is
$$d\mathcal K=\sum_{j=1}^3 2\cv{x}_i dy_i
=\ii(\lambda^{-1}-\lambda)dy_1+(\lambda^{-1}+\lambda)d\theta.$$
Observe that for all values of the parameter $x$ we have $\lambda\mathcal K(x)\in\mathcal W^{\geq 0}_{\R}$.
Using Proposition \ref{Pro:decomposition} from Appendix \ref{appendix:division}
with $(\mu_1,\mu_2)=(1,-1)$, we decompose $\lambda\mathcal K$ as
$$\lambda\mathcal K(x)=\mathcal K_{-1}(x)+
\lambda\mathcal K_0(x)
+(\lambda^2-1) \wc{\mathcal K}(x)$$
with $(\mathcal K_{-1}(x),\mathcal K_0(x))\in\R^2$ and $\wc{K}(x)\in\mathcal W^{\geq 0}_{\R}$.
We want to solve $\mathcal K_{-1}=0$ and $\wc{\mathcal K}=0$.
From the formula for $d\mathcal K$ we obtain
\[\begin{split}
&d\mathcal K_{-1}=2d\theta\\
&d\wc{\mathcal K}=-\ii d y_1+d\theta.\end{split}\]
The partial derivative of $(\wc{\mathcal K},\mathcal K_{-1})$
with respect to $(y_1,\theta)$ is an isomorphism
from $\ii\mathcal W^{\geq 0}_{\R}\times\R$ to $\mathcal W^{\geq 0}_{\R}\times\R$.
Hence the implicit function theorem gives rise to a $T>0$ such that there exists unique 
$(y_1(t),\theta(t))\in\ii \mathcal W^{\geq 0}_{\R}\times\R$ in a neighborhood of
$(0,\cv{\theta})$ with $\mathcal K\in\mathcal W^0$ for all $t$ with $|t|<T$.
\end{proof}
\begin{remark}
A priori, $T$ depends on the angle $\varphi$.
However, the partial derivatives, computed in the above proof, remain invertible in the limit cases
$\varphi=0$ and $\varphi=\pi/2$.
So provided the functions $\mathcal F_i$ and $\mathcal H_i$ extend smoothly to $\varphi =0$ and $\varphi=\pi/2$, the implicit function theorem will give a uniform $T>0$ such that
 for all $\varphi\in[0,\pi/2]$ and $|t|<T$ the Problem \eqref{monodromy-problem3} has a unique solution
 $(x(t),\theta(t))$.
We prove that this is indeed the case in Section \ref{limitvarphi}.
\end{remark}
\begin{proposition}
\label{prop:symmetries}\begin{enumerate}
\item The solution $(x(t),\theta(t))$ given by Proposition \ref{prop:IFT} has the following
parity with respect to $t$:
\[x(-t)(-\lambda)=-x(t)(\lambda)\quad\text{ and }\quad
\theta(-t)=\pi-\theta(t).\]
\item Moreover, $x(t,\varphi)$ and $\theta(t,\varphi)$ have the following symmetries as a function of $\varphi$,
\begin{equation*}
\begin{split}
&x_1(-t,\tfrac{\pi}{2}-\varphi)=x_1(t,\varphi),\quad
x_2(-t,\tfrac{\pi}{2}-\varphi)=x_3(t,\varphi),\quad
x_3(-t,\tfrac{\pi}{2}-\varphi)=x_2(t,\varphi)\\
&\theta(-t,\tfrac{\pi}{2}-\varphi)=\theta(t,\varphi).
\end{split}
\end{equation*}
\end{enumerate}
In particular, in the case of $\varphi=\tfrac{\pi}{4}$ combining (1) and (2) gives
$$x_1(t)(\lambda)=-x_1(t)(-\lambda),\quad
x_2(t)(\lambda)=-x_3(t)(-\lambda),\quad\text{ and }\quad
\theta(t)=\frac{\pi}{2},$$
thus $\varphi=\tfrac{\pi}{4}$ yields a family of minimal surfaces.
\end{proposition}
\begin{proof}
Given parameters $(x,\theta)$, consider the parameters $(\widefrown{x},\widefrown{\theta})$ defined by
$$\widefrown{x}(\lambda)=-x(-\lambda),\quad
\quad\text{and}\quad
\widefrown{\theta}=\pi-\theta.$$
Note that at the central value,
$\widefrown{\cv{x}}=\cv{x}$ and $\widefrown{\cv{\theta}}=\cv{\theta}$.
We then have by inspection
\begin{align*}
&\eta_{-t,\widefrown{x}}(-\lambda)=\eta_{t,x}(\lambda)\\
&\Phi_{-t,\widefrown{x}}(-\lambda)=\Phi_{t,x}(\lambda)\\
&\wh{\mathfrak p}(-t,\widefrown{x})(-\lambda)=-\wh{\mathfrak p}(t,x)(\lambda)\\
&\mathcal F_1(-t,\widefrown{x})(-\lambda)=-\mathcal F_1(t,x)(\lambda)\\
&\mathcal H_1(-t,\widefrown{x},\widefrown{\theta})
=\Re\,\widehat{\mathfrak p}(-t,\widefrown{x})(e^{\ii\widefrown{\theta}})
=-\Re\,\widehat{\mathfrak p}(t,x)(-e^{\ii(\pi-\theta)})
=-\Re\,\overline{\widehat{\mathfrak p}(t,x)(e^{\ii\theta})}
=-\mathcal H_1(t,x,\theta)\\
&\mathcal K(\widefrown{x})(-\lambda)=\mathcal K(x)(\lambda)
\end{align*}
and similar statements hold for $\widehat{\mathfrak q}$, $\mathcal F_2$ and $\mathcal H_2$.
Therefore, if $(t,x,\theta)$ solves the Monodromy Problem \eqref{monodromy-problem3} then $(-t,\widefrown{x},\widefrown{\theta})$ also solves the Monodromy Problem.
By uniqueness in the implicit function theorem, $\widefrown{x}(-t)=x(t)$ and
$\widefrown{\theta}(-t)=\theta(t)$ from which
point (1) follows.
\medskip

To prove point (2), consider this time the parameter $\widefrown{x}$ and the angle $\widefrown{\varphi}$ defined by
$$\widefrown{x}_1=x_1,\quad
\widefrown{x}_2=x_3,\quad
\widefrown{x}_3=x_2\quad\text{and}\quad
\widefrown{\varphi}=\tfrac{\pi}{2}-\varphi.$$
Note that at the central value,
$\widefrown{x}(0,\widefrown{\varphi})=x(0,\varphi)$. Let $\iota(z)=\ii z$.
To state the dependance of objects on the angle $\varphi$, we add a subscript to all objects, for example we write
$p_{j,\varphi}$, $\omega_{j,\varphi}$ and $\eta_{t,x,\varphi}$.
Then $\iota(p_{j,\varphi})=p_{j+1,\widefrown{\varphi}}$, where the indices are considered mod 4. This gives
$$\iota^*\left(\frac{dz}{z-p_{j+1,\widefrown{\varphi}}}\right)=\frac{\ii dz}{\ii z-\ii p_{j,\varphi}}=\frac{dz}{z-p_{j,\varphi}}$$
hence
\begin{equation}
\label{eq:iota*omega}
\iota^*\omega_{1,\widefrown{\varphi}}=-\omega_{1,\varphi},\quad
\iota^*\omega_{2,\widefrown{\varphi}}=-\omega_{3,\varphi}\quad\text{and}\quad
\iota^*\omega_{3,\widefrown{\varphi}}=\omega_{2,\varphi}.
\end{equation}
Therefore, we obtain for the pull-back of the potential under $\iota$
\begin{eqnarray*}
\iota^*\left(\eta_{-t,\widefrown{x},\widefrown{\varphi}}\right)
&=&-t\begin{pmatrix} x_1 (-\omega_{1,\varphi})&
x_3(-\omega_{3,\varphi})+\ii x_2 \omega_{2,\varphi}\\
x_3(-\omega_{3,\varphi})-\ii x_2\omega_{2,\varphi}& - x_1(-\omega_{1,\varphi})\end{pmatrix}
=S^{-1}\eta_{t,x,\varphi}S
\end{eqnarray*}
with
$$S=\begin{pmatrix}e^{\ii\pi/4}&0\\0&e^{-\ii\pi/4}\end{pmatrix}.$$
The same holds for the extended frame
$$\iota^*\Phi_{-t,\widefrown{x},\widefrown{\varphi}}=S^{-1}\Phi_{t,x,\varphi}S.$$
Evaluating at $z=1$ we obtain 
$$\mathcal Q(-t,\widefrown{x},\widefrown{\varphi})=S^{-1} \mathcal P(t,x,\varphi)S$$
which gives
$$\mathfrak q(-t,\widefrown{x},\widefrown{\varphi})=-\mathfrak p(t,x,\varphi).$$

Consequently, if $(t,x,\theta)$ solves the Monodromy Problem \eqref{monodromy-problem3} for the angle $\varphi$ then $(-t,\widefrown{x},\theta)$ solves the Monodromy Problem for the angle $\widefrown{\varphi}$ and point (2) also follows by the uniqueness in the implicit function theorem.
\end{proof}

\subsection{Building the surface}
\label{section:building}
In this section, we fix $\varphi\in(0,\frac{\pi}{2})$ and consider the solution $(x(t,\varphi),\theta(t,\varphi))$ of Problem \eqref{monodromy-problem3}
given by Proposition \ref{prop:IFT}. Let
$$\psi(t)=t\sqrt{\mathcal K(x(t,\varphi))}.$$
At $t=0$, we have $\mathcal K(\cv{x})=1$ so that $\psi'(0)=1$. Hence $\psi$ is a diffeomorphism in a neighborhood of $0$.
\begin{proposition}\label{prop:building}
For $g\in\N$ large enough let
\begin{equation}
\label{eq:t}
t=\psi^{-1}\left(\tfrac{1}{2g+2}\right).
\end{equation}
Then the lift of $\eta_{t,x(t,\varphi)}$ to $M_{g,\varphi}$ has apparent singularities at $\wt p_j$ and defines a conformal CMC immersion $f_{g,\varphi}:M_{g,\varphi}\to\S^3$ with the following properties:
\begin{enumerate}
\item $f_{g,\varphi}$ has mean curvature
$H=\operatorname{cotan}(\theta(t,\varphi)).$
\item The Willmore energy of $f_{g,\varphi}$ is given by
$$\mathcal W(f_{g,\varphi})=8\pi\left[1-\mathcal K(x(t,\varphi))^{-1/2}\left(\cos(\varphi)x_2^0(t,\varphi)-\sin(\varphi)x_3^0(t,\varphi)\right)\right].$$
\item  Up to an isometry of $\S^3$, the image of the quarter disk $\{\Re(z)\geq 0$, $\Im(z)\geq 0$, $|z|\leq 1\}$ under $f_{g,\varphi}$
is bounded by four symmetry curves. Each of the four curves is contained in one hyperplane given by $\{x_2=0\}$,
$\{x_3=0\}$, $\{x_4=0\}$ or $\{\arg(x_1+\ii x_2)=\frac{\pi}{g+1}\}$, respectively, 
where $(x_1,x_2,x_3,x_4)$ denote the coordinates in $\R^4$ (not to be confused with the parameters $(x_1,x_2,x_3)$. In other words, $f_{g,\varphi}$ has the same planar symmetries as the Lawson surface $\xi_{1,g}$.

\item The image of $f_{g,\pi/4}$ is the Lawson minimal surface $\xi_{1,g}$ of genus $g$.

\item As $g\to\infty$, the blow-up $g\times(f_{g,\varphi}-\Id)$ converges in the simply connected domain $\mathcal U$ (given by $\C$ with radial rays from $z=p_j$ to $z=\infty$ removed) to a Scherk minimal surface in the tangent space of $\S^3\cong SU(2)$ at $\Id$, identified with $\R^3$.
The Scherk surface has period $2\pi$ and angle $2\varphi$ between the wings.

\item As $g\to\infty$, the image of $f_{g,\varphi}$ converges to the union of the
great sphere given by $\{\sin(\varphi)x_3+\cos(\varphi)x_4=0\}$ and the great sphere given by 
$\{\sin(\varphi)x_3-\cos(\varphi)x_4=0\}$. These two spheres intersect along the great circle $\{x_3=x_4=0\}$ with an angle $2\varphi$. The convergence is smooth away from the intersection circle.
\end{enumerate}
\end{proposition}
\begin{remark}
We will show in Proposition \ref{prop:willmore} that the Willmore energy of $f_{g,\varphi}$ is below $8\pi$ for large $g$, ensuring that $f_{g,\varphi}$ is embedded.
\end{remark}

\begin{proof}\mbox{}
The Riemann surface $M_{g,\varphi}$ is given as an algebraic curve  by the equation
$$y^{g+1} = \frac{(z-p_1)(z-p_3)}{(z-p_2)(z-p_4)}.$$
Let $\pi\colon M_{g,\varphi}\to\C P^1$ denote the projection onto the $z-$plane, namely
$\pi(y,z)=z$.
Then $\pi$ is a $(g+1)$-sheeted covering totally branched over $p_1,\cdots,p_4$.
Let $\widetilde{p}_j=\pi^{-1}(p_j),$ $\widetilde{\eta}_g=\pi^*\eta_{t,x(t)}$ and
$\widetilde{\Phi}_g$ be the solution of $d\widetilde{\Phi}_g=\widetilde{\Phi}_g\widetilde{\eta}_g$ with initial condition $\widetilde{\Phi}_g(\widetilde{0})=U(t,\varphi)$,
where $\pi(\widetilde{0})=0$ and the unitarizer
$U$ is given by Proposition \ref{prop:monodromy-reduction}. We first prove that $\widetilde{\Phi}_g$ solves the Monodromy Problem \eqref{monodromy-problem}.
Let $\gamma\in\pi_1(M_{g,\varphi}\setminus\{\widetilde{p}_1,\cdots,\widetilde{p}_4\},\widetilde{0})$.
Using the generators $\gamma_1,\gamma_4,\gamma_3,\gamma_4$ of the fundamental group of the 4-punctured sphere, we decompose $\pi\circ\gamma$ as
$$\pi\circ\gamma=\prod_{j=1}^\ell(\gamma_{i_j})^{n_j}$$
for some indices $i_1,\cdots,i_{\ell}\in \{1,2,3,4\}$ and integers $n_1, ..., n_{\ell} \in \Z$ satisfying
\begin{equation}
\label{eq:congruence}\sum_{j=1}^\ell (-1)^{i_j} n_j\equiv 0\mod (g+1).
\end{equation}
Equation \eqref{eq:congruence} comes from the condition that the closed curve $\pi \circ\gamma$ on the quotient is induced from a closed curve $\gamma$ on $M_{g,\varphi},$ i.e., it stems from the monodromy of the covering map $\pi$,
see \cite[Section 4.2]{HHT}.
Then
$$\mathcal M(\widetilde{\Phi}_g,\gamma)=\prod_{j=1}^\ell U M_{i_j}(t)^{n_j}U^{-1}$$
so $\mathcal M(\widetilde{\Phi}_g,\gamma)\in\Lambda SU(2)$.
By our choice of $t$, we have
$$\det(tA_1)=-t^2\mathcal K(x(t))=-s^2\quad\text{ with }\quad s=\tfrac{1}{2g+2}$$
so $t A_1$ has eigenvalues $\pm s$.
At the Sym points, all $M_{i_j}$ are diagonal and have the same eigenvalues $e^{\pm 2\pi \ii s}$.
From the $\delta$ symmetry, we have $M_3=M_1$ and $M_4=M_2$, and since
$M_1 M_2 M_3 M_4=\Id$, 
$M_1=M_2^{-1}=M_3=M_4^{-1}$.
Hence $\mathcal M(\widetilde{\Phi}_g,\gamma)\mid_{\lambda=\lambda_1,\lambda_2}$ is diagonal and has eigenvalues
$$\exp\big(\pm 2\pi\ii s\sum_{j=1}^\ell(-1)^{i_j} n_j\big)=\pm 1$$
thanks to Equations \eqref{eq:t}, \eqref{eq:congruence} and $s=\frac{1}{2g+2}$.
Hence the Monodromy Problem \eqref{monodromy-problem} is solved
and the Sym-Bobenko formula defines an immersion $f_{g,\varphi}$ on
$M_{g,\varphi}\setminus \{\widetilde{p}_1,\cdots,\widetilde{p}_4\}$ with mean curvature 
$$H= \ii\frac{\lambda_1+\lambda_2}{\lambda_1-\lambda_2}
=\ii \frac{e^{\ii \theta} + e^{-\ii\theta}}{e^{\ii \theta} - e^{-\ii\theta}}=\operatorname{cotan}(\theta(t,\varphi)).$$

Next we prove that $\widetilde{p}_1,\cdots,\widetilde{p}_4$ are apparent singularities.
Let $w$ be a local coordinate in a neighborhood of $\widetilde{p}_1$ such that
$w^{g+1}=z-p_1$. Then using $t\sqrt{\mathcal K}=\frac{1}{2g+2}$ we have
\begin{equation}
\label{eq:wteta}
\widetilde{\eta}_g=t\,(g+1)A_1\frac{dw}{w}+O(w^g)dw
=\frac{1}{2\sqrt{\mathcal K}}\begin{pmatrix}x_1&x_2+\ii x_3\\x_2-\ii x_3&-x_1\end{pmatrix}
\frac{dw}{w}+O(w^g)dw.\end{equation}
Consider the local gauge
$$G_1=\begin{pmatrix}1&0\\k&1\end{pmatrix}\begin{pmatrix}
w^{-1/2}&0\\0&w^{1/2}\end{pmatrix}
\quad\text{with}\quad
k=\frac{\sqrt{\mathcal K}-x_1}{x_2+\ii x_3}.$$
A computation gives
$$\widetilde{\eta}_g . G_1=\frac{1}{2 \sqrt{\mathcal K}}\begin{pmatrix}0&x_2+\ii x_3\\
\frac{x_1^2+x_2^2+x_3^2-\mathcal K}{(x_2+\ii x_3)\,w^2}&0\end{pmatrix}dw
+O(w^g)dw$$
which is holomorphic at $w=0$ as $\mathcal K=x_1^2+x_2^2+x_3^2$.
This ensures that $f_{g,\varphi}$ extends to a regular immersion at $\widetilde{p}_1$. Similarly, $f_{g,\varphi}$ extends
 to $\widetilde{p}_2$, $\widetilde{p}_3$ and $\widetilde{p}_4$ as a regular immersion by replacing $(x_1,x_2,x_3)$ by $(-x_1,-x_2,x_3)$, $(x_1,-x_2,-x_3)$ and $(-x_1,x_2,-x_3),$  respectively,  in the definition of $k$.
\begin{remark} At $t=0$, we have
$$k=-e^{\ii\varphi}\frac{\lambda-\ii}{\lambda+\ii},$$
and hence the gauge $G_1$ has in fact a pole at $\lambda=-\ii$. For small $t\neq 0$, $G_1$ may have a pole in the unit $\lambda$-disk close to $-\ii$. Therefore, we need to apply the $r$-Iwasawa decomposition (for some $0<r<1$, see for example \cite{McI}) instead of the ordinary Iwasawa decomposition. This does not alter the corresponding immersion.
\end{remark}

\subsubsection*{Proof of Point (2)} The Willmore energy of a compact CMC immersion $f$ in the 3-sphere constructed from a meromorphic DPW potential $\eta$ is given
by
\begin{equation}\label{areaasresidue}
\mathcal W(f)=4\pi\sum_{j=1}^n\operatorname{Res}_{q_j}\operatorname{trace}(\eta_{-1}G_{j,1} G_{j,0}^{-1}),\end{equation}
where $q_1,\dots,q_n$ are the poles of the potential $\eta=\sum_{k=-1}^{\infty}\eta_k\lambda^k$ and $G_j=\sum_{k=0}^\infty G_{j,k}\lambda^k$ is a local gauge defined near $q_j$ such that
$\eta.G_j$ is regular at $q_j$.
This is proved in \cite[Corollary 17]{HHT} in the minimal case (where the Willmore energy is of course the area) and in \cite{volume} in the general CMC case.

As our gauges involve square roots,
we need to work on a double covering where the gauges are well-defined. Hence,
in order to apply \cite[Corollary 17]{HHT} directly, we work on a double covering
$\widehat M_{g,\varphi}\to  M_{g,\varphi}$
branched at $\widehat{p}_1,\cdots,\widehat{p}_4$ and let $\widehat{\eta}_g$ be the pullback of $\widetilde{\eta}_g$ to $\widehat M_{g,\varphi}$.
Using $v=\sqrt{w}$ as local coordinate in a neighborhood of $\widehat{p}_1$, we have
$$\Res_{\widehat{p}_1}\widehat{\eta}_{g,-1}=\frac{1}{\sqrt{\mathcal K}}\begin{pmatrix}x_{1,-1}& x_{2,-1}+\ii x_{3,-1}\\
x_{2,-1}-\ii x_{3,-1}&-x_{1,-1}\end{pmatrix}$$
$$G_{1,1}G_{1,0}^{-1}=\begin{pmatrix}1&0\\k_1-k_0&1\end{pmatrix}$$
$$k_0=\frac{-x_{1,-1}}{x_{2,-1}+\ii x_{3,-1}},\qquad
k_1=\frac{\sqrt{\mathcal K}-x_{1,0}}{x_{2,-1}+\ii x_{3,-1}}+\frac{x_{1,-1}(x_{2,0}+\ii x_{3,0})}{(x_{2,-1}+\ii x_{3,-1})^2}$$
\begin{eqnarray*}
\operatorname{Res}_{\widehat{p}_1}\operatorname{trace}\left(\widehat{\eta}_{g,-1}G_{1,1} G_{1,0}^{-1}\right)
&=&\tfrac{1}{\sqrt{\mathcal K}}(x_{2,-1}+\ii x_{3,-1})(k_1-k_0)\\
&=&1+\tfrac{1}{\sqrt{\mathcal K}}\left(x_{1,-1}-x_{1,0}+\frac{x_{1,-1}(x_{2,0}+\ii x_{3,0})}{x_{2,-1}+\ii x_{3,-1}}\right)\\
&=&1+\tfrac{1}{\sqrt{\mathcal K}}\left(x_{1,-1}-x_{1,0}-e^{\ii\varphi}(x_{2,0}+\ii x_{3,0})\right).
\end{eqnarray*}
By replacing $(x_1,x_2,x_3)$ by $(-x_2,-x_2,x_3)$, $(x_1,-x_2,-x_3),$ and $(-x_1,x_2,-x_3)$ respectively, we obtain
$$\operatorname{Res}_{\widehat{p}_2}\operatorname{trace}\left(\widehat{\eta}_{g,-1}G_{2,1}G_{2,0}^{-1}\right)=1+\tfrac{1}{\sqrt{\mathcal K}}\left(-x_{1,-1}+x_{1,0}+e^{-\ii\varphi}(-x_{2,0}+\ii x_{3,0})\right)$$
$$\operatorname{Res}_{\widehat{p}_3}\operatorname{trace}\left(\widehat{\eta}_{g,-1}G_{3,1}G_{3,0}^{-1}\right)=1+\tfrac{1}{\sqrt{\mathcal K}}\left(x_{1,-1}-x_{1,0}-e^{\ii\varphi}(x_{2,0}+\ii x_{3,0})\right)$$
$$\operatorname{Res}_{\widehat{p}_4}\operatorname{trace}\left(\widehat{\eta}_{g,-1}G_{4,1}G_{4,0}^{-1}\right)=1+\tfrac{1}{\sqrt{\mathcal K}}\left(-x_{1,-1}+x_{1,0}+e^{-\ii\varphi}(-x_{2,0}+\ii x_{3,0})\right).$$
Adding all four residues, we obtain, 
$$2\operatorname{Area}(f_{g,\varphi})=16\pi\left(1-\tfrac{1}{\sqrt{\mathcal K}}(x_{2,0}\cos(\varphi)-x_{3,0}\sin(\varphi))\right),$$
where the factor $2$ is due to us working on a double cover.
\subsubsection*{Proof of Point (3)} To describe the symmetries explicitly, we consider the following identification of $\S^3$ with $SU(2)$:
\begin{equation}
\label{eq:S3model}
(x_1,x_2,x_3,x_4)\in\S^3\subset\R^4\;\longleftrightarrow
\matrix{x_1+\ii x_2&x_3+\ii x_4\\ -x_3+\ii x_4&x_1-\ii x_2}\in SU(2).
\end{equation}
\begin{remark}
Note that the $(x_1, x_2, x_3, x_4)$ denote the standard coordinates of $\R^4$ and not the parameter $x=(x_1, x_2, x_3)$ of the DPW potential in this subsection.
\end{remark}
We consider the branch of the extended frame $\widetilde{\Phi}$ in the simply connected domain $\mathcal U$ such that
$\widetilde{\Phi}(0)=U$, and denote $(F,B)$ its Iwasawa decomposition. We study the symmetries
$\sigma$, $\sigma\circ\delta$ and $\sigma\circ\tau$.
\begin{itemize}
\item{\em The symmetry $\sigma(z)=\overline{z}$.}
 Since $\sigma^*\eta=\overline{\eta}$ and $\widetilde{\Phi}(z=0)=U=\overline{U}$ by Proposition \ref{prop:symmetriesU}, we have
$\sigma^*\widetilde{\Phi}=\overline{\widetilde{\Phi}}$. Hence
$\sigma^* F=\overline{F}$ and
$$\sigma^* f= \overline{F}(\lambda_1)\overline{F}(\lambda_2)^{-1}=
\overline{F(\lambda_2)}\;\overline{F(\lambda_1)^{-1}}
=\overline{f^{-1}}.$$
In the model \eqref{eq:S3model}, this corresponds to the symmetry $x_3\to -x_3$. Hence
the real line is mapped to a symmetry curve in the $x_3=0$ hyperplane.
\item {\em The symmetry $\sigma\circ\delta(z)=-\overline{z}$.}
Since $\delta^*\Phi=D^{-1}\Phi D$ in $\mathcal U$, we have
$$(\sigma\circ\delta)^*\widetilde{\Phi}=D^{-1}\overline{\widetilde{\Phi}} D.$$
$$(\sigma\circ\delta)^*F=D^{-1}\overline{F}D$$
$$(\sigma\circ\delta)^*f
=D^{-1}\overline{f^{-1}} D.$$
This corresponds to the symmetry $x_4\to -x_4$. Hence the imaginary axis is mapped to a symmetry curve in the $x_4=0$ hyperplane.
\item {\em The symmetry $\sigma\circ\tau(z)=\frac{1}{\overline{z}}$ in the east sector.}
 \begin{figure}[h]
\centering  \vspace{-0.3cm}
 \includegraphics[height=0.25\textwidth]{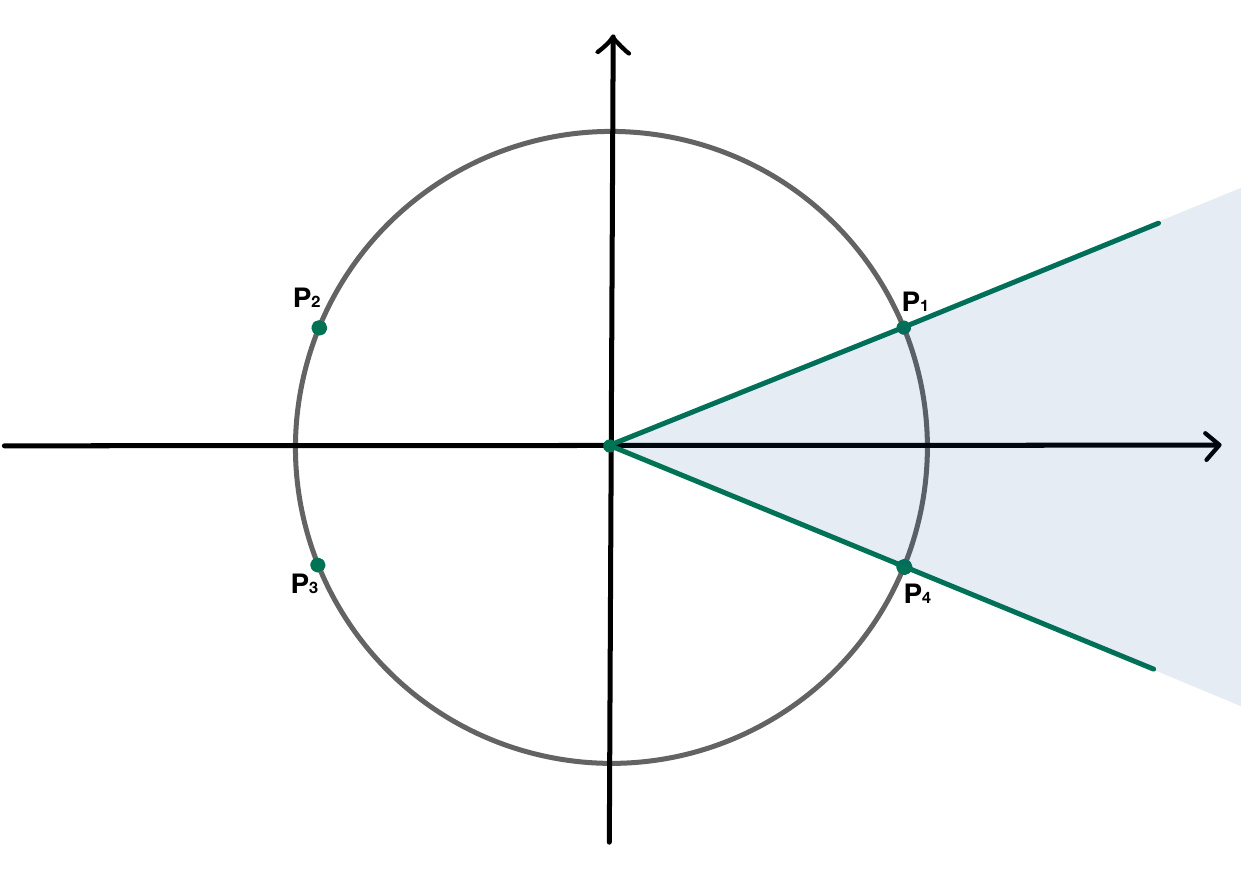} 
  \vspace{-0.5cm}
 \caption{The blue shaded region shows the east sector considered in this subsection.}\label{eastsector}
\end{figure}
Here one must be careful since the domain $\mathcal U, $ see Figure \ref{domain}, is not invariant under the symmetry $\tau$. We first study the symmetry $\sigma\circ\tau$ in the sector $-\varphi<\arg(z)<\varphi$ which contains the positive real 
axis and is preserved by $\sigma\circ\tau$. From $(\sigma\circ\tau)^*\eta=C^{-1}\overline{\eta}C$,
we have
$$(\sigma\circ\tau)^*\widetilde{\Phi}=R C^{-1}\overline{\widetilde{\Phi}}C$$
for some $R\in\Lambda SL(2,\C)$. Evaluating at $z=1$, we obtain
$$U\mathcal P=R C^{-1}\overline{U\mathcal P}C=RC^{-1}U\mathcal P C.$$
Hence
$$R=U\mathcal P C^{-1}\mathcal P^{-1} U^{-1} C=(UL_3 U^{-1})C\in\Lambda SU(2)$$
because $U$ unitarizes $L_3$. Moreover, from the properties of $\mathcal P$ and $U$, $R$ is diagonal at the Sym points and $\overline{R}=R$. Hence we can write
$$R(\lambda_1)=\matrix{e^{\ii\alpha}&0\\0&e^{-\ii\alpha}}\quad\text{ and }\quad
R(\lambda_2)=\overline{R(\lambda_1)}=R(\lambda_1)^{-1}$$
for some real $\alpha$. To compute the Iwasawa decomposition of $(\sigma\circ\tau)^*\widetilde{\Phi}$,
we write $C^{-1}\overline{B}(\lambda=0)C=Q_0 R_0$ with $Q_0\in SU(2)$ and $R_0$ upper triangular with positive diagonal. Then
$$(\sigma\circ\tau)^*\widetilde{\Phi}=\underbrace{\left(RC^{-1} \overline{F} C Q_0\right)}_{\in\Lambda SU(2)}\underbrace{\left(Q_0^{-1}C^{-1}\overline{B}C\right)}_{\in\Lambda^+_{\R}SL(2,\C)}$$
so 
$$(\sigma\circ\tau)^*F=RC^{-1} \overline{F} C Q_0$$
$$(\sigma\circ\tau)^*f=R(\lambda_1)C^{-1}\overline{f^{-1}}C R(\lambda_2)^{-1}$$
which corresponds to the symmetry $x_1+\ii x_2\to e^{2\ii\alpha}(x_1-\ii x_2)$.
Hence the image of the arc $-\varphi\leq\arg(z)\leq\varphi$ on the unit circle is a symmetry curve in the plane $\arg(x_1+\ii x_2)=\alpha$.

\item {\em The symmetry $\sigma\circ\tau(z)=\frac{1}{\overline{z}}$ in the north sector.}

 \begin{figure}[h] \vspace{-0.3cm}
\centering 

 \includegraphics[height=0.25\textwidth]{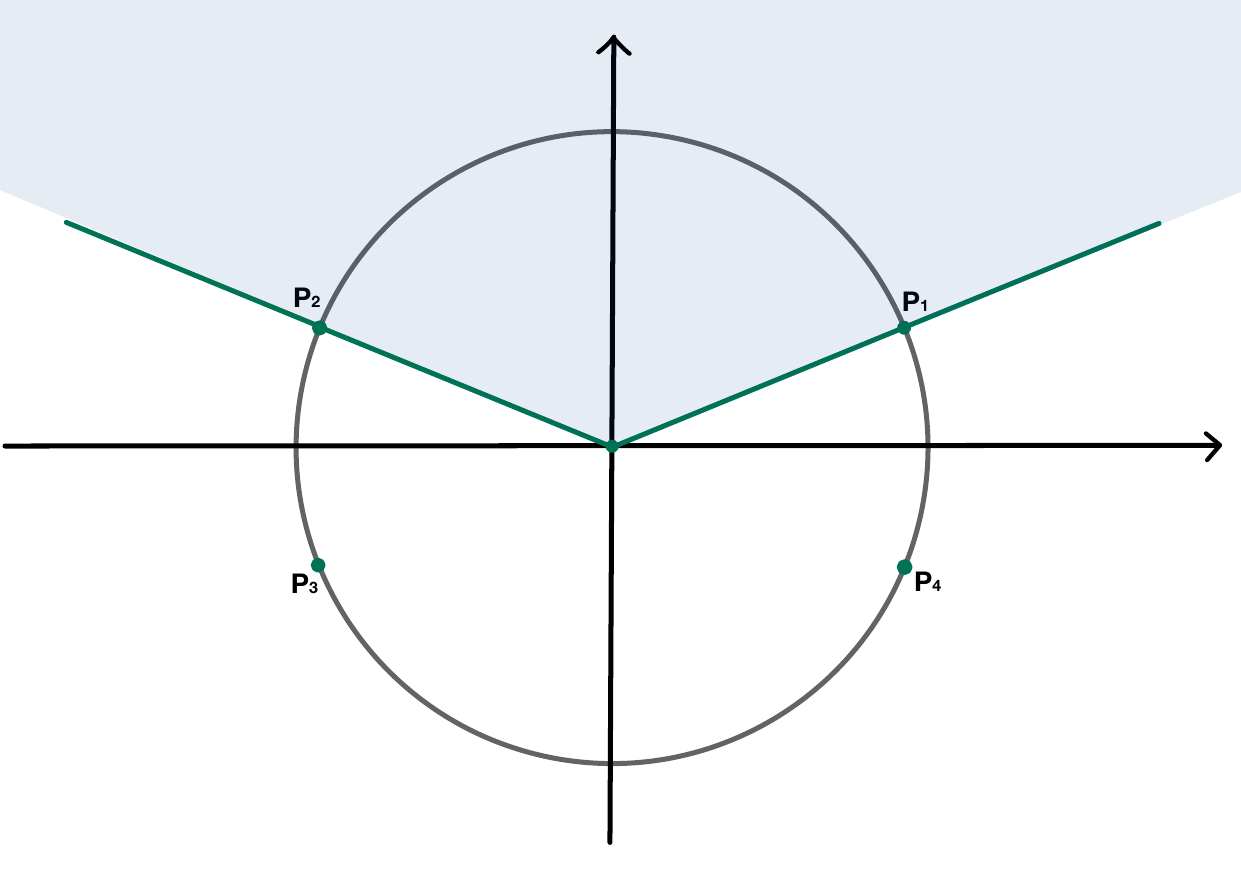} 
 \vspace{-0.5cm}
\caption{The blue shaded region shows the north sector considered in this subsection.}\label{northsector}
\end{figure}
\noindent
In the same way, in the sector $\varphi<\arg(z)<\pi-\varphi$ which contains the positive imaginary 
axis, we have
$$(\sigma\circ\tau)^*\widetilde{\Phi}=R' C^{-1}\overline{\widetilde{\Phi}}C$$
for some $R'\in\Lambda SL(2,\C)$. Evaluating at $z=\ii$, we obtain
$$U\mathcal Q=R' C^{-1}\overline{U\mathcal Q}C=R'C^{-1}UD^{-1}\mathcal Q D C.$$
Hence using that $U$ commutes with $D$
$$R'=U\mathcal Q CD\mathcal Q^{-1} U^{-1} DC=(UL_1 U^{-1})DC\in\Lambda SU(2).$$
Moreover, from the properties of $U$ and $\mathcal Q$, $R'$ is diagonal at the Sym points and
$\overline{R'}=D^{-1}R' D$. Hence we may write
$$R'(\lambda_1)=\matrix{e^{\ii\beta}&0\\0&e^{-\ii\beta}}\quad\text{ and }\quad
R'(\lambda_2)=R'(\lambda_1)^{-1}$$
and by the same argument as in Point (c), the image of the arc $\varphi\leq \arg(z)\leq \pi-\varphi$ on the unit circle is a symmetry curve in the plane $\arg(x_1+\ii x_2)=\beta$.
The angle $\beta$ is related to $\alpha$ by observing that the monodromy of $\widetilde{\Phi}$ along $\gamma_1$ is equal to
$$U M_1 U^{-1}=\widetilde{\Phi}(+\infty)\widetilde{\Phi}(+\ii\infty)^{-1}=R(R')^{-1}.$$
Hence $R(R')^{-1}$ has eigenvalues $e^{\pm 2\pi\ii s}$ which gives
$$\alpha-\beta=\pm 2\pi s=\pm\frac{\pi}{g+1}$$
proving point (3) of Proposition \ref{prop:building} up to a rotation.
\end{itemize}
\subsubsection*{Proof of Point (4)} If $\varphi=\frac{\pi}{4}$, $f_{g,\pi/4}$ is minimal and $M_{g,\pi/4}$ has the additional symmetry
$\iota(z)=i z$ which we study in the same way as the symmetries in Point (3).
For a function $u\in\mathcal W$, we use the notation $u^{\dagger}(\lambda)=u(-\lambda)$.
\begin{lemma}
If $\varphi=\frac{\pi}{4}$, the unitarizer $U$ satisfies $U^{\dagger}=U$, i.e., $U$ is an even function.
\end{lemma}
\begin{proof}
By Proposition \ref{prop:symmetries} we have
$$x_1=-x_1^{\dagger},\quad x_2=-x_3^{\dagger},\quad x_3=-x_2^{\dagger}$$
Using \eqref{eq:iota*omega}, we obtain
\begin{equation}
\label{eq:iota*eta}
\iota^*\eta=t\matrix{x_1^{\dagger}\omega_1&x_3^{\dagger}\omega_3-\ii x_2^{\dagger}\omega_2\\
x_3^{\dagger}\omega_3+\ii x_2^{\dagger}\omega_2&-x_1^{\dagger}\omega_1}=S^{-1}\eta^{\dagger}S\end{equation}
with $S$ as defined in the proof of Proposition \ref{prop:symmetries}.
Hence
$$\iota^*\Phi=S^{-1}\Phi^{\dagger} S$$
which gives at $z=1$
$$\mathcal Q=S^{-1}\mathcal P^{\dagger} S.$$
Then, using that $U$ commutes with $S$
$$U^{\dagger} L_1 (U^{\dagger})^{-1}=U^{\dagger}
S^{-1}\mathcal P^{\dagger} S CD S^{-1}(\mathcal P^{\dagger})^{-1}S (U^{\dagger})^{-1}
=S^{-1}(U L_3 U^{-1})^{\dagger} S\in\Lambda SU(2).$$
In the same way,
$$U^{\dagger}L_3 (U^{\dagger})^{-1}=U^{\dagger} S\mathcal Q^{\dagger} S^{-1} C S (\mathcal Q^{\dagger})^{-1} S^{-1}(U^{\dagger})^{-1}=
S(U L_1 U^{-1})^{\dagger} S^{-1}\in\Lambda SU(2).$$
Since $\iota(\gamma_1)=\gamma_2$, we have
$$M_2=S^{-1} M_1^{\dagger}S$$
$$U^{\dagger}L_4 (U^{\dagger})^{-1}=U^{\dagger}S M_2^{\dagger} S^{-1}(U^{\dagger})^{-1}
=S (U M_2 U^{-1})^{\dagger}S^{-1}\in\Lambda SU(2).$$
Finally, since $U$ is diagonal, $U^{\dagger} L_2 (U^{\dagger})^{-1}=D$.
By uniqueness of the unitarizer $U$, we have $U^{\dagger} U^{-1}\in\Lambda SU(2)$.
Since it is also in $\Lambda^+_{\R} SL(2,\C)$, we have $U^{\dagger}U^{-1}=\Id.$
\end{proof}
Returning to the proof of Point (4), we have
by equation \eqref{eq:iota*eta} and $U=U^{\dagger}$
$$\iota^*\widetilde{\Phi}=S^{-1}\widetilde{\Phi}^{\dagger} S$$
$$\iota^* F=S^{-1}F^{\dagger}S$$
$$\iota^* f= S^{-1} F(-\lambda_2)F(-\lambda_1)^{-1}S=S^{-1}f^{-1}S,$$
and finally for the symmetry $\iota\circ\sigma(z)=\ii\overline{z}$:
$$(\iota\circ\sigma)^*f=\sigma^*(S^{-1}f^{-1} S)=S^{-1}\overline{f} S.$$
This corresponds to the isometry
$$(x_1,x_2,x_3,x_4)\mapsto (x_1,-x_2,-x_4,-x_3).$$
Hence the image of the segment $[0,p_1]$ lies on the great circle defined by the equations $x_2=0,\;x_3+x_4=0$.
Furthermore, from $\widetilde{\Phi}(+\ii\infty)=\iota^*\widetilde{\Phi}(+\infty)$ we obtain
$$R' C^{-1} U C=S^{-1}R^{\dagger}C^{-1}UC S.$$
Evaluating at $\lambda_1$, we obtain $\beta=-\alpha$ (as should be expected).
Using the symmetries of Point (3), the image of the east sector is bounded by a geodesic 4-gon with two angles of $\frac{\pi}{g+1}$ at $f(p_1)$ and $f(p_4)$ and two right angles at $f(0)$ and $f(\infty)$: the boundary of the fundamental piece of Lawson surface $\xi_{1,g}$. The conclusion follows from the uniqueness of the solution to the Plateau problem in this situation (see the details of the argument in \cite[Theorem 5]{HHT}).
\subsubsection*{Proof of Point (5)} We fix $\varphi$ and use an index $t$ to emphasize the dependence on the parameter $t$, namely, we denote by $\eta_t=\eta_{t,x(t)}$ the DPW potential, by  $U(t)$ the unitarizer (which is chosen to be positive), by
$\widetilde{\Phi}_g=U(t) \Phi_{t,x(t)}$ the extended frame, by $F_t$ the unitary part of $\widetilde{\Phi}_g$, and by
$f_t=F_t\left(e^{\ii\theta(t)}\right)F_t\left(e^{-\ii\theta(t)}\right)^{-1}$ the immersion, which is well defined in the simply connected domain $\mathcal U$, see Figure \ref{domain}.
At $t=0$, we have $\eta_0=0$ and $\Phi_0=\Id$, hence $U(0)=\Id$ and $f_0=\Id$.
We prove that the limit
$$\lim_{t\to 0}\tfrac{1}{2t}(f_t-\Id)=\tfrac{1}{2}f'\mid_{t=0}$$
parametrizes a Scherk surface in  $\mathcal U$, which proves Point (5) since
$g\sim\frac{1}{2t}$ as $g\to\infty$.
The argument is similar to the proof of Theorem 4 in \cite{minoids}.

We have the following general formula for the differential of a CMC immersion into $\S^3$ constructed by the DPW method:
$$df=F(\lambda_1)\left[(\overline{\lambda_1}-\overline{\lambda_2})B^0 \eta^- (B^0)^{-1} 
+(\lambda_2-\lambda_1)(\overline{B^0\eta^-(B^0)^{-1}})^T
\right]F(\lambda_2)^{-1}$$
where $\lambda_1,\lambda_2$ are the Sym-points, $\eta^-$ is the negative part of the DPW potential,
$F$ is the unitary part of $\Phi$ and $B^0$ is the positive part of $\Phi$ evaluated at $\lambda=0$.
We differentiate this equation with respect to $t$ and obtain at $t=0$\begin{eqnarray*}
&&\lefteqn{df'=-\ii\left((\eta')^-+(\overline{\eta'\mbox{}^T})^-\right)}\\
&&=\matrix{
\omega_1-\overline{\omega_1}&
 \ii\sin(\varphi)(\omega_2+\overline{\omega_2})-\cos(\varphi)(\omega_3+\overline{\omega_3})\\
\ii\sin(\varphi)(\omega_2+\overline{\omega_2})+\cos(\varphi)(\omega_3+\overline{\omega_3})&
-\omega_1+\overline{\omega_1}},
\end{eqnarray*}
because $\eta_0=0$, $F_0=B_0=\Id$ and $\theta(0)=\frac{\pi}{2}$. In the identification of $SU(2)$ with $\S^3$ given by \eqref{eq:S3model}, this translates to
$$df'=(0,-2\,\Re(\ii\omega_1),-2\cos(\varphi)\Re(\omega_3),2\sin(\varphi)\Re(\omega_2))
=\Re(0,\phi_3,\phi_2,\phi_1)$$
with
\begin{align*}
\phi_1&=2\sin(\varphi)\omega_2=\frac{4\sin(2\varphi)(z^2-1)dz}{z^4-2\cos(2\varphi) z^2+1}\\
\phi_2&=-2\cos(\varphi)\omega_3=\frac{-4\ii\sin(2\varphi)(z^2+1)dz}{z^4-2\cos(2\varphi) z^2+1}\\
\phi_3&=-2\ii\,\omega_1=\frac{8\sin(2\varphi)z\,dz}{z^4-2\cos(2\varphi) z^2+1}\,.
\end{align*}
We have $\phi_1^2+\phi_2^2+\phi_3^2=0$ and the Gauss map is given by
$$G=\frac{\phi_1+\ii\phi_2}{\phi_3}=z.$$
With $Ox_2$ as the ``vertical'' direction, this is the Weierstrass Representation of a Scherk surface with period $4\pi$ and angle $2\varphi$.
Scaling by $\frac{1}{2}$ we obtain Point (5).
\subsubsection*{Proof of Point (6)}
Consider again the coordinate $w$ in a neighbourhood of $\wt{p}_1$ such that
$w^{g+1}=z-p_1$.
We shall prove that the image of the disk $|w|<1$ converges to a hemisphere as
$g\to\infty$.
\begin{lemma}
\label{lemma:Phi-infty}
 In the disk $\mathbb D= \{w \in \C\; |\;  |w|<1\}$, we have
$$\lim_{g\to\infty}\wt{\Phi}_g=\wt{\Phi}_\infty=\exp\left(\tfrac{1}{2}\cv{A}_1\log(w)\right)$$
where $\cv{A}_1$ denotes the matrix $A_1$ at the central value of the parameters.
The convergence is uniform on compact subsets of the disk $\mathbb D$.
\end{lemma}
\begin{proof}
By Equation \eqref{eq:wteta}, we have on compact subsets of the disk $\mathbb D$
$$\lim_{g\to\infty}\wt{\eta}_g=\cv{A}_1\frac{dw}{2w}:=\wt{\eta}_{\infty}.$$
Hence to prove the lemma, it suffices to compute the limit of $\wt{\Phi}_g$ at the point $w=\frac{1}{2}$, which corresponds in $\Sigma$ to the point $q_1(t)=p_1+(\frac{1}{2})^{g+1}$, where $t$ and $g$ are related by Equation \eqref{eq:t}.
We define $\Psi_t:\mathcal U\to\Lambda SL(2,\C)$ by
$$\Psi_t(z)=\exp\left(tA_1\log(1-z/p_1)\right).$$
By Equation \eqref{eq:t}, we have
$$\lim_{t\to 0}\Psi_t(q_1(t))=\exp\left(\tfrac{1}{2}A_1\log(\tfrac{1}{2})\right).$$
Moreover, $\Psi_t$ is uniformly bounded on the segment $[0,q_1(t)]$.
We compare $\Phi_t$ and $\Psi_t$ on this segment using the variation of constants method. Writing $\Phi_t=Y_t\Psi_t$, we have
$$d\Phi_t=Y_t\Psi_t\eta_t=dY_t \,\Psi_t+Y_t\Psi_t tA_1\frac{dz}{z-p_1}.$$
Hence
$$dY_t=\sum_{j=2}^4 Y_t\Psi_t\, tA_j\frac{dz}{z-p_j}\Psi_t^{-1}.$$
Integrating the equation and taking norm, we obtain on the segment $[0,q_1]$ the estimate
$$\|Y_t-\Id\|\leq C|t|\int_0^z \|Y_t\|$$
for some uniform constant $C$.
By Gronwall inequality this therefore yields
$$\|Y_t(q_1(t))-\Id\|\leq C|t|\exp(C|t|).$$
Hence,
$$\lim_{t\to 0}\Phi_t(q_1(t))=\exp\left(\tfrac{1}{2}A_1\log(\tfrac{1}{2})\right).$$
Since $\lim_{t\to 0} U_t=\Id$ by Theorem \ref{unitarizer} we have
$$\lim_{g\to\infty}\wt{\Phi}_g(w=\tfrac{1}{2})=\exp\left(\tfrac{1}{2}A_1\log(\tfrac{1}{2})\right).$$
\end{proof}
\begin{remark}
The content of Lemma \ref{lemma:Phi-infty} is that $\wt{\Phi}_{\infty}(w=1)=\Id$.
Since the point $w=1$ corresponds to $z=p_1+1$ and
$\lim_{t\to 0}\Phi_t(p_1+1)=\Id$, the lemma seems trivial. But since the convergence to $\wt{\eta}_\infty$ only holds in compact subsets of
$\mathbb D$, it is not possible to conclude in that way.
\end{remark}
We now identify the limit immersion $f_{\infty}$ on the disk $\mathbb D$.
Since $\cv{A}_1$ is hermitian on the unit circle $|\lambda|=1$, we have for any angle $\alpha$
$$\wt{\Phi}_\infty(we^{\ii\alpha})=R_\alpha\wt{\Phi}_{\infty}(w)
\quad\text{ with }\quad
R_\alpha=\exp\left(\tfrac{1}{2}\,\ii\alpha \cv{A}_1\right)\in\Lambda SU(2).$$
Hence by the Sym-Bobenko formula
$$f_\infty(w e^{\ii\alpha})=R_\alpha(\ii)f_{\infty}(w)R_\alpha(-\ii)^{-1}
=\matrix{e^{\ii\alpha/2}&0\\0&e^{-\ii\alpha/2}}
f_{\infty}(w)\matrix{e^{\ii\alpha/2}&0\\0&e^{-\ii\alpha/2}},$$
so it suffices to study the immersion for $w$ being real.
Further, 
$$\wt{\Phi}_\infty(w)=\frac{\ii}{4 \lambda 
   \sqrt{w}}
   \matrix{
(\lambda -\ii)^2-(\lambda +\ii)^2 w
 & e^{-\ii \varphi }\left(\lambda ^2+1\right) 
   (1-w) \\
    e^{\ii \varphi }\left(\lambda ^2+1\right)  (w-1)
 & - (\lambda +\ii)^2+(\lambda -\ii)^2w 
}.$$
Computing the Iwasawa decomposition
of $\wt{\Phi}_{\infty} = F_{\infty}B_{\infty}$ for $w\in\R$ gives
\begin{align*}
&F_{\infty}(w)=\\
r&\small\matrix{
 \ii\lambda^{-1} (w-1)(w^2+1)+(w+1)^3+2 \ii \lambda  (w^2-w)
   & e^{-\ii \varphi } (w-1) \left(-2 \ii \lambda^{-1} w +w^2-1+\ii\lambda (w^2+1)\right)
   \\
 e^{\ii \varphi } (w-1) \left(\ii\lambda^{-1}(w^2+1)+1-w^2-2\ii\lambda w\right)
 & -2\ii\lambda^{-1}(w^2-w)+ (w+1)^3 -\ii \lambda  (w-1)
   (w^2+1)
}\end{align*}
$$
B_{\infty}(w)=\frac{r}{\sqrt{w}}\matrix{
-\ii \lambda  \left(w^4-1\right)+w^4+6 w^2+1 &
   e^{-\ii \varphi } \left(-\left(w^2-1\right)^2-\ii \lambda 
   \left(w^4-1\right)\right) \\
 e^{\ii \varphi }\,2 \ii \lambda   (w^3-w) & 
 4 (w^3+w)+2\ii \lambda  (w^3-w)
}$$
with
$$r=\frac{1}{2 \sqrt{w^6+7 w^4+7 w^2+1}}.$$
\begin{remark}
It is clear that $F_{\infty}\in\Lambda SU(2)$ and $B_{\infty}\in\Lambda^+_{\R} SL(2,\C)$. One can verify that $F_{\infty} B_{\infty}=\wt{\Phi}_{\infty}$ by expanding the product.
The Iwasawa decomposition was explicitly computed using another observation. The entries of $\wt{\Phi}_\infty$ are degree-1 Laurent polynomials in $\lambda$. In this case, the entries of the unitary part $F_{\infty}\in\Lambda SU(2)$ are Laurent polynomials of the same degree.  Thus $F_{\infty}$ is entirely determined by its entries $F_{11}$ and $F_{12}$
which are given by 6 unknown complex numbers. We define $B_{\infty}=F_{\infty}^{-1}\wt{\Phi}_{\infty}$ and solve the equations given by $B_{\infty}\in\Lambda^+_{\R} SL(2,\C)$ using Mathematica.
\end{remark}
The Sym-Bobenko formula then gives after simplification for $w\in\R$
$$f_{\infty}(w)=\frac{1}{w^2+1}\matrix{
 2 w& e^{-\ii \varphi } \left(1-w^2\right)
   \\
 e^{\ii \varphi } \left(w^2-1\right)& 2 w
}\,.$$
Hence for $w$ in the unit disk $\mathbb D$ we have
$$f_{\infty}(w)=\frac{1}{|w|^2+1}\matrix{
 2 w& e^{-\ii \varphi } \left(1-|w|^2\right)
   \\
 e^{\ii \varphi } \left(|w|^2-1\right)& 2\overline{w}
}\,,$$
from which we can deduce that the image of $\mathbb D$ under $f_{\infty}$ is the hemisphere
given by
$$\begin{cases}
\sin(\varphi)x_3+\cos(\varphi)x_4=0\\
\cos(\varphi)x_3-\sin(\varphi)x_4>0.
\end{cases}$$
Introducing similar coordinates $w_j$ in a neighborhood of each $\wt{p}_j$, and
using the symmetries $\sigma$ and $\delta$, we conclude that
the image of the unit disk $|w_2|<1$ is the hemisphere
$$\begin{cases}
\sin(\varphi)x_3-\cos(\varphi)x_4=0\\
\cos(\varphi)x_3+\sin(\varphi)x_4>0.
\end{cases}$$
and the images of the unit disks $|w_3|<1$ and $|w_4|<1$ are the complementary hemispheres of the above two hemispheres.
This concludes the proof of Proposition \ref{prop:building}.
\end{proof}

\section{Degenerating conformal type} \label{limitvarphi}
The aim of this section is to strengthen the conclusion of Proposition \ref{prop:IFT} by proving that $T$ is uniform with respect to $\varphi$.
By Proposition \ref{prop:symmetries}, it suffices to consider
$\varphi\in[0,\frac{\pi}{4}]$.

\begin{proposition}\label{prop:uniformT}
There exist $T>0$ such that for all $t\in (-T,T)$ and  $\varphi \in [0, \tfrac{\pi}{4}]$,
Problem \eqref{monodromy-problem3} has a unique solution $(x(t,\varphi),\theta(t,\varphi))$.
Moreover, $x(t, \varphi)$ and $\theta(t,\varphi)$ are analytic functions of $t$, $\varphi$ and $\varphi \log(\varphi)$, and at $\varphi=0$, we have
$x(t,0)=\cv{x}(\varphi=0)$ for all $t$.
As a consequence, there exists $g_0\in\N$ such that the immersion $f_{g,\varphi}:M_{g,\varphi}\to\S^3$
exists for all $g\geq g_0$ and $\varphi\in (0,\tfrac{\pi}{4})$.
Finally, for fixed $g$, the image of $f_{g,\varphi}$ converges to a doubly covered great sphere when $\varphi\to 0$
\end{proposition}
In this section, we use an additional index $\varphi$ to denote the dependence of objects on the parameter $\varphi$, e.g.,  $\eta_{t,x,\varphi}$ and $\Phi_{t,x,\varphi}$.
Since the partial differentials in the proof of Proposition \ref{prop:IFT} do not depend on $\varphi$, we only need to show that  the maps $\mathcal F_1,$ $\mathcal F_2$, $\mathcal H_1,$ and $\mathcal H_2$ remain smooth enough to apply a version of the implicit function theorem at $\varphi = 0$. For $\varphi\rightarrow 0$ the limit Riemann surface is given by $$\Sigma_0 = \C P^1\setminus\{\pm  1\}.$$
Hence the quantities corresponding to $\mathcal Q$, i.e., $\mathcal F_2$ and $\mathcal H_2,$ remain well-defined and depend smoothly on $\varphi$ in the $\varphi \rightarrow 0$ limit.
To be explicit, we have at $\varphi=0$:
\begin{equation}
\label{eq:eta-phi0}
\eta_{t, x,0} = t(A_1+A_4)\frac{dz}{z-1}+ t(A_2+A_3)\frac{dz}{z+1}
=2t\matrix{0&x_2\\x_2&0}\left(\frac{dz}{z-1}-\frac{dz}{z+1}\right)
\end{equation}
\begin{equation*}
\begin{split}
\Phi_{t,x,0}&=\exp\left(2t\begin{pmatrix}0&x_2\\x_2&0\end{pmatrix}\log\left(\frac{1-z}{1+z}\right)\right)\\
\mathcal Q(t,x,0)&=\Phi_{t,x,0}(z=\ii)=
\matrix{\cos(\pi t x_2)&-\ii\sin(\pi tx_2)\\-\ii\sin(\pi tx_2)&\cos(\pi tx_2)}\\
\mathfrak q(t,x,0)&=\sin(2\pi  t x_2).
\end{split}
\end{equation*}
Applying the implicit function theorem to the equations $\mathcal F_2=0$ and $\mathcal H_2=0$ uniquely determines the parameter $y_2\in \mathcal W_\R^{\geq0}$
as a smooth function of $(t,\varphi)$ in a neighborhood of $(0,0)$ and
the remaining parameters $(y_1,y_3,\theta)$ (where $x_i=\cv{x}_i+y_i$ as in the proof of Proposition \ref{prop:IFT}).
Moreover, at $\varphi=0$, the solution is $y_2=0$.
Since also $\cv{x}_2=0$ at $\varphi=0$, we have
$x_2=0$ at $\varphi=0$ for all $(t,y_1,y_3)$.
Dealing with the limits of the remaining equations is more difficult, as the limit potential has a pole at $z=1$.
\subsection{The asymptotic of $\mathcal P$} \mbox{}
In this section, we assume that $y_2$ has been determined by solving the equations
$\mathcal F_2=0$ and $\mathcal H_2=0$ and we denote
$\eta_{t,\varphi}$ the resulting potential, $\Phi_{t,\varphi}$ the corresponding solution
and $\mathcal P(t,\varphi)=\Phi_{t,\varphi}(z=1)$, not writing the dependence on the remaining parameters $y_1$ and $y_3$.
It turns out that $\mathcal P=\Phi(z=1)$ does not extend smoothly at $\varphi=0$, but
rather extends as an analytic function of $\varphi$ and $\varphi\log\varphi$, in the following sense:
\begin{definition}\cite{nodes}
 Let $f(x)$ be a function of the real variable $x\geq 0$. We say that f is an analytic function of
$x$ and $x\log x$ if there exists a real analytic function of two variables $g(x,y)$ defined in a neighborhood of $(0,0)$
in $\R^2$ such that 
$$f(x) = g(x,x\log x) \quad  \text{ for } x > 0 \quad \text{ and } \quad f(0) = g(0,0).$$
 \end{definition}
 \begin{remark}
 The terminology used in \cite{nodes} is that of a smooth function of $x$ and $x\log x$.
 Here we need that $g$ is analytic, which ensures that it is unique. Indeed, if a real analytic function
 $g$ satisfies $g(x,x\log x)=0$ for $x>0$, then $g=0$. This is false in the $C^{\infty}$ class.
 \end{remark}

Let $\Gamma$ denote the straight line from $z=0$ to $z=1.$
Fix some positive numbers $0<\varepsilon<\varepsilon_0<1$.
To study the behaviour of $\mathcal P(t,\varphi)$ for $\varphi\to 0^+$, we assume
$\varphi<\varepsilon^2$ and subdivide the curve $\Gamma$ into  
\[\Gamma=\Gamma_0\cdot\Gamma_{1,\varphi}\cdot\Gamma_{2,\varphi}\]
with
\begin{equation}
\begin{split}
\Gamma_0&\colon  s\mapsto (1-\varepsilon)s\qquad\text{from $0$ to $1-\varepsilon$}\\
\Gamma_{1,\varphi}&\colon  s\mapsto (1-\varepsilon)(1-s)+(1-\tfrac{\varphi}{\varepsilon})s\qquad\text{from $1-\varepsilon$ to $1-\tfrac{\varphi}{\varepsilon}$}\\
\Gamma_{2,\varphi}&\colon s\mapsto 1+\tfrac{\varphi}{\varepsilon}(s-1)
\qquad\text{from $1-\tfrac{\varphi}{\varepsilon}$ to $1$.}
\end{split}
\end{equation}
For a flat connection $d+\eta$ and a curve $\gamma:[0,1]\to\Sigma$, we denote $\Pi(\eta,\gamma)$ the principal solution of $\eta$ along $\gamma$ (see \cite{taylor}), namely the value at $s=1$ of the solution of
$$\begin{cases}
Y'(s)=Y(s)\eta(\gamma(s))\gamma'(s)\\
Y(0)=\Id\end{cases}.$$

Then
$$\mathcal P(t,\varphi)=\Pi(\eta_{t,\varphi},\Gamma)
=\Pi(\eta_{t,\varphi},\Gamma_0)\,
\Pi(\eta_{t,\varphi},\Gamma_{1,\varphi})\,
\Pi(\eta_{t,\varphi},\Gamma_{2,\varphi})
.$$

The principal solution along $\Gamma_0$ is clearly an analytic function of $\varphi$ in a neighborhood of $0$, since the path $\Gamma_0$ is fixed and
$\eta_{t,\varphi}$ depends analytically on $\varphi$ on $\Gamma_0$.
Moreover, at $\varphi=0$ we have $\eta_{t,0}=0$ so
\begin{equation}
\label{eq:principal-Gamma0}
\Pi(\eta_{t,0},\Gamma_0)=\Id.
\end{equation}
It is more delicate for the paths $\Gamma_{1,\varphi}$ and $\Gamma_{2,\varphi}$:
we will see that the principal solution along $\Gamma_{1,\varphi}$ extends as an analytic function of $\varphi$ and $\varphi\log\varphi$ at $\varphi=0$, while the principal solution along $\Gamma_{2,\varphi}$
is an analytic function of $\varphi$.

\subsection{Principal solution along $\Gamma_{2,\varphi}$}$\;$\\
\label{section:principal-Gamma2}
To analyse the $\varphi\to 0$
limit of the principal solution along $\Gamma_{2,\varphi}$ we consider for $\varphi>0$ the diffeomorphism
\[\psi_\varphi\colon D(1,1)\to D(1,\tfrac{2\varphi}{\varepsilon}); \quad w\mapsto 1+\tfrac{2\varphi}{\varepsilon}(w-1),\]
where $D(1, R)$ is the disc of center $1$ and radius $R$.

Then
\[\wh{\Gamma}_2:=(\psi_\varphi)^{-1}\circ\Gamma_{2,\varphi}\colon [0,1] \longrightarrow D(1,1); \quad \,s\longmapsto  \tfrac{1}{2}(1+s) \]
is independent of $\varphi$.
The pullback potential $\wh{\eta}_{t,\varphi}=\psi_{\varphi}^*\eta_{t,\varphi}$
has simple poles at $\psi_{\varphi}^{-1}(p_j)$ for $j= 1, ...,  4$.

We have
\begin{equation*}
\begin{split}
&\lim_{\varphi\to 0}\psi_{\varphi}^{-1}(p_1(\varphi))=1+\tfrac{\varepsilon}{2}\ii,\quad
\lim_{\varphi\to 0}\psi_{\varphi}^{-1}(p_4(\varphi))=1-\tfrac{\varepsilon}{2}\ii,\\
&\lim_{\varphi\to 0}\psi_{\varphi}^{-1}(p_2(\varphi))=
\lim_{\varphi\to 0}\psi_{\varphi}^{-1}(p_3(\varphi))=\infty.
\end{split}
\end{equation*}
So $\wh{\eta}_{t,\varphi}$ extends analytically to $\varphi=0$ with simple poles at
$1\pm\frac{\varepsilon}{2}\ii$. At $\varphi=0$, we have $x_2=0$, and thus 
$A_4=-A_1$ and
\begin{equation}
\label{eq:limit-eta}
\wh{\eta}_{t,0}=tA_1\left(\frac{dw}{w-1-\tfrac{\varepsilon}{2}\ii}
-\frac{dw}{w-1+\tfrac{\varepsilon}{2}\ii}\right).
\end{equation}
Hence
$\Pi(\eta_{t,\varphi},\Gamma_{2,\varphi})=\Pi(\wh{\eta}_{t,\varphi},\wh{\Gamma}_2)$
extends analytically to $\varphi=0$ and
\begin{equation}
\label{eq:principal-Gamma2}
\Pi(\eta_{t,\varphi},\Gamma_{2,\varphi})\mid_{\varphi=0}=
\exp\left(t A_1\int_{1/2}^1 \wh{\eta}_{t,0}\right)
=\exp\left(\pi\ii t A_1-tA_1 \log\left(\frac{1+\varepsilon\ii}{1-\varepsilon\ii}\right)\right)\end{equation}
where we take  principal value of the logarithm on $\C\setminus\R^+$.

\subsection{Principal solution along $\Gamma_{1,\varphi}$}\label{subseubgamma0} $\;$\\
We apply \cite[Theorem 5]{nodes} which we restate here as Theorem \ref{theorem:philogphi} with adjusted notations.
To use this result, it is necessary to view $\varphi$ as a complex number.

\begin{remark}
For $\varphi\in\C$ in a neighborhood of $0$, the potential $\eta_{t,x,\varphi}$
is well defined and depends holomorphically on $\varphi$. The poles
$\pm e^{\pm\ii\varphi}$ no longer lie on the unit circle breaking the symmetries.
Nevertheless, the complexified equations $\mathcal F_2=0$ and $\mathcal H_2=0$ can be solved using the implicit function theorem. The solution $y_2$ then lies in $\mathcal W^{\geq 0}$ instead of $\mathcal W^{\geq 0}_{\R}$ and depends holomorphically on $\varphi$.
(See Section \ref{section:complexification} for the complexification of the equations).
\end{remark}

For principal solution along $\Gamma_{1,\varphi}$, consider $\varphi \in D(0,\varepsilon_0^2)\subset\C$ and the diffeomorphism 
$$\psi_\varphi \colon \mathcal A_\varphi \longrightarrow \mathcal A_\varphi, \quad z \longmapsto 1+ \frac{\varphi}{z-1},$$
where
$\mathcal A_{\varphi}$ is the annulus
\[\mathcal A_\varphi=\{z\in\C\mid  \tfrac{|\varphi|}{\varepsilon_0}<|z-1|<\varepsilon_0\}.\] 
Furthermore, consider the change of parameter $\varphi= e^\omega$ with $\Re(\omega)<2\log\varepsilon_0$ and let $\beta_\omega$ denote the spiral curve from $z= 1-\varepsilon$ to $z = 1-\tfrac{\varphi}{\varepsilon}$ defined by
\[\beta_\omega \colon [0,1] \rightarrow \mathcal A_\varphi, \quad \beta_\omega(s) =1 -\varepsilon^{1-2s}e^{s \omega}.\]
Note that
\[\psi_{\varphi}\circ\beta_\omega(s)=\beta_{\omega}(1-s),\]
and for real $\omega$, the path $\beta_{\omega}$ is homotopic to
$\Gamma_{1,\varphi}$.
Also $\beta_{\omega+2\pi\ii}$ is homotopic to
$\gamma\cdot \beta_{\omega}$ where
\[\gamma  \colon [0,1] \rightarrow \mathcal A_\varphi; \; \gamma(s)=1- \varepsilon e^{2 \pi \ii s}
\]
parametrizes the circle of radius $\varepsilon$ around $z=1$.

 \begin{theorem}[Theorem 5 in \cite{nodes}]
\label{theorem:philogphi}
With the notations introduced above
consider a family of DPW potentials $\eta_\varphi$ in $\mathcal A_{\varphi}$
depending holomorphically on $\varphi\in D(0,\varepsilon_0^2)$.
Let $\wh \eta_{\varphi} := \psi_\varphi^* \eta_{\varphi}$. Assume that there exists $\sl(2, \C)$-valued 1-forms $\eta_0$ and $\wh \eta_0$ holomorphic in the disk $D(1,\varepsilon_0)$ such that 
$$\lim_{\varphi \rightarrow 0} \eta_{\varphi} = \eta_{0} \quad \text{and} \quad  \lim_{\varphi \rightarrow 0} \wh \eta_{\varphi} = \wh \eta_0$$
on compact subsets of the punctured disk $D^*(1,\varepsilon_0).$
Then we have for $|\varphi|$ small enough:
\begin{enumerate}
\item The function $\wt F$ defined by
$$\wt F(\omega)=\Pi(\eta_{t,\varphi},\gamma)^{ - \frac{\omega}{2\pi i} }\Pi(\eta_{t,\varphi},\beta_\omega )$$
satisfies $\wt F(\omega + 2\pi\ii) = \wt F(\omega)$ and therefore descends to a holomorphic function $F$ on $D(0,\varepsilon_0^2)$ by  $F(e^\omega):=\wt F(\omega)$.
\item The function $F$ extends holomorphically  to  $\varphi= 0$ with 
$$F(0) = \Pi(\eta_0 ,1-\varepsilon,1)\Pi(\wh \eta_0 ,1,1-\varepsilon),$$
where $\Pi(\eta_0, 1-\varepsilon, 1 )$ denotes the principal solution of $\eta_0$ along 
the straight line from $1-\varepsilon$ to $1$.
\item Consequently, for $\varphi> 0$, the function $\Pi(\eta_\varphi,\beta_\omega)$ extends to an analytic function of $\varphi$ and $\varphi\log\varphi$ with value $F(0)$ at $\varphi = 0$. 
\end{enumerate}
\end{theorem}

Returning to our problem, we have $|1-p_1|\simeq |\varphi|$ for small $\varphi$ so
$\mathcal A_{\varphi}$ does contain neither $p_1$, nor the other singularities $p_2$, $p_3$ or $p_4$.
Hence $\eta_{t,\varphi}$ is holomorphic in $\mathcal A_{\varphi}$.
By Equation \eqref{eq:eta-phi0} and recalling that $x_2=0$ at $\varphi=0$,
we have
$$\lim_{\varphi\to 0}\eta_{t,\varphi}=0.$$
Since $\psi_{\varphi}$ is involutive, the pullback potential $\wh{\eta}_{t,\varphi}=\psi_{\varphi}^*\eta_{t,\varphi}$ has poles at
$\psi_{\varphi}(p_j)$ for $j = 1, ..., 4$ and
\begin{align*}
&\lim_{\varphi\to 0} \psi_{\varphi}(p_1(\varphi))=1-\ii,\quad
\lim_{\varphi\to 0} \psi_{\varphi}(p_4(\varphi))=1+\ii\\
&\lim_{\varphi \to 0} \psi_{\varphi}(p_2(\varphi))=\lim_{\varphi\to 0} \psi_{\varphi}(p_3(\varphi))=1.\end{align*}
Hence, using again that $x_2=0$ at $\varphi=0$ which implies $A_4=-A_1$ and $A_3=-A_2$ we obtain
$$\lim_{\varphi\to 0}\wh{\eta}_{t,\varphi}
=t A_1\left(\frac{dz}{z-1+\ii}-\frac{dz}{z-1-\ii}\right).$$
Consequently, the limit $\wh{\eta}_{t,0}$
is holomorphic on $D(1,\varepsilon_0)$.
By Point (3) of Theorem \ref{theorem:philogphi}, $\Pi(\eta_{t,\varphi},\Gamma_{1,\varphi})$ extends to an analytic function of $\varphi$ and $\varphi\log\varphi$, with value
\begin{equation}
\label{eq:principal-Gamma1}
\Pi(\eta_{t,\varphi},\Gamma_{1,\varphi})\mid_{\varphi=0}=
\exp\left(t A_1\int_{1}^{1-\varepsilon}\wh{\eta}_{t,0}\right)
=\exp\left(t A_1\log\left(\frac{\varepsilon-\ii}{\varepsilon+\ii}\right)-\pi\ii t A_1\right)
\end{equation}
at $\varphi=0$, 
where we again choose the principal value of the logarithm $\C\setminus\R^+$.
\subsection{Conclusion}
We have proved that $\mathcal P$ extends as an analytic function of $\varphi$
and $\varphi\log\varphi$ at $\varphi=0$.
Moreover, collecting the results of \eqref{eq:principal-Gamma0},
\eqref{eq:principal-Gamma2} and \eqref{eq:principal-Gamma1}, we have at 
$\varphi=0$
\begin{equation}
\label{eq:P-varphi0}
\mathcal P\mid_{\varphi=0}=\exp(\pi \ii t A_1).
\end{equation}
Hence the functions $\mathcal F_1$ and $\mathcal H_1$ extend as analytic functions of 
$t$, $\varphi$, $\phi=\varphi\log\varphi$ and the remaining parameters
$y_1$, $y_2$ and $\theta$. 
We solve the equations
$\mathcal F_1=0$, $\mathcal H_1=0$ and $\mathcal K$ constant using the implicit function theorem to determine $(y_1,y_3,\theta)$ as analytic functions of
$(t,\varphi,\phi)$ in a neighborhood of $(0,0,0)$ and then specialize to $\phi=\varphi\log\varphi$.
This proves that the solution is an analytic function of $t$, $\varphi$
and $\varphi\log\varphi$ and yields a solution of the Monodromy Problem \eqref{monodromy-problem3} for every $\varphi $ in the compact interval $[0, \tfrac{\pi}{4}]$ and $t$ in a uniform interval $(-T,T)$.
Moreover, by Equation \eqref{eq:P-varphi0},
the solution at $\varphi=0$ is
$x(t,0)=\cv{x}(0)$ and $\theta(t)=\frac{\pi}{2}$ for all $t$, because $\cv{A}_1$ is hermitian on the unit circle and diagonal at $\lambda=\pm \ii$.
\subsection{Convergence to a doubly covered sphere}
With the notations of Section \ref{section:principal-Gamma2}, let
$\wh{\Phi}_{t,\varphi}=\psi_{\varphi}^*\wt{\Phi}_{t,\varphi}$.
At $\varphi=0$, the unitarizer satisfies $U_{t,0}=\Id$ for all $t$ by Theorem \ref{unitarizer}.

Since
$\psi_{\varphi}(1)=1$, we have by Equation \eqref{eq:P-varphi0}
$$\wh{\Phi}_{t,0}(1)=\exp\left(\pi \ii t \cv{A}_1\right).$$
Consider the change of variable
$$v=\frac{w-1-\tfrac{\varepsilon}{2}\ii}{w-1+\tfrac{\varepsilon}{2}\ii}\,.$$
By Equation \eqref{eq:limit-eta},
$$\wh{\eta}_{t,0}(w)=t\cv{A}_1\frac{dw}{w}.$$
Hence
$$\wh{\Phi}_{t,0}(w)=\exp\left(t\cv{A}_1\log(w)\right).$$
At $\varphi=0$, we have $\mathcal K(t,0)=1$ so $t=\frac{1}{2(g+1)}$.
Let $\pi:\C P^1\to\C P^1$ be the branched covering defined by $\pi(w)=w^{g+1}$. Then
$\pi^*\wh{\Phi}_{t,0}$ is precisely given by Lemma \ref{lemma:Phi-infty}.
(Note that here, $g$ is fixed and $\varphi\to 0$, whereas in Lemma \ref{lemma:Phi-infty}, $\varphi$ was fixed and $g\to\infty$.)
We have seen in Lemma \ref{lemma:Phi-infty} that the corresponding immersion
parametrizes the great sphere $\{x_4=0\}\subset S^3$ (since here $\varphi=0$).
Unfolding the various changes of variables, this means that a certain small neighborhood of
$z=1$ converges to this great sphere.
Using the $\delta$ symmetry, the symmetric neighborhood of $z=-1$ converges to the same great sphere. This concludes the proof of Proposition \ref{prop:uniformT}.

\section{Expansion of \texorpdfstring{$\mathfrak p$}{\textbackslash{}frak p} and \texorpdfstring{$\mathfrak q$}{\textbackslash{}frak q} in term of iterated integrals}\label{iterated}
\subsection{Iterated integrals}
Let $\alpha_1,\cdots,\alpha_n$ be closed 1-forms on a Riemann surface $\Sigma$ and
$\gamma:[0,1]\to \Sigma$ be a piecewise smooth path.
The iterated integral of $\alpha_1,\cdots,\alpha_n$ along $\gamma$ is defined by (\cite{chen})
$$\int_{\gamma}\alpha_1\cdots\alpha_n:=\int_{0<t_1<\cdots<t_n<1}\alpha_1(\gamma_1(t_1))(\gamma'_1(t_1))dt_1\cdots\alpha_n(\gamma_n(t_n))(\gamma'_n(t_n))dt_n.$$
If $\mathcal U\subset\C$ is a star-shaped domain with respect to the origin, we define $\int_0^z \alpha_1\cdots\alpha_n$ as the iterated
integral along the straight segment from $0$ to $z$.
Using the change of variable $t_i=s_i t_n$ for $1\leq i\leq n-1$, it satisfies the recursive property
$$\int_0^z \alpha_1\cdots\alpha_n=\int_{w=0}^z\left(\int_0^w\alpha_1\cdots\alpha_{n-1}\right)\alpha_n(w)$$
which may also be used as its definition.
In our case, $\mathcal U$ is the complex plane minus the radial rays from $p_i$ to $\infty$ as defined in Section \ref{section:half-trace-coord} and we define
for $z\in \mathcal U$
$$\Omega_{i_1,\cdots,i_n}(z)=\int_0^z \omega_{i_1}\cdots\omega_{i_n}.$$
We call $n$ the depth of the $\Omega$-value $\Omega_{i_1,\cdots,i_n}$.
\subsection{A pattern for the expansion of $\mathfrak p$ and $\mathfrak q$} Let $\Phi_{t, x}$ be the extended frame with $\Phi_{t,x}(z=0) = \Id$ and $\Phi_{t,x}^{i,j}$ denote its entries.
Define $X_{t,x}=(X_{t,x}^1,X_{t,x}^2,X_{t,x}^3):\mathcal U\to\mathcal W^3$ by
$$X_{t,x}^1=\Phi_{t,x}^{11}\Phi_{t,x}^{21}-\Phi_{t,x}^{12}\Phi_{t,x}^{22}$$
$$X_{t,x}^2=\ii(\Phi_{t,x}^{11}\Phi_{t,x}^{21}+\Phi_{t,x}^{12}\Phi_{t,x}^{22})$$
$$X_{t,x}^3=\Phi_{t,x}^{11}\Phi_{t,x}^{22}+\Phi_{t,x}^{12}\Phi_{t,x}^{21}.$$
We are interested in
$$\mathfrak p(t,x)=X_{t,x}^1(1)\quad\text{ and }\quad \mathfrak q(t,x)=X_{t,x}^2(\ii).$$
It is quite remarkable that $X_{t,x}$ satisfies a first order ODE system.
Indeed,  a direct computation using
$$d\Phi_{t,x}=\Phi_{t,x}\,t\matrix{ x_1\omega_1& x_2\omega_2+ \ii x_3\omega_3\\x_2\omega_2-\ii x_3\omega_3&-x_1\omega_1}$$
gives
\begin{equation}
\label{inducedODEspin}
dX_{t,x}=X_{t,x}\,t\alpha_x
\end{equation}
with
$$\alpha_x=2 \ii \left(\begin{array}{ccc}0 &x_1 \omega_1&x_3\omega_3\\
-x_1 \omega_1 & 0 &- x_2 \omega_2\\
- x_3 \omega_3 & x_2 \omega_2& 0\end{array}\right).$$
Moreover, at $z=0$, we have $\Phi_{t,x}(0)=I_3$ so $X_{t,x}(0)=(0,0,1)$.
\begin{remark}
The ODE \eqref{inducedODEspin} is in fact induced by the spin covering $\mathrm{SL}(2,\C)\to\mathrm{SO}(3,\C).$ To be more explicit, 
consider the complex inner product
\[\langle A,B \rangle:=-\frac{1}{2}\tr(AB)\]
on the complex 3-dimensional vector space $\mathfrak{sl}(2,\C),$ and its orthonormal basis
\[m_1=\begin{pmatrix}0&\ii\\\ii&0\end{pmatrix},\quad 
m_2=\begin{pmatrix}0&-1\\1&0\end{pmatrix},\quad 
m_3=\begin{pmatrix}-\ii&0\\0&\ii\end{pmatrix}.\]
The spin covering is given by the adjoint representation  
\[\begin{array}{ll}\text{Ad}\colon &\mathrm{SL}(2,\C)\longrightarrow \mathrm{SO}(3,\C)\\
&g\longmapsto (A\mapsto gAg^{-1}).\end{array}\]
With respect to the basis $(m_1,m_2,m_3)$,
\[\Phi=\begin{pmatrix}\Phi_{11}&\Phi_{12} \\\Phi_{21}&\Phi_{22}\end{pmatrix}\]
is mapped to
\[\text{Ad}(\Phi)=\begin{pmatrix}
\tfrac{1}{2}(\Phi_{11}^2-\Phi_{12}^2-\Phi_{21}^2+\Phi_{22}^2)
&\tfrac{\ii}{2}(\Phi_{11}^2+\Phi_{12}^2-\Phi_{21}^2-\Phi_{22}^2)
&\Phi_{11}\Phi_{12}-\Phi_{21}\Phi_{22}\\
-\tfrac{\ii}{2}(\Phi_{11}^2-\Phi_{12}^2+\Phi_{21}^2-\Phi_{22}^2)
&\tfrac{1}{2}(\Phi_{11}^2+\Phi_{12}^2+\Phi_{21}^2+\Phi_{22}^2)
&-\ii(\Phi_{11}\Phi_{12}+\Phi_{21}\Phi_{22})\\
\Phi_{11}\Phi_{21}-\Phi_{12}\Phi_{22}
&\ii(\Phi_{11}\Phi_{21}+\Phi_{12}\Phi_{22})
&\Phi_{11}\Phi_{22}+\Phi_{12}\Phi_{21} \\
\end{pmatrix},\]
and the last row of $\text{Ad}\Phi$ is in fact $X_{t,x}$.
If $\Phi$ solve the differential equation $d\Phi=\Phi\eta,$ then
\[d\text{Ad}(\Phi)=\text{Ad}(\Phi) [\eta,\cdot],\] which gives \eqref{inducedODEspin} with respect to the basis $(m_1,m_2,m_3)$.
\end{remark}

We define
\[M_1=\left(\begin{array}{ccc} 0 & 2\ii & 0\\-2\ii & 0 &0\\0&0&0\end{array}\right),\quad
M_2=\left(\begin{array}{ccc} 0&0&0\\0&0&-2\ii\\0&2\ii&0\end{array}\right)\quad\text{ and }\quad
M_3=\left(\begin{array}{ccc} 0&0&2\ii\\0&0&0\\-2\ii&0&0\end{array}\right)\]
such that
$$\alpha_x=\sum_{i=1}^3 x_i M_i\omega_i.$$
\begin{proposition}
\label{prop:expandpq}
The analytic functions $\mathfrak p$ and $\mathfrak q$ can be expanded as
$$\mathfrak p(t,x)=\sum_{n=1}^{\infty} \sum_{i_1,\cdots,i_n} t^n x_{i_1}\cdots x_{i_n}(M_{i_1}\cdots M_{i_n})_{31}\Omega_{i_1,\cdots,i_n}(1)$$
$$\mathfrak q(t,x)=\sum_{n=1}^{\infty} \sum_{i_1,\cdots,i_n} t^n x_{i_1}\cdots x_{i_n}(M_{i_1}\cdots M_{i_n})_{32}\Omega_{i_1,\cdots,i_n}(i)$$
where $(M_{i_1}\cdots M_{i_n})_{ij}$ means the $(ij)$-the entry of the matrix $(M_{i_1}\cdots M_{i_n})$ and the summation is over all $(i_1,\cdots,i_n)\in\{1,2,3\}^n$.
\end{proposition}
\begin{proof}
Let $Y_{t,x}:\mathcal U\to\Lambda SL(2,\C)$ be the solution of the Cauchy Problem
$$\begin{cases}
d_z Y_{t,x}= Y_{t,x}\, t\alpha_x\\
Y_{t,x}(z=0)=I_3.\end{cases}$$
By Leibniz formula we have for $n\geq 1$,
$$  d\frac{\partial^n Y_{t,x}}{\partial t^n}\mid_{t=0}=
\frac{\partial^n (t Y_{t,x} \alpha_x)}{\partial t^n}\mid_{t=0}=
n\frac{\partial^{n-1} Y_{t,x}}{\partial t^{n-1}}\mid_{t=0}\alpha_x.$$
Hence integrating from $0$ to $z$ yields
$$\frac{\partial^n Y_{t,x}(z)}{\partial t^n}\mid_{t=0}=
n \int_0^z \frac{\partial^{n-1} Y_{t,x}}{\partial t^{n-1}}\mid_{t=0}\alpha_x.$$
Using the recursive definition of iterated integrals we obtain through induction
\begin{eqnarray*}
\frac{\partial^n Y_{t,x}(z)}{\partial t^n}\mid_{t=0}&=&n! \int_0^z (\alpha_x)^n\\
&=&n!\int_0^z\sum_{i_1,\cdots,i_n} x_{i_1}\cdots x_{i_n}M_{i_1}\cdots M_{i_n}\omega_{i_1}\cdots \omega_{i_n}\\
&=&n!\sum_{i_1,\cdots,i_n} x_{i_1}\cdots x_{i_n}M_{i_1}\cdots M_{i_n}\Omega_{i_1,\cdots,i_n}(z).
\end{eqnarray*}
The proposition then follows from $\mathfrak p(t,x)=Y_{t,x}^{31}(1)$ and $\mathfrak q(t,x)=Y_{t,x}^{32}(\ii)$.
\end{proof}
The following proposition characterizes the $\Omega$-values which appear in
the expansions of $\mathfrak p$ and $\mathfrak q$ in Proposition \ref{prop:expandpq}.
\begin{proposition}
\label{prop:pattern}
Consider the following graph:
$$\xymatrix{
&e_2\ar@{-}[ld]_2\ar@{-}[rd]^1\\
e_3 \ar@{-}[rr]^3&&e_1}$$
Then
\begin{enumerate}
\item $(M_{i_1}\cdots M_{i_n})_{31}\neq 0$ if and only if $(i_1,\cdots,i_n)$ labels the edges of a path from $e_3$ to
$e_1$.
\item $(M_{i_1}\cdots M_{i_n})_{32}\neq 0$ if and only if $(i_1,\cdots,i_n)$ labels the edges of a path from $e_3$ to $e_2$.
\item All $\Omega$-values which appear in the expansion of $\mathfrak p$ and $\mathfrak q$ are purely imaginary complex numbers.
\end{enumerate}
\end{proposition}
For example, the only $\Omega$-values of depth at most 4 appearing in the expansion of $\mathfrak p$ are (integrated from $z=0$ to $z=1$):
$$
\Omega_{3},\quad
\Omega_{2,1},\quad
\Omega_{2,2,3},\quad
\Omega_{3,1,1},\quad
\Omega_{3,3,3},\quad
\Omega_{2,1,1,1},\quad
\Omega_{2,2,2,1},\quad
\Omega_{2,1,3,3},\quad
\Omega_{3,1,2,3},\quad
\Omega_{3,3,2,1},$$
and for $\mathfrak q$ they are  (integrated from $z=0$ to $z=\ii$):
$$\Omega_{2},\quad
\Omega_{3,1},\quad
\Omega_{2,1,1},\quad
\Omega_{3,3,2},\quad
\Omega_{2,2,2},\quad
\Omega_{2,1,3,2},\quad
\Omega_{2,2,3,1},\quad
\Omega_{3,1,2,2},\quad
\Omega_{3,1,1,1},\quad
\Omega_{3,3,3,1}.$$
\begin{proof}
 Let $(e_1,e_2,e_3)$ be the canonical basis of $\C^3$ (seen as line vectors).
From
$$\begin{array}{lllll}
e_1 M_1=2\ii e_2 && e_1 M_2=0 &&e_1 M_3=2\ii e_3\\
e_2 M_1=-2\ii e_1&&e_2 M_2=-2\ii e_3&& e_2 M_3=0\\
e_3 M_1=0&&e_3 M_2=2\ii e_2 && e_3 M_3=-2\ii e_1
\end{array}$$
we see that $e_i M_j$ is non-zero 
if and only if there is an edge labelled $j$ adjacent to $e_i$, in which case $e_i M_j$ is $\pm 2\ii$ times its endpoint. So by induction
$e_3 M_{i_1}\cdots M_{i_n}$ is nonzero if and only if $(i_1,\cdots,i_n)$ labels a path 
starting at $e_3$ on the graph, in which case it is $\pm (2\ii)^n$ times its endpoint.
Points 1 and 2 follow.
\medskip

Regarding Point 3, we have on the segment $[0,1]$ that $\omega_1\in \ii\R$,
$\omega_2\in \R$ and $\omega_3\in \ii\R$. All paths from $e_3$ to $e_1$ must have an odd number of edges that are labelled with 1 or 3. Therefore, all corresponding $\Omega$-values at $z=1$ are purely imaginary.
In the same way, we have on the segment $[0,\ii]$ that $\omega_1\in \ii\R$, 
$\omega_2\in \ii\R$ and $\omega_3\in\R$. All paths from $e_3$ to $e_2$ have an odd number of edges that are labelled with 1 or 2. Therefore, all corresponding $\Omega$-values at $z=\ii$ must also be purely imaginary.
\end{proof}

\section{First order derivatives} \label{sec:1storder}
Let $(x(t),\theta(t))$ be the solution of the Monodromy Problem \eqref{monodromy-problem3}.
In this section, we compute the first order derivatives $x'$, $\theta'$ with respect to $t$ at $t=0$.
From this, we obtain first order estimates for the mean curvature and Willmore energy of $f_{g,\varphi}$.
\subsection{First order derivatives of the parameters}
\begin{proposition}
\label{prop:derivatives}
The first order derivatives of the parameters at $t=0$ are given by
\begin{align*}
x_1'&=\frac{1}{\pi}\sin(\varphi)\cos(\varphi)(\Omega_{2,1}(1)+\Omega_{3,1}(\ii))(\lambda^2+1)\\
x_2'&=-\frac{\ii}{\pi}\cos(\varphi)\Omega_{3,1}(\ii)(\lambda^2+1)+\frac{2\ii}{\pi}\sin^2(\varphi)\cos(\varphi)(\Omega_{2,1}(1)+\Omega_{3,1}(\ii))\\
x_3'&=-\frac{\ii}{\pi}\sin(\varphi)\Omega_{2,1}(1) (\lambda^2+1)+\frac{2 \ii}{\pi}\cos^2(\varphi)\sin(\varphi)(\Omega_{2,1}(1)+\Omega_{3,1}(\ii))\\
\theta'&=\frac{-2\ii}{\pi}\sin(\varphi)\cos(\varphi)(\Omega_{2,1}(1)+\Omega_{3,1}(\ii))\,.
\end{align*}
\end{proposition}
We will compute the required $\Omega$-values in Section \ref{section:depth2integrals}.
\begin{proof}

Define
$$\begin{cases}
\wt{\mathfrak p}(t)=\wh{\mathfrak p}(t,x(t))=\frac{1}{t}\mathfrak p(t,x(t))\\
\wt{\mathfrak q}(t)=\wh{\mathfrak q}(t,x(t))\\
\mathcal H_1(t)=\wh{\mathfrak p}(t,x(t))(\lambda=e^{\ii\theta(t)})\\
\mathcal H_2(t)=\wh{\mathfrak q}(t,x(t))(\lambda=e^{\ii\theta(t)})\\
\mathcal K(t)=x_1(t)^2+x_2(t)^2+x_3(t)^2.\end{cases}$$
By Proposition \ref{prop:pattern}, the first order expansion of $\widehat{\mathfrak p}$ and $\widehat{\mathfrak q}$ are
$$\widehat{\mathfrak p}(t,x)=2\pi x_3+ 4 t\, \Omega_{2, 1}(1) x_1 x_2 +O(t^2),\qquad
\widehat{\mathfrak q}(t,x)=2\pi x_2 + 4 t\, \Omega_{3, 1}(\ii) x_1 x_3+O(t^2).$$
Hence at $t=0$,
$$\wt{\mathfrak p}'=2\pi x_3'+4 \,\Omega_{2,1}(1)\cv{x}_1\cv{x}_2,\qquad
\wt{\mathfrak q}'=2\pi x_2'+4 \,\Omega_{3,1}(\ii)\cv{x}_1\cv{x}_3.$$
Since $x_3'=y_3'\in\mathcal W^{\geq 0}$, this gives
\begin{eqnarray*}
(\wt{\mathfrak p}'-\wt{\mathfrak p}'^*)^+&=&2\pi (x_3')^+ + 4(\Omega_{2,1}(1)\cv{x}_1\cv{x_2}-\overline{\Omega_{2,1}(1)}\cv{x}_1^*\cv{x}_2^*)^+\\
0&=&2\pi (x_3')^+ +2\ii\sin(\varphi)\Omega_{2,1}(1)\lambda^2.\end{eqnarray*}
Hence
$$(x_3')^+=\frac{-\ii}{\pi}\sin(\varphi)\Omega_{2,1}(1)\lambda^2$$
and in the same way,
$$(x_2')^+=\frac{-\ii}{\pi}\cos(\varphi)\Omega_{3,1}(\ii)\lambda^2.$$
Since $\mathcal H_1(t)$ for all $t$, we have $\mathcal H_1'=0$, so
\begin{eqnarray}
0&=&\wt{\mathfrak p}'(\lambda=\ii)+\frac{\partial \wt{\mathfrak p}}{\partial\lambda}\mid_{\lambda=\ii} (-\theta')\nonumber\\
&=&2\pi \left((x'_3)^0+(x'_3)^+(\ii)\right)+4\Omega_{2,1}(1)\cv{x}_1(\ii)\cv{x}_2(\ii)
+2\pi\cos(\varphi)\theta'\nonumber\\
&=&2\pi (x_3')^0+2\pi \cos(\varphi)\theta'+2\ii\sin(\varphi)\Omega_{2,1}(1).
\label{eq:deriv1}
\end{eqnarray}
In the same way, the equation $\mathcal H_2'=0$ gives
\begin{equation}
\label{eq:deriv2}
2\pi (x_2')^0+2\pi \sin(\varphi)\theta'+2\ii\cos(\varphi)\Omega_{3,1}(\ii)=0.
\end{equation}
Finally, we compute at $t=0$
\begin{eqnarray*}
\lambda\mathcal K'&=&\sum_{j=1}^32\lambda\cv{x}_jx_j'\\
&=&\ii(1-\lambda^2)x_1'-(\lambda^2+1)\sin(\varphi)\left((x_2')^0+(x_2')^+\right)-(\lambda^2+1)\cos(\varphi)\left((x_3')^0+(x_3')^+\right)\\
&=&\ii(1-\lambda^2)x_1'-(\lambda^2+1)\sin(\varphi)(x_2')^0-(\lambda^2+1)\cos(\varphi)(x_3')^0\\
&&+\frac{\ii}{\pi}(\lambda^2+1)\lambda^2\sin(\varphi)\cos(\varphi)(\Omega_{2,1}+\Omega_{3,1}).
\end{eqnarray*}
Since $\mathcal K$ is constant (with respect to $\lambda$), the remainder of the division of $\lambda\mathcal K'$ by $(\lambda^2-1)$ gives
$$\lambda\mathcal K'=-2\sin(\varphi)(x'_2)^0-2\cos(\varphi)(x'_3)^0+\frac{2\ii}{\pi}\sin(\varphi)\cos(\varphi)(\Omega_{2,1}(1)+\Omega_{3,1}(\ii)).$$
Hence $\mathcal K'=0$ (in agreement with Proposition \ref{prop:symmetries}) and
\begin{equation}
\label{eq:deriv3}
\sin(\varphi)(x'_2)^0+\cos(\varphi)(x'_3)^0=\frac{\ii}{\pi}\sin(\varphi)\cos(\varphi)(\Omega_{2,1}+\Omega_{3,1})\,.
\end{equation}
Solving the system of equations \eqref{eq:deriv1}, \eqref{eq:deriv2}, \eqref{eq:deriv3} determine $\theta'$, $(x'_2)^0$
and $(x'_3)^0$ to be as stated in Proposition \ref{prop:derivatives}.
Finally, the quotient of $\lambda\mathcal K'$ by $(\lambda^2-1)$ gives
$$0=-\ii x'_1-\sin(\varphi)(x'_2)^0-\cos(\varphi)(x'_3)^0+\frac{\ii}{\pi}\sin(\varphi)\cos(\varphi)(\Omega_{2,1}(1)+\Omega_{3,1}(\ii))(\lambda^2+2)$$
which together with equation \eqref{eq:deriv3} determines $x'_1$.
\end{proof}
\subsection{Mean curvature}
In this section and the next one, we need to emphasize the dependence of objects on the angle $\varphi$
so we write $x(t,\varphi)$ and $\theta(t,\varphi)$.
Let $H_{g,\varphi}$ be the mean curvature of $f_{g,\varphi}$. From Proposition \ref{prop:symmetries}  we have $H_{g,\varphi}=-H_{g,\tfrac{\pi}{2}-\varphi}$.
\begin{proposition}
\label{prop:H}
For $g \gg1 $ the mean curvature $H_{g,\varphi}$ is strictly positive for all 
$\varphi\in (0, \frac{\pi}{4})$.
\end{proposition}
\begin{proof}
We have
$$H_{g,\varphi}=\operatorname{cotan}(\theta(t,\varphi))\quad\text{with}\quad
t\sqrt{\mathcal K(t,\varphi)}=s=\frac{1}{2g+2}.$$ 
We want to find a uniform $\varepsilon>0$ such that for
$$\forall t\in(0,\varepsilon),\quad\forall\varphi\in(0,\frac{\pi}{4}),\quad 0<\theta(t,\varphi)<\frac{\pi}{2}.$$ 
By Propositions \ref{prop:derivatives} and \ref{prop:depth2-integrals}, we have
$$\frac{\partial\theta}{\partial t}(0,\varphi)=2\sin(2\varphi)\log(\tan(\varphi))$$
which is negative for $0<\varphi<\frac{\pi}{4}$.
So the existence of $\varepsilon$ is ensured for $\varphi$ lying  in any proper subinterval of $(0,\frac{\pi}{4})$.
To study $\theta$ near $\varphi=0$ and $\varphi=\frac{\pi}{4}$, 
define
$$\wh{\theta}(t,\varphi)=\frac{1}{t}\left(\theta(t,\varphi)-\frac{\pi}{2}\right)$$
which extends analytically to $t=0$ with
$$\wh{\theta}(0,\varphi)=2\sin(2\varphi)\log(\tan(\varphi)).$$
We have $\frac{\partial \wh{\theta}}{\partial\varphi}(0,\frac{\pi}{4})=4$.
Since $\wh{\theta}(t,\frac{\pi}{4})=0$ for all $t$,
a first order Taylor expansion of $\wh{\theta}$ at $(t,\frac{\pi}{4})$ gives
$$\wh{\theta}(t,\varphi)=(4+O(t))(\varphi-\tfrac{\pi}{4})+O((\varphi-\tfrac{\pi}{4})^2)$$
where the second $O$ is uniform with respect to $t$,
so $\wh{\theta}(t,\varphi)<0$ for $t$ small enough and $\varphi<\frac{\pi}{4}$ close
enough to $\frac{\pi}{4}$.

By Proposition \ref{prop:uniformT}, $\wh{\theta}$ extends as an analytic function of $\varphi$ and
$\varphi\log(\varphi)$ in a neighborhood of $\varphi=0 \in \C$. Thus there exists an analytic function $\Theta(t,\varphi,\phi)$ such that
$$\wh{\theta}(t,\varphi)=\Theta(t,\varphi,\varphi\log(\varphi)).$$
At $t=0$, we can rewrite
$$\wh{\theta}(0,\varphi)=2\sin(2\varphi)\log\left(\frac{h(\varphi)}{\cos(\varphi)}\right)+4h(2\varphi)\varphi\log(\varphi)$$
with $h(x)=\frac{\sin(x)}{x}$, which extends analytically to $x=0$.
Hence by uniqueness of $\Theta$,
$$\Theta(0,\varphi,\phi)=2\sin(2\varphi)\log\left(\frac{h(\varphi)}{\cos(\varphi)}\right)+4h(2\varphi)\phi.$$

We have
$\frac{\partial\Theta}{\partial\varphi}(0,0,0)=0$ and
$\frac{\partial\Theta}{\partial\phi}(0,0,0)=4$.
Since $\Theta(t,0,0)=0$ for all $t$, a first order Taylor expansion at $(t,0,0)$ gives
$$\Theta(t,\varphi,\phi)=O(t)\varphi+(4+O(t))\phi+O(\varphi^2+\phi^2).$$
We substitute $\phi=\varphi\log(\varphi)$ and obtain
$$\wh\theta(t,\varphi)=(4+O(t))\varphi\log(\varphi)+O((\varphi\log(\varphi))^2)$$
where the second $O$ is uniform with respect to $t$.
Hence $\wh\theta(t,\varphi)<0$ for $t$ and $\varphi>0$ small enough.
\end{proof}

\subsection{Willmore energy and Area}$\;$\\
Recall that $t$ is related to $s=\frac{1}{2g+2}$ by
$t\sqrt{\mathcal K(t,\varphi)}=s.$
Since $\mathcal K'=0$, we have
$$\mathcal K(t,\varphi)=1+O(t^2)$$
 so $s\sim t$. By Point (2) of Proposition \ref{prop:building}, Propositions \ref{prop:derivatives} and \ref{prop:depth2-integrals}, we obtain
\begin{eqnarray*}
\mathcal W(f_{g,\varphi)}&=&
8\pi\left(1-\cos(\varphi)(x'_2)^0 s+\sin(\varphi)(x'_3)^0 s+O(s^2)\right)\\
&=&8\pi\left(1+s\cos^2(\varphi)\tfrac{\ii}{\pi}\Omega_{3,1}(\ii)-s\sin^2(\varphi)\tfrac{\ii}{\pi}\Omega_{2,1}(1)+O(s^2)\right)\\
&=&8\pi\left(1+2 s\cos^2(\varphi)\log(\cos(\varphi))+2s\sin^2(\varphi)\log(\sin(\varphi))+O(s^2)\right).
\end{eqnarray*}
In particular, in the minimal case $\varphi=\frac{\pi}{4}$
\begin{equation}\label{eq:lawarea}\mathcal W(f_{g,\pi/4})=\operatorname{Area}(\xi_{1,g})=8\pi\left(1-\ln(2) s+O(s^2)\right)\end{equation}
which we already obtained in \cite{HHT}. 
\begin{proposition}\label{prop:willmore}
For $g$ large enough, the immersions $f_{g,\varphi}$ satisfies  $\mathcal W(f_{g,\varphi})<8\pi$ for all $\varphi\in(0,\frac{\pi}{2})$.
As a consequence, the surfaces $f_{g,\varphi}$ are embedded by the
Li-Yau estimate \cite{LiYau}.
\end{proposition}
\begin{proof} 
Fix $g\gg1$ such that the complete family of CMC surfaces $f_{g,\varphi}$ of genus $g$ exists.
We claim that the Willmore energy $\mathcal W(f_{g,\varphi})$ is strictly monotonic for $\varphi\in(0,\tfrac{\pi}{4}).$
Note that $\mathcal W(f_{g,\varphi})=\mathcal W(f_{g,\pi/2-\varphi})$ by Proposition \ref{prop:symmetries}.
Then the proposition follows, since the Willmore energy for the Lawson surfaces $\xi_{1,g} = f_{g,\pi/4}$ is strictly below $8\pi$ and the Willmore energy at $\varphi=0$ is $8\pi$.

It is well-known that CMC surfaces are critical points of the 
Willmore functional under conformal variations \cite{BPP}.
The cotangent space to the Teichm\"uller space at the Riemann surface $M$ can be identified with the space of holomorphic quadratic differentials $H^0(M,K^2)$.
The Lagrange-multiplier (see \cite[Corollary 16 and Remark 4]{BPP}) for the Willmore functional is given by $\tfrac{1}{2}HQ$,
as the Hopf differential $Q$ is holomorphic for CMC surfaces. Let $\Pi$ denote the projection from the space of immersions to the Teichm\"uller space and consider the image of the map $\Pi(f_{g,\varphi})$ which is a real $1-$dimensional submanifold, as $\varphi$ determines the conformal structure of the surface. For the symmetric surfaces the space of  holomorphic and symmetric quadratic differentials  is real 2 dimensional. Moreover, it is well known that $\Pi$  fails to be submersive at isothermic, and hence at CMC, surfaces. Therefore, the pairing between
$Q$ and $\frac{\partial}{\partial\varphi}$ is non-degenerate and  the Willmore functional is monotonic along the
family of CMC surfaces as long as $H\neq0.$ Thus, the result follows from Proposition \ref{prop:H}.
\end{proof}
\begin{remark}
Proposition \ref{prop:willmore} can also be proven   analogously to Proposition  \ref{prop:H} using the first order derivatives.
\end{remark}

\begin{corollary}
For all $\varphi \in (0, \tfrac{\pi}{2})$ and $g \geq g_0$, there exists a
 constrained Willmore minimizer in the conformal class of $f_{g,\varphi}$.\end{corollary}
\begin{proof}
This follows from Proposition \ref{prop:willmore} and \cite{KuwertSchatzle}, which gives the existence of Willmore energy minimizer with prescribed conformal class 
provided that the infimum energy in the conformal class is below $8 \pi.$
\end{proof}
\subsection{Computation of the depth-2 integrals}
\label{section:depth2integrals}
\begin{proposition}
\label{prop:depth2-integrals}
For $\varphi \in (0, \tfrac{\pi}{2})$ we have 
$$\Omega_{2,1}(1)=2\pi \ii\log(\sin(\varphi))\quad\text{ and }\quad
\Omega_{3,1}(\ii)=-2\pi\ii \log(\cos(\varphi)).$$
\end{proposition}
\begin{proof}
Let $\gamma$ be the closed curve given by the composition of the real half-line from $0$ to $+\infty$ with the imaginary half-line from $+\ii\infty$ to $0$.
We compute the iterated integral $\int_{\gamma}\omega_2\omega_1$ in two different ways: using the symmetries and by applying the Residue Theorem.
Recall that for $\tau(z)=\frac{1}{z}$ we have the symmetries
$\tau^*\omega_1=-\omega_1$ and $\tau^*\omega_2=\omega_2$.
Since $\tau(1)=1$ we have $\tau^*\Omega_2=\Omega_2$.
Then using the change of variable rule we obtain
\[\int_1^{+\infty}\Omega_2\omega_1=-\int_0^1\tau^*(\Omega_2\omega_1)
=\int_0^1\Omega_2\omega_1.\]
Hence
\[\int_0^{+\infty}\Omega_2\omega_1=
2\int_0^1\Omega_2\omega_1=2\,\Omega_{2,1}(1).\]
Let $\nu(z)=\frac{-1}{z}$. We have the symmetries
$\nu^*\omega_1=-\omega_1$ and $\nu^*\omega_2=-\omega_2$.
Since $\nu(\ii)=\ii$, we have $\nu^*\Omega_2=2\Omega_2(\ii)-\Omega_2$.
Then
\[\int_{\ii}^{\ii \infty}\Omega_2\omega_1=
-\int_0^{\ii}(2\Omega_2(\ii)-\Omega_2)(-\omega_1)=2\,\Omega_2(\ii)\Omega_1(\ii)-\int_0^{\ii}\Omega_2\omega_1.\]
Hence
\[\int_0^{\ii\infty}\Omega_2\omega_1=2\,\Omega_2(\ii)\Omega_1(\ii).\]
Note that in these computations, $\Omega_2$
denotes the primitive of $\omega_2$ on the simply connected domain $\mathcal U$.
On the other hand, the analytic continuation of $\Omega_2$ along $\gamma$ is equal to $\Omega_2+2\pi\ii$
on the second segment from $\ii\infty$ to $0$. Hence using Equation \eqref{eq:OmegaI}
\begin{eqnarray*}
\int_{\gamma}\omega_2\omega_1&=&\int_0^{\infty}\Omega_2\omega_1-\int_0^{\ii\infty}
(\Omega_2+2\pi \ii)\omega_1\\
&=&2\,\Omega_{2,1}(1)-2\Omega_2(\ii)\Omega_1(\ii)-4\pi\ii\Omega_1(\ii)\\
&=&2\,\Omega_{2,1}(1)-4\pi\varphi.
\end{eqnarray*}

Next we compute the integral using the Residue Theorem.
Define for $j=1,2$
$$\wt\omega_j=\omega_j-\frac{dz}{z-p_1}\quad\text{ and }\quad
\wt\Omega_j=\int_0^z\wt\omega_j.$$
Then $\wt\Omega_j$ is a well-defined holomorphic function in the quadrant bounded
by $\gamma$. By the Residue Theorem and integration by parts
\begin{eqnarray*}
\int_{\gamma}\omega_2\omega_1&=&
\int_{\gamma}\left(\wt\Omega_2+\log\left(1-\tfrac{z}{p_1}\right)\right)\left(\wt\omega_1+\tfrac{dz}{z-p_1}\right)\\
&=&\int_{\gamma}\wt\Omega_2\tfrac{dz}{z-p_1}+\int_{\gamma}\log\left(1-\tfrac{z}{p_1}\right)\wt\Omega_1'+\int_{\gamma}\log\left(1-\tfrac{z}{p_1}\right)\tfrac{dz}{z-p_1}\\
&=&2\pi\ii \,\wt\Omega_2(p_1)+\left[\log\left(1-\tfrac{z}{p_1}\right)\wt\Omega_1\right]_{\gamma(0)}^{\gamma(1)}-\int_{\gamma}\wt\Omega_1\tfrac{dz}{z-p_1}+\left[\tfrac{1}{2}\left(\log\left(1-\tfrac{z}{p_1}\right)\right)^2\right]_{\gamma(0)}^{\gamma(1)}\\
&=&2\pi\ii(\wt\Omega_2(p_1)-\wt\Omega_1(p_1))-2\pi^2\\
&=&4\pi\ii\log\left(\frac{1-\tfrac{p_1}{p_4}}{1-\tfrac{p_1}{p_3}}\right)-2\pi^2\\
&=&4\pi\ii\log(\sin(\varphi))-4\pi\varphi.
\end{eqnarray*}

Comparing the two expressions we obtain
$$\Omega_{2,1}(1)=2\pi\ii\log(\sin(\varphi)).$$
The second $\Omega$-value can be deduced by symmetry.
As in the proof of Proposition \ref{prop:symmetries} let $\widefrown{\varphi}=\frac{\pi}{2}-\varphi$ and $\iota(z)=\ii z$. Equation \eqref{eq:iota*omega} then yields
$$\int_0^{\ii}\omega_{3,\widefrown{\varphi}}\,\omega_{1,\widefrown{\varphi}}=\int_0^1\iota^*(\omega_{3,\widefrown{\varphi}}\,\omega_{1,\widefrown{\varphi}})
=-\int_0^1\omega_{2,\varphi}\omega_{1,\varphi}.$$
Hence
$$\Omega_{3,1}(\ii)(\tfrac{\pi}{2}-\varphi)=-\Omega_{2,1}(1)(\varphi).$$
\end{proof}
\section{Higher order derivatives}\label{sec:higheroder}
In this section, we present an algorithm to compute the derivatives of the parameters with respect to $t$ to any order. In the most symmetric case $\varphi=\pi/4$, the algorithm is simple enough that the derivatives of order 2 and 3 can be computed by hand. This gives the expansion of the area of Lawson minimal surfaces to order 3 in term of $\Omega$-values of depth at most 4.
\subsection{Iterative algorithm to compute the derivatives}
\label{section:algorithm}
We expand the solution $x(t),\theta(t)$ and $\mathcal K(t)$, which are analytic in $t$, as
$$x_j(t)=\sum_{n=0}^{\infty} x_{j,n} t^n,\qquad
\theta(t)=\sum_{n=0}^{\infty} \theta_n t^n,\qquad
\mathcal K(t)=\sum_{n=0}^{\infty}\mathcal K_n t^n$$
so $x_j^{(n)}=n! \,x_{j,n}$, but it will be more convenient to use the coefficients $x_{j,n}$ than the derivative
$x_j^{(n)}$.
We also use the following notation : the index $[k]$ denotes the expansion to order $k$ in $t$, so for example
$$x_{j,[k]}=\sum_{\ell=0}^k x_{j,\ell}t^{\ell}.$$
\begin{proposition}
\label{prop:algorithm}
For $n\geq 1$, the terms $x_{1,n}$, $x_{2,n}$ and $x_{3,n}$ are polynomials in $\lambda$ of degree at most $n+1$.
Moreover, their coefficients, as well as the ones of $\theta_n$ and $\mathcal K_n$, can be expressed as functions of $\varphi$ and $\Omega$-values of depth at most $(n+1)$ if $n$ is odd and $n$ if $n$ is even.
\end{proposition}
\begin{proof}
The proof describes an algorithm to compute $x_{j,n}$, $\theta_n$ and $\mathcal K_n$ for all $n$.
 Fix $n\geq 1$ and assume that $x_{j,k}$ and $\theta_k$ have already been computed for all
$k<n$ and satisfy the conclusion of Proposition \ref{prop:algorithm}.
The method to compute $x_{j,n}$ and $\theta_n$ is the same as in the proof of Proposition \ref{prop:derivatives}
(where $n=1$), except that there are more lower order terms.
As in the proof of Proposition \ref{prop:derivatives} we denote by
$$\wt{\mathfrak p}(t)=\wh{\mathfrak p}(t,x(t))=t^{-1}\mathfrak p(t,x(t)).$$
By Proposition \ref{prop:expandpq} we can write
\begin{equation}
\label{eq:expandp}\wt{\mathfrak p}(t)=2\pi x_3(t)+\wt{\mathfrak p}_{\low}(t) + o(t^n)
\end{equation}
with
$$\wt{\mathfrak p}_{\low}(t)=\sum_{k=1}^n \sum_{i_1,\cdots,i_{k+1}}t^k  x_{i_1,[n-k]}\cdots x_{i_{k+1},[n-k]}(M_{i_1}\cdots M_{i_{k+1}})_{31} \Omega_{i_1,\cdots,i_{k+1}}(1).$$
Here the index ``lower" denotes a quantity which only depends on terms of order $<n$, i.e., terms satisfying the induction hypothesis. Note that the coefficient of $t^n$ in $t^k x_{i_1,[n-k]}\cdots x_{i_{k+1},[n-k]}$ is
$$\sum_{j_1+\cdots+j_{k+1}=n-k} x_{i_1,j_1}\cdots x_{i_{k+1},j_{k+1}}.$$
Each term in this sum is a Laurent polynomial in $\lambda, $ where 
the negative degree is at most $k$ (because $x_{j,0}=\cv{x}_j$ are Laurent polynomials with negative degree one).
By the induction hypothesis, the positive degree is at most
$$\sum_{\ell=1}^{k+1} (j_{\ell}+1)=(n-k)+(k+1)=n+1.$$
Hence the coefficient of $t^n$ in $\wt{\mathfrak p}$ is
$$\wt{\mathfrak p}_n=2\pi x_{3,n}+\wt{\mathfrak p}_{\low,n}$$
where $\wt{\mathfrak p}_{\low,n}$ is a Laurent polynomial in $\lambda$ of degree at most $n+1$.
This gives
$$0=(\wt{\mathfrak p}_n-\wt{\mathfrak p}_n^*)^+
=2\pi x_{3,n}^+ +(\wt{\mathfrak p}_{\low,n}-\wt{\mathfrak p}_{\low,n}^*)^+$$
which determines
$$x_{3,n}^+=\frac{-1}{2\pi}(\wt{\mathfrak p}_{\low,n}-\wt{\mathfrak p}_{\low,n}^*)^+$$
This is a (true) polynomial of degree at most $(n+1)$.
Moreover, by the induction hypothesis, all $\Omega$-values in $\wt{\mathfrak p}_{\low,n}$ have depth
at most $n+1$, and the only terms of depth $n+1$ are of the form
\begin{equation}
\label{eq:term}
\cv{x}_{i_1}\cdots \cv{x}_{i_{n+1}}(M_{i_1}\cdots M_{i_{n+1}})_{31}\Omega_{i_1,\cdots,i_{n+1}}(1).
\end{equation}
Now $(M_{i_1}\cdots M_{i_{n+1}})_{31}$ is either zero or $\pm (2\ii)^{n+1}$. If $n$ is even, this is an imaginary number. By Proposition \ref{prop:pattern}, $\Omega_{i_1,\cdots,i_{n+1}}(1)$ is also imaginary.
Moreover, $\cv{x}_j=\cv{x}_j^*$ for $j=1,2,3$. So each term of the form \eqref{eq:term}
is invariant under $*$ so disappears in
$x_{3,n}^+$ if $n$ is even. In the same way,
$$x_{2,n}^+=\frac{-1}{2\pi}(\wt{\mathfrak q}_{\low,n}-\wt{\mathfrak q}_{\low,n}^*)^+$$
with
$$\wt{\mathfrak q}_{\low}(t)=\sum_{k=1}^n \sum_{i_1,\cdots,i_{k+1}}t^k  x_{i_1,[n-k]}\cdots x_{i_{k+1},[n-k]}(M_{i_1}\cdots M_{i_{k+1}})_{32} \Omega_{i_1,\cdots,i_{k+1}}(\ii).$$
We expand $\mathcal H_1(t)$ to order $n$ as follows.
First of all,
\begin{eqnarray*}
\cv{x}_3(e^{\ii\theta(t)})&=&-\cos(\varphi)\cos(\theta(t))\\
&=&-\cos(\varphi)(\cos(\theta_{[n-1]})\cos(\theta_n t^n)-\sin(\theta_{[n-1]})\sin(\theta_n t^n))+o(t^n)\\
&=&-\cos(\varphi)(\cos(\theta_{[n-1]})-\theta_n t^n)+o(t^n).
\end{eqnarray*}
By Equation \eqref{eq:expandp} evaluated at $\lambda=e^{i\theta(t)}$:
\begin{eqnarray*}\mathcal H_1(t)&=&2\pi x_3(e^{\ii\theta(t)})+\wt{\mathfrak p}_{\low}(e^{\ii\theta(t)})+o(t^n)\\
&=&2\pi\sum_{k=0}^n x_{3,k}(e^{\ii\theta_{[n-k]}})t^k+\wt{\mathfrak p}_{\low}(e^{\ii\theta_{[n-1]}})+o(t^n)\\
&=&2\pi\cos(\varphi)\theta_n t^n + 2\pi x_{3,n}^0 t^n+\mathcal H_{1,\low}(t)+o(t^n)
\end{eqnarray*}
where $\mathcal H_{1,\low}$ contains all the terms which are already known under the induction hypothesis:
$$\mathcal H_{1,\low}(t)=
-2\pi\cos(\varphi)\cos(\theta_{[n-1]})+2\pi\sum_{k=1}^{n-1}x_{3,k}(e^{\ii\theta_{[n-k]}})t^k
+2\pi x_{3,n}^+(\ii)t^n+\wt{\mathfrak p}_{\low}(e^{\ii\theta_{[n-1]}}).$$
We expand the function $\mathcal H_{1,\low}(t)$ in series to order $n$ in $t$ and denote
$\mathcal H_{1,\low,n}$ the coefficient of $t^n$.
(To get an idea of the complexity of this step, remember that $\theta_{[k]}$ denotes the expansion of $\theta(t)$ to order $k$ in
$t$, and $x_{3,k}$ is a polynomial of degree at most $k+1$ in $\lambda$.)
Since $\mathcal H_1(t)=0$ for all $t$, we obtain, by looking at the coefficient of $t^n$, the equation
\begin{equation}
\label{eq:H1n}
2\pi \cos(\varphi)\theta_n+ 2\pi x_{3,n}^0+\mathcal H_{1,\low,n}=0.
\end{equation}
In the same way, expanding $\mathcal H_2(t)$ gives the equation
\begin{equation}
\label{eq:H2n}
2\pi \sin(\varphi)\theta_n+ 2\pi x_{2,n}^0+\mathcal H_{2,\low,n}=0
\end{equation}
with
$$\mathcal H_{2,\low}(t)=
-2\pi\sin(\varphi)\cos(\theta_{[n-1]})+2\pi\sum_{k=1}^{n-1}x_{2,k}(e^{\ii\theta_{[n-k]}})t^k 
+2\pi x_{2,n}^+(\ii)t^n+\wt{\mathfrak q}_{\low}(e^{\ii\theta_{[n-1]}}).$$
The coefficient of $t^n$ in $\mathcal K(t)$ is
\begin{eqnarray*}
\mathcal K_n&=&\sum_{j=1}^3 \sum_{k=0}^{n}x_{j,k}x_{j,n-k}\\
&=& 2\cv{x}_1 x_{1,n}+ 2\cv{x}_2 x_{2,n}^0+2\cv{x}_{3}x_{3,n}^0+
\mathcal K_{\low,n}
\end{eqnarray*}
where $\mathcal K_{\low,n}$ contains again all the terms which are already known:
$$\mathcal K_{\low,n}=2\cv{x}_2 x_{2,n}^+ + 2 \cv{x}_3 x_{3,n}^+  +\sum_{j=1}^3 \sum_{k=1}^{n-1}x_{j,k}x_{j,n-k}.$$
Multiplying by $\lambda$, we obtain the equation
\begin{equation}
\label{eq:Kn}
\lambda{\mathcal K}_n=\ii(1-\lambda^2)x_{1,n}-\sin(\varphi)(\lambda^2+1)x_{2,n}^0-\cos(\varphi)(\lambda^2+1)x_{3,n}^0+\lambda\mathcal K_{\low,n}.
\end{equation}
Observe that $\mathcal K(t)$ is constant in $\lambda$, so  $\mathcal K_n$ is a constant.
Using the induction hypothesis, it is not hard to see that
$\lambda\mathcal K_{\low,n}$ is a (true) polynomial of degree at most $n+3$.
Let $Q_n$, $R_n$ be the quotient and remainder of the division of $\lambda\mathcal K_{\low,n}$
by $(\lambda^2-1)$. The quotient and remainder of the division of \eqref{eq:Kn}
by $(\lambda^2-1)$ give us respectively the equations
\begin{equation}
\label{eq:quotient}
-\ii x_{1,n}-\sin(\varphi)x_{2,n}^0-\cos(\varphi)x_{3,n}^0+Q_n=0
\end{equation}
\begin{equation}
\label{eq:remainder}
-2\sin(\varphi)x_{2,n}^0-2\cos(\varphi)x_{3,n}^0+R_n=\lambda{\mathcal K}_n.
\end{equation}
Equation \eqref{eq:remainder} at $\lambda=0$ gives
\begin{equation}
\label{eq:remainder0}
\sin(\varphi)x_{2,n}^0+\cos(\varphi)x_{3,n}^0=\tfrac{1}{2}R_n(0).
\end{equation}
Equations \eqref{eq:quotient} and \eqref{eq:remainder0} give
$$x_{1,n}=\tfrac{\ii}{2} R_n(0)-\ii Q_n$$
which is a polynomial of degree at most $n+1$.
The equations \eqref{eq:H1n}, \eqref{eq:H2n} and \eqref{eq:remainder0} determine
$\theta_n$, $x_{2,n}^0$ and $x_{3,n}^0$. In particular
$$\theta_n=-\frac{\sin(\varphi)}{2\pi}\mathcal H_{2,\low,n}-\frac{\cos(\varphi)}{2\pi}\mathcal H_{1,\low,n}-\frac{R_n(0)}{2}.$$
Finally, equation \eqref{eq:remainder} determines the constant $\mathcal K_n$ as the coefficient of $\lambda$ in $R_n$.
\end{proof}
\begin{remark}
By Proposition \ref{prop:symmetries}, we have for all $n\geq 1$ and $j=1,2,3$
$$x_{j,n}(-\lambda)=(-1)^{n+1}x_{j,n}(\lambda)\quad\text{ and }\quad
\theta_n=(-1)^{n+1}\theta_n.$$
Hence if $n$ is even, $x_{j,n}^0=0$ and $\theta_n=0$. So the even steps of the algorithm are much simpler as we can avoid computing $\mathcal H_{1,\low}$ and $\mathcal H_{2,\low}$ to determine $x_{2,n}^0$, $x_{3,n}^0$ and $\theta_n$.
Also, the even order terms of the expansion of the Willmore energy and the mean curvature are zero for the same reason. 
\end{remark}
\subsection{Derivatives of order 2 and 3 in the minimal case}\label{mincasereduce}
The algorithm presented in Section \ref{section:algorithm} is much simpler in the case $\varphi=\tfrac{\pi}{4}$.
Indeed, we then have by Proposition \ref{prop:symmetries}
$$x_1(-t)=x_1(t)\quad\text{ and }\quad
x_2(-t)=x_3(t).$$
Hence $x_{1,n}=0$ if $n$ is odd and $x_{3,n}=(-1)^n x_{2,n}$, so there are less computations to do.
Moreover, $\theta(t)=\pi/2$ for all $t$ so $\theta_n=0$ for $n\geq 1$ and the formula for $x_{3,n}^0$ simplifies to
\begin{equation*}
\begin{split}
\mathcal H_{1,\low,n}&=
2\pi x_{3,n}^+(i)+\wt{\mathfrak p}_{\low,n}(i)\\
x_{3,n}^0&=\frac{-1}{2\pi}\mathcal H_{1,\low,n}.
\end{split}
\end{equation*}
\begin{proposition}
\label{prop:derivatives-minimal}
If $\varphi=\pi/4$, then
\begin{align*}
x_{1,1}&=x_{1,2}=x_{1,3}=\mathcal K_1=\mathcal K_2=\mathcal K_3=0\\
x_{3,1}&=-x_{2,1}=\tfrac{-\ii}{\pi\sqrt{2}}\Omega_{2,1}(\lambda^2+1)=\tfrac{-\log(2)}{\sqrt{2}}(\lambda^2+1)\\
x_{3,2}&=x_{2,2}=\tfrac{-1}{\sqrt{2}\pi^2}\Omega_{2,1}^2(\lambda^3+\lambda)
=\tfrac{\log(2)^2}{\sqrt{2}}(\lambda^3+\lambda)\\
x_{3,3}&=-x_{2,3}=\tfrac{-\ii\sqrt{2}}{2\pi^3}\Omega_{21}^3(\lambda^4-1)
-\tfrac{\sqrt{2}}{2\pi^2}\Omega_{2,1}\Omega_{3,1,1}(\lambda^4-2\lambda^2-3)\\
&-\tfrac{\sqrt{2}}{4\pi^2}\Omega_{2,1}(\Omega_{2,2,3}-3\Omega_{3,3,3})(\lambda^2+1)^2
+\tfrac{\ii\sqrt{2}}{2\pi}\Omega_{2,1,1,1}(\lambda^4-2\lambda^2-3)\\
&-\tfrac{\ii\sqrt{2}}{4\pi}(\Omega_{2,2,2,1}-\Omega_{3,1,2,3}+\Omega_{2,1,3,3}+\Omega_{3,3,2,1})(\lambda^2+1)^2
\end{align*}
where all $\Omega$ integrals are evaluated at $z=1$.
\end{proposition}
\begin{proof} We follow the algorithm described in Section \ref{section:algorithm}.
Proposition \ref{prop:expandpq} gives
\begin{eqnarray*}
\lefteqn{\widehat{\mathfrak p}(t,x)=2\pi x_3+4t\,\Omega_{2,1} x_1 x_2
-8\ii t^2\,\left(\Omega_{3,1,1}x_1^2x_3+\Omega_{2,2,3}x_2^2x_3+\Omega_{3,3,3}x_3^3\right)}\\
&&+16t^3\,\left(\Omega_{2,1,1,1}x_1^3x_2+\Omega_{2,2,2,1}x_1x_2^3-\Omega_{3,1,2,3}x_1x_2x_3^2+\Omega_{2,1,3,3}x_1x_2x_3^2+\Omega_{3,3,2,1}x_1x_2x_3^2\right)\\
&&+O(t^4)\,.
\end{eqnarray*}
\subsubsection{First step ($n=1$)}
$$\wt{\mathfrak p}_{\low,1}=4\, \Omega_{2,1}\cv{x}_1\cv{x}_2=\frac{-\ii}{\sqrt{2}}\Omega_{2,1}(\lambda^{-2}-\lambda^2),\qquad x_{3,1}^+=\frac{-\ii}{\pi\sqrt{2}}\Omega_{2,1}\lambda^2$$
$$\mathcal H_{1,\low,1}=\frac{2\pi \ii}{\pi\sqrt{2}} \Omega_{2,1},\qquad
x_{3,1}^0=\frac{-\ii}{\pi\sqrt{2}}\Omega_{2,1}.$$
\subsubsection{Second step ($n=2$)}
\begin{align*}
\wt{\mathfrak p}_{\low,2}&=4\,\Omega_{2,1}(\cv{x}_1 x_{2,1} + \cv{x}_2 x_{1,1})
-8\ii(\Omega_{3,1,1}\cv{x}_1^2\cv{x}_3
+\Omega_{2,2,3}\cv{x}_2^2\cv{x}_3+\Omega_{3,3,3}\cv{x}_3^3)\\
&=\frac{\sqrt{2}}{\pi}\Omega_{2,1}^2(\lambda^3-\lambda^{-1})
-8\ii(\Omega_{3,1,1}\cv{x}_1^2\cv{x}_3
-\Omega_{2,2,3}\cv{x}_2^2\cv{x}_3-\Omega_{3,3,3}\cv{x}_3^3)\\
x_{3,2}&=x_{3,2}^+=\frac{-\sqrt{2}}{2\pi^2}\Omega_{2,1}^2(\lambda^3+\lambda)\\
\mathcal K_{\low,2}&=4\cv{x}_3 x_{3,2}+\sum_{j=1}^3 (x_{j,1})^2\\
&=\frac{-4}{2\sqrt{2}}(\lambda^{-1}+\lambda)\frac{-1}{\sqrt{2}\pi^2}\Omega_{2,1}^2(\lambda^3+\lambda)
+\frac{-2}{2\pi^2}\Omega_{2,1}^2(\lambda^2+1)^2\\
&=0\,.
\end{align*}
Hence $x_{1,2}=0$ and $\mathcal K_2=0$. Note this is a  cancellation which does not follow directly from parity.
\subsubsection{Third step ($n=3$)}
\begin{eqnarray*}
\lefteqn{\wt{\mathfrak p}_{\low,3}=4\,\Omega_{2,1}\cv{x}_1 x_{2,2}
-8\ii(\Omega_{3,1,1}\cv{x}_1^2x_{3,1}+\Omega_{2,2,3}(\cv{x}_2^2x_{3,1}+2\cv{x}_2\cv{x}_3x_{2,1})
+3\,\Omega_{3,3,3}\cv{x}_3^2 x_{3,1})}\\
&&+16\,\Omega_{2,1,1,1}\cv{x}_1^3\cv{x}_2+16\,\Omega_{2,2,2,1}\cv{x}_1\cv{x}_2^3
+16\,(-\Omega_{3,1,2,3}+\Omega_{2,1,3,3}+\Omega_{3,3,2,1})\cv{x}_1\cv{x}_2\cv{x}_3^2\\
&=& \tfrac{\ii \sqrt{2}}{\pi^2}\Omega_{2,1}^3(\lambda^4-1)
+\tfrac{\sqrt{2}}{\pi}\Omega_{2,1}\Omega_{3,1,1}(\lambda^4-\lambda^2-1+\lambda^{-2})\\
&&+\tfrac{\sqrt{2}}{2\pi}\Omega_{2,1}(\Omega_{2,2,3}-3\Omega_{3,3,3})(\lambda^4+3\lambda^2+3+\lambda^{-2})
+\tfrac{\ii\sqrt{2}}{2}\Omega_{2,1,1,1}(-\lambda^4+2\lambda^2-2\lambda^{-2}+\lambda^{-4})\\
&&-\tfrac{\ii\sqrt{2}}{4}(\Omega_{2,2,2,1}-\Omega_{3,1,2,3}+\Omega_{2,1,3,3}+\Omega_{3,3,2,1})(\lambda^4+2\lambda^2-2\lambda^{-2}-\lambda^{-4}).
\end{eqnarray*}
Hence
\begin{eqnarray*}
x_{3,3}^+&=&\tfrac{-\ii\sqrt{2}}{2\pi^3}\Omega_{21}^3\,\lambda^4
-\tfrac{\sqrt{2}}{2\pi^2}\Omega_{2,1}\Omega_{3,1,1}(\lambda^4-2\lambda^2-1)\\
&&-\tfrac{\sqrt{2}}{4\pi^2}\Omega_{2,1}(\Omega_{2,2,3}-3\Omega_{3,3,3})(\lambda^4+2\lambda^2+3)
+\tfrac{\ii\sqrt{2}}{2\pi}\Omega_{2,1,1,1}(\lambda^4-2\lambda^2)\\
&&-\tfrac{\ii\sqrt{2}}{4\pi}(\Omega_{2,2,2,1}-\Omega_{3,1,2,3}+\Omega_{2,1,3,3}+\Omega_{3,3,2,1})(\lambda^4+2\lambda^2).
\end{eqnarray*}
Moreover, $\wt{\mathfrak p}_{\low,3}(\ii)=0$ so $x_{3,3}(\ii)=0$ and this determines $x_{3,3}$ as in Proposition \ref{prop:derivatives-minimal}.
\end{proof}
\begin{corollary}
\label{cor:alpha3}
$$\text{Area}(\xi_{1,g})=8\pi(1-\alpha_1 s -\alpha_3 s^3)+ O(s^5)$$
with $s=\frac{1}{2g+2}$,
$\alpha_1=\log(2)$
and
$$\alpha_3=\tfrac{-\ii}{\pi^3}\Omega_{2,1}^3
+\tfrac{1}{2\pi^2}\Omega_{2,1}(\Omega_{2,2,3}-6\Omega_{3,1,1}-3\Omega_{3,3,3})
+\tfrac{\ii}{2\pi}(6\Omega_{2,1,1,1}+\Omega_{2,2,2,1}-\Omega_{3,1,2,3}+\Omega_{2,1,3,3}+\Omega_{3,3,2,1}).$$
\end{corollary}
\begin{proof}
Since $\mathcal K(t)=1+O(t^4)$, we have $t=s+O(s^5)$.
By Proposition \ref{prop:building}
$$\text{Area}(\xi_{1,g})=8\pi\big(1+\sqrt{2}(x_{3,1}^0 s +x_{3,3}^0 s^3)\big)+O(t^5).$$
Evaluating $x_{3,1}$ and $x_{3,3}$ at $\lambda=0$ gives the result.
\end{proof}
\subsection{Implementation}
\label{section:implementation}
In the general case of $\varphi\neq\pi/4$, the computation of the order 3 derivatives is extremely tedious.
The algorithm has been implemented in Mathematica. We give the output of the algorithm in Appendix \ref{appendix-order3}.

The algorithm can also be used to compute the coefficients $\alpha_n$ numerically.
We have been able to do so up to $\alpha_{21}$. The results are presented in Appendix \ref{appendix:numalpha}.
In general, the algorithm expresses the coefficient $\alpha_n$ in terms of $\Omega$-values of depth up to $n+1$. This raises the question of how to compute these $\Omega$-values numerically.

As iterated integrals, they can be computed using a variant of the Simpson method. This is how the numerical value of $\alpha_3$ was first computed and the number $\frac{9}{4}\zeta(3)$ was recognized. The precision is however not very good. Each $\Omega$-value of depth $n$ can be expressed quite naturally as a sum of $4^n$ Multiple Polylogarithms (in fact only $3^n$, see Appendix \ref{appendix:nummpl}).
Multiple Polylogarithms are implemented in most computer algebra systems and this allows for the computation of the coefficients $\alpha_n$ with arbitrary precision.
The coefficient $\alpha_3$ was computed with 500 digits precision, confirming the value
$\frac{9}{4}\zeta(3)$, before we found a mathematical proof. On the other hand, this does not permit computing $\Omega$-values of large depth $n$, as the number $3^n$ of MPLs gets too large.

In the minimal case $\varphi=\frac{\pi}{4}$, each $\Omega$-value which appears in the output of the algorithm (characterized by Proposition \ref{prop:pattern}) can in fact be expressed as a single alternating multiple zeta value. This is not trivial and is the subject of Section \ref{section:multizetas}. This observation permits the efficient computation of each required $\Omega$-value and is how computing the coefficients up to $\alpha_{21}$ could be realized.
\begin{remark}
One issue is that the algorithm computes the derivatives of the parameters with respect to $t$ so expresses the area of Lawson surface $\xi_{1,g}$ as a series in $t$ (whose relation to $g$ is not explicit), whereas we want a series in $s=\frac{1}{2g+2}$.
We solve this problem as follows: since the algorithm also computes $\mathcal K(t)$ as a series in $t$, we substitute 
$s= t\sqrt{\mathcal K(t)}$ in the series \eqref{eq:area-series} and identify with the series in $t$ computed by the algorithm to obtain a triangular system of equations determining the coefficients $\alpha_k$.
\end{remark}

\section{\texorpdfstring{$\Omega$}{Omega}-values as alternating multiple zeta values}
\label{section:multizetas}
In the minimal case $\varphi=\frac{\pi}{4}$, we can express the $\Omega$-values which appear
in the expansion of $\mathfrak p$ in Proposition \ref{prop:expandpq} as alternating multiple zeta values.
We start with  recalling some definitions and notations.

\begin{definition}[Alternating multiple zeta values]
	Let \( n_1,\ldots,n_d \in \mathbb{Z}_{>0} \) be positive integers, and let \( \eps_1,\ldots,\eps_n \in \{ \pm 1 \} \).  The \emph{alternating multiple zeta value} (alternating MZV) with indices \( n_1,\ldots,n_d \) and signs \( \eps_1,\ldots,\eps_d \) is defined as follows,
	\[
		\zeta(n_1,\ldots,n_d; \eps_1,\ldots,\eps_d) \coloneqq \sum_{0 < k_1 < \cdots < k_d} \frac{\eps_1^{k_1} \cdots \eps_d^{k_d}}{k_1^{n_1} \cdots k_d^{n_d}} \,.
	\]
	The \emph{weight} of \( \zeta(n_1,\ldots,n_d; \eps_1,\ldots,\eps_d) \) is given by \( w = n_1 + \cdots + n_d \), the sum of the indices, and the \emph{depth} thereof is \( d \), the number of indices.  (Note that this becomes the definition of the multiple polylogarithm \( \Li_{n_1,\ldots,n_d}(\eps_1,\ldots,\eps_d) \), if we allow arbitrary \( \eps_i \) with \( |\eps_i| \leq 1 \), and the condition \( (n_d,\eps_d) \neq (1,1) \) for convergence.)
\end{definition}

The alternating MZV \( \zeta(n_1,\ldots,n_d; \eps_1,\ldots,\eps_d) \) converges if and only if \( (n_d,\eps_d) \neq (1,1) \).
It will be convenient to simplify the notation for \( \zeta(n_1,\ldots,n_d; \eps_1,\ldots,\eps_d) \) as follows.  When \( \eps_i = -1 \), write \( \overline{n_i} \) as the corresponding index, otherwise \( \eps_i = 1 \), write the index \( n_i \) as usual.  For example
\[
	\zeta(n_1,\overline{n_2},n_3,\overline{n_4},\overline{n_5}) = \zeta(n_1,n_1,n_3,n_4,n_5; 1, -1, 1, -1, -1) \,.
\]

One can represent an alternating MZV as an iterated integral, over the following family of differential forms 
\[
	\eta_0(t) = \frac{\mathrm{d}{t}}{t}  \,,  \quad	\eta_1(t) = \frac{\mathrm{d}{t}}{t-1}  \,,  \quad \eta_{-1}(t) = \frac{\mathrm{d}{t}}{t-(-1)}  \,.
\]
By termwise integration of the resulting geometric series expansion of \( \eta_{\pm 1}(t_i) \) in the iterated integral (see \cite[Theorem 2.2]{goncharovMPL}), one has
\begin{equation}\label{eqn:mzvtoint}
	\begin{aligned}
	& \zeta(n_1,\cdots,n_d; \eps_1, \ldots, \eps_d) \\
	& = (-1)^d \int_{0}^1 \eta_{\eps_1 \cdots \eps_d} \, \eta_0^{n_1-1} \, \eta_{\eps_2 \cdots \eps_d} \, \eta_0^{n_2-1} \cdots \eta_{\eps_d} \,  \eta_0^{n_d - 1} \,,
\end{aligned}
\end{equation}
along the straight-line path \( 0 \to 1 \).
The goal of this section is to establish the following proposition, evaluating each \( \Omega \)-value which appears in the expansion of $\mathfrak{p}$ (Proposition \ref{prop:expandpq}), as \( \ii \pi \) times a single alternating MZV (expressed in the iterated integral form), with indices in \( \{\overline{1}, 2, \overline{2}\} \).
In this section, all $\Omega$-values are evaluated at $z=1$.

\begin{proposition}\label{prop:omegaeval}
	Consider \( \Omega_{i_1,\ldots, i_n} \), such that \( (i_1,\ldots,i_n) \) forms a path from \( e_3 \) to \( e_1 \) in the graph of Proposition \ref{prop:pattern}.  Let the path be \( e_{v_0} = e_3 \to e_{v_1} \to e_{v_2} \to \cdots \to e_{v_{n-1}} \to e_1 = e_{v_n} \), then
	\[
		\Omega_{i_1,\ldots, i_n} = (-1)^{\#\{ i_j = 1\}} \ii\pi \cdot \int_{f(v_0)=0}^{f(v_n)=1} \eta_{f(v_1)} \cdots  \eta_{f(v_{n-1})} \,,
	\]
	where \( f(1) = 1 \), \( f(2) = -1 \), \( f(3) = 0 \).
\end{proposition}

To more clearly illustrate this proposition, we give the following example.  Further examples can be found in Eq.~\eqref{eqn:wt4omegas} below, in the context of computing \( \alpha_3 \).

\begin{example}
	Consider the \( \Omega \)-value,
	\[
		\Omega_{312222132133}
	\]
	whose indices \( (i_j)_{j=1}^{12} = (3,1,2,2,2,2,1,3,2,1,3,3) \) corresponds to the path
	\[
		e_3 \xrightarrow{3}
		e_1 \xrightarrow{1}
		e_2 \xrightarrow{2}
		e_3 \xrightarrow{2}
		e_2 \xrightarrow{2}
		e_3 \xrightarrow{2}
		e_2 \xrightarrow{1}
		e_1 \xrightarrow{3}
		e_3 \xrightarrow{2}
		e_2 \xrightarrow{1}
		e_1 \xrightarrow{3}
		e_3 \xrightarrow{3}
		e_1  \,.
	\]
	The vertices \( e_{v_j} \) of this path are given by \( (v_j)_{j=0}^{12} = (3,1,2,3,2,3,2,1,3,2,1,3,1) \).

	 Applying \( f \) as defined above, gives \vspace{-0.5em}
	\[
		(f(v_j))_{j=0}^{12} = (\underbrace{0}_{\mathclap{\text{lower bound}}},\overbrace{1,-1,0,-1,0,-1,1,0,-1,1,0}^{\mathclap{\text{differential forms}}},\underbrace{1}_{\mathclap{\text{upper bound}}})
	\]
	as the lower bound, differential form indices, and upper bound appearing in the iterated integral expression, respectively.  We count also that \( \#\{i_j=1\} = 3 \).  So we obtain
	\begin{align*}
		\Omega_{312222132133} &= (-1)^3 \cdot \ii\pi \cdot \int_0^1 \eta_1 \, \eta_{-1} \,\eta_0 \,\eta_{-1} \,\eta_0 \,\eta_{-1} \,\eta_1 \,\eta_0 \,\eta_{-1} \,\eta_1 \,\eta_0
	\end{align*}
	After converting from the iterated integral to alternating MZV form via Eq.~\eqref{eqn:mzvtoint}, with depth $d=7$ equal to the number of non-zero forms in the iterated integral, we obtain
	\[
		\Omega_{312222132133} = \ii\pi \cdot \zeta(\overline{1},2,2,\overline{1},\overline{2},\overline{1},2) \,.
	\]
\end{example}

\begin{remark}
	Since \( v_i \neq v_{i+1} \), as the graph of Proposition \ref{prop:pattern} has no loops, the string \( f(v_1),\ldots,f(v_{n-1}) \) starts with a non-zero entry, ends with an entry \( \neq 1 \), and never contains two consecutive 0's, 1's or $-1$'s.  Hence when in alternating MZV form, the arguments are exactly from the set \( \{\overline{1}, 2, \overline{2}\} \) (which correspond respectively to the following sequences inside the integral: \( \overline{1} \leftrightarrow (\eta_{\pm1}) \eta_{\mp1} \), \( 2 \leftrightarrow (\eta_{\pm1} \eta_0) \eta_{\pm 1} \) and \( \overline{2} \leftrightarrow (\eta_{\pm1} \eta_0) \eta_{\mp1} \), where the sequence in brackets is the part which gives the size of the index, and the next form determines whether or not that index is barred).  This hence establishes the structural result involving \( \{\overline{1}, 2, \overline{2} \}\) indices given in Theorem \ref{thm1}.
\end{remark}

As a starting point, consider the M\"obius transformation r, which maps \( (0, -1, \infty, 1) \) to \( (p_1, p_2, p_3, p_4) \) respectively. Explicitly, this is given by
\begin{equation}\label{eqn:moeb}
	z = r(w) = p_3 \frac{w+\ii}{w-\ii} \,.	
\end{equation}
We compute that
\begin{align*}
	\widetilde{\omega_1} \coloneqq r^\ast \omega_1 &= \Big(  {-} \frac{1}{w-(-1)} + \frac{1}{w} - \frac{1}{w-1} \Big) \mathrm{d}w \\
	\widetilde{\omega_2} \coloneqq r^\ast \omega_1 &= \Big( {-} \frac{1}{w-(-1)} + \frac{1}{w} + \frac{1}{w-1} \Big) \mathrm{d}w \\
	\widetilde{\omega_3} \coloneqq r^\ast \omega_1 &= \Big(  \phantom{{+}} \frac{1}{w-(-1)} + \frac{1}{w} - \frac{1}{w-1} \Big) \mathrm{d}w \,,
\end{align*}
and observe that \( 0 = r(-\ii), 1 = r(-1 + \sqrt{2}) \).  So we also have the iterated integral expression
\begin{align}
\label{eqn:omega:3n}
	\Omega_{i_1,\ldots, i_n} = \int_{-\ii}^{-1+\sqrt{2}} \widetilde{\omega_{i_1}}(w) \cdots \widetilde{\omega_{i_n}}(w)  
\end{align}
along the straight-line path from \( -\ii \) to \( -1 + \sqrt{2} \).  (The pullback of the straight-line path \( 0 \to 1 \) is homotopic, relative to the end points, to this straight-line path from \( -\ii \) to \( -1 + \sqrt{2} \), and by the homotopy invariance of these iterated integrals, we may deform the original path to the straight-line path.)

\paragraph{\bf Iterated Beta Integrals:}  We must now consider the iterated beta integrals developed by Hirose and Sato \cite{hsBetaIntegral}.  Introduce the differential form
\[
	F_{x,y}^{\alpha,\beta}(t) \coloneqq \frac{(x-y)^{\alpha-\beta} \mathrm{d}t}{(t-x)^\alpha (t - y)^{1-\beta}} \,,
\]
and the associated \emph{iterated beta integral}, with \( x_i \neq x_{i+1} \),
\[
	\beta_\gamma\Big( 
	\genfrac{.}{.}{0pt}{0}{\alpha_0}{x_0} \Big| 
	\genfrac{.}{.}{0pt}{0}{\alpha_1}{x_1} \Big| 
	\cdots \Big|
	\genfrac{.}{.}{0pt}{0}{\alpha_{n-1}}{x_{n-1}} \Big| 
	\genfrac{.}{.}{0pt}{0}{\alpha_n}{x_n} 
	\Big)
	\coloneqq \frac{
	\int_{\gamma} F_{x_0,x_1}^{\alpha_0,\alpha_1}\, F_{x_1,x_2}^{\alpha_1,\alpha_2}\, \cdots \, F_{x_{n-1},x_n}^{\alpha_{n-1},\alpha_n} 
	}{
		\int_\gamma F_{x_0,x_n}^{\alpha_0, \alpha_{n}}
	}
	\,,
\]
along some fixed path \( \gamma \) from \( x_0 \) to \( x_n \).  A useful observation is that the adjacent 1-forms \( F_{x_i,x_{i+1}}^{\alpha_i,\alpha_{i+1}} \) and \( F_{x_{i+1},x_{i+2}}^{\alpha_{i+1},\alpha_{i+2}} \) share the pole \( x_{i+1} \); this plays a key role for Hirose and Sato, when they derive a formula for the total differential of the iterated beta integral.

One of their main theorems on iterated beta integrals is the following invariance in the upper parameter.

\begin{theorem}[Hirose and Sato, \cite{hsBetaIntegral}, in preparation]
	The iterated beta integral is invariant under translation of all upper parameters by a constant \( c \),
	\[
			{\beta}_\gamma\Big( 
		\genfrac{.}{.}{0pt}{0}{\alpha_0}{x_0} \Big| 
		\genfrac{.}{.}{0pt}{0}{\alpha_1}{x_1} \Big| 
		\cdots \Big|
		\genfrac{.}{.}{0pt}{0}{\alpha_{n-1}}{x_{n-1}} \Big| 
		\genfrac{.}{.}{0pt}{0}{\alpha_n}{x_n} 
		\Big) =
			{\beta}_\gamma\Big( 
		\genfrac{.}{.}{0pt}{0}{\alpha_0 + c}{x_0} \Big| 
		\genfrac{.}{.}{0pt}{0}{\alpha_1 + c}{x_1} \Big| 
		\cdots \Big|
		\genfrac{.}{.}{0pt}{0}{\alpha_{n-1}+c}{x_{n-1}} \Big| 
		\genfrac{.}{.}{0pt}{0}{\alpha_n+c}{x_n} 
		\Big) \,.
	\]

	\begin{proof}[Proof idea {\rm (for details, see Hirose and Sato \cite{hsBetaIntegral})}]
		Goncharov showed \cite[Theorem 2.1]{goncharovMPL} that the iterated integral
		\[
			\int_{x_0}^{x_{n+1}} \eta_{x_1}  \cdots  \eta_{x_n}
		\]
		has a total differential of the form
		\[
			\mathrm{d} \int_{x_0}^{x_{n+1}} \eta_{x_1} \cdots \eta_{x_n} = \sum_{i=0}^n \Big( \int_{x_0}^{x_{n+1}} \eta_{x_1} \cdots  \widehat{\eta_{x_i}} \cdots \eta_{x_n} \Big)  \cdot \mathrm{d} \log\Big(\frac{x_i - x_{i+1}}{x_i - x_{i-1}}\Big) \,,
		\]
		where \( \widehat{\eta_{x_i}} \) denotes that \( \eta_{x_i} \) is omitted from the integrand.  This is a direct calculation by noting that after differentiating with respect to \( x_i \) the integral \( \int_{t_{i-1} \leq t_i \leq t_i+1} \eta_{x_i}(t_i) \) (somewhere inside the iterated integral) becomes
		\[
			\int_{t_{i-1} \leq t_i \leq t_{i+1}} \frac{\mathrm{d}t_i}{(t_i-x_i)^2} = \frac{1}{t_{i-1} - x_i} - \frac{1}{t_{i+1} -  x_i} \,.
		\]
		Then by applying the partial fractions identities
		\begin{align*}
			& \frac{\mathrm{d}t_{i-1}}{(t_{i-1} - x_{i-1})(t_{i-1} - x_i)} = \frac{1}{x_{i-1} - x_i} \Big( \frac{\mathrm{d}t_{i-1}}{t_{i-1} - x_{i-1}} - \frac{\mathrm{d}t_{i-1}}{t_{i-1} - x_i} \Big) \\
			& \frac{\mathrm{d}t_{i+1}}{(t_{i+1} - x_{i+1})(t_{i+1} - x_i)} = \frac{1}{x_{i+1} - x_i} \Big( \frac{\mathrm{d}t_{i+1}}{t_{i+1} - x_{i+1}} - \frac{\mathrm{d}t_{i+1}}{t_{i+1} - x_i} \Big) \,,
		\end{align*}
		one can carry out the remaining integrations (over \( t_1,\ldots,\widehat{t_i},\ldots,t_n \)), and obtain the above formula for the total differential.
		
		For iterated beta integrals, the key property that adjacent 1-forms \( F_{x_i,x_{i+1}}^{\alpha_i,\alpha_{i+1}} \) and \( F_{x_{i+1},x_{i+2}}^{\alpha_{i+1},\alpha_{i+2}}  \) share the pole \( x_i \) allows Hirose and Sato to derive a similar Goncharov-type formula for the total differential in this case.  Namely, if \( \alpha_i \) are fixed, then:
		\[
			\mathrm{d} \beta_\gamma\Big( 
			\genfrac{.}{.}{0pt}{0}{\alpha_0}{x_0} \Big| 
			\genfrac{.}{.}{0pt}{0}{\alpha_1}{x_1} \Big| 
			\cdots \Big|
			\genfrac{.}{.}{0pt}{0}{\alpha_{n-1}}{x_{n-1}} \Big| 
			\genfrac{.}{.}{0pt}{0}{\alpha_n}{x_n} 
			\Big) =
			\sum_{i=1}^n \beta_\gamma\Big( 
			\genfrac{.}{.}{0pt}{0}{\alpha_0}{x_0} \Big| 
			\genfrac{.}{.}{0pt}{0}{\alpha_1}{x_1} \Big| 
			\cdots \Big|
			\widehat{\genfrac{.}{.}{0pt}{0}{\alpha_i}{x_i}}
			\Big| \cdots \Big|
			\genfrac{.}{.}{0pt}{0}{\alpha_{n-1}}{x_{n-1}} \Big| 
			\genfrac{.}{.}{0pt}{0}{\alpha_n}{x_n} 
			\Big) \mathrm{d} \beta\Big( 
			\genfrac{.}{.}{0pt}{0}{\alpha_{i-1}}{x_{i-1}} \Big| 
			\genfrac{.}{.}{0pt}{0}{\alpha_i}{x_i} \Big| 
			\genfrac{.}{.}{0pt}{0}{\alpha_{i+1}}{x_{i+1}} 
			\Big) \,,
		\]
		where the case \( n = 2 \), can be given in terms of the usual \( \mathrm{d}\log \) directly
		\[
			\mathrm{d} \beta_\gamma\Big( 
			\genfrac{.}{.}{0pt}{0}{\alpha_{0}}{x_{0}} \Big| 
			\genfrac{.}{.}{0pt}{0}{\alpha_1}{x_1} \Big| 
			\genfrac{.}{.}{0pt}{0}{\alpha_{2}}{x_{2}} 
			\Big) = | x_0 - x_1 |^{\alpha_0 - \alpha_1} | x_1 - x_2 |^{\alpha_1 - \alpha_2} | x_2 - x_0 |^{\alpha_2 - \alpha_0}  \mathrm{d}\log\Big( \frac{z_1 - z_2 }{ z_1 - z_0 } \Big) \,.
		\]
		(In particular, this is independent of the path, hence the missing subscript \( \gamma \) on \( \mathrm{d}\beta \) on the right-hand side of the total differential above.)
		
		Supposing that \( \beta_\gamma\big( \genfrac{.}{.}{0pt}{1}{\alpha_{0}}{x_{0}} \big| \cdots \big|
		\genfrac{.}{.}{0pt}{1}{\alpha_m}{x_m} \big) \) is invariant under translation of \( \alpha_i \mapsto \alpha_i + c \)  for \( m < n \), one sees from the differential equation, that
		\begin{equation}\label{eqn:betadiff}
			\beta_\gamma\Big( \genfrac{.}{.}{0pt}{0}{\alpha_{0}+c}{x_{0}} \Big| \cdots \Big|
			\genfrac{.}{.}{0pt}{0}{\alpha_n+c}{x_n} \Big) - \beta_\gamma\Big( \genfrac{.}{.}{0pt}{0}{\alpha_{0}}{x_{0}} \Big| \cdots \Big|
			\genfrac{.}{.}{0pt}{0}{\alpha_n}{x_n} \Big) = \text{constant} \,,
		\end{equation}
		as a function of the \( x_i \).  They then consider a fixed \( 1 \leq i \leq n-1 \), and depending on whether \( \alpha_{i-1},\alpha_i,\alpha_{i+1} \) satisfy one of three conditions \( \Re(1 - \alpha_{i-1} + \alpha_{i+1} ) > 0, \Re(\alpha_{i-1}- \alpha_i) > 0 , \Re(\alpha_i - \alpha_{i+1}) > 0 \) (depending only on differences), show that \( \beta_\gamma\big( \genfrac{.}{.}{0pt}{1}{\alpha_{0}(+c)}{x_{0}} \big| \cdots \big|
		\genfrac{.}{.}{0pt}{1}{\alpha_m(+c)}{x_m} \big) \to 0 \) with the limit \( x_i \to x_{i-1} , x_i \to x_{i+1} \) or \( x_i \to \infty \) respectively.  (Analytic continuation then gives this for all \( \alpha_i \).)  So both terms in Equation \eqref{eqn:betadiff} would go to 0, and the constant of integration vanishes, which inductively establishes the invariance for all \( n \).
	\end{proof}
	
\end{theorem}

By setting \( \alpha,\beta = 0\), we have
\[
	F_{x,y}^{0,0}(t) = \frac{\mathrm{d}t}{t-y} = \eta_y(t) \,.
\]
The following Lemma establishes the behaviour of the iterated beta integral as \( \alpha \to 0^+ \).

\begin{lemma}\label{lem:beta0}
	Assume that \( x_0 \neq x_n \).  As \( \alpha \to 0^+ \), we have
	\[
		\lim_{\alpha \to 0^+} {\beta}_\gamma\Big( 
		\genfrac{.}{.}{0pt}{0}{\alpha}{x_0} \Big| 
		\genfrac{.}{.}{0pt}{0}{\alpha}{x_1} \Big| 
		\cdots \Big|
		\genfrac{.}{.}{0pt}{0}{\alpha}{x_{n-1}} \Big| 
		\genfrac{.}{.}{0pt}{0}{\alpha}{x_n}  
		\Big) 
		=
		 \int_{x_0}^{x_{n}} \eta_{x_1} \cdots \eta_{x_{n-1}}
	\]

\begin{proof}
Recall that one can multiply two iterated integrals \( \int_{a}^{b} \) by considering how the integration indices \( a < t_1 < \cdots < t_n < b \) and \( a < s_1 < \cdots < s_m < b \) are interleaved (measure zero sets \( s_i = t_j \) can be neglected).  This gives rise to the shuffle-product of iterated integrals. Using this shuffle-product, we have
\begin{align*}
	& {\beta}_\gamma\Big( 
	\genfrac{.}{.}{0pt}{0}{\alpha}{x_0} \Big| 
	\genfrac{.}{.}{0pt}{0}{\alpha}{x_1} \Big| 
	\cdots \Big|
	\genfrac{.}{.}{0pt}{0}{\alpha}{x_{n-1}} \Big| 
	\genfrac{.}{.}{0pt}{0}{\alpha}{x_n}  
	\Big) \\
	& = \frac{\int_{x_0}^{x_{n}}  F_{x_0,x_1}^{\alpha,\alpha} \, F_{x_1,x_2}^{\alpha,\alpha} \cdots  F_{x_{n-1},x_n}^{\alpha,\alpha}}{\int_{x_0}^{x_{n}} F_{x_0,x_n}^{\alpha,\alpha} } \\[1ex]
	& =  \frac{
		\begin{aligned}[t]
		& \textstyle\int_{x_0}^{x_{n}} F_{x_0,x_1}^{\alpha,\alpha} \, F_{x_1,x_2}^{\alpha,\alpha} \cdots  F_{x_{n-2},x_{n-1}}^{\alpha,\alpha} \cdot \textstyle\int_{x_0}^{x_n} F_{x_{n-1},x_n}^{\alpha,\alpha} \\[-0.5ex]
		& \hspace{5em} - \textstyle\sum_{i=0}^{n-2} \textstyle\int_{x_0}^{x_n} F_{x_0,x_1}^{\alpha,\alpha} \cdots 
		F_{x_{i-1},x_i}^{\alpha,\alpha} \, F_{x_{n-1}, x_n}^{\alpha,\alpha} \, F_{x_{i}, x_{i+1}}^{\alpha,\alpha} \cdots
		 F_{x_{n-2},x_{n-1}}^{\alpha,\alpha}
	\end{aligned}
}{\int_{x_0}^{x_n} F^{\alpha,\alpha}_{x_0,x_n} }
\end{align*}
The summation in the numerator is finite (as \( \alpha \to 0^+ \)), as the last differential form goes to \( \eta_{x_{n-1}} \) which has no pole at \( x_n \) and the logarithmic singularity in the \( t_i \) leads to a convergent integral \( \int_0^1 \log(t) \mathrm{d}{t} \).  So in the limit, this will be dominated by the denominator \( \int_{x_0}^{x_n} F^{\alpha,\alpha}_{x_0,x_n} \to \infty \), as \( \alpha \to 0^+ \). 

On the other hand,
\[
	\lim_{\alpha\to0^+} \frac{ \int_{x_0}^{x_n} F_{x_{n-1},x_n}^{\alpha,\alpha}}{\int_{x_0}^{x_n} F_{x_{0},x_n}^{\alpha,\alpha}} 
	= \lim_{\alpha\to0^+} 
	 \frac{\int_{x_0}^{x_{n-1}} F_{x_{n-1},x_n}^{\alpha,\alpha} + \int_{x_{n-1}}^{x_n} F_{x_{n-1},x_n}^{\alpha,\alpha}}{\int_{x_0}^{x_n} F_{x_{0},x_n}^{\alpha,\alpha}}
= 1 \,.
\]
This holds because \( x_{n} \neq x_0 \) means \( \int_{x_0}^{x_{n-1}} F_{x_{n-1},x_n}^{\alpha,\alpha} \to  \int_{x_0}^{x_{n-1}} \eta_{x_n} \), which is finite, and
\begin{align*}
	\int_{x_{n-1}}^{x_{n}}  F_{x_{n-1},x_n}^{\alpha,\alpha}
	&= \int_{x_{n-1}}^{x_{n}} \frac{\mathrm{d}t}{(t-x_{n-1})^{\alpha} (t - x_n)^{1-\alpha}} = \int_{0}^1 \frac{\mathrm{d}t}{t^\alpha(t-1)^{1-\alpha}}
	= \int_{x_0}^{x_n} F_{x_0,x_n}^{\alpha,\alpha}
	 \\
	&= -\ii\pi - \pi \cot(a \pi) \,
\end{align*}

so both the numerator and denominator go to infinity at the same rate.  We therefore obtain
\[
	\lim_{\alpha\to0^+} {\beta}_\gamma\Big( 
	\genfrac{.}{.}{0pt}{0}{\alpha}{x_0} \Big| 
	\genfrac{.}{.}{0pt}{0}{\alpha}{x_1} \Big| 
	\cdots \Big|
	\genfrac{.}{.}{0pt}{0}{\alpha}{x_{n-1}} \Big| 
	\genfrac{.}{.}{0pt}{0}{\alpha}{x_n}  
	\Big)
	 = \int_{x_0}^{x_n}  \eta_{x_1} \cdots \eta_{x_{n-1}} \,,
\]
as claimed.
\end{proof}
\end{lemma}

We want to consider the case \( x_i \in \{ 0, \pm 1 \} \), and the result of the invariance of iterated beta integrals here.

Introduce the curve
\[
	\mathcal{C} = \{ (u,v,t) \in \mathbb{C}^3 \mid u^2 = t(t-1), v^2 = t(t+1) \} \,,
\]
then a computation via Riemann-Roch shows that \( \mathcal{C} \) has genus 0, and so it is rationally parametrisable.  In particular, this is given by taking
\begin{align*}
	t &= g(\xi) = \frac{(1 + \xi^2)^2}{4\xi (\xi^2-1)}
	\\
 u &= \sqrt{t(t-1)} = \frac{\left(\xi ^2+1\right) \left(\xi ^2-2 \xi -1\right)}{4 \xi (\xi^2 -1)}
\\
 v &= \sqrt{t(t+1)} = \frac{\left(\xi ^2+1\right) \left(\xi ^2+2 \xi -1\right)}{4 \xi (\xi^2 -1) }
 .\end{align*}

Then we have
\begin{align*}
	F_{0,1}^{\half,\half} &= \frac{\mathrm{d}t}{\sqrt{t(t-1)}} = \frac{\mathrm{d}t}{u} \\
	F_{0,-1}^{\half,\half} &= \frac{\mathrm{d}t}{\sqrt{t(t+1)}} = \frac{\mathrm{d}t}{v} \\
	F_{1,-1}^{\half,\half} &= \frac{\mathrm{d}t}{\sqrt{(t-1)(t+1)}} = \frac{t \mathrm{d}t}{u v}
\end{align*}

 and we can compute the pullbacks of the above differential forms to be
\begin{align*}
	g^\ast F_{0,1}^{\half,\half} &= \widetilde{\omega_2}(\xi) \\
	g^\ast F_{0,-1}^{\half,\half} &= \widetilde{\omega_3}(\xi) \\
	g^\ast F_{1,-1}^{\half,\half} &= -\widetilde{\omega_1}(\xi)\,.
\end{align*}

Moreover, we note \( g(-\ii) = 0 \), and \( g(-1 + \sqrt{2}) = -1 \).  Therefore, if \( i_1 \to \cdots \to i_n \) is a path from \( e_{v_0} = e_3 \to e_{v_1} \to \cdots \to e_{v_{n-1}} \to e_{v_n} = e_{1} \), we can check that, with \( \phi(3) = 0, \phi(2) = 1, \phi(1) = -1 \), one has
 \begin{align*} 
F^{\half,\half}_{0,1} = F^{\half,\half}_{\phi(3),\phi(2)} = F^{\half,\half}_{\phi(2),\phi(3)} &\leftrightarrow \text{edges labelled 2 between \( e_2 \) and \( e_3 \)} \\
F^{\half,\half}_{0,-1} = F^{\half,\half}_{\phi(3),\phi(1)} = F^{\half,\half}_{\phi(1),\phi(3)} &\leftrightarrow \text{edges labelled 3 between \( e_1 \) and \( e_3 \)} \\
F^{\half,\half}_{-1,1} = F^{\half,\half}_{\phi(1),\phi(2)} = F^{\half,\half}_{\phi(2),\phi(1)} &\leftrightarrow \text{edges labelled 1 between \( e_1 \) and \( e_2 \)} \\
\end{align*}
as \( F \) is symmetric in the lower arguments \( x, y \) when \( \alpha = \beta = \half \), and each edge can be traversed in both directions.
So
\begin{align*}
	\Omega_{i_1,\ldots, i_n} &= \int_{-\ii}^{-1 + \sqrt{2}} \widetilde{\omega_{i_1}} \cdots \widetilde{\omega_{i_n}} \\
	&= (-1)^{\# \{ i_j = 1 \}} \int_{0 = \phi(v_0)}^{-1=\phi(v_n)} F^{\half, \half}_{\phi({v_0}),\phi({v_1})} \, F^{\half, \half}_{\phi({v_1}),\phi({v_2})} \cdots F^{\half, \half}_{\phi({v_{n-1}}),\phi({v_n})}  \\
	&= (-1)^{\# \{ i_j = 1 \}} \int_0^{-1} F^{\half,\half}_{0,-1} 
	\cdot {\beta}_\gamma\Big( 
\genfrac{.}{.}{0pt}{0}{\tfrac{1}{2}}{\phi(v_0)} \Big| 
\genfrac{.}{.}{0pt}{0}{\tfrac{1}{2}}{\phi(v_1)} \Big| 
\cdots \Big|
\genfrac{.}{.}{0pt}{0}{\tfrac{1}{2}}{\phi(v_{n-1})} \Big| 
\genfrac{.}{.}{0pt}{0}{\tfrac{1}{2}}{\phi(v_n)}  
\Big)\,.
\end{align*}
This has written the \( \Omega \)-value of interest, as an iterated beta integral (up to the normalisation factor above).  However, one can check directly
\[
	\int_0^{-1} F^{\half,\half}_{0,-1} = \int_{0}^{-1} \frac{\mathrm{d}t}{\sqrt{t(t+1)}} = \ii \pi \,.
\]
Now since the iterated beta integral is invariant under shifting the upper parameters by a constant, we subtract \( \tfrac{1}{2} \) everywhere, and apply Lemma \ref{lem:beta0}, to find
\begin{align*}
	\Omega_{i_1,\ldots,i_n} &= (-1)^{\# \{ i_j = 1 \}} \ii \pi \cdot {\beta}_\gamma\Big( 
	\genfrac{.}{.}{0pt}{0}{0}{\phi(v_0)} \Big| 
	\genfrac{.}{.}{0pt}{0}{0}{\phi(v_1)} \Big| 
	\cdots \Big|
	\genfrac{.}{.}{0pt}{0}{0}{\phi(v_{n-1})} \Big| 
	\genfrac{.}{.}{0pt}{0}{0}{\phi(v_n)}  
	\Big) \\
	&=  (-1)^{\# \{ i_j = 1 \}} \ii \pi \int_{\phi(v_0)}^{\phi(v_{n})} \eta_{\phi(v_1)} \cdots \eta_{\phi(v_{n-1})}\,.
\end{align*}
As a final step, we want to change variables in order to set the upper bound equal to \( 1 \), instead of \( -1 = \phi(v_n) = \phi(1) \).  Make the change of variables \( t \mapsto -t \) in the iterated integral, then \( \eta_{x}(-t) = \tfrac{\mathrm{d}(-t)}{(-t) - x} = \tfrac{\mathrm{d}t}{t + x} = \eta_{-x}(t) \), and the bounds change from \( \phi(v_0) = 0 \) and \( \phi(v_n) = -1 \) to \( -\phi(v_0) = 0 \) and \( -\phi(v_n) = 1 \).  So
\begin{align*}
\Omega_{i_1,\ldots,i_n} 
&=  (-1)^{\# \{ i_j = 1 \}} \ii \pi \cdot \int_{-\phi(v_0)}^{-\phi(v_n)} \eta_{-\phi(v_1)} \cdots \eta_{-\phi(v_{n-1})} \,.
\end{align*}
If we therefore take \( f(j) = -\phi(j) \), namely \( f(1) = 1, f(2) = -1 , f(3) = 0 \), we have obtained the result of Proposition \ref{prop:omegaeval}. \hfill \qed
 
\paragraph{\bf The required special values for \( \alpha_3 \).}

In order to evaluate \( \alpha_3 \), we require the evaluations for all valid \( \Omega \)-values of weight \( \leq 4 \), namely
\[
	\Omega_3,\,  \Omega_{21}, \,\Omega_{311},\, \Omega_{223},\, \Omega_{333}, \,\Omega_{2111}, \,\Omega_{2221},\, \Omega_{3123},\, \Omega_{2133},\, \Omega_{3321} \,.
\]
From Proposition~\ref{prop:omegaeval}, we find
\begin{equation}\label{eqn:wt4omegas}
\begin{aligned}
	\Omega_3 &= \ii \pi \cdot 1 \\
	\Omega_{21} &= -\ii \pi \cdot \textstyle\int_0^1 \eta_{-1}  = \ii \pi \cdot \zeta(\overline{1}) \,, \\
	\Omega_{311} &= \ii \pi \cdot \textstyle\int_0^1 \eta_{1}\,\eta_{-1} = \ii \pi \cdot \zeta(\overline{1}, \overline{1}) \,, \\
	\Omega_{223} &= \ii \pi \cdot \textstyle\int_0^1\eta_{-1}\, \eta_{0} = - \ii \pi \cdot \zeta(\overline{2}) \,, \\
	\Omega_{333} &= \ii \pi \cdot \textstyle\int_0^1 \eta_{1}\, \eta_{0} = - \ii \pi \cdot \zeta(2) \,, \\
	\Omega_{2111} &= -\ii \pi \cdot \textstyle\int_0^1 \eta_{-1}\, \eta_{1}\, \eta_{-1} = \ii \pi \cdot \zeta(\overline{1}, \overline{1}, \overline{1}) \,, \\
	\Omega_{2221} &= -\ii \pi \cdot \textstyle\int_0^1 \eta_{-1}\, \eta_{0}\, \eta_{-1} = - \ii \pi \cdot \zeta(2, \overline{1}) \,, \\
	\Omega_{3123} &= -\ii \pi \cdot \textstyle\int_0^1 \eta_{1}\, \eta_{-1}\, \eta_{0} = - \ii \pi \cdot \zeta(\overline{1}, \overline{2}) \,, \\
	\Omega_{2133} &= -\ii \pi \cdot \textstyle\int_0^1 \eta_{-1}\, \eta_{1}\, \eta_{0} = -\ii \pi \cdot \zeta(\overline{1}, 2) \,, \\
	\Omega_{3321} &= -\ii \pi \cdot \textstyle\int_0^1 \eta_{1}\, \eta_{0}\, \eta_{-1} = - \ii \pi \cdot \zeta(\overline{2}, \overline{1})\,.
\end{aligned}
\end{equation}

Using well-known evaluations, and standard relations amongst alternating MZV's, we can express all of the above values in terms of \( \pi, \log(2) \) and \( \zeta(3) \).  For completeness, we establish these evaluations directly here, using elementary means.  (In particular, we will avoid using the extended double-shuffle relations, which requires a more detailed discussion on regularisation to rigorously define \cite{IKZ}.)

Firstly,
\[
	\zeta(\overline{1}) = \sum_{n=1}^\infty \frac{(-1)^{k}}{k} = -\log(2)
\]
is well-known and elementary.  It is also well known that
\[
	\zeta(2) = \sum_{k=1}^\infty \frac{1}{k^2} = \frac{\pi^2}{6} \,;
\]
on the other hand \( \zeta(3) \) must (apparently) be left as it is.  Then directly from the series definition, we have the stuffle-product structure
\[
	\zeta(a)\zeta(b) = \zeta(a,b) + \zeta(b,a) + \zeta(a \oplus b) \,,
\]
where with barred entries \( a \oplus b \) adds the values, and counts the number of bars modulo 2, i.e. \( 5 \oplus 3 = 8, 5 \oplus \overline{3} = \overline{5} \oplus 3 = \overline{8} \) and \( \overline{5} \oplus \overline{3} = 8\).  So \(
	\zeta(\overline{1})^2 = 2 \zeta(\overline{1}, \overline{1}) + \zeta(2) \,,
\)
which gives
\[
	\zeta(\overline{1}, \overline{1}) = \frac{1}{2}\log(2)^2 - \frac{1}{2} \zeta(2) \,.
\]

Next, by splitting into the odd and even terms, we have that
\[
	\zeta(\overline{2}) = \sum_{k=1}^\infty \frac{(-1)^k}{k^2} = \sum_{k=1}^\infty \frac{-1}{k^2} + 2 \sum_{k=1}^\infty \frac{1}{(2k)^2} = -(1 - 2^{1-2}) \zeta(2) = -\frac{\pi^2}{12}
\]
More generally \( \zeta(\overline{k}) = -(1 - 2^{1-k}) \zeta(k) \), a version of the so-called level 2 distribution relations in depth 1.

Now, we can consider
\begin{align*}
	\zeta(\overline{1})^3 &= 6\zeta(\overline{1},\overline{1},\overline{1}) + 3\big(\zeta(\overline{1},2) +  \zeta(2,\overline{1})\big) + \zeta(\overline{3}) \\
	&= 6\zeta(\overline{1},\overline{1},\overline{1}) + 3\zeta(\overline{1})\zeta(2) - 2\zeta(\overline{3}) \,.
\end{align*}
Whence
\[
	\zeta(\overline{1},\overline{1},\overline{1}) = -\frac{1}{6} \log^3(2) + \frac{\pi^2}{12} \log(2) - \frac{1}{4} \zeta(3) \,.
\]
This approach works more generally to evaluate any MZV of the form \( \zeta(\overbrace{a, a, \ldots, a}^{\text{$n$ times}}) \) as a polynomial in \( \zeta(a), \zeta(a \oplus a), \ldots, \zeta(\underbrace{a \oplus \cdots \oplus a}_{\text{$n$ times}}) \).

Now, we derive the remaining evaluations by comparing the stuffle-product of alternating MZV's (multiplying their sum representation), with the shuffle-product (multiplying their integral representation).  Consider
\begin{align*}
	 \zeta(\overline{1})\zeta(\overline{2}) & \overset{\mathrm{st}}{=} \zeta(\overline{1},\overline{2}) + \zeta(\overline{2},\overline{1}) + \zeta(3) \\[1ex]
	 \zeta(\overline{1})\zeta(2) &= \int_0^1 \eta_{-1} \int_0^1 \eta_1 \, \eta_0 
	 \overset{\mathrm{sh}}{=} \int_0^1 \eta_{-1} \, \eta_1 \, \eta_0 
	+ \int_0^1 \eta_1 \, \eta_{-1} \, \eta_0 
	+ \int_0^1 \eta_{1} \, \eta_0 \, \eta_{-1} \\
	& = \zeta(\overline{1},2) + \zeta(\overline{1},\overline{2}) + \zeta(\overline{2}, \overline{1})\,.
\end{align*}
Then subtracting one from the other implies
\begin{align*}
	\zeta(\overline{1},2) & = \zeta(\overline{1})\zeta(2) - \zeta(\overline{1})\zeta(\overline{2}) + \zeta(3)  
	= -\frac{\pi^2}{4} \log(2) + \zeta(3) \,.
\end{align*}
From this evaluation and the stuffle-product \( \zeta(\overline{1})\zeta(2) = \zeta(\overline{1},2) + \zeta(2,\overline{1}) + \zeta(\overline{3}) \), we immediately obtain
\[
	\zeta(2,\overline{1}) = \frac{\pi^2}{12} \log(2) - \frac{1}{4} \zeta(3) \,.
\]
Then consider
\begin{align*}
	\zeta(\overline{2}) \zeta(\overline{1}) & = \int_0^1 \eta_{-1}\int _0^1 \eta_{-1} \, \eta_0  
	 \overset{\mathrm{sh}}{=} 2 \int_0^1 \eta_{-1}\, \eta_{-1} \, \eta_0 
	+ \int_0^1 \eta_{-1} \, \eta_0 \, \eta{-1}  \\
	& = 2\zeta(1,\overline{2}) + \zeta(2,\overline{1})\,.
\end{align*}
With the evaluation of \( \zeta(2,\overline{1}) \) above, this implies
\[
	\zeta(1,\overline{2}) = \frac{1}{8} \zeta(3) \,.
\]
From the iterated integral representation, we can change variables via \( t \mapsto 1-t \) (and then reverse the path) to obtain
\[
	\zeta(1,2) = \int_0^1 \eta_1 \, \eta_1 \, \eta_0  = \int_1^0 \eta_0 \, \eta_0 \, \eta_1  = - \int_0^1 \eta_1 \, \eta_0 \, \eta _0  = \zeta(3) \,.
\]
By considering the level 2 distribution relation in depth 2
\begin{align*}
	\zeta(1,2) + \zeta(\overline{1},2) + \zeta(1,\overline{2}) + \zeta(\overline{1},\overline{2}) &= \sum_{k_1 < k_2} \frac{(1 + (-1)^{k_1})(1 + (-1)^{k_2})}{k_1 k_2^2} \\
	& = \sum_{k_1 < k_2} \frac{4}{(2k_1) (2k_2)^2} = \frac{1}{2} \zeta(1,2) \,,
\end{align*}
we can then obtain
\[
	\zeta(\overline{1},\overline{2}) = \frac{\pi^2}{4} \log(2) - \frac{13}{8} \zeta(3) \,.
\]
Finally from the stuffle-product \( \zeta(\overline{1})\zeta(\overline{2}) = \zeta(\overline{1},\overline{2}) + \zeta(\overline{2},\overline{1}) + \zeta(3) \), we obtain
\[
	\zeta(\overline{2},\overline{1}) = -\frac{\pi^2}{6}\log(2) + \frac{5}{8} \zeta(3) \,.
\]
This has established all of the necessary evaluations; summarising we obtain the following \( \Omega \)-values
\begin{equation}\label{eqn:wt3omega}
\begin{alignedat}{2}
\Omega _3 & =\ii \pi \,, &         \Omega _{2,1,1,1}&=-\frac{\ii \pi }{4} \zeta (3)-\frac{\ii \pi }{6}  \log ^3(2)+\frac{ \ii \pi ^3}{12} \log (2) \,, \\
\Omega _{2,1} & =-\ii \pi  \log (2) \,, &            \Omega _{2,2,2,1}&=\frac{\ii \pi  }{4}\zeta (3)-\frac{\ii \pi ^3}{12}  \log (2) \,, \\
\Omega _{3,1,1} &=\frac{\ii \pi}{2} \log ^2(2)-\frac{\ii \pi ^3}{12} \,, \hspace{4em} &         \Omega _{3,1,2,3}&=\frac{13 \ii \pi  }{8}\zeta(3)-\frac{\ii \pi ^3}{4}  \log (2) \,, \\
\Omega _{2,2,3} &=\frac{\ii\pi ^3}{12} \,,  &          \Omega _{2,1,3,3}&=\frac{\ii \pi ^3}{4} \log (2)-\ii \pi  \zeta (3) \,, \\
\Omega _{3,3,3} &=-\frac{\ii \pi ^3}{6} \,, &    \Omega _{3,3,2,1}&=\frac{\ii \pi ^3 }{6} \log (2)-\frac{5 \ii \pi  }{8}\zeta (3)\,.
\end{alignedat}
\end{equation}
By substitution of the above results in the formula for $\alpha_3$ in Corollary \ref{cor:alpha3}, we obtain
$$\alpha_3 = \frac{9}{4} \zeta(3).$$
Using a computer, we can calculate further coefficients $\alpha_k$, both formally and numerically.
We give the output of these computations in Appendix \ref{appendix:numalpha}.
\begin{remark}
A more complicated proof that \( \alpha_3 = \frac{9}{4} \zeta(3) \), which does not rely on the currently in-preparation work of Hirose and Sato \cite{hsBetaIntegral}, is available upon request.
\end{remark}

\def\even{\mbox{\scriptsize even}}
\def\odd{\mbox{\scriptsize odd}}
\def\Lip{\operatorname{Lip}}
\def\weight{\varrho}
\def\triplenorm{|\hspace{-0.3mm}|\hspace{-0.3mm}|}
\def\bigtriplenorm{\big|\hspace{-0.5mm}\big|\hspace{-0.5mm}\big|}
\section{Quantitative Implicit Function Theorem} \label{sec:quantativeimplicitfunction}
In this section we show how to estimate the convergence radius of the Taylor series of the DPW potential $\eta_{t, x(t)}$ at $t=0$ for $\varphi = \pi/4$. This also gives an estimate on the convergence radius for the series \eqref{eq:area-series} giving the area of Lawson minimal surfaces.
To do this, we complexify the problem and rewrite the Monodromy Problem as a system of holomorphic equations in the complex parameters $(t,x)$. Then we estimate the radius $T$ such that the solution $x(t)$ obtained from the implicit function theorem exists for all $t\in D(0,T)$. By holomorphicity, the series expansion of $x(t)$ at $t=0$ will have convergence radius at least $T$. There are various Quantitative Implicit Function results which estimate how large $T$ can be (see for example Kantorovitch theory or Smale alpha-theory in chapter 3 of \cite{dedieu}).
However, we obtain much better results by working barehand with the Contraction Mapping Principle.
The estimate we obtain relies on the numerical estimation (with proven error bounds) of a large number of $\Omega$-values.
\subsection{Parity ansatz}
For $u\in\mathcal W$, we define
$$u_{\even}(\lambda)=\tfrac{1}{2}(u(\lambda)+u(-\lambda))$$
$$u_{\odd}(\lambda)=\tfrac{1}{2}(u(\lambda)-u(-\lambda))$$
and denote $\mathcal W_{\even}$ the subspace of even functions.
The notations are combined in the obvious way (for example $\mathcal W^{\geq 0}_{\R,\even}$
denotes the space of even functions with real coefficients which extend holomorphically to the unit disk).

From now on, we fix $\varphi=\pi/4$.
By the last point of Proposition \ref{prop:symmetries}, the solution to Problem \eqref{monodromy-problem3}
has the following parity with respect to $\lambda$:
\begin{align*}
x_1(\lambda)&=-x_1(-\lambda)\\
x_2(\lambda)&=-x_3(-\lambda).
\end{align*}
We impose these symmetries a priori and define
\begin{align*}
x_1&=\cv{x}_1+\lambda \ii u_1\\
x_2&=\cv{x}_2-u_2+\lambda u_3\\
x_3&=\cv{x}_3+ u_2+\lambda u_3
\end{align*}
where $u_1,u_2,u_3$ are parameters in $\mathcal W^{\geq 0}_{\R,\even}$.
We denote $u=(u_1,u_2,u_3)$. With a slight abuse of notations, we denote the potential by
$\eta_{t,u}$ instead of $\eta_{t,x(u)}$ and likewise we use $\mathfrak p(t,u)$ instead of $\mathfrak p(t,x(u))$.
A computation gives
\begin{equation}
\label{eq:K}
\mathcal K(u)=x_1^2+x_2^2+x_3^2=1+(\lambda^2-1)u_1-\sqrt{2}(\lambda^2+1)u_3
-\lambda^2 u_1^2+2 u_2^2+2\lambda^2u_3^2,
\end{equation}
so $\mathcal K(u)\in\mathcal W^{\geq 0}_{\even}$ is even.
\subsection{Reformulation of the monodromy problem}
\label{section:quantitative-reformulation}
We define the operator
$$\mathcal D:\mathcal W^{\geq 0}_{\even}\to\mathcal W^{\geq 0}_{\even}\qquad
\mathcal D(u)(\lambda)=\frac{u(\lambda)-u(1)}{\lambda^2-1}.$$
Using Proposition \ref{Pro:decomposition} and the isometry
$\mathcal W_{\rho^2}^{\geq 0}\to\mathcal W_{\rho,\even}^{\geq 0}$ defined by
$u\mapsto (\lambda\mapsto u(\lambda^2))$, $\mathcal D$ is a bounded operator with norm
$\|\mathcal D\|=\frac{1}{\rho^2-1}$.
We define $\mathcal F=(\mathcal F_j)_{1\leq j\leq 4}$ by
\begin{equation}
\label{eq:monodromy-problem4}
\begin{cases}
\mathcal F_1(t,u)=\mathcal D(\mathcal K(u))\\
\mathcal F_2(t,u)=(\wh{\mathfrak p}(t,u)-\wh{\mathfrak p}(t,u)^*)^+_{\even}\\
\mathcal F_3(t,u)=(\wh{\mathfrak p}(t,u)-\wh{\mathfrak p}(t,u)^*)_{\odd}^+\\
\mathcal F_4(t,u)=\wh{\mathfrak p}_{\even}(\lambda=\ii).\end{cases}
\end{equation}
\begin{proposition}
Problem \eqref{monodromy-problem3} is equivalent to $\mathcal F(t,u)=0$
\end{proposition}
\begin{proof}
Clearly, Problem \eqref{monodromy-problem3} implies $\mathcal F(t,u)=0$. For the converse, assume that $\mathcal F(t,u)=0$.
Then $\mathcal F_1=0$ implies that $\mathcal K(u)$ is constant while
$\mathcal F_2=\mathcal F_3=0$ implies that $(\wh{\mathfrak p}-\wh{\mathfrak p}^*)^+=0$.
Since $\mathfrak p=\mathfrak p^*=\overline{\mathfrak p}$
$$\wh{\mathfrak p}(\ii)=\wh{\mathfrak p}^*(\ii)=\overline{\wh{\mathfrak p}(\ii)}=\wh{\mathfrak p}(-\ii)$$
so $\wh{\mathfrak p}_{\odd}(\ii)=0$. Hence $\mathcal F_4=0$ implies that 
$\wh{\mathfrak p}(\ii)=0$.
By the proof of Point 2 of Proposition \ref{prop:symmetries}, we have for $\varphi=\pi/4$
$$\mathfrak q(t,u)(\lambda)=-\mathfrak p(t,u)(-\lambda)$$
and the corresponding statements for $\mathfrak q$ follow.
\end{proof}
\subsection{Complexification}
\label{section:complexification}
We redefine the star operator on $\mathcal W$ as the holomorphic extension of the star operator on
$\mathcal W_{\R}$, namely
$$u^*(\lambda)=\sum_{k\in\Z} u_k\lambda^{-k},$$
without the complex conjugation of $u_k$.
With that in mind, \eqref{eq:monodromy-problem4} defines a holomorphic map, which we denote the same by the same letter
$$\mathcal F:\C\times(\mathcal W^{\geq 0}_{\even}) ^3\to (\mathcal W)^{\geq 0}_{\even}\times(\mathcal W)^{>0}_{\even}\times(\mathcal W)^{>0}_{\odd}\times\C.$$
\begin{remark} The complexified equation $\mathcal F(t,u)=0$ does not have any geometrical meaning, if $t$ is not real. In particular, it does not imply that the monodromy is unitary, so there is no corresponding minimal immersion into $\mathrm{SU}(2)$.\end{remark}
\subsection{Reformulation as a fixed point problem}
Let $E=(\mathcal W^{\geq 0}_{\even})^3$
and consider the partial differential
$$L=d_u \mathcal F(0,0):E\longrightarrow
(\mathcal W^{\geq 0}_{\even})\times (\mathcal W^{>0}_{\even})
\times (\mathcal W^{>0}_{\odd})\times\C$$
given by
$$L=\left(
du_1-\sqrt{2}\,\mathcal D\big((\lambda^2+1)du_3\big),\;
2\pi du_2^+,\;
2\pi \lambda du_3,\;
2\pi du_2(\ii)\right).$$
The operator $L$ is an isomorphism whose inverse is given by
$$L^{-1}(v_1,v_2,v_3,v_4)=
\left(
v_1+\tfrac{\sqrt{2}}{2\pi}\mathcal D\big((\lambda^{-1}+\lambda)v_3\big),\;
\tfrac{1}{2\pi}(v_2-v_2(\ii)+v_4),\;
\tfrac{1}{2\pi\lambda} v_3
\right).$$
Following the proof of the implicit function theorem, we consider the map
$$\begin{array}{lcl}\mathcal G&:&\C\times E\to E\\
&&(t,u)\mapsto L^{-1}\mathcal F(t,u)-u\end{array}$$
so $\mathcal F(t,u)=0$ is equivalent to $-\mathcal G(t,u)=u$.
The components of $\mathcal G$ are given by
\begin{align*}
\mathcal G_1(t,u)&=\mathcal D\left[\mathcal K(u)+\tfrac{\sqrt{2}}{2\pi}(\lambda^{-1}+\lambda)(\wh{\mathfrak p}(t,u)-\wh{\mathfrak p}(t,u)^*)^+_{\odd}\right]-u_1\\
\mathcal G_2(t,u)&=\frac{1}{2\pi}\left[(\wh{\mathfrak p}(t,u)-\wh{\mathfrak p}(t,u)^*)^+_{\even}-(\wh{\mathfrak p}(t,u)-\wh{\mathfrak p}(t,u)^*)^+_{\even}(\lambda=\ii)+\wh{\mathfrak p}(t,u)_{\even}(\lambda=\ii)\right]-u_2\\
\mathcal G_3(t,u)&=\frac{1}{2\pi\lambda}\big(\wh{\mathfrak p}(t,u)-\wh{\mathfrak p}(t,u)^*)^+_{\odd}-u_3.
\end{align*}
Our goal is to find $T$ and $R=(R_1,R_2,R_3)$
such that for $t\in\C$, $|t|\leq T$, the map $u\mapsto \mathcal G(t,u)$ preserves the box
$$B_R=\{u\in E\mid \|u_j\|_{\rho}\leq R_j, 1\leq j\leq 3\}$$
and is contracting (for a norm yet to defined on $E$, see Section \ref{section:Lipschitz}).
The proof of the implicit function theorem ensures the existence of such $T>0$ and $R>0$. We want to investigate how large $T$ can be. For this we need to estimate $\|\mathcal G_i\|_{\rho}$ and their Lipschitz constants.
\subsection{Estimating $\mathcal G$}
\label{section:estimationG}
In this section, we assume that $|t|\leq T$ and $\|u_i\|_{\rho}\leq R_i$ for $i=1,2,3$ and
give estimates of $\|\mathcal G_i\|_{\rho}$ in terms of $T$, $R$, $\rho$.
Fix some integer $n\geq 1$.
By Proposition \ref{prop:expandpq}, the order $n-1$ expansion of $\wh{\mathfrak p}=t^{-1}\mathfrak p$ is
$$\wh{\mathfrak p}(t,u)=\sum_{k=0}^{n-1}\sum_{i_1,\cdots,i_{k+1}} t^k x_{i_1}(u)\cdots x_{i_{k+1}}(u)
(M_{i_1}\cdots M_{i_{k+1}})_{31}\,\Omega_{i_1,\cdots,i_{k+1}}(1)
+\mathcal R_n(t,u)$$
with $\mathcal R_n(t,u)=O(|t|^n).$
Replacing each $x_i$ by its expression in term of $u$, we obtain an expression of the form
$$\wh{\mathfrak p}(t,u)=\sum_{k=0}^{n-1} \sum_{|\alpha|\leq k+1} a_{k,\alpha} t^k u^{\alpha}+\mathcal R_n(t,u)$$
where $\alpha=(\alpha_1,\alpha_2,\alpha_3)\in\N^3$.
We use the following standard notations for multi-indices:
$$u^\alpha=u_1^{\alpha_1}u_2^{\alpha_2} u_3^{\alpha_3}$$
$$|\alpha|=\alpha_1+\alpha_2+\alpha_3$$
$$R^{\alpha}=R_1^{\alpha_1}R_2^{\alpha_2} R_3^{\alpha_3}.$$
The coefficients $a_{k,\alpha}$ are Laurent polynomials of degree at most $k+1$ in $\lambda$ whose coefficients can be computed in terms of $\Omega$-values.
In the same way, we rewrite \eqref{eq:K} as
$$\mathcal K(u)=\sum_{k,\alpha}b_{k,\alpha} t^k u^\alpha$$
where of course $b_{k,\alpha}=0$ if $k>0$ or $\alpha>2$, but this notation is consistent with the notation for $\wh{\mathfrak p}$. 
Moreover, later on we will add a $t$-dependent correcting term to each $x_i$ and $\mathcal K$ will depend on $t$, so it is better to do the estimate in this general form.
We write
$$\mathcal G_i(t,u)=\sum_{k,\alpha} \mathcal G_i^{k,\alpha}+\mathcal G_i^{\mathcal R}$$
where $\mathcal G_i^{k,\alpha}$ and  $\mathcal  G_i^{\mathcal R}$ denote respectively the contributions of the terms $t^k u^\alpha$ and $\mathcal R_n$ to $\mathcal G_i$, where $\mathcal G_i^{k, \alpha}$ are given by
\begin{align*}
\mathcal G_1^{k,\alpha}&=t^k\mathcal D\left[b_{k,\alpha}u^\alpha+\tfrac{\sqrt{2}}{2\pi}(\lambda^{-1}+\lambda)(a_{k,\alpha}u^{\alpha}-(a_{k,\alpha}u^{\alpha})^*)^+_{\odd}\right]\\
\mathcal G_2^{k,\alpha}&=\frac{t^k}{2\pi}\left[\big(a_{k,\alpha}u^{\alpha}-(a_{k,\alpha}u^{\alpha})^*\big)^+_{\even}-\big(a_{k,\alpha}u^{\alpha}-(a_{k,\alpha} u^{\alpha})^*\big)^+_{\even}(\lambda=\ii)+(a_{k,\alpha}u^{\alpha})_{\even}(\lambda=\ii)\right]\\
\mathcal G_3^{k,\alpha}&=\frac{t^k}{2\pi\lambda}\big(a_{k,\alpha}u^{\alpha}-(a_{k,\alpha}u^{\alpha})^*\big)^+_{\odd}
\end{align*}
to which one must subtract $u_1$ from $\mathcal G_1^{0,(1,0,0)}$,
$u_2$ from $\mathcal G_2^{0,(0,1,0)}$ and 
$u_3$ from $\mathcal G_3^{0,(0,0,1)}$.
We estimate each term depending on $T$, $R$ and the coefficients $a_{k,\alpha}$ and $b_{k,\alpha}$.
\begin{enumerate}
\item By definition of $\mathcal G_i$, the linear terms with $k=0$ and $|\alpha|=1$ vanish.
\item The terms with $\alpha=0$ are evaluated explicitly:
\begin{align*}
\|\mathcal G_1^{k,0}\|_{\rho}&=T^k\big\|\mathcal D\left[b_{k,0}+\tfrac{\sqrt{2}}{2\pi}(\lambda^{-1}+\lambda)(a_{k,0}-a_{k,0}^*)^+_{\odd}\right]\big\|_{\rho}\\
\|\mathcal G_2^{k,0}\|_{\rho}&= \frac{T^k}{2\pi}\left\|(a_{k,0}-a_{k,0}^*)^+_{\even}-(a_{k,0}-a_{k,0}^*)^+_{\even}(\ii)+(a_{k,0})_{\even}(\ii)\right\|_{\rho}\\
\|\mathcal G_3^{k,0}\|_{\rho}&=\frac{T^k}{2\pi\rho}\left\|(a_{k,0}-a_{k,0}^*)^+_{\odd}\right\|_{\rho}
\end{align*}
\item We estimate the remaining terms the best we can, using the following elementary facts:
$$u\in\mathcal W\Rightarrow \| (u-u^*)^+\|_{\rho}\leq \|u\|_{\rho}$$
$$u\in\mathcal W^{>0}\Rightarrow \|\lambda^{-1}u\|_{\rho}\leq \rho^{-1}\|u\|_{\rho}$$
$$u\in\mathcal W^{>0}_{\even}\Rightarrow
|u(\ii)|\leq \rho^{-2}\|u\|_{\rho}.$$
Estimating $\mathcal G_3^{k,\alpha}$ is straightforward:
$$\|\mathcal G_3^{k,\alpha}\|_{\rho}
\leq\frac{T^k}{2\pi\rho}\|(a_{k,\alpha} u^{\alpha})_{\odd}\|_{\rho}
\leq \frac{ T^k R^{\alpha}}{2\pi\rho}\|a_{k,\alpha,\odd}\|_{\rho}.$$
For $\mathcal G_2^{k,\alpha}$, we can take advantage of some cancellations.
 Write
$$f(\lambda)=(a_{k,\alpha}u^{\alpha})_{\even}=\sum_{\ell\in\Z}f_{2\ell}\lambda^{2\ell}.$$
Then
$$(f-f^*)^+=
\sum_{\ell\geq 1}(f_{2\ell}-f_{-2\ell})\lambda^{2\ell}$$
$$(f-f^*)^+(\ii)=
\sum_{\ell\geq 1}(f_{2\ell}-f_{-2\ell})\ii^{2\ell}$$
$$f(\ii)=f_0+\sum_{\ell\geq 1}(f_{2\ell}+f_{-2\ell})\ii^{2\ell}$$
$$\mathcal G_2^{k,\alpha}=\frac{t^k}{2\pi}\big(f_0+\sum_{\ell\geq 1}(f_{2\ell}-f_{-2\ell})\lambda^{2\ell}+2\sum_{\ell\geq 1}f_{-2\ell}\,\ii^{2\ell}\big)$$
and finally
$$\|\mathcal G_2^{k,\alpha}\|_{\rho}\leq \frac{T^k}{2\pi}\left(\|f\|_{\rho}+2\rho^{-2}\|f^-\|_{\rho}\right)
\leq \frac{T^k R^\alpha}{2\pi}\left(\|a_{k,\alpha,\even}\|_{\rho}+2 \rho^{-2}\|a_{k,\alpha,\even}^-\|_{\rho}\right).$$
To estimate $\mathcal G_1^{k,\alpha}$, we decompose
$$a_{\alpha,k,\odd}=a_{k,\alpha,\odd}^+ + a_{k,\alpha,\odd}^-$$
and observe that $(a_{k,\alpha,\odd}^+ u^{\alpha})^{*+}=0$ so
\begin{eqnarray*}
\mathcal G_1^{k,\alpha}&=&t^k\mathcal D\left[
\left(b_{k,\alpha}+\tfrac{\sqrt{2}}{2\pi}(\lambda^{-1}+\lambda)a_{k,\alpha,\odd}^+\right)u^{\alpha}\right]\\
&+&t^k\mathcal D\left[\tfrac{\sqrt{2}}{2\pi} (\lambda^{-1}+\lambda)\left(a_{k,\alpha,\odd}^- u^{\alpha}-(a_{k,\alpha,\odd}^- u^{\alpha})^*\right)^+\right]
\end{eqnarray*}
and we can take advantage of possible cancellations in the first term.
This gives the estimate
$$\|\mathcal G_1^{k,\alpha}\|_{\rho}\leq \frac{T^k R^\alpha}{\rho^2-1}\left(
\big\|b_{k,\alpha}+\tfrac{\sqrt{2}}{2\pi}(\lambda^{-1}+\lambda)a^+_{k,\alpha,\odd}\big\|_{\rho}
+\tfrac{\sqrt{2}}{2\pi}(\rho^{-1}+\rho)\|a_{k,\alpha,\odd}^-\|_{\rho}\right).$$
\end{enumerate}
\subsection{Estimating the Lipschitz constants}
\label{section:Lipschitz}
Let $\weight=(\weight_1,\weight_2,\weight_3)\in(0,\infty)^3$ be positive weights.
We define a weighted norm on $E$ by
$$\|u\|_{\weight,\rho}=\max_{1\leq i\leq 3}\weight_i^{-1}\|u_i\|_{\rho}.$$
Then the Lipschitz constant of the map $\mathcal G$ on the box $B_R$ is
$$\Lip(\mathcal G)=\max_{u,v\in B_R\atop |t|\leq T}\frac{\|\mathcal G(t,u)-\mathcal G(t,v)\|_{\weight,\rho}}
{\|u-v\|_{\weight,\rho}}
=\max_{1\leq i\leq 3}\weight_i^{-1}\Lip(\mathcal G_i)$$
where 
$$\Lip(\mathcal G_i)=\sup_{u,v\in B_R\atop |t|\leq T}\frac{\|\mathcal G_i(t,u)-\mathcal G_j(t,v)\|_{\rho}}{\|u-v\|_{\weight,\rho}}.$$
The map $\mathcal G$ is contracting (with respect to $u\in B_R$) if $\Lip(\mathcal G_i)<\weight_j$ for $1\leq i\leq 3$.
The weights are additional parameters which help to have a contracting map.
In this section we explain how to estimate the Lipschitz constant of each term $\mathcal G_i^{k,\alpha}$.
The differential of $u\mapsto u^{\alpha}=u_1^{\alpha_1}u_2^{\alpha_2}u_3^{\alpha_3}$ applied to $v\in E$ is
$$ d(u^{\alpha})\cdot v=\sum_{i=1}^3 \alpha_i u_i^{\alpha_i-1} v_i\prod_{j\neq i} u_j^{\alpha_j}.$$
Taking norms on both sides and using $\|v_i\|\leq \weight_i\|v\|_{\weight,\rho}$, we obtain for $u\in B_R$:
$$\|d(u^{\alpha})\cdot v\|_{\weight,\rho}\leq \|v\|_{\weight,\rho}\sum_{i=1}^3 \alpha_i R_i^{\alpha_i-1} \weight_i \prod_{j\neq i}R_j^{\alpha_j}=\|v\|_{\weight,\rho}\,d(R^{\alpha})\cdot \weight.$$
By the mean value inequality, for $u,v\in B_R$
$$\|u^{\alpha}-v^{\alpha}\|_{\weight,\rho}\leq \|u-v\|_{\weight,\rho}\,d(R^{\alpha})\cdot \weight.$$
By the results of Section \ref{section:estimationG}, since each term $\mathcal G_i^{k,\alpha}$ is linear in $u^{\alpha}$:
\begin{align*}
\Lip(\mathcal G_1^{k,\alpha})&\leq \frac{T^k}{\rho^2-1}\left[
\big\|b_{k,\alpha}+\tfrac{\sqrt{2}}{2\pi}(\lambda^{-1}+\lambda)a^+_{k,\alpha,\odd}\big\|_{\rho}
+\tfrac{\sqrt{2}}{2\pi}(\rho^{-1}+\rho)\|a_{k,\alpha,\odd}^-\|_{\rho}\right]\,d(R^{\alpha})\cdot \weight\\
\Lip(\mathcal G_2^{k,\alpha})&\leq \frac{T^k}{2\pi}\left[\|a_{k,\alpha,\even}\|+2 \rho^{-2}\|a_{k,\alpha,\even}^-\|\right]d(R^{\alpha})\cdot\weight\\
\Lip(\mathcal G_3^{k,\alpha})&\leq \frac{T^k}{2\pi\rho}\|a_{k,\alpha,\odd}\|_{\rho}\,d(R^{\alpha})\cdot \weight.
\end{align*}
In other words, $\Lip(\mathcal G_i^{k,\alpha})$ is obtained by taking the differential of the estimate of
$\mathcal G_i^{k,\alpha}$ with respect to $R$ applied to the vector $\weight$.
\subsection{Estimating the remainder terms $\mathcal G_i^{\mathcal R}$}
\label{section:remainder}
We first estimate the remainder $\mathcal R_n$ of $\mathfrak p$ using Gronwall inequality.
\begin{proposition} There exists constants $C_0(R,\rho)$ and $C_1(n,R,\rho)$,
$C_2(n,R,\rho,\weight)$ and $C_3(\rho,\weight)$ such that for $\|u_i\|_{\rho}\leq R_i$
and $|t|\leq T$: 
\label{prop:gronwall}
$$\|\mathcal R_n(t,u)\|_{\rho}\leq e^{C_0 \,T}C_1\,T^n $$
and
$$\Lip(\mathcal R_n)\leq e^{C_0\,T}C_2 T^n + e^{2 C_0\, T}C_1C_3 T^{n+1}.$$
\end{proposition}
The proof explains how to compute the constants.
\begin{proof}
We equip the space of line vectors $\mathcal M_{1,3}(\mathcal W)$ with the sup norm:
$$\|X\|_{\rho}=\max_{1\leq j\leq 3}\|X_j\|_{\rho}$$
and $\mathcal M_3(\mathcal W)$
with the associated matrix norm
$$\triplenorm A\triplenorm_{\rho}=\max_{1\leq j\leq 3}\sum_{i=1}^3 \|A_{ij}\|_{\rho}$$
so that $\|XA\|_{\rho}\leq \|X\|_{\rho}\triplenorm A\triplenorm_{\rho}$.
With the notations of the proof of Proposition \ref{prop:expandpq}, let
$$Z_{t,x}(z)=Y_{t,x}(z)-\sum_{k=0}^n \frac{t^k}{k!}Y_{t,x}^{(k)}(z)
=Y_{t,x}(z)-\sum_{k=0}^n t^k \int_0^z(\alpha_x)^k.$$
We have
$$\mathcal R_n(t,u)=t^{-1}Z_{t,x(u)}^{31}(1).$$
For $z\in [0,1]$
\begin{eqnarray*}
d Z_{t,x}(z)&=&Y_{t,x}(z) t\alpha_x(z)-\left(\sum_{k=1}^n t^k\int_0^z (\alpha_x)^{k-1}\right)\alpha_x(z)\\
&=&t\left(Y_{t,x}(z)-\sum_{k=0}^{n-1} t^k \int_0^z (\alpha_x)^k\right)\alpha_x(z)\\
&=&t\left(Z_{t,x}(z)+t^n \int_0^z(\alpha_x)^n\right)\alpha_x(z).
\end{eqnarray*}
Integrating and using the recursive definition of iterated integrals gives
\begin{equation}
\label{eq:Ztx}
Z_{t,x}(z)=t\int_0^z Z_{t,x}\alpha_x+t^{n+1} \int_0^z (\alpha_x)^{n+1}.
\end{equation}
Multiplying on the left by $e_3=(0,0,1)$ and taking norms:
$$\|e_3 Z_{t,x}(z)\|_{\rho}\leq |t|\int_0^z \|e_3 Z_{t,x}\|_{\rho}\,\triplenorm \alpha_x\triplenorm_{\rho}+|t|^{n+1} \int_0^z \|e_3(\alpha_x)^{n+1}\|_{\rho}.$$
By Gronwall inequality, for all $z\in[0,1]$
\begin{equation}
\label{eq:gronwall}
\|e_3 Z_{t,x}(z)\|_{\rho}\leq 
\exp\left(|t|\int_0^1\triplenorm \alpha_x\triplenorm_{\rho}\right)
|t|^{n+1}\int_0^1\big\|e_3 (\alpha_x)^{n+1}\big\|_{\rho}
.
\end{equation}
We need to estimate the two integrals.
For $u\in B_R$, we have the straightforward estimate
\begin{equation}
\label{eq:estimate-xi}
\|x_i(u)\|_{\rho}\leq c_i(R,\rho)\quad \text{ with } \quad c_1=\rho+\rho R_1\quad\text{ and }\quad
c_2=c_3=\tfrac{\rho}{\sqrt{2}}+R_2+\rho R_3.
\end{equation}
For $z\in[0,1]$, we have
$$|\omega_1|=\frac{4z}{1+z^4},\quad
|\omega_2|=\frac{2 \sqrt{2}(1-z^2)}{1+z^4} \; \leq \;
|\omega_3|=\frac{2\sqrt{2}(1+z^2)}{1+z^4}.$$
Hence
$$\triplenorm\alpha_x\triplenorm_{\rho}=2\max\{
\|x_1\|_{\rho}\,|\omega_1|,\|x_2\|_{\rho}\,|\omega_2|\}+
2\|x_3\|_{\rho}\,|\omega_3|
\leq 2\max\{
c_1|\omega_1|,c_2|\omega_2|\}+
2c_3|\omega_3|.
$$
We have
$$c_1|\omega_1|<c_2|\omega_2|\;\Leftrightarrow\;
z<z_0=\frac{\sqrt{c_1^2+2 c_2^2}-c_1}{\sqrt{2}c_2}.$$
Hence
\begin{eqnarray*}
\int_0^1 \triplenorm\alpha_x\triplenorm_{\rho}
&\leq &2 c_2\int_0^{z_0}|\omega_2|+2 c_1\int_{z_0}^{1}|\omega_1|+2c_3\int_0^1 |\omega_3|\\
&=&2 c_2\log\left(\frac{z_0^2+\sqrt{2}z_0+1}{z_0^2-\sqrt{2}z_0+1}\right)
+c_1(\pi-4\arctan(z_0^2))+2\pi c_3=C_0(\rho,R).
\end{eqnarray*}
The other (iterated) integral is estimated as follows:
$$\int_0^z\|e_3 (\alpha_x)^{n+1}\|_{\rho}\leq\sum_{i_1,\cdots,i_{n+1}}
\|x_{i_1}\cdots x_{i_{n+1}}\|_{\rho}
\,\|e_3 M_{i_1}\cdots M_{i_{n+1}}\|\,\int_0^1|\omega_{i_1}\cdots\omega_{i_{n+1}}|.$$
Observe that $\omega_1$, $\omega_2$ and $\omega_3$ are real or imaginary but have constant argument on $[0,1]$, so
\begin{equation}
\label{eq:absOmega}\int_0^1 |\omega_{i_1}\cdots\omega_{i_{n+1}}|=
\big|\int_0^1\omega_{i_1}\cdots\omega_{i_{n+1}}\big|
=|\Omega_{i_1,\cdots,i_{n+1}}(1)|.\end{equation}
For $u\in B_R$, we can expand $x_{i_1}\cdots x_{i_{n+1}}$ as a polynomial in $u_1$, $u_2$, $u_3$
and estimate its norm in function of $\rho,R$. This gives an inequality of the form
$$\max_{z\in[0,1]}\big\|e_3 \int_0^z (\alpha_x)^{n+1}\big\|_{\rho}\leq C_1(n,R,\rho).$$
\begin{remark}
\label{remark:numAbsOmega}
Computing $C_1$ requires the numerical evaluation
of $\Omega_{i_1,\cdots,i_{n+1}}$ where $(i_1,\cdots,i_{n+1})$ labels the edges on the graph of Proposition \ref{prop:pattern} from the vertex $e_3$ to any other vertices,
not just $e_1$ as in Section \ref{section:multizetas}. These $\Omega$-values can in general not be expressed in term of MZVs. We compute them using Multiple Polylogarithms instead (see Appendix \ref{appendix:nummpl}).

\end{remark}
By \eqref{eq:gronwall}, we obtain
$$\|\mathcal R_n\|_{\rho}\leq \|t^{-1} e_3 Z_{t,x}(1)\|_{\rho}\leq 
e^{C_0\,T} C_1 T^n.$$
To prove the second point of Proposition \ref{prop:gronwall}, differentiate Equation \eqref{eq:Ztx}
with respect to $u_i$, take norms and use Gronwall inequality again:
$$\frac{\partial Z_{t,x}(z)}{\partial u_i}=
t\int_0^z \frac{\partial Z_{t,x}}{\partial u_i}\alpha_x
+t\int_0^z Z_{t,x}\frac{\partial \alpha_x}{\partial u_i}
+t^{n+1} \int_0^z\frac{\partial}{\partial u_i} (\alpha_x)^{n+1}$$
$$\big\|e_3 \frac{\partial Z_{t,x}(z)}{\partial u_i}\big\|_{\rho}\leq 
|t|\int_0^z \big\|e_3 \frac{\partial Z_{t,x}}{\partial u_i}\big\|_{\rho}\,\triplenorm \alpha_x\triplenorm_{\rho}
+|t|\int_0^z \|e_3 Z_{t,x}\|_{\rho}\bigtriplenorm\frac{\partial \alpha_x}{\partial u_i}\bigtriplenorm_{\rho}
+|t|^{n+1}\int_0^z\big\|e_3 \frac{\partial}{\partial u_i} (\alpha_x)^{n+1}\big\|_{\rho}$$
$$\big\|e_3 \frac{\partial Z_{t,x}(1)}{\partial u_i}\big\|_{\rho}\leq 
e^{C_0 |t|}\left(
|t|\int_0^1 \|e_3 Z_{t,x}\|_{\rho}\bigtriplenorm\frac{\partial \alpha_x}{\partial u_i}\bigtriplenorm_{\rho}
+|t|^{n+1}\int_0^1\big\|e_3 \frac{\partial}{\partial u_i} (\alpha_x)^{n+1}\big\|_{\rho}\right).$$
We need to estimate the two integrals. By Inequality \eqref{eq:gronwall}
$$\int_0^z \|e_3 Z_{t,x}\|_{\rho}\bigtriplenorm\frac{\partial \alpha_x}{\partial u_i}\bigtriplenorm_{\rho}
\leq e^{C_0 |t|} C_1|t|^{n+1}\int_0^1\bigtriplenorm\frac{\partial \alpha_x}{\partial u_i}\bigtriplenorm_{\rho}.$$
We compute
\begin{align*}
\bigtriplenorm\frac{\partial \alpha_x}{\partial u_1}\bigtriplenorm_{\rho}
&=\int_0^1 2\rho|\omega_1|=2\rho|\Omega_1(1)|=\pi\rho=C_{3,1}\\
\bigtriplenorm\frac{\partial \alpha_x}{\partial u_2}\bigtriplenorm_{\rho}
&=\int_0^1 2(|\omega_2|+|\omega_3|)=2(|\Omega_2(1)|+|\Omega_3(1)|)=2\pi+4\log(\sqrt{2}+1)=C_{3,2}\\
\bigtriplenorm\frac{\partial \alpha_x}{\partial u_3}\bigtriplenorm_{\rho}
&=\int_0^1 2\rho(|\omega_2|+|\omega_3|)=2\pi\rho+4\rho\log(\sqrt{2}+1)=C_{3,3}.
\end{align*}
The other (iterated) integral is estimated as before:
$$\int_0^1\big\|e_3 \frac{\partial}{\partial u_i} (\alpha_x)^{n+1}\big\|_{\rho}
\leq\sum_{i_1,\cdots,i_{n+1}}\bigtriplenorm\frac{\partial}{\partial u_i}(x_{i_1}\cdots x_{i_{n+1})}\bigtriplenorm_{\rho}\,\|e_3M_{i_1}\cdots M_{i_{n+1}}\|\;|\omega_{i_1}\cdots\omega_{i_{n+1}}|.$$
Each term $x_{i_1}\cdots x_{i_{n+1}}$ is estimated by expanding it as a polynomial in $u$, differentiating with respect to
$u_i$ and then taking its norm. Using the inequality \eqref{eq:absOmega}, this gives an estimate of the form
$$\int_0^1\big\|e_3 \frac{\partial}{\partial u_i} (\alpha_x)^{n+1}\big\|_{\rho}
\leq C_{2,i}(n,R,\rho)$$
where each $C_{2,i}$ can be computed by estimating $\Omega$-values of
depth $n+1$ (see Remark \ref{remark:numAbsOmega}).
By the mean value inequality, for $u,v\in B_R$:
\begin{eqnarray*}
\|\mathcal R_n(u)-\mathcal R_n(v)\|_{\rho}&\leq &|t|^{-1}\sum_{i=1}^3
\big\|e_3 \frac{\partial Z_{t,x}(z)}{\partial u_i}\big\|_{\rho}\|u_i-v_i\|_{\rho}\\
&\leq &\sum_{i=1}^3 |t|^{-1}e^{C_0 |t|}\left(
|t|e^{C_0 |t|} C_1 |t|^{n+1}C_{3,i}\weight_i+|t|^{n+1} C_{2,i}\weight_i\right)
\|u-v\|_{\weight,\rho}
\end{eqnarray*}
which proves Proposition \ref{prop:gronwall} with
$$C_2=\sum_{i=1}^3 C_{2,i}\weight_i\quad\text{and}\quad
C_3=\sum_{i=1}^3 C_{3,i}\weight_i.$$
\end{proof}
\begin{proposition}
For $\|u_i\|\leq R_i$ and $|t|\leq R$, we have for $1\leq i\leq 3$:
$$\|\mathcal G_i^{\mathcal R}(t,u)\|_{\rho}\leq C^{\mathcal R}_i e^{C_0 \,T}C_1\,T^n $$
and
$$\Lip(\mathcal G_i^{\mathcal R})\leq C^{\mathcal R}_i (e^{C_0\,T}C_2 T^n + e^{2 C_0\, T}C_1C_3 T^{n+1})$$
with
$$C^{\mathcal R}_1=\frac{\sqrt{2}(\rho^{-1}+\rho)}{2\pi(\rho^2-1)},\quad
C^{\mathcal R}_2=\frac{(1+2\rho^{-2})}{2\pi}\quad\text{and}\quad
C^{\mathcal R}_3=\frac{1}{2\pi\rho}.$$
\end{proposition}
\begin{proof}
We have
$$\mathcal G_1^{\mathcal R}(t,u)=\frac{\sqrt{2}}{2\pi}\mathcal D\left[(\lambda^{-1}+\lambda)(\mathcal R_n(t,u)-\mathcal R_n(t,u)^*)^+_{\odd}\right]$$
$$\|\mathcal G_1^{\mathcal R}(t,u)\|_{\rho}\leq\frac{\sqrt{2}}{2\pi}\frac{(\rho^{-1}+\rho)}{(\rho^2-1)}\|\mathcal R_n(t,u)\|_{\rho}$$
$$\|\mathcal G_1^{\mathcal R}(t,u)-\mathcal G_1^{\mathcal R}(t,v)\|_{\rho}
\leq \frac{\sqrt{2}}{2\pi}\frac{(\rho^{-1}+\rho)}{(\rho^2-1)}\|\mathcal R_n(t,u)-\mathcal R_n(t,v)\|_{\rho}.$$
The conclusion for $\mathcal G_1^{\mathcal R}$ follows from Proposition \ref{prop:gronwall}.
The proofs for $\mathcal G_2^{\mathcal R}$ and $\mathcal G_3^{\mathcal R}$ are similar.
\end{proof}
\subsection{Estimating the genus}
Assume that we are given some positive number $T,R$ such that for $t\in D(0,T)$, Problem \eqref{eq:monodromy-problem4} has a unique solution $u(t)\in B_R$.
Let
$$\psi(t)=t\sqrt{\mathcal K(u(t))}.$$
Recall that the Lawson surface $\xi_{1,g}$ is obtained by taking $t=\psi^{-1}(\frac{1}{2g+2})$.
Since $\mathcal K(0)=1$, it is clear that $\psi$ is a diffeomorphism in a neighborhood of $0$.
In this section, we give a quantitative estimate on the size of this neighborhood.
For this, we need to know how close to $1$ the number $\mathcal K(u(t))$ (which does not depend on $\lambda$) is.
Note that we do not have access to this number,
so we instead estimate $\mathcal K(u)$, evaluated at $\lambda=1$, for all $u\in B_R$.
(The choice to evaluate at $\lambda=1$ is rather arbitrary, but seems to give better results than $\lambda=0$.)
With the notations of Section \ref{section:estimationG}, we have
$$\mathcal K(u)(\lambda=1)=\sum_{k,\alpha} b_{k,\alpha}(1) t^k u^{\alpha}(1).$$
Since $b_{0,0}=1$, we have for $u\in B_R$
$$|\mathcal K(u)(\lambda=1)-1| \leq \sum_{(k,\alpha)\neq (0,0)} |b_{k,\alpha}(1)| \,T^k R^{\alpha}=C_{\mathcal K}(T,R).$$
\begin{proposition}
\label{prop:diffeo}
Assume that $C_{\mathcal K}<1$ and let $T'=T\sqrt{1-C_{\mathcal K}}$.
Then the function $\psi$ is biholomorphic from $\psi^{-1}(D(0,T'))$ to $D(0,T')$.
Furthermore, $\psi^{-1}$ is real on $(-T',T')$.
\end{proposition}
\begin{proof}
Since $\mathcal K(u(t))$ is constant, it is equal to its value at $\lambda=1$, so
$$|\mathcal K(u(t))-1|\leq C_{\mathcal K}.$$

By Lemma \ref{lemma:diffeo} below with $a=\tfrac{1}{2}$ and $z=\mathcal K(u(t))-1$, we have for $t\in D(0,T)$
$$|\sqrt{\mathcal K(t)}-1|\leq 1-\sqrt{1-C_{\mathcal K}}.$$
For $s\in D(0,T')$, consider the functions
$$f_s(t)=\psi(t)-s\quad\text{and}\quad g_s(t)=t-s.$$
Then for $|t|=T$ :
$$|f_s(t)-g_s(t)|=|t|\,\big|\sqrt{\mathcal K(u(t))}-1\big|\leq T\big(1-\sqrt{1-C_{\mathcal K}}\big)=T-T'
<T-|s|\leq |g_s(t)|.$$
By Rouch\'e Theorem, $f_s$ and $g_s$ have the same number of zeros in $D(0,T)$,
and it is clear that $g_s$ has precisely one zero. Hence the equation $\psi(t)=s$ has a unique solution in $D(0,T)$ so $\psi$ is bijective from
$\psi^{-1}(D(0,T'))$ to $D(0,T')$ and since it is holomorphic, it is biholomorphic.
By symmetry, $\psi$ is real on $(-T,T)$, hence $\psi(\overline{t})=\overline{\psi(t)}$.
It follows that $\psi^{-1}(\overline{t})=\overline{\psi^{-1}(t)}$, so $\psi^{-1}$ is real on
$(-T',T')$.
\end{proof}
\begin{lemma} 
\label{lemma:diffeo}
For $z\in D(0,1)$ and $a\geq 0$,
$$|(1+z)^a-1|\leq \begin{cases}
1-(1-|z|)^a\quad\text{if $a\leq 1$}\\
(1+|z|)^a-1\quad\text{if $a\geq 1$}\end{cases}.$$
\end{lemma}
We use the lemma with $a>1$ in Section \ref{section:area}.
\begin{proof} Simply write
$$(1+z)^a-1=\int_0^1 a(1+sz)^{a-1}zds.$$
If $a-1\geq 0$, then
$$|(1+z)^a-1|\leq \int_0^1 a(1+s|z|)^{a-1}|z|ds=(1+|z|)^a-1.$$
If $a-1\leq 0$, then
$$|(1+z)^a-1|\leq \int_0^1 a(1-s|z|)^{a-1}|z|ds=1-(1-|z|)^a.$$
\end{proof}
\begin{corollary}
\label{corollary:diffeo}
Under the hypothesis of Proposition \ref{prop:diffeo},
The series \eqref{eq:area-series} converges for
$$g> \frac{1}{2 T\sqrt{1-C_{\mathcal K}}}-1 = \operatorname{genus}(T,R).$$
\end{corollary}
\subsection{Results}
The results of Sections \ref{section:estimationG}, \ref{section:Lipschitz} and \ref{section:remainder}
 give us estimates of the following form,
for $|t|\leq T$ and $u\in B_R$:
$$\|\mathcal G_i(t,u)\|\leq C_{\mathcal G_i}(n,T,R,\rho)$$
$$Lip(G_i)\|\leq C_{\Lip_i}(n,T,R,\weight,\rho)$$
where $(T,R,\weight,\rho)$ are formal variables.
Fix some number $\kappa<1$. We want that
\begin{equation}
\label{eq:constraintG}
C_{\mathcal G_i}(n,T,R,\rho)\leq \kappa R_i\quad\text{for $1\leq i\leq 3$}
\end{equation}
so that $\mathcal G$ preserves the box $B_R$, and
\begin{equation}
\label{eq:constraintLip}
C_{\Lip_i}(n,T,R,\weight,\rho)\leq \kappa \weight_i\quad\text{for $1\leq i\leq 3$}
\end{equation}
so that $\mathcal G$ is contractible for the norm $\|\cdot\|_{\weight,\rho}$ in the box $B_R$.
We minimize the function $\operatorname{genus}(T,R)$ (defined in Corollary \ref{corollary:diffeo})
under the constraints
\eqref{eq:constraintG} and \eqref{eq:constraintLip}, with respect to the variables
$(T,R,w,\rho)$.
This gives the following results, depending on the chosen value of $n$ (which is the order at which we have expanded $\mathfrak p$):
$$\begin{array}{|c|c|}
\hline
n&\text{genus}\\
\hline
1 & 94.697\\
2 & 17.1829\\
3 &  9.39386\\
4 & 7.3087\\
5 & 6.91425\\
6 & 6.86426\\
\hline
\end{array}$$
\begin{remark} 
\begin{enumerate}
\item This is implemented in Mathematica.
\item We choose $\kappa=0.99999$ and use Mathematica \verb$Minimize$ function to minimize
the genus under the constraints \eqref{eq:constraintG} and \eqref{eq:constraintLip}. This gives us a set of values for the parameters
$(T,R,\weight,\rho)$.
We compute again $C_{\mathcal G_i}$ and $C_{\Lip_i}(n,T,R,\weight,\rho)$
for this value of the parameters, this time using interval-arithmetic, and check that the constraints
\eqref{eq:constraintG} and \eqref{eq:constraintLip} are satisfied (for a slightly larger $\kappa<1$).
We obtain an interval-arithmetic proof of the claimed results.
\item For $n=6$, the remainder terms $\mathcal G_i^{\mathcal R}$ are of order $10^{-4}$ so further increasing $n$ will not improve significantly the results.
\end{enumerate}
\end{remark}

\subsection{Further results using the derivatives of $x(t)$}
\label{section:correction-derivative}
Observe that if $x(t)$ is a solution of an equation of the form $f(t,x)=0$ then by the mean value inequality
$$f\big(t,\sum_{k=0}^N x^{(k)}(0)\frac{t^k}{k!}\big)=O(|t|^{N+1}).$$
We choose some integer $N< n$ and change the definition of $x$ to
\[x_1(t,u)=\cv{x}_1+\sum_{k=1}^N x_{1,k}t^k+\lambda u_1\]
\[x_j(t,u)=\cv{x}_j+(-1)^{j+1} u_2+\lambda u_3+\sum_{k=1}^N x_{j,k} t^k
\quad\text{ for $j=2,3$}\]
where the coefficients $x_{j,k}=x_j^{(k)}(0)/k!$ can be computed numerically
using the algorithm presented in Section \ref{section:algorithm}.
This kills all terms $t^k u^0$ in $\mathcal G_i$ for $k\leq N$.
The functions $\mathcal G_i$ are estimated in the same way. Note that
$\mathcal K$ now depends on $t$, which is why we took care to estimate $\mathcal G_1$
in this more general setup.
Also, we need to change the estimates of $\|x_j\|_{\rho}$ in Equation \eqref{eq:estimate-xi} ; the constants $c_j$ now also depend on $T$.
This gives the following results, with $N=n-1$:
$$\begin{array}{|c|c|c|}
\hline
n&\text{genus}\\
\hline
2 & 16.2129\\
3 &  7.97941\\
4 &  5.16451\\
5 &  4.01048 \\
6 &  3.46739\\
7 & 3.28229\\
\hline
\end{array}$$
\subsection{Quadratic corrections}
\label{section:correction-quadratic}
To make further progress, we may try to kill the quadratic terms in $\mathcal G$, namely terms of the form $t^k u^{\alpha}$ with $k+|\alpha|=2$ (and $|\alpha|\geq 1$ since we explain in Section \ref{section:correction-derivative} how to kill the terms with $\alpha=0$ by using derivatives).
By Proposition \ref{prop:expandpq}, the first order expansion of $\wt{\mathfrak p}$ is
$$\wh{\mathfrak p}(t,u)=
2\pi x_3 -4\pi \ii t\log(2) x_1 x_2+O(t^2).$$
Using the definition of $x_1$ and $x_2$, the quadratic term in $2 \ii t x_1 x_2$ is
$$\tfrac{1}{\sqrt{2}}(\lambda^2+1)tu_1
+ (\lambda^{-1}-\lambda)tu_2
+(\lambda^2-1)tu_3.$$
Therefore, adding
$$\frac{\log(2)}{\sqrt{2}}(\lambda^2+1)tu_1
+\log(2)(\lambda^2-1)tu_3$$
to $x_3$ (and subtracting it to $x_2$ to preserve parity) kills the terms $tu_1$ and $tu_3$ in $\wt{\mathfrak p}$. Since these are even terms, we gain in $\mathcal G_2$ and we do not loose in $\mathcal G_1$ nor $\mathcal G_3$, at least at the quadratic order.
In other words, we replace the former definitions of $x_2$ and $x_3$ by
$$x_j(t,u)=\cv{x}_j
+(-1)^{j+1} \left(u_2
+\frac{\log(2)}{\sqrt{2}}(\lambda^2+1)tu_1
+\log(2)(\lambda^2-1)tu_3\right)
+\lambda u_3
+\sum_{k=1}^N x_{j,k}t^k.$$
The constants $c_j$ in Equation \eqref{eq:estimate-xi} as well as the constants $C_{j,k}$ in Section
\ref{section:remainder} must be updated to take into account the new terms.
This gives the following results, still using $N=n-1$ derivatives:
$$\begin{array}{|c|c|c|}
\hline
n&\text{genus}\\
\hline
6 & 3.04143\\
7 & 2.81835\\
8 & 2.65404\\
\hline
\end{array}.$$
\begin{remark} We cannot do the same for the term $tu_2$. Indeed, we cannot add
$\log(2)\lambda^{-1}t u_2$ because we cannot add negative powers of $\lambda$ to $x_3$.
We could subtract $\log(2)\lambda t u_2$ to make $\mathcal G_3$ smaller. But this actually makes
$\mathcal G_1$ larger because we loose the benefit of a cancellation between $a_{1,(0,1,0),\odd}^+$ and $b_{1,(0,1,0)}$
in the evaluation of $\|\mathcal G_1^{1,(0,1,0)}\|_{\rho}$, and this is not compensated by the gain in $\mathcal G_3$.
\end{remark}
\begin{remark}
It is shown in \cite{HH3}
 that there exists for $s \in (0, \tfrac{1}{4}]$  a real analytic family of Fuchsian DPW potentials $\tilde \eta_s$ which gives rise to all Lawson surfaces $\xi_{1,g}$, $g\geq1$ for  $s=\tfrac{1}{2g+2}$ this. Up to the
 reparametrization $s=\psi(t)$, the potentials $\tilde\eta_s$ coincide with 
$\eta_{t,x(t)}$ from Proposition \ref{prop:building}, whenever the power series expansion of $\tilde\eta_{t,x(t)}$ at $t=0$ converges. 
This means $\eta_{s=\psi(t)}$ is the analytic continuation of $\eta_{t,x(t)}$. However, in order to compute the area of $\xi_{1,2}$ explicitly it would be important to know whether the power series of $\eta_{t,x(t)}$ at $t=0$ converges for all $t$ with $\psi(t)\in[0,\tfrac{1}{6}]$ corresponding to $g\geq 2.$ Though we cannot prove this at the moment, we do conjecture that the convergence radius should cover the
Lawson surface $\xi_{1,2}$ and the Clifford torus as well. In fact, when plugging in $g=1$ into the Taylor series \eqref{eq:area-series} of order 21 of the area function at $t=0$, we obtain $2 \pi^2,$ the area of the Clifford torus, up to an error of $10^{-3}.$
\end{remark}

\section{The area of Lawson minimal surfaces \texorpdfstring{$\xi_{1,g}$}{xi\_\{1,g\}}} \label{sec:monotonicity}
The goal in this section is to approximate the area of Lawson surfaces $\xi_{1,g}$ of genus $\geq 3$ using the series \eqref{eq:area-series} with explicit bounds on the error.
The coefficients $\alpha_k$ of \eqref{eq:area-series} have been computed up to $k=21$, see Section \ref{section:implementation} and Appendix \ref{appendix:numalpha} for the numerical values.
We estimate the error by estimating the remaining coefficients $\alpha_k$
for $k\geq 23$ using the Cauchy estimate from complex analysis.
\label{section:area}
\subsection{Estimating $|\alpha_k|$}
We use the setup of Section \ref{section:correction-derivative} respectively
\ref{section:correction-quadratic} with $N$ derivatives.
\begin{proposition}
\label{prop:cauchy} There exists a constant $C_{\mathcal A}(N)$ such that for
$k>N$,
$$|\alpha_k|\leq \frac{C_{\mathcal A}}{(T')^k}.$$
\end{proposition}
The proof explains how to compute the constant $C_{\mathcal A}$.
\begin{proof}
By Proposition \ref{prop:building}, the area of the Lawson surface $\xi_{1,g}$ is equal to
$8\pi\mathcal A(s)$ with $s=\frac{1}{2g+2}$ and
$$\mathcal A(s)=1-\frac{1}{\sqrt{2\mathcal K(t)}}(x_2^0(t)-x_3^0(t))\quad\text{with}\quad t\sqrt{\mathcal K(t)}=s.$$
In the setup of Section \ref{section:correction-quadratic} (just remove the $\log(2)$ terms in the setup of
Section \ref{section:correction-derivative})
$$\mathcal A(s)=1-\frac{\sqrt{2}}{\sqrt{\mathcal K(t)}}\left(-u_2^0(t)-\frac{\log(2)}{\sqrt{2}} t u_1^0(t)
+\log(2) t u_3^0(t)+\sum_{k=1}^N x_{2,k}^0 t^k\right).$$
Consider the holomorphic function defined for $s\in D(0,T')$ by
\begin{eqnarray*}
h(s)&=&\mathcal A(s)-1 + \sqrt{2}\sum_{k=1}^N x_{2,k}^0 s^k\\
&=&\frac{\sqrt{2}}{\sqrt{\mathcal K(t)}}\left(u_2^0(t)+\frac{\log(2)}{\sqrt{2}} t u_1^0(t)
-\log(2) t u_3^0(t)+\sum_{k=1}^N x_{2,k}^0 \left(\sqrt{\mathcal K(t)}s^k-t^k\right)\right).\end{eqnarray*}
Observe that $h$ and $\mathcal A$ differ by a polynomial so
$h^{(k)}=\mathcal A^{(k)}$ for $k>N$.
For $s\in D(0,T')$, we have by Lemma \ref{lemma:diffeo},
$$\left|\sqrt{\mathcal K(t)}s^k-t^k\right|
=|t|^k\left|\mathcal K(t)^{(k+1)/2}-1\right|\leq T^k\left( (1+C_{\mathcal K})^{(k+1)/2}-1\right).$$
Hence the function $h$ is bounded by $\mathcal C_A$ in the disk $D(0,T')$, with
$$C_{\mathcal A}=\frac{\sqrt{2}}{\sqrt{1-C_{\mathcal K}}}\left(R_2+\frac{\log(2)}{\sqrt{2}}T R_1+\log(2)TR_3
+\sum_{k=1}^N |x_{2,k}^0| T^k \left( (1+C_{\mathcal K})^{(k+1)/2}-1\right)\right).
$$
By the Cauchy Estimate, we have for all $k\in\N$:
$$|h^{(k)}(0)|\leq C_{\mathcal A}\frac{k!}{(T')^k}.$$
Hence for $k>N$,
$$|\alpha_k|=\frac{1}{ k!}|\mathcal A^{(k)}(0)|=\frac{1}{ k!}|h^{(k)}(0)|\leq \frac{C_{\mathcal A}}{(T')^k}.$$
\end{proof}
\subsection{Area estimates for $g\geq 3$}
We use the series \eqref{eq:area-series} truncated to order $21$ to compute an approximate
value of $\text{Area}(\xi_{1,g})$ for $g\geq 3$. By Proposition \ref{prop:cauchy}, the error in the genus $g$ case is bounded by
$$8\pi\sum_{k=23\atop\text{$k$ odd}}^{\infty}|\alpha_k| s^k\leq
8\pi\sum_{k=23\atop\text{$k$ odd}}^{\infty}C_{\mathcal A}\left(\tfrac{s}{T'}\right)^k
=8\pi C_{\mathcal A}\frac{\left(\tfrac{s}{T'}\right)^{23}}{1-\left(\tfrac{s}{T'}\right)^2}
\quad\text{ with }\quad s=\frac{1}{2g+2}.$$
This gives the following results (where $T'$ and $C_{\mathcal A}$ are computed in the setup of
Section \ref{section:correction-quadratic} with $n=8$ and $N=7$)
$$\begin{array}{|c|c|c|}
\hline
\text{genus}&\text{approximate area}&\text{bound on error}\\
\hline
3& 22.82027709& 0.244537\\
4& 23.32191299&0.000512743\\
5&23.64134581&5.732114\; 10^{-6}\\
6&23.86347454&1.4302993 \;10^{-7}\\
7&24.02726927&6.096336 \;10^{-9}\\
8&24.15322275&3.847452\; 10^{-10}\\
9&24.25318196&3.2867174 \;10^{-11}\\
10&24.33449044&3.574938\; 10^{-12}\\
\hline
\end{array}$$
\begin{proposition}
\label{prop:monotonicity}
The area of $\xi_{1,g}$ is strictly increasing in $g$ for $g\geq 3$.
\end{proposition}
\begin{proof}
The previous results ensure that $\text{Area}(\xi_{1,g})$ is increasing for 
$3\leq g\leq 9$.
For $g\geq 9$ we have $s\leq T''=\frac{1}{20}$.
By Taylor formula, for $s\in(0,T'')$, there exists $\sigma\in (0,s)$ such that
$$\mathcal A'(s)=\sum_{k=0}^6\mathcal A^{(k+1)}(0)\frac{s^k}{k!}
+\mathcal A^{(8)}(\sigma)\frac{s^7}{7!}
=-\sum_{k=1\atop \text{$k$ odd}}^5
k\alpha_k s^{k-1}+\mathcal A^{(8)}(\sigma)\frac{s^7}{7!}.$$
Using the Cauchy Estimate we have
$$|\mathcal A^{(8)}(\sigma)|\leq  C_{\mathcal A}\frac{8!}{(T'-T'')^8}.$$
Since $\alpha_1=\log(2)$, $\alpha_3>0$ and $\alpha_5>0$,
$$\mathcal A'(s)\leq -\log(2)+7|\alpha_7| (T'')^6+8\frac{ C_{\mathcal A} (T'')^7}{(T'-T'')^8}\leq -0.668205.$$
The last inequality is obtained using interval arithmetic in the setup of Section \ref{section:correction-quadratic} with $n=8$ and $N=7$.
Hence $\mathcal A$ is decreasing on $[0,T'']$ so $\mbox{Area}(\xi_{1,g})$ is an increasing function of $g$ for $g\geq 9$.
\end{proof}
\subsection{Area estimate in the genus 2 case}
\begin{proposition}
\label{prop:genus2}
We have
$$\text{Area}(\xi_{1,1})<\text{Area}(\xi_{1,2})<\text{Area}(\xi_{1,3}).$$
\end{proposition}
\begin{proof}
Since $\xi_{1,1}$ is the Clifford torus, we have
$$\text{Area}(\xi_{1,1})=\mathcal W(\xi_{1,1})<\mathcal W(\xi_{1,2})=\text{Area}(\xi_{1,2})$$
by the solution of the Willmore conjecture \cite{MN}.
To prove the other inequality, we construct a simplicial approximation of $\xi_{1,2}$. 
The genus-$2$ Lawson surface $\xi_{1,2}$ is obtained using reflection and rotation of a fundamental piece which is the solution of the Plateau problem for a geodesic $4$-gon $\Gamma_2$, i.e., it has the least area among all surfaces with the same boundary. The aim is to find a triangulated surface with boundary $\Gamma_2$, i.e.,  each face is a geodesic triangle,  and show that its area is below the (coarse) lower bound $22.57<22.82027709-0.245$ obtained for the genus-$3$ surface $\xi_{1,3}.$

To compute the area we make use of the Gau\ss-Bonnet formula applied to a geodesic triangle $\Delta_{ABC}$  in $\S^3$ with vertices $A, B, C $ and angles $\alpha, \beta, \gamma$ which gives
$$\text{Area}(\Delta_{ABC}) = \int_{\Delta_{A,B,C}} K_{S^3} dA =  \alpha + \beta + \gamma - \pi.$$
The angle $\alpha$ at the vertex $A$ given by the two arc-length parametrized geodesics, $\gamma_{AB}$ and $\gamma_{AC}$, connecting $A$ to $B$ and $C$ respectively  is given by 
$$\alpha = \cos^{-1}\left (<\gamma_{AB}'(0), \gamma_{AC}'(0)>_{\R^4}\right).$$
Let $\S^3 \subset \mathbb C^2 \cong \mathbb R^4$ and consider the four vertices of the geodesic polygon $\Gamma_2$ given by
$$P_1 = (1,0), \quad P_2 = (\ii,0), \quad Q_1 = (0,1), \quad Q_2= (0, e^{\ii \pi/3}).$$
Let $\gamma_{AB}$ denote the edge connecting the vertices $A \in \{P_1,P_2, Q_1, Q_2\}$ and $B \in \{P_1,P_2, Q_1, Q_2\}.$ Then for $i,j \in \{1,2\}$ the four edges of the 4-gon $\Gamma$ are
$$\gamma_{P_i, Q_j} (s)= \sin(s) P_i + \cos(s)Q_j.$$
 \begin{figure}[h]
\centering 
 \includegraphics[height=0.3\textwidth]{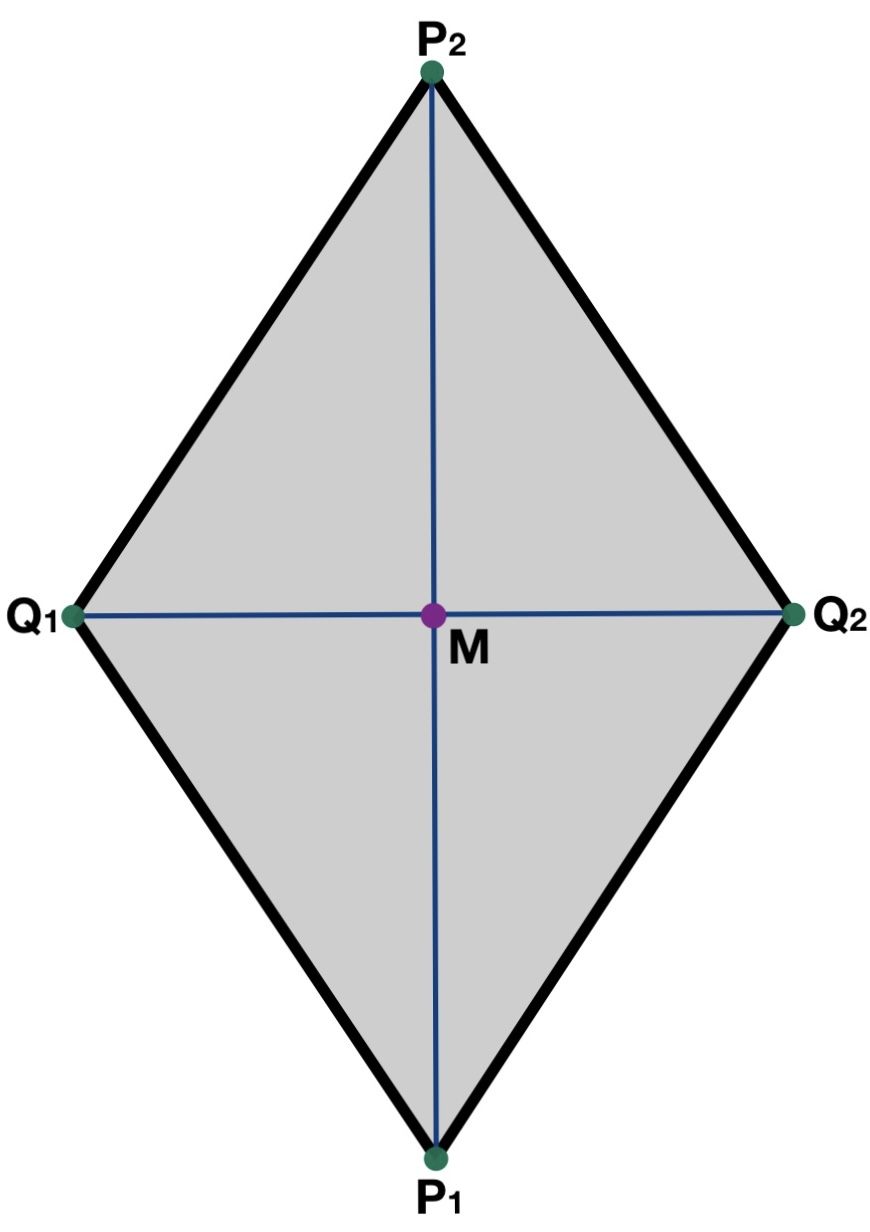}\hspace{1cm}
  \includegraphics[height=0.3\textwidth, angle=-0.9]{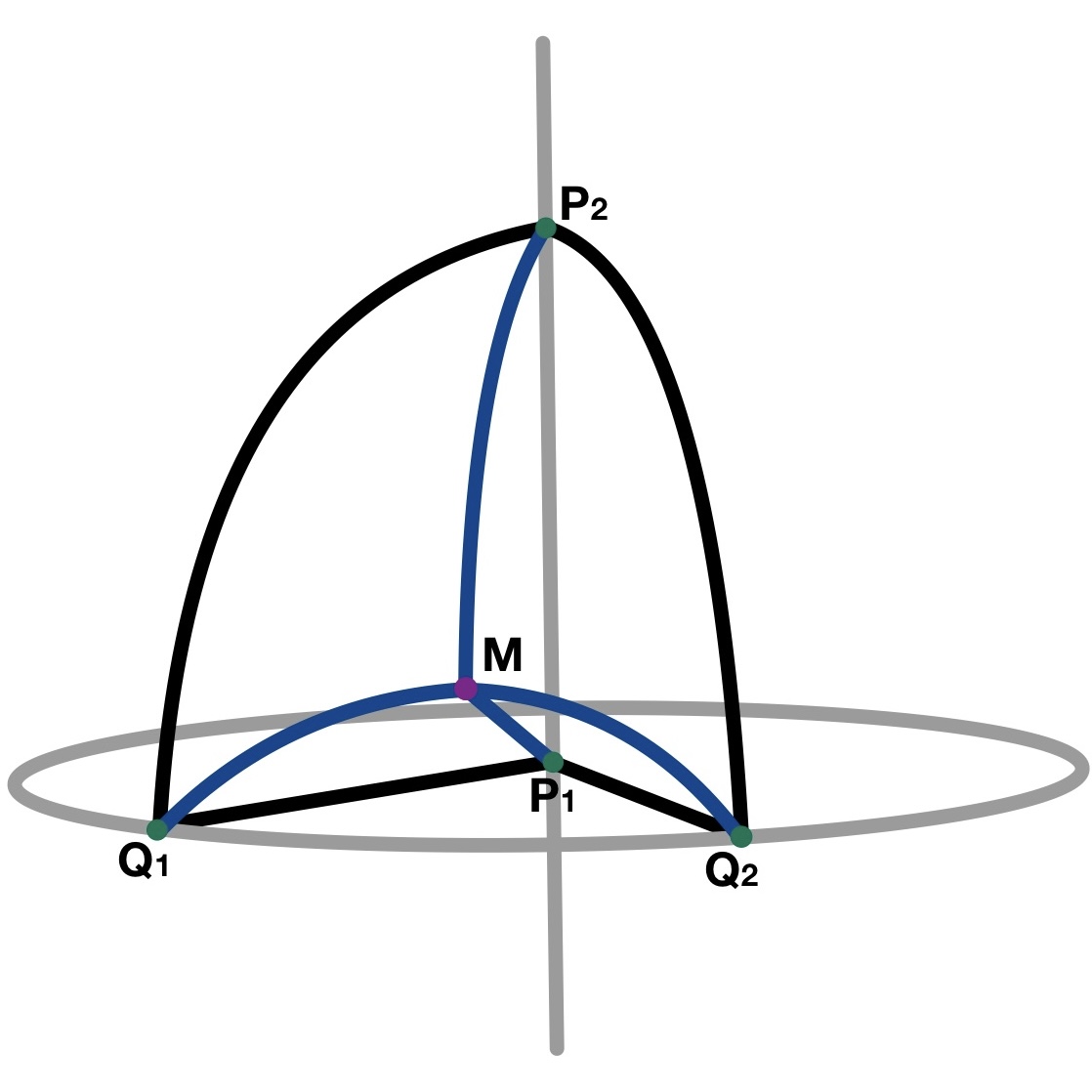}
  \caption{The left picture shows the domain of the triangulated surface. The labels of the vertices corresponds to their image in $S^3$. The right picture shows how the triangulated surface in $S^3$ stereographically projected to $\mathbb R^3$ look like.}
   \label{picture1st}
\end{figure}
For visualization purposes we consider the stereographic projection of $\S^3$ to $\mathbb R^3$ such that the $(x,y)-$plane and the unit 2-sphere are images of 
totally geodesic 2-spheres in $\S^3.$ The boundary curve has two reflectional symmetries, which the Plateau solution inherits as well, namely the angle bisection geodesic 2-sphere of the angles at the $P_i$, 
$$\sigma_P \colon \S^3 \longrightarrow \S^3, (z, w) \longmapsto (  z , e^{\ii\pi /3}  \bar w),$$ 
and the one of the angles at the $Q_i, $ 
$$\sigma_Q\colon \S^3 \longrightarrow \S^3, (z, w) \longmapsto (\ii \bar z , w).$$ 
Thus as a first refinement of the triangulation, we add a fifth vertex $M,$ which should be fixed under both $\sigma_P$ and $\sigma_Q$, as indicated in Figure \ref{picture1st}. In other words, we choose $M$ lying on the fix point set of both symmetries, which is a geodesic $\gamma$ parametrized by
$$M=\cos(s_0)(e^{\ii \pi /4}, 0) + \sin(s_0)(0, e^{\ii \pi /6}) \subset \mathbb C^2,$$
for some real parameter $s_0 \in [0, \pi/2)$
 \begin{figure}[h]
\centering 
 \includegraphics[height=0.35\textwidth]{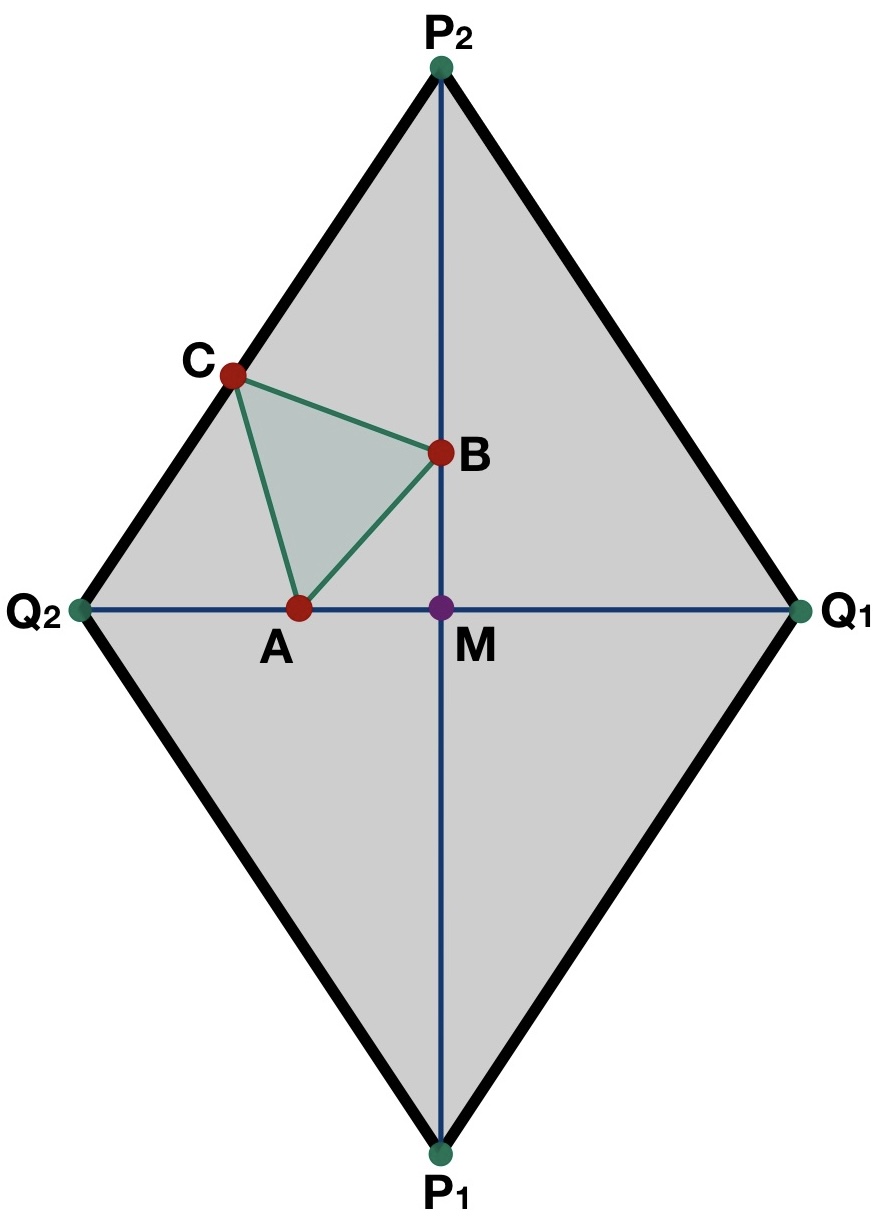} \hspace{0.2cm}
 \includegraphics[height=0.35\textwidth]{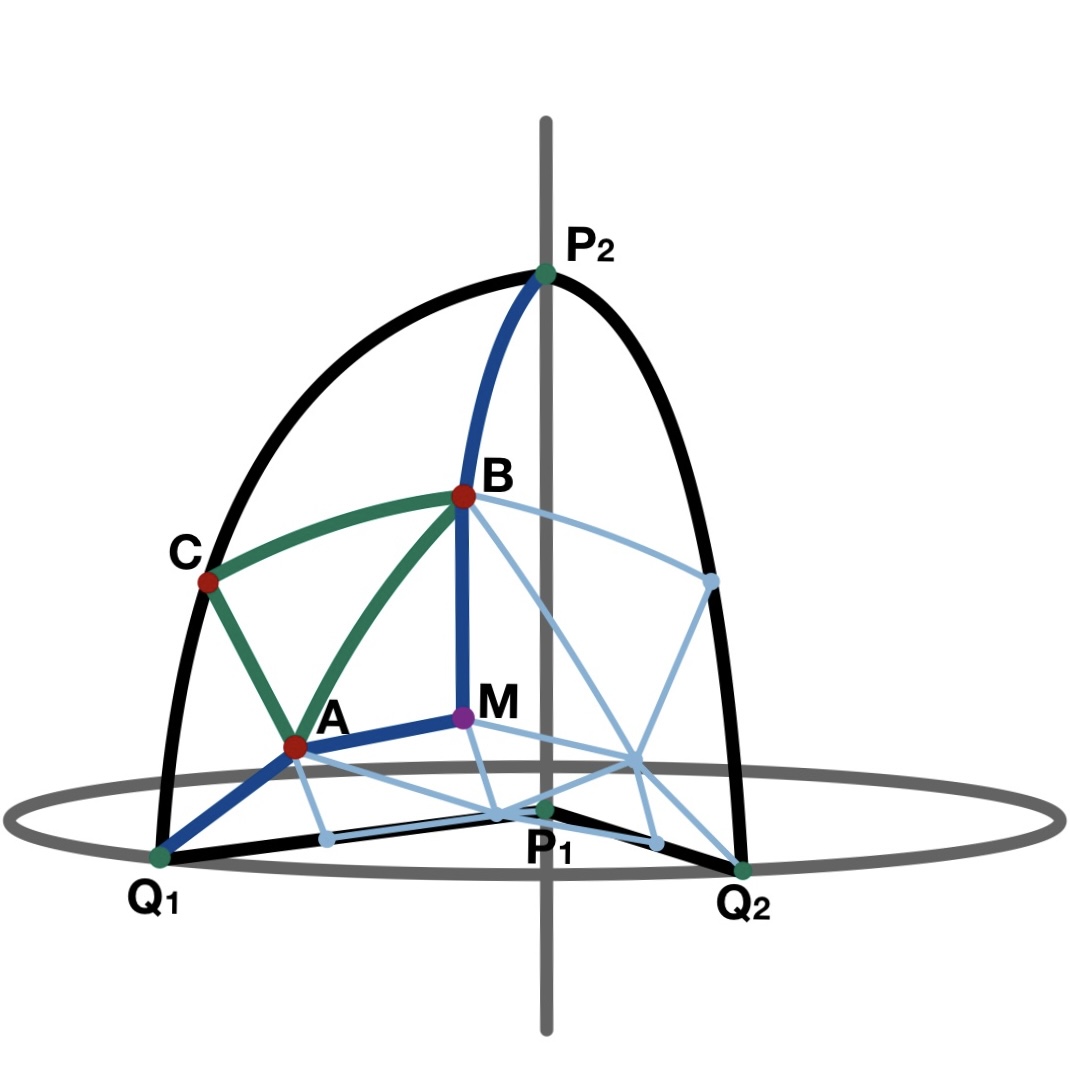}
 \caption{The black lines on the surface are the fixpoints of the two reflection symmetries of the plateau solution. The green line indicates the intersection line of the two geodesic spheres which are fixed under these reflexions. The image of fifth vertex is chosen on this green line such that the total area of the four (congruent) geodesic triangles are minimized.}\label{symmetry}
\end{figure}

To further refine, we add three points $A, B,C$ with 

$$A= \cos(s_1)(e^{\ii \pi/4}, 0) +  \sin(s_1)(0, e^{\ii t_1})$$

lying in the fix point set of $\sigma_Q$, 

$$B = \cos({s_2}) (0, e^{\ii \pi/6})+  \sin(s_2)(e^{\ii t_2}, 0) $$
lying in the fix point set of $\sigma_P$, and

\[C:= (0, \cos(s_3) , \sin(s_3), 0)\]
lying on the geodesic between $Q_1$ and $P_2 $ for some real parameters $s_1, s_2, s_3, t_1, t_2$ lying in the interval $[0, \pi/2).$  

The Lawson surface $\xi_{1,2}$ consists of $12$ congruent copies of the fundamental piece. Moreover, due to the choice of the vertices, every fundamental piece splits into four copies of congruent triangulations, see Figure \ref{symmetry}. Thus
 \begin{equation}
\begin{split}
\text{Area}  (\xi_{1, 2})  &\leq 48 \left[\text{Area} (\Delta_{Q_1AC}) + \text{Area} (\Delta_{AMB}) + \text{Area} (\Delta_{ABC}) + \text{Area} (\Delta_{BCP_2})\right] 
\end{split}
\end{equation}

Using the Mathematica \verb$Minimize$ function, we obtain the following values for the six parameters:

$$s_0 = 1.13641, s_1 = 1.27441, s_2 = 0.848594, s_3 = 1.06941, t_1 = 0.134219, t_2 = 1.4755.$$

This gives the following bound on the area
 \begin{equation}
\begin{split}
\text{Area}  (\xi_{1, 2})  &\leq 48 \left[\text{Area} (\Delta_{Q_1AC}) + \text{Area} (\Delta_{AMB}) + \text{Area} (\Delta_{ABC}) + \text{Area} (\Delta_{BCP_2})\right] \\
&< 22.45 < 22.57 \leq\text{Area}(\xi_{1,3}),
\end{split}
\end{equation}
concluding the proof of Proposition \ref{prop:genus2}.
\end{proof}
\appendix
\section{Euclidean division in the Banach algebra \texorpdfstring{$\mathcal W_{\rho}$}{\textbackslash{}cal W\_rho}}
\label{appendix:division}
In the Banach algebra $\mathcal W_{\rho}$, there is a euclidean division by polynomials with roots in the disk $\D_{\rho}$. This results has been used in Sections \ref{section:IFT} and \ref{section:quantitative-reformulation}.
\begin{proposition}
\label{Pro:decomposition}
Let $\mu_1,\cdots,\mu_n\in\D_{\rho}$. For any $u\in\mathcal W_{\rho}^{\geq 0}$, there exists a unique
pair $(q,r)\in\mathcal W_{\rho}^{\geq 0}\times \C[\lambda]$ such that
$$u=(\lambda-\mu_1)\cdots(\lambda-\mu_n) q + r\quad \text{ and }\quad \deg(r)<n.$$
Moreover, the quotient $q$ has norm bounded by
$$\|q\|_{\rho}\leq \frac{1}{(\rho-|\mu_1|)\cdots(\rho-|\mu_n|)}\|u\|_{\rho}.$$
\end{proposition}
\begin{proof}
It suffices to prove the case $n=1$ as the general case follows by induction.
Let $\mu=\mu_1\in\D_{\rho}$ and $u\in\mathcal W^{\geq 0}_{\rho}$.
Define $r=u(\mu)\in\C$ and
$$q(\lambda)=\frac{u(\lambda)-u(\mu)}{\lambda-\mu}.$$
We have
$$q(\lambda)=\sum_{k=0}^{\infty} u_k\frac{\lambda^k-\mu^k}{\lambda-\mu}=
\sum_{k=0}^{\infty}\sum_{j=0}^{k-1}u_k\lambda^j\mu^{k-1-j}$$
$$\|q\|_{\rho}\leq \sum_{k=0}^{\infty}\sum_{j=0}^{k-1}|u_k|\rho^j|\mu|^{k-1-j}
=\sum_{k=0}^{\infty}|u_k|\frac{\rho^k-|\mu|^k}{\rho-|\mu|}\leq
\frac{\|u\|_{\rho}}{\rho-|\mu|}.$$
\end{proof}
\section{Character Variety}\label{AppCV}
We are interested in the following  Betti moduli space for $\traceL\in(0,2)$
\[\mathcal M^\traceL_B:=\{(L_1,L_2,L_3)\in\mathrm{SL}(2,\C)^3\mid\;\tr(L_j)=0,\,-\tr(L_3L_2L_1)=\traceL\}/\mathrm{SL}(2,\C)\]
where $\mathrm{SL}(2,\C)$ acts by overall conjugation.
The space can be naturally identified with the moduli space of representations $\mathcal M^\traceL$ from the first fundamental group of the 4-punctured sphere into $\mathrm{SL}(2,\C).$
The local conjugacy classes of a representation $\rho \in \mathcal M^\traceL$ is determined by
\[\tr(\rho(\gamma_1))=\tr(\rho(\gamma_2))=\tr(\rho(\gamma_3))=0\quad\text{and}\quad  \tr(\rho(\gamma_4))=\traceL\]
via $\rho(\gamma_j)=L_j$ for $j=1,\dots,3.$ 
In fact, since $L_1^2=L_2^2=L_3^2=-\Id$ we have $(L_1L_2L_3)^{-1}=-L_3L_2L_1$, and $\tr(\rho(\gamma_4))=-\tr(L_3L_2L_1).$

Fixing the conjugation up to some diagonal freedom, we can assume  without loss of generality that
\begin{equation}\label{defL2}L_2=\left(\begin{smallmatrix} \ii&0\\0&-\ii\end{smallmatrix}\right).\end{equation}
We call an element $[\rho]\in\mathcal M^\traceL_B$ irreducible if it is induced by an irreducible representation (generated by $L_1,L_2,L_3$).
\begin{lemma}\label{lem:irr}
Every element of $\mathcal M^\traceL_B$ is irreducible. In particular, $L_1$ and $L_3$ cannot be simultaneously upper (or lower) triangular and neither of the two can be diagonal.
\end{lemma}
\begin{proof}
A direct computation shows that, with the choice of $L_2$, if $L_1$ is upper (respectively lower) triangular and  $L_3$ is upper (respectively lower) triangular, or if either of the two matrices are diagonal, then the trace of $L_4:=-L_3L_2L_1$ must be $0$, which is excluded by assumption.
\end{proof}

\begin{definition}
The three traces $x,y,z\in\C^3$ defined by
\[x=\tr(L_1L_2),\quad y=\tr(L_2L_3),\quad z=\tr(L_2L_4)\]
are called trace-coordinates of the representation $\rho=\rho(L_1,L_2,L_3)$.
\end{definition}
Clearly, $(x,y,z)$ are invariant under conjugation, and define functions on $\mathcal M_B^\traceL.$
As a direct consequence of Lemma \ref{lem:irr}, we only have to consider the following two cases:
\begin{enumerate}
\item[   Type I:] $L_1$ is either upper or lower triangular, then it has diagonal entries $\pm \ii$ or $\mp\ii$ and off-diagonal entries being either 1 and 0 or 0 and 1;
\item[   Type II:] $L_1$ is not triangular. Then up to conjugation with a diagonal matrix, there exists a  $ x\in\C$ with $x^2\neq4$ such that $L_1$ is given by
\begin{equation}\label{defL1}L_1=\begin{pmatrix} -\tfrac{\ii}{2}  x&1\\ \tfrac{1}{4} \left( x^2-4\right)&\frac{\ii   x}{2}\end{pmatrix}.\end{equation}
\end{enumerate}

\begin{lemma}\label{lem:xpm2}
With the notations above we have $x=\tr(L_1L_2).$ In particular, $[\rho(L_1,L_2,L_3)]$ is of type I if and only if $x^2=4.$
More specifically we have

\begin{enumerate}
\item If $x=2$, $L_1$ is upper triangular if and only 
if  $z=-y-\ii\,\traceL$ and lower triangular if and only if $z=-y+\ii\,\traceL$. 
\item If  $x=-2$, $L_1$ is upper triangular if and only 
if  $z=y-\ii\,\traceL$ and lower triangular if and only if $z=y+\ii\,\traceL$. 
\end{enumerate}
Thus every solution to $x^2= 4$ satisfies  $\traceL^2+(y\pm z)^2=0 $ and is given by some representation $\rho(L_1,L_2,L_3)$ of type I which is unique up to conjugation.
\end{lemma}
\begin{proof}
This follows by elementary computations in each of the individual cases.
\end{proof}

We aim to show that the trace coordinates $(x,y,z)$ of a representation
solve \begin{equation}\label{poly}
P^\traceL(x,y,z):=\traceL^2+\left(x^2+y^2+z^2\right)+x y z-4=0,
\end{equation}
and conversely, every solution 
 $(x,y,z)$ uniquely determines a representation up to conjugation.
 For $x=\pm2$, this statement is equivalent to Lemma \ref{lem:xpm2}.

For $L_1$ of type II we can assume without loss of generality
\begin{equation}\label{defL3} L_3:=
\begin{pmatrix}
 -\tfrac{\ii y}{2}  & a \\
b & \frac{\ii y}{2} \\
\end{pmatrix}\end{equation}
with $\det( L_3)-1=\frac{1}{4} \left(-4 a b+y^2-4\right)=0.$
Then $\tr(L_2 L_3)=y$
and 
\begin{equation*}
\begin{split} \traceL&=-\tr( L_3L_2L_1)=\frac{1}{4} \ii a x^2-\ii (a+b)\,\\
 z&=\tr(L_2L_4) = -\tr(L_2 L_3L_2L_1)=a-b-\tfrac{1}{4}a x^2-\tfrac{1}{2}xy.
\end{split}
\end{equation*}

We  directly compute
\begin{equation}\label{psep}P^{ \traceL}(x,y, z)=\frac{1}{4} \left(x^2-4\right) \left(4 a b-y^2+4\right) = 0.\end{equation}

Conversely, given $(x,y,z)$ satisfying $P^{\traceL} (x,y,z)  =0$ and $x^2 \neq 4,$ the triple $(L_1, L_2, L_3)$ given by \eqref{defL1}, \eqref{defL2} and \eqref{defL3} with $a$ and $b$ determined by $\det (L_3)=1$ and $\traceL=-\tr( L_3L_2L_1)=\frac{1}{4} \ii a x^2-\ii (a+b)$ is a representation with these trace coordinates.

Together with Lemma \ref{lem:xpm2}
we obtain the first part of the following theorem:

\begin{theorem} 
Let $\traceL>0$. Then, every representation  of the Betti moduli space $[\rho]\in\mathcal M^\traceL_B$ is irreducible.
Moreover, $\mathcal M^\traceL_B$  and the relative character variety
\[\mathcal X^\traceL:=\{(x,y,z)\in\C^3\mid P^\traceL(x,y,z)=0\}\]
are complex manifolds, which are biholomorphic to each other
via the trace coordinates.
\end{theorem}
\begin{proof}
It remains to show that the complex analytic spaces  $\mathcal M^\traceL_B$ and $\mathcal X^\traceL$ are smooth. Biholomorphicity then follows from both spaces being complex manifolds using the same coordinates. \\

By the implicit function theorem,
the character variety $\mathcal X^\traceL$ (for $\traceL\in(0,2)$) is smooth. 
The space $\mathcal M^\traceL_B$ has 
on the open subset $x^2\neq 4$
a natural smooth structure using the normal forms of $L_1$ and $L_3$ given in \eqref{defL1} and \eqref{defL3}.  For $x \rightarrow \pm2,$ the coordinate for $L_1$ extends smoothly to the upper triangular case. To glue in the $L_1$ being lower triangular case, we need to change the coordinates using a conjugation by a diagonal matrix, which normalizes the lower left entry to 1.  
\end{proof}
\subsection{Unitary representations} 
If $[(L_1, L_2, L_3)]\in\mathcal M^\traceL_B$ is unitary up to conjugation then the trace coordinates automatically satisfy $(x,y,z)\in[-2,2].$ But to obtain unitary representations this condition can be slightly relaxed.
\begin{theorem}\label{unitarizer}
Let $[(L_1,L_2,L_3)] \in\mathcal M^\traceL_B$. Then there exists an unitarizer $U$ such that $U^{-1} L_j U \in \SU(2) $ if and only if
\begin{equation}\label{unitarycond}(x,y,z)\in U^\traceL:=[-2,2]^2\times \R\,\cap\, \mathcal X^\traceL=(-2,2)^2\times \R\,\cap\, \mathcal X^\traceL.\end{equation}
The unitarizer $U$ depends smoothly on the trace coordinate $x$ and we have $U =\Id$ at $x = 0$.
Moreover, the set $U^\traceL$ of unitary points constitutes a  connected component of $\mathbb R^3\cap \mathcal X^\traceL.$
\end{theorem}
\begin{proof}
Since
$P^\traceL(\pm2,y,z)=\traceL^2+(y\pm z)^2,$ we have for $x=\pm2$ and $y\in\R$
that $z\notin\R.$ This gives that unitary representations only occur in type II.
Analogously, for $y=2$ and $x\in\R$ we have $z\notin\R.$
Hence, for $(x,y,z)$ satisfying \eqref{unitarycond}, we have
\[x^2<4 \quad\text{and}\quad y^2<4.\]

Therefore, conjugating the normal forms \eqref{defL1}, \eqref{defL2},  and \eqref{defL3} by the real positive matrix (referred to as unitarizer)
\[U:=\left(\begin{smallmatrix} 
 \frac{\sqrt{2}}{\sqrt[4]{4-x^2}} & 0 \\
 0 & \frac{\sqrt[4]{4-x^2}}{\sqrt{2}} \\
\end{smallmatrix}\right)\]
we have
\[U^{-1}L_1U \; \text{ and }\; U^{-1}L_2U\in\mathrm{SU}(2),\quad
\text{and}
\quad U^{-1}L_3U=
\begin{pmatrix}
 -\tfrac{\ii y}{2}  & \frac{1}{2} a \sqrt{4-x^2} \\
 \frac{2 b}{\sqrt{4-x^2}} & \frac{\ii y}{2} \\
\end{pmatrix}.\]
Since $a$ and $b$ are determined by
\[\det (L_3) = \frac{y^2}{4}-ab = 1\quad\text{ and }\quad z=a-b-\tfrac{1}{4}a x^2-\tfrac{1}{2}xy\]
we obtain that there exists two pairs $(a_\pm , b_\pm)$ of solutions given by
\[a_{\pm}=\frac{\pm2 \ii\traceL+x y+2 z}{4-x^2}\quad\text{and}\quad b_{\pm}=\frac{\pm 2 \ii\traceL -x y-2 z}{4}\]
which gives that 
$U^{-1}L_3U$ is unitary as well.

Finally, $U^\traceL$ is connected, as every element $(x,y,z)\in U^\traceL$ can be connected within $U^\traceL$ to either of the points $(0,0,\pm \sqrt{4-\traceL^2}),$ by sending $(x,y) \rightarrow 0$
and these two points can be connected within $U^\traceL$ via the curve $\gamma(r)=(\sin(r)\sqrt{4-\traceL^2}, 0, \cos(r)\sqrt{4-\traceL^2})$. It constitutes a whole connected component of real representations, because the boundary points $x\rightarrow \pm2$ and $y\rightarrow \pm2$ do not give rise to real representations.
\end{proof}

\section[Numerical evaluation of multiple polylogarithms]{Numerical evaluation of multiple polylogarithms}
\label{appendix:nummpl}
In the process of deriving the results in this paper, one needs to numerically evaluate many alternating multiple zeta values (more generally: cyclotomic multiple zeta values, or general multiple polylogarithms).  

Multiple polylogarithms are implemented in many computer algebra systems and calculators, for example \texttt{Maple}, \texttt{gp/pari}, \texttt{GiNaC}, to name a few.  The implementation used in \texttt{GiNaC} is described in \cite{numMPLWeinzierl}, more references can be found in \cite[\oldS7.48 \texttt{inifcns\_nstdsums.cpp} File Reference]{ginac}.  \texttt{GiNaC} does not guarantee accuracy, but computed digits can be checked by increasing precision.  (Implementation bugs were fixed at various points.)  The implementation in \texttt{gp/pari}	is documented in \cite[{\bf polylogmult(s, {z}, {t = 0})}]{gppari}, where accuracy is guaranteed to a certain number of bits when the algorithm converges.  (A warning is given there: the \texttt{gp/pari} algorithm for multiple polylogarithms might not converge even at moderate roots of unity, and raises an error.)  Finally, a \texttt{Maple} implementation is described in \cite{numMplFrellesvig}; it is not clear if this comprises the built-in routines \cite{maple}.  

Since these different implementations can be cross-checked, one can be essentially certain that the results are accurate.  Nevertheless, we would like to be self-contained, and give proven accuracy.  An approach to evaluating alternating MZV's is described in \cite{BBBL} using the H\"older convolution formula to express the result via \emph{geometrically convergent} multiple polylogarithms.
  This is essentially the path composition of iterated integrals; an extension of this approach to evaluate arbitrary cyclotomic multiple zeta values was described in a talk by Hirose \cite{hiroseCompMZV}.  We recall the details of this setup and give the necessary bounds to establish a disc containing the resulting values. \medskip

\paragraph{\bf Truncated MPL's} Introduce the \emph{truncated} multiple polylogarithm:
\[
\Li_{N;a_1,\ldots,a_d}(x_1,\ldots,x_d) \coloneqq \sum_{0 < n_1 < \cdots < n_d \leq N} \frac{x_1^{n_1} \cdots x_d^{n_d}}{n_1^{a_1} \cdots n_d^{a_d}} \,,
\]
where the indices \( a_1,\ldots,a_d \geq 1 \in \mathbb{Z} \).
Write \( \delta_j = \prod_{i=j}^d x_i \).  Note that
\[
\Li_{N;a_1,\ldots,a_d}(x_1,\ldots,x_d) = \sum_{0 < n_1 < \cdots < n_d \leq N} \frac{(\delta_1)^{n_1} (\delta_2)^{n_2-n_1} \cdots (\delta_{d-1})^{n_{d-1} - n_{d-2}} (\delta_d)^{n_d - n_{d-1}}}{n_1^{a_1} \cdots n_d^{a_d}} \,.
\]
In the region \( | \delta_j | \leq 1 \), \( j = 1, \ldots d \), and \( (a_d,x_d) \neq (1,1) \), the series is convergent as \( N \to \infty \), and the limit defines the multiple polylogarithm,
\[
\Li_{a_1,\ldots,a_d}(x_1,\ldots,x_d) = \lim_{N\to\infty} \Li_{N;a_1,\ldots,a_d}(x_1,\ldots,x_d) \,.
\]
For details see \cite[\oldS2.3, and Corollary 2.3.10]{zhaoBook}.

\paragraph{\bf Geometrically convergent MPL's} Fix \( 0 < \alpha < 1 \), and suppose \( |\delta_j| \leq \alpha \), for \( j = 1, \ldots, d \).  Then 
\[
\Li_{N;a_1,\ldots,a_d}(x_1,\ldots,x_d)
\]
converges absolutely, and is bounded by some multiple of the geometric series \( \sum_{n_d=1}^{\infty} \beta^{n_d} \), for any \( \alpha < \beta < 1 \).

\begin{proof}
	By the trivial estimate \( n_i \geq 1 \), and \( |\delta_i| \leq \alpha \), we have
	\begin{align*}
	\big| \Li_{N;a_1,\ldots,a_d}(x_1,\ldots,x_d) \big| \leq \sum_{0 < n_1 < \cdots < n_d \leq N} \underbrace{\alpha^{n_1} \, \alpha^{n_2-n_1}  \,\cdots  \,\alpha^{n_{d-1} - n_{d-2}}  \,\alpha^{n_d - n_{d-1}}}_{= \alpha^{n_d}}.
	\end{align*}
	Explicitly extracting the sum over \( n_d \), we can write
	\[
	= \sum_{n_d=1}^N \bigg( \sum_{0<n_1<\cdots<n_{d-1} < n_d} 1  \bigg) \alpha^{n_d} 
	< \sum_{n_d=1}^N n_d^{d-1} \alpha^{n_d} \,.
	\]
	This converges by comparison with the geometric series \( \sum_{n_d=1}^N \beta^{n_d} \).
\end{proof}

\paragraph{\bf Estimate of MPL tails}	Fix \( 0 < \alpha < 1 \), and suppose \( |\delta_j| \leq \alpha \), for \( j = 1, \ldots, d \).  Then following bound holds
\[
\Big| \Li_{a_1,\ldots,a_d}(x_1,\ldots,x_d) - \Li_{N;a_1,\ldots,a_d}(x_1,\ldots,x_d) \Big| \leq 
\alpha^N \cdot \sum_{i=1}^d  \frac{N^{i-1}}{N^{a_i + \cdots + a_d}} \cdot \Big( \frac{\alpha}{1-\alpha} \Big)^{d - (i-1)}
\]

\begin{proof}
	We have
	\begin{align*}
	& \Big| \Li_{a_1,\ldots,a_d}(x_1,\ldots,x_d) - \Li_{N;a_1,\ldots,a_d}(x_1,\ldots,x_d) \Big|
	\\
	& \leq \sum_{\substack{0 < n_1 < \cdots < n_d \\ n_d > N}} \frac{|\delta_1|^{n_1} |\delta_2|^{n_2-n_1} \cdots |\delta_d|^{n_d - n_{d-1}}}{n_1^{a_1} \cdots n_d^{a_d}}  \\
	&	 =  \sum_{i=1}^d  \sum_{\substack{0 < n_1 < \cdots < n_{i-1} \leq N \\ N < n_i < \cdots < n_d}} \frac{|\delta_1|^{n_1} |\delta_2|^{n_2-n_1} \cdots |\delta_d|^{n_d - n_{d-1}}}{n_1^{a_1} \cdots n_d^{a_d}} 
	\end{align*}
	
	For \( i = 1 \), using the trivial bounds \( n_i \geq N \) in the denominator, and \( |\delta_i|^{n_{i+1} - n_i} \leq \alpha^{n_{i+1} - n_i} \) in the numerator, since the exponents \( n_{i+1} - n_i \geq 1 \), we can write
	\begin{align*}
	& \sum_{\substack{N < n_1 < \cdots < n_d}} \frac{|\delta_1|^{n_1} |\delta_2|^{n_2-n_1} \cdots |\delta_d|^{n_d - n_{d-1}}}{n_1^{a_1} \cdots n_d^{a_d}} 
	\leq \frac{1}{N^{a_1 + \cdots + a_d}}\!\!\! \sum_{N < n_1 < \cdots < n_d} \!\!\!  \alpha^{n_1} \, \alpha^{n_2 - n_1} \, \cdots \, \alpha^{n_d - n_{d-1}}
	\end{align*}
	By the substitution \( n_i = N + \ell_1 + \cdots + \ell_i \), \( \ell_i \geq 1 \), this is simply
	\[
	= \frac{1}{N^{a_1 + \cdots + a_d}} \sum_{\ell_1,\ldots,\ell_d = 1}^{\infty} \alpha^{N+\ell_1} \alpha^{\ell_2} \cdots \alpha^{\ell_d}
	= \frac{\alpha^N}{N^{a_1 + \cdots + a_d}}  \cdot \Big(\frac{\alpha}{1 - \alpha}\Big)^{d}\,.
	\]
	
	For \( 2 \leq i \leq d \), write \( |\delta_i|^{n_i - n_{i-1}} = |\delta_i|^{N-n_{i-1}} \cdot |\delta_i|^{n_i - N} \), and break the numerator at this point, giving
	\begin{align*}
	& \sum_{\substack{0 < n_1 < \cdots < n_{i-1} \leq N \\ N < n_i < \cdots < n_d}} \frac{|\delta_1|^{n_1} |\delta_2|^{n_2-n_1} \cdots |\delta_d|^{n_d - n_{d-1}}}{n_1^{a_1} \cdots n_d^{a_d}}  \\
	& = 
	\begin{aligned}[t] 
	\sum_{\substack{0 < n_1 < \cdots < n_{i-1} \leq N}} & \frac{|\delta_1|^{n_1} |\delta_2|^{n_2-n_1} \cdots |\delta_{i-1}|^{n_{i-1} - n_{i-2}} |\delta_i|^{N - n_{i-1}}}{n_1^{a_1} \cdots n_{i-1}^{a_{i-1}}} \\[-1ex]
	& \cdot \sum_{\substack{N < n_i < \cdots < n_{d}}} \frac{|\delta_i|^{n_i - N} |\delta_{i+1}|^{n_{i+1} - n_i} \cdots |\delta_d|^{n_d - n_{d-1}}}{n_i^{a_i} \cdots n_d^{a_d}}\,.  \end{aligned} 
	\end{align*}
	The same argument as above (for the \( N < n_1 < \cdots < n_d \) tail when \( i = 1 \)), gives the following bound for the second factor
	\[
	\sum_{\substack{N < n_i < \cdots < n_{d}}} \frac{|\delta_i|^{n_i - N} |\delta_{i+1}|^{n_{i+1} - n_i} \cdots |\delta_d|^{n_d - n_{d-1}}}{n_i^{a_i} \cdots n_d^{a_d}}
	< \frac{1}{N^{a_i + \cdots + a_d}} \cdot \Big( \frac{\alpha}{1-\alpha} \Big)^{d - (i-1)} \,.
	\]
	On the other hand, apply the trivial bound \( n_1,\ldots,n_i \geq 1 \) and \( |\delta_i| \leq \alpha \) (since the exponents are positive), to the first factor, giving
	\begin{align*}
	& \sum_{\substack{0 < n_1 < \cdots < n_{i-1} \leq N}}  \frac{|\delta_1|^{n_1} |\delta_2|^{n_2-n_1} \cdots |\delta_{i-1}|^{n_{i-1} - n_{i-2}} |\delta_i|^{N - n_{i-1}}}{n_1^{a_1} \cdots n_{i-1}^{a_{i-1}}} 
	\\
	& \leq \sum_{\substack{0 < n_1 < \cdots < n_{i-1} \leq N}} \alpha^{N} \quad \leq \quad N^{i-1} \alpha^N\,.
	\end{align*}
	Hence we obtain
	\[
	\sum_{\substack{0 < n_1 < \cdots < n_{i-1} \leq N \\ N < n_i < \cdots < n_d}} \frac{|\delta_1|^{n_1} |\delta_2|^{n_2-n_1} \cdots |\delta_d|^{n_d - n_{d-1}}}{n_1^{a_1} \cdots n_d^{a_d}} \leq  \frac{N^{i-1} \alpha^N }{N^{a_i + \cdots + a_d}} \cdot \Big( \frac{\alpha}{1-\alpha} \Big)^{d - (i-1)}\,.
	\]
	Overall, we then have
	\[
	\Big| \Li_{a_1,\ldots,a_d}(x_1,\ldots,x_d) - \Li_{N;a_1,\ldots,a_d}(x_1,\ldots,x_d) \Big| \leq \alpha^N \cdot \sum_{i=1}^d  \frac{N^{i-1}}{N^{a_i + \cdots + a_d}} \cdot \Big( \frac{\alpha}{1-\alpha} \Big)^{d - (i-1)} \,.
	\]
	This majorisation goes to 0 as \( N \to \infty \), as the exponential \( \alpha^N \), \( 0 < \alpha < 1 \), dominates any power of \( N \).
\end{proof}

\paragraph{\bf Efficient evaluation of truncated MPL's} 

In order to evaluate truncated multiple polylogarithms efficiently, we proceed as follows.  (This is a well-known approach.)  Write
\begin{align*}
\Li_{N;a_1,\ldots,a_d}(x_1,\ldots,x_d) &= \sum_{n_d = 1}^N \Li_{n_d; a_1,\ldots, a_{d-1}}(x_1,\ldots,x_{d-1}) \cdot \frac{x_d^{n_d}}{n_d^{a_d}}  \\
& = \Li_{N-1;a_1,\ldots,a_d}(x_1,\ldots,x_d) + \Li_{N; a_1,\ldots, a_{d-1}}(x_1,\ldots,x_{d-1}) \cdot \frac{x_d^{N}}{N^{a_d}} \,,
\end{align*}
Then initialize the vector \( (v_r)_{r=0}^d \) as follows, where by convention \( \Li_{0;\emptyset}(\emptyset) = 1 \) is the only sensible value to assign.
\begin{align*}
\tag*{\texttt{Initialise:}} \\
(v_r)_{r=0}^d = \big( \Li_{0;\emptyset}(\emptyset), \Li_{0;a_1}(x_1), \ldots, \Li_{0;a_1,\ldots,a_d}(x_1,\ldots,x_d) \big) = ( 1, 0, \ldots, 0) \,.
\end{align*}
For \( i = 1, \ldots, N \) iterate as follows, computing the new values \( (v_r')_{r=0}^d \) and replacing the old vector \( (v_r)_{r=0}^d \) with the new values.  After iteration \( i \), \((v_r)_{r=0}^i \) contains \( (\Li_{i,a_1,\ldots,a_r}(x_1\ldots,x_r))_{r=0}^d \).
\begin{align*} 
\tag*{\texttt{Iterate \( 1 \leq i \leq N \):}} \\
(v_r)_{r=0}^d \leftarrow (v_r')_{r=0}^d = \big( 1, v_1 + v_0 \cdot \frac{x_1^i}{i^{n_1}},  v_2 + v_1 \cdot \frac{x_2^i}{i^{n_2}}, \ldots,  v_d + v_{d-1} \cdot \frac{x_d^i}{i^{n_d}},  \Big) \,.
\end{align*}
Obtain the final result as follows.
\[
\tag*{\texttt{Return:}}
v_d = \Li_{N;a_1,\ldots,a_d}(x_1,\ldots,x_d) \,.
\]

This computes the value of \( \Li_{N;a_1,\ldots,a_d}(x_1,\ldots,x_d) \) in \( O(Nd) \) steps, rather than the na\"ive \( O(N^d) \) obtained from iterating over all indices \( 0 < n_1 < \cdots < n_d \leq N \) directly.

\paragraph{\bf Reduction to geometrically convergent MPL's}

The final task is to reduce any multiple polylogarithm \( \Li_{a_1,\ldots,a_d}(x_1,\ldots,x_d) \) to an expression in such geometrically convergent MPL's whose values can be computed efficiently, and whose tails can be bounded explicitly.  By disc arithmetic, one can carry the explicit error bounds forward to obtain a value for \( \Li_{a_1,\ldots,a_d}(x_1,\ldots,x_d) \) with proven error bounds.

Recall that any multiple polylogarithm can be written as an iterated integral (c.f. Equation \eqref{eqn:mzvtoint}, for the special case where \( x_i = \pm 1 \); more generally \cite[Theorem 2.2]{goncharovMPL}), as follows.
\[
\Li_{a_1,\ldots,a_d}(x_1,\ldots,x_d) = (-1)^d \int_0^1 \eta_{\delta_1^{-1}} \eta_0^{a_1-1}  \eta_{\delta_2^{-1}} \eta_0^{a_2-1}  \cdots \eta_{\delta_d^{-1}} \eta_0^{a_d-1} \,,
\]
where
\[
\eta_z(t) = \frac{\mathrm{d}t}{t - z} \,, \quad \delta_j = \prod\nolimits_{i=j}^d x_i \text{ (as before)}\,.
\]
(Note that \( |\delta_j| \leq 1 \) is part of the condition implying the MPL is convergent.)  The viewpoint to take now is that we can compute \( \alpha \) (the parameter telling us the rate of geometric convergence of the given MPL) as
\begin{equation}\label{eqn:mpl:geom}
\alpha = \frac{1}{\min_j |\delta_j^{-1} | } = \frac{|1-0|}{\min_j |\delta_j^{-1}-0 |} \,. 
\end{equation}
Since the iterated integral
\[
\int_{y_0}^{y_{n+1}} \eta_{y_1} \cdots \eta_{y_n} 
\]
is invariant under affine transformations \( y_i \mapsto \lambda y_i + \mu \), it can always be rescaled to be an integral from \( y_0 = 0 \) to \( y_n = 1 \), via \( y_i \mapsto \frac{y_i - y_0}{y_n - y_0} \).  Consider the circle centered at the lower bound of integration \( p = y_0 \), with circumference containing the upper bound of integration \( q = y_{n+1} \).  This has radius \( r = |q-p| \).  The number of radii from \( q \) to the nearest pole \( y_i \) ($\neq p$) gives \( \alpha^{-1} \); this quantity is affine invariant and reduces to \eqref{eqn:mpl:geom} when \( y_0 = 0, y_{n+1} = 1 \). 
\begin{center}
	\begin{tikzpicture}
	
	\coordinate (P) at (0,0);
	\node[below left] at (P) {$p$};
	
	\coordinate (Q) at ({1*cos(23)},{1*sin(23)});
	\coordinate (Qb) at ({0.95*cos(23)},{0.95*sin(23)});
	\node[above right] at (Q) {$q$};
	
	\draw (P) circle (1);
	
	\draw[dotted, blue] (P) circle (1.8);
	\coordinate (D) at ({1.8*cos(104)},{1.8*sin(104)});
	\coordinate (Db) at ({1.75*cos(104)},{1.75*sin(104)});

	\coordinate (F) at ({1.75*cos(-87)},{1.75*sin(-87)});
	\coordinate (G) at ({1*cos(-84)},{1*sin(-84)});
	\node[label={[align=center]below:{\small convergence}}] at (G){};
	
	\node[label={[align=left]below:{\small convergence \\[-0.5ex] at rate $\alpha$}}] at (F){};
	
	\draw[->, >=latex] (P) -- (Qb);
	\path (P) -- (Q) coordinate[pos=0.5] (R);
	\node[above] at (R) {$r$};

	\draw[->, >=latex, blue, dashed] (P) -- (Db);
	\path (P) -- (D) coordinate[pos=0.8] (E);
	\node[right] at (E) {$r\alpha^{-1}$};
	
	\fill[blue] (P) circle (2pt);
	\fill (Q) circle (2pt);
	
	\fill[blue] (D) circle (2pt);
	\node[above] at (D) {$y_i$ \rlap{(nearest)}};
	
	\end{tikzpicture}
\end{center}

The goal is therefore to re-express the integral \( \int_{p}^q \) (as a sum of products of other integrals) in such a way as to obtain geometric convergence (with \( \alpha \) as small as feasible) everywhere.  We have some relations for iterated integrals to assist us, firstly:
\begin{align}
\tag{Path composition} 
\int_{p}^q \alpha_1 \cdots \alpha_n &= \sum_{i=0}^n \int_{p}^r \alpha_1 \cdots \alpha_i \cdot \int_{r}^q \alpha_{i+1} \cdots \alpha_d  \,.
\end{align}

If \( q \) is not a singularity of any of the differential forms, then the nearest singularity \( \delta_j^{-1} \) to \( q \) is some distance \( \eps > 0\).  The region \( \{ x \colon |x-p| \geq 1 \} \cap \{ x \colon |x-q| \geq \eps \} \) is some positive distance  \( d \) away from the radius \( p \to q \) (as both are compact).  By decomposing the path \( p \to q \) at equally spaced points \( s_1,\ldots,s_k \) (say), we can replace the circle of radius \( |q-p| \) centered at \( p \) with circles of arbitrarily small radius \( \frac{1}{k} \) centered at \( p, s_1, \ldots, s_k \) along the path \( p \to q \).  For each of these integrals, the convergence rate is then \( < \frac{d}{k} \), which can be made arbitrarily small.

\begin{center}
	\begin{tikzpicture}[scale=2]
	
	\coordinate (zero) at (0,0);
	\node[below left] at (zero) {$p$};
	
	\coordinate (pone) at (1,0);
	\node[below right] at (pone) {$q$};

	\coordinate (r1) at (0.2,0);
	\coordinate (r2) at (0.4,0);
	\coordinate (r3) at (0.6,0);
	\coordinate (rk) at (0.8,0);
	

	\draw[densely dotted, red] (zero) circle (0.2);
	\draw[densely dotted, red] (pone) circle (0.2);
	\draw[densely dotted, red] (r1) circle (0.2);
	\draw[densely dotted, red] (rk) circle (0.2);
	\draw[densely dotted, red] (0.4,0) circle (0.2);
	\draw[densely dotted, red] (0.6,0) circle (0.2);
	
	\path[name path=circleA] (0,0) circle (1);
	\path[name path=circleB] (1,0) circle (0.3);
	
	\path[name intersections={of=circleA and circleB, by={A,B}}];
	
	\draw (A) to[anticlockwise arc centered at=zero] (B) to[anticlockwise arc centered at=pone] (A);
	\draw[densely dotted, gray] (A) to[clockwise arc centered at=zero] (B) to[clockwise arc centered at=pone] (A);
	
	\fill[blue] (0.95,0.3)  node[above right, black] {\scriptsize $\delta_j^{-1}$}   circle (1pt);
	
	\draw[->, >=latex] (zero) to[out=15, in=165] (pone);
	
	\draw[->, >=latex, red] (zero) to[] (r1);
	\draw[->, >=latex, red] (r1) to[] (r2);
	\draw[dotted, red] (r2) to[] (r3);
	\draw[->, >=latex, red] (r3) to[] (rk);
	\draw[->, >=latex, red] (rk) to[] (pone);
	
	\fill[blue] (zero) circle (1pt);
	\fill[black] (pone) circle (1pt);
	
	\fill[black] (r1) circle (0.5pt);
	\fill[black] (r2) circle (0.5pt);
	\fill[black] (r3) circle (0.5pt);
	\fill[black] (rk) circle (0.5pt);
	
	\end{tikzpicture}
\end{center}

However, if \( q \) is a singularity of some of the differential forms, consider still the nearest singularity \( \delta_j^{-1} \) to \( q \), which is a distance \( \eps > 0 \) away.  (If no singularity exists, take \( \eps = \frac{|q-p|}{2} \).)  Take point \( s \) on the radius \( p\to q \), distance \( \eps \) from \( q \), and apply path decomposition.  The integral \( \int_p^s \) is handled as before, to obtain arbitrarily fast geometric convergence.  To handle \( \int_s^q \), apply path reversal:
\begin{align}
\tag{Path reversal} 
\int_{s}^q \alpha_1 \cdots \alpha_n &= (-1)^n \int_{q}^s \alpha_n \cdots \alpha_1 \,.
\end{align}

\begin{center}
	\begin{tikzpicture}[scale=2]
	
	\coordinate (zero) at (0,0);
	\node[below left] at (zero) {$p$};
	
	\coordinate (pone) at (1,0);
	\node[below right] at (pone) {$q$};
	
	\coordinate (rk) at (0.7,0);
	\node[below right] at (rk) {$s$};
	

	\draw[densely dotted, red] (zero) circle (0.7);
	
	\path[name path=circleA] (0,0) circle (1);
	\path[name path=circleB] (1,0) circle (0.3);
	
	\path[name intersections={of=circleA and circleB, by={A,B}}];
	
	\draw (A) to[anticlockwise arc centered at=zero] (B) to[anticlockwise arc centered at=pone] (A);
	\draw[densely dotted, gray] (A) to[clockwise arc centered at=zero] (B) to[clockwise arc centered at=pone] (A);

	\draw[densely dotted, red] (pone) circle (0.3);
	
	\fill[blue] (0.95,0.3)  node[above right, black] {\scriptsize $\delta_j^{-1}$}   circle (1pt);
	
	\draw[->, >=latex] (zero) to[out=15, in=165] (pone);
	
	\draw[->, >=latex, red] (zero) to[] (rk);
	\draw[->, >=latex, red] (pone) to[] (rk);
	
	\fill[blue] (zero) circle (1pt);
	\fill[blue] (pone) circle (1pt);
	\fill[black] (0.7,0) circle (1pt);

	\end{tikzpicture}
\end{center}

Then one can apply path decomposition to \( \int_q^s \) to obtain arbitrarily fast geometric convergence.

\paragraph{\bf Case: Alternating multiple zeta values:} 

Alternating multiple zeta values are multiple polylogarithms at \( x_i = \pm 1 \).  The differential forms therefore have singularities are \( 0, \pm 1 \).  To speed up convergence, express the path \( 0 \to 1 \) as the concatenation of \( 0 \to \frac{1}{2} \) with the inverse of \( 1 \to \frac{1}{2} \).

\begin{center}
	\begin{tikzpicture}[scale=2]
	
	\coordinate (zero) at (0,0);
	\coordinate (zerou) at (0.03,0.03);
	\node[below left] at (zero) {$0$};
	
	\coordinate (pone) at (1,0);
	\coordinate (poneb) at (0.95,0);
	\coordinate (poneu) at (0.97,0.03);
	\node[below right] at (pone) {$1$};

	\coordinate (half) at (0.5,0);
	\coordinate (halfb) at (0.45,0);
	\coordinate (halfc) at (0.55,0);
	\node[below, xshift=2mm] at (half) {\small $\tfrac{1}{2}$};
	
	\coordinate (mone) at (-1,0);
	\node[below left] at (mone) {$-1$};
	
	\draw[] (zero) circle (1);

	\draw[densely dotted, red] (zero) circle (0.5);
	\draw[densely dotted, red] (pone) circle (0.5);
	
	\draw[->, >=latex] (zerou) to[in=165, out=15] (poneu);
	
	\draw[->, >=latex, red] (zero) -- (halfb);
	\draw[->, >=latex, red] (pone) -- (halfc);
	
	\fill[blue] (zero) circle (1pt);
	\fill[blue] (pone) circle (1pt);
	\fill[blue] (mone) circle (1pt);
	
	\fill (half) circle (1pt);
	
	\end{tikzpicture}
\end{center}

The integrals \( \int_0^{1/2} \) have geometric convergence rate \( \frac{1}{2} \), as both (other) singularities \( \pm 1 \) are distance 1 from 0.  The integrals \( \int_{1}^{1/2} \) have geometric convergence rate \( \frac{1}{2} \) as the nearest (other) singularity to 1 is 0, with distance 1.  

We therefore express any alternating MZV as a sum of multiple polylogarithms which converge geometrically, at rate \( \alpha = \frac{1}{2} \).

\paragraph{\bf Case: \( \Omega \)-values:} The original definition of \( \Omega_{i_1,\ldots,i_n} \) expresses the result as a sum of \( 4^n \) integrals \( \int_0^1 \) over differential forms with singularities at \( p_k = \exp(2 \pi \mathrm{i} (2k-1) / 8)\), \( k = 1, \ldots, 4 \).

To speed up convergence, decompose the path \( 0 \to 1 \) as the concatenation of \( 0 \to \frac{1}{3} \), with \( \frac{1}{3} \to \frac{2}{3} \) then the inverse of \( 1 \to \frac{2}{3} \).  The singularity at \( p_1 = \exp(2 \pi \mathrm{i}/8) \) is the nearest one to each start points of integral \( 0, 1/3 \) or \( 1 \); it has distance \( 1, \frac{1}{3} \sqrt{10-3\sqrt{2}} = 0.799817\ldots, \sqrt{2-\sqrt{2}} = 0.765367\ldots\) respectively.  The slowest rate of geometric convergence we obtain is then \( \alpha = \frac{1}{3} / \sqrt{2-\sqrt{2}} = 0.435521\ldots \).

\begin{center}
	\begin{tikzpicture}[scale=2]
	
	\coordinate (zero) at (0,0);
	\coordinate (zerou) at (0.03,0.03);
	\node[below left] at (zero) {$0$};
	
	\coordinate (pone) at (1,0);
	\coordinate (poneb) at (0.95,0);
	\coordinate (poneu) at (0.97,0.03);
	\node[below right] at (pone) {$1$};

	\coordinate (third) at (0.33,0);
	\coordinate (thirdb) at (0.30,0);
	\coordinate (thirdc) at (0.36,0);
	\node[below, xshift=2mm] at (third) {\small $_{\tfrac{1}{3}}$};

	\coordinate (twothird) at (0.66,0);
	\coordinate (twothirdb) at (0.63,0);
	\coordinate (twothirdc) at (0.69,0);
	\node[below, xshift=2mm] at (twothird) {\small $_{\tfrac{2}{3}}$};

	\draw[] (zero) circle (1);

	\draw[densely dotted, red] (zero) circle (0.33);
	\draw[densely dotted, red] (third) circle (0.33);
	\draw[densely dotted, red] (pone) circle (0.33);
	
	\draw[->, >=latex] (zerou) to[in=165, out=15] (poneu);
	
	\draw[->, >=latex, red] (zero) -- (thirdb);
	\draw[->, >=latex, red] (thirdc) -- (twothirdb);
	\draw[->, >=latex, red] (pone) -- (twothirdc);
	
	\fill (pone) circle (1pt);
	\fill (zero) circle (1pt);
	
	\coordinate (p1) at ({1*cos(45)},{1*sin(45)});
	\coordinate (p2) at ({1*cos(135)},{1*sin(135)});
	\coordinate (p3) at ({1*cos(-135)},{1*sin(-135)});
	\coordinate (p4) at ({1*cos(-45)},{1*sin(-45)});
	\node[below left] at (p1) {$p_1$};
	\node[below right] at (p2) {$p_2$};
	\node[above right] at (p3) {$p_3$};
	\node[above left] at (p4) {$p_4$};
	
	\fill[blue] (p1) circle (1pt);
	\fill[blue] (p2) circle (1pt);
	\fill[blue] (p3) circle (1pt);
	\fill[blue] (p4) circle (1pt);
	
	\fill (third) circle (1pt);
	\fill (twothird) circle (1pt);
	
	\end{tikzpicture}
\end{center}

We therefore express any \( \Omega \)-value as a sum of multiple polylogarithms which converge geometrically, at rate \( \alpha = \frac{44}{100} \).

\paragraph{\bf Case: \( \Omega \)-values (refined):}  Recall the alternative formula for \( \Omega_{i_1,\ldots,i_n} \) in \eqref{eqn:omega:3n} expresses the result as a sum of \( 3^n \) integrals \( \int_{-\mathrm{i}}^{-1+\sqrt{2}} \) over differential forms with singularities at \( 0, \pm 1 \).

In this case, to speed up convergence, we deform the straight-line path \( -\mathrm{i} \to -1+\sqrt{2} \) (via homotopy invariance) to pass through \( r_1 = \frac{3}{10} - \frac{9}{16} \mathrm{i} \) and \( r_2 = \frac{1}{2} -  \frac{5}{16} \mathrm{i} \) (chosen rational for convenience), and express it as the concatenation of \( -\mathrm{i} \to r_1  \), with the reverse of \( r_2 \to r_1 \), and then \( r_2 \to -1 + \sqrt{2} \).

The nearest singularity to \( -\mathrm{i} \) is at 0, so \( \int_{-\mathrm{i}}^{r_1} \) has \( \alpha = \frac{|- \mathrm{i} - r_1|}{|-\mathrm{i} - 0|} = 0.542984\ldots \).  For \( r_2 \) nearest singularity is at 0 or 1, so \( \int_{r_2}^{r_1} \) has  \( \alpha = \frac{|r_2 - r_1|}{|r_2 - 0|} = 0.542984\ldots \) and \( \int_{r_2}^{-1+\sqrt{2}} \)  has \( \alpha = \frac{|r_2 - (-1 + \sqrt{2})|}{|p_2 - 0|} = 0.549606\ldots \). 

\begin{center}
	\begin{tikzpicture}[scale=2]
	
	\coordinate (start) at (0,-1);
	\coordinate (zerou) at (0.03,0.03);
	\node[below left] at (start) {$-\mathrm{i}$};
	
	\coordinate (end) at (0.4142,0);
	\coordinate (endb) at (0.4139,-0.03);
	\node[above] at (end) {\scriptsize $-1{+}\sqrt{2}$};

	\coordinate (r1) at (0.3,-0.56);
	\coordinate (r2) at (0.5,-0.3125);
	\node[below, xshift=1mm, yshift=0mm] at (r1) {\scriptsize $r_1$};
	\node[below, xshift=1mm] at (r2) {\scriptsize $r_2$};
	
	\draw[->, >=latex] (start) to[in=-135,out=90] (endb);
	
	\draw[->, >=latex, red] (start) -- (r1);
	\draw[->, >=latex, red] (r2) -- (r1);
	\draw[->, >=latex, red] (r2) -- (endb);
	
	\draw[] (-2,0) -- (2,0);
	
	\draw[densely dotted, red] (start) circle (0.53);
	\draw[densely dotted, red] (r2) circle (0.320);
	\draw[densely dotted, red] (r2) circle (0.324);
	
	\fill (start) circle (1pt);
	\fill (end) circle (1pt);
	
	\coordinate (p1) at (-1,0);
	\coordinate (p2) at (0,0);
	\coordinate (p3) at (1,0);
	\node[below left] at (p1) {$-1$};
	\node[below right] at (p2) {$0$};
	\node[above right] at (p3) {$1$};
	
	\fill[blue] (p1) circle (1pt);
	\fill[blue] (p2) circle (1pt);
	\fill[blue] (p3) circle (1pt);
	
	\fill (r1) circle (1pt);
	\fill (r2) circle (1pt);
	
	\end{tikzpicture}
\end{center}

We therefore express any \( \Omega \)-value as a sum of multiple polylogarithms which converge geometrically, at rate \( \alpha = \frac{55}{100} \).  Although this is slower than previously, we only have to expand \( \Omega_{i_1,\ldots,i_n} \)  as \( 3^n \) integrals initially, which is a significant saving.

\section{Higher order Taylor expansions}
\subsection{Third order derivatives in the CMC case}
\label{appendix-order3}
Using Mathematica, we can compute the third order derivatives of the parameters for arbitrary angle $\varphi$.
This gives the following formulas for the coefficient $\mathcal W_3$ in the expansion of the Willmore energy \eqref{eq:Willmore-series} and for the coefficient $H_3$ of $\frac{1}{(2g+2)^3}$ in the expansion of the mean curvature of $f_{g,\varphi}$:

\begin{align*}
&\mathcal W_3=
-\frac{\ii}{\pi^3}\sin^4(\varphi)\left(2\cos(2\varphi)+1\right)\Omega_{2,1}(1)^3
+\frac{\ii}{2\pi ^3} \sin^2(\varphi )\left(3 \cos (2 \varphi )+3 \cos (4 \varphi
   )+4\right)\Omega_{2,1}(1)^2\Omega_{3,1}(\ii)
\\&
-\frac{\ii}{\pi^3}\cos^4(\varphi)\left(2\cos(2\varphi)-1\right)\Omega_{3,1}(\ii)^3
-\frac{\ii}{2\pi ^3} \cos ^2(\varphi ) \left(-3 \cos (2 \varphi )+3 \cos (4 \varphi
   )+4\right)\Omega_{2,1}(1)\Omega_{3,1}(\ii)^2
\\&
-\frac{1}{4\pi^2}\sin^2(2\varphi)\Omega_{2,1}(1)
\left(
3\Omega_{3,3,3}(1)
-2\Omega_{2,1,1}(\ii)
+2\Omega_{3,3,2}(\ii)
\right)
\\&
-\frac{1}{4\pi^2}\sin^2(2\varphi)\Omega_{3,1}(\ii)
\left(
3\Omega_{2,2,2}(\ii)
-2\Omega_{3,1,1}(1)
+2\Omega_{2,2,3}(1)
\right)
\\&
+\frac{1}{\pi^2}\sin^2(\varphi)(\cos(2\varphi)-2)\Omega_{2,1}(1)\Omega_{3,1,1}(1)
-\frac{1}{\pi^2}\sin^4(\varphi)\Omega_{2,1}(1)\Omega_{2,2,3}(1)
\\&
-\frac{1}{\pi^2}\cos^2(\varphi)(\cos(2\varphi)+2)\Omega_{3,1}(\ii)\Omega_{2,1,1}(\ii)
-\frac{1}{\pi^2}\cos^4(\varphi)\Omega_{3,1}(\ii)\Omega_{3,3,2}(\ii)
\\&
+\frac{\ii}{4\pi} \sin^2(2\varphi)
\left(
\Omega_{2,1,3,3}(1)
-\Omega_{3,1,2,3}(1)
+\Omega_{3,3,2,1}(1)
+\Omega_{2,1,3,2}(\ii)
-\Omega_{2,2,3,1}(\ii)
-\Omega_{3,1,2,2}(\ii)
\right)
\\&
+\frac{\ii}{\pi}\sin^4(\varphi)\Omega_{2,2,2,1}(1)
+\frac{3\ii}{\pi} \sin^2(\varphi)\Omega_{2,1,1,1}(1)
-\frac{\ii}{\pi} \cos^4(\varphi)\Omega_{3,3,3,1}(\ii)
-\frac{3\ii}{\pi}\cos^2(\varphi)\Omega_{3,1,1,1}(\ii)
\end{align*}
\begin{align*}
&H_3=
-\frac{2\ii}{\pi^3}{\sin^2(\varphi)\sin(4\varphi)\Omega_{2,1}(1)^3}
+\frac{8\ii}{\pi^3}\sin(\varphi)(3\cos(2\varphi)-2)\cos^3(\varphi)\Omega_{2,1}(1)\Omega_{3,1}(\ii)^2
\\&
+\frac{2\ii}{\pi^3}\cos^2(\varphi)\sin(4\varphi)\Omega_{3,1}(\ii)^3
-\frac{8\ii}{\pi^3}\sin^3(\varphi)(2+3\cos(2\varphi))\cos(\varphi)\Omega_{2,1}(1)^2\Omega_{3,1}(\ii)
\\&
+\frac{4}{\pi^2}\sin(\varphi)\cos^3(\varphi)
\left(
2\Omega_{3,1}(\ii)\Omega_{3,1,1}(1)
-2\Omega_{3,1}(\ii)\Omega_{2,2,3}(1)
+\Omega_{3,1}(\ii)\Omega_{3,3,2}(\ii)
-3\Omega_{2,1}(1)\Omega_{3,3,3}(1)
\right)
\\&
+\frac{4}{\pi^2}\sin^3(\varphi)\cos(\varphi)
\left(
2\Omega_{2,1}(1)\Omega_{3,3,2}(\ii)
-2\Omega_{2,1}(1)\Omega_{2,1,1}(\ii)
-\Omega_{2,1}(1)\Omega_{2,2,3}(1)
+3\Omega_{3,1}(\ii)\Omega_{2,2,2}(\ii)
\right)
\\&
+\frac{1}{\pi^2}\sin(4\varphi)
\left(\Omega_{3,1}(\ii)\Omega_{2,1,1}(\ii)
+\Omega_{2,1}(1)\Omega_{3,1,1}(1)
\right)
+\frac{2\ii}{\pi}\sin(2\varphi)\left(
\Omega_{2,1,1,1}(1)
+\Omega_{3,1,1,1}(\ii)
\right)
\\&
+\frac{4i}{\pi}\sin(\varphi)\cos^3(\varphi)
\left(
\Omega_{2,1,3,3}(1)
-\Omega_{3,1,2,3}(1)
+\Omega_{3,3,2,1}(1)
+\Omega_{3,3,3,1}(\ii)
\right)
\\&
+\frac{4\ii}{\pi}\sin^3(\varphi)\cos(\varphi)
\left(
-\Omega_{2,1,3,2}(\ii)
+\Omega_{2,2,2,1}(1)
+\Omega_{2,2,3,1}(\ii)
+\Omega_{3,1,2,2}(\ii)
\right)
\end{align*}

\subsection{Computation of \texorpdfstring{$\alpha_k$}{alpha\_k} for \texorpdfstring{$k>3$}{k>3}}
\label{appendix:numalpha}
When specializing to the Lawson surfaces with $\varphi = \tfrac{\pi}{4}$ we have been able to compute numerically the coefficient $\alpha_k$ in the area expansion \eqref{eq:area-series} up to $\alpha_{21}$:
\begin{align*}
\alpha_1&\simeq
0.693147180559945309417232121458176568075500134360255254120680\\
\alpha_3 & \simeq
2.704628032109087142149410863400762479221219157766122484032610\\
\alpha_5 & \simeq
3.699626994497618439893380135471044617736329548309105157162310\\
\alpha_7 & \simeq
-53.1688000602634657601186493744463143722221041377109549606883\\
\alpha_9 & \simeq
-459.565676371488633633252895256096561995526272030689845199417\\
\alpha_{11} & \simeq
-260.931729774858246058852756835445016841900749580577223718493\\
\alpha_{13}&  \simeq   
26311.75666632241667824049728000376568318761694887921531627959\\
\alpha_{15} & \simeq
219897.7526067197482348266274038050133501624360107896585815548\\
\alpha_{17} & \simeq   
-204390.987496916879876223326569020676825058179523091704555104\\
\alpha_{19} & \simeq   
-19346782.5372543220622302604976526258798242712500787552866514\\
\alpha_{21} & \simeq   
-148960589.720279268862574700035701683223669243796252922710520\\
\end{align*}
Using the alternating MZV Data Mine \cite{DM} we can also give analytic formulas (in terms of a set of algebra generators) for the coefficients $\alpha_k$ up to \( \alpha_{11} \), in particular:

{\allowdisplaybreaks
	\begin{align*}
	&\alpha_5 \begin{aligned}[t]
	& = -8 \zeta(1,1,\overline{3})+\tfrac{121}{16}\zeta (5)+\tfrac{2 \pi ^2}{3} \zeta (3) -21 \zeta (3) \log ^2(2)  \\
	& = 16\Li_5\!\big(\tfrac{1}{2}\big)+16 \Li_4\!\big(\tfrac{1}{2}\big) \log (2) - \tfrac{11}{16} \zeta (5) - 14 \zeta (3) \log ^2(2)-\tfrac{8}{3} \zeta(2) \log ^3(2)+\tfrac{8 }{15} \log ^5(2) 
	\end{aligned} \\[1ex]
	&\alpha_7 \begin{aligned}[t] 
	& = \begin{aligned}[t]
	& -256 \zeta (1,1,1,1,\overline{3})
	+\tfrac{1392}{17} \zeta (1,1,\overline{5})
	+\tfrac{720}{17} \zeta (1,3,\overline{3})
	+128 \log ^2(2) \zeta (1,1,\overline{3})
	\\ &
	+28 \zeta (3) \zeta (1,\overline{3}) 
	{} +\tfrac{296921}{1088}\zeta (7)
	-\tfrac{418 \pi ^2}{51} \zeta (5)
	-\tfrac{473 \pi ^4}{765} \zeta (3)
	-\tfrac{109}{2} \zeta (5) \log ^2(2)
	\\ &
	+\tfrac{280}{3} \zeta (3) \log ^4(2)
	-\tfrac{32 \pi ^2}{3} \zeta (3) \log ^2(2)
	-112 \zeta (3)^2 \log (2) \end{aligned} 
	 \end{aligned}
	\end{align*}
}

\end{document}